\documentclass[reqno]{amsart}
\usepackage[utf8]{inputenc}

\title{Global well-posedness of the stochastic Abelian-Higgs equations in two dimensions}
\author{Bjoern Bringmann}
\address{Bjoern Bringmann, School of Mathematics, Institute for Advanced Study, Princeton, NJ 08540 \& Department of Mathematics, Princeton University, Princeton, NJ 08544}
\email{bjoern@ias.edu}
\author{Sky Cao}
\address{Sky Cao, Department of Mathematics, Massachusetts Institute of Technology, Cambridge, MA 02139}
\email{skycao@mit.edu}

\usepackage{SYM}
\usepackage{SYM-diagrams}
\usepackage{hyperref}
\usepackage{mhequ}

\newcommand{\covd}{\mbf{D}}
\newcommand{\spacetime}{J}
\newcommand{\leray}{\mrm{P}_\perp}
\newcommand{\qkernel}{Q}
\newcommand{\gaugerenorm}{\frac{1}{8\pi}}
\newcommand{\resgauge}[1]{R_{\mrm{gauge, #1}}}
\newcommand{\solutionmap}{\mbb{S}}
\newcommand{\spc}{\mrm{spc}}
\newcommand{\resnorm}{\mrm{cshe}}
\newcommand{\moll}{\chi}
\newcommand{\massp}{p^{\mrm{m}}}
\newcommand{\eucmoll}{\chi^{\mrm{euc}}}
\newcommand{\eucheat}{p^{\mrm{euc}}}
\newcommand{\Cgauge}{\mathsf{C}_{\mrm{g}}}

\newcommand{\power}{p}
\newcommand{\st}{\mrm{st}}
\newcommand{\mDuh}{\Duh_{\mrm{m}}}

\begin{document}

\begin{abstract}
We prove the global well-posedness of the stochastic Abelian-Higgs equations in two dimensions. The proof is based on a new covariant approach, which consists of two parts: First, we introduce covariant stochastic objects. The covariant stochastic objects and their multi-linear interactions are controlled using covariant heat kernel estimates. Second, we control nonlinear remainders using a covariant monotonicity formula, which is inspired by earlier work of Hamilton.
\end{abstract}

\maketitle

\tableofcontents

\section{Introduction}

In the last decade, there has been tremendous interest in singular stochastic partial differential equations (SPDEs). The local well-posedness of many singular SPDEs can now be shown using general methods, such as regularity structures \cite{BCCH21,BHZ19,CH16,H14}, para-controlled calculus \cite{GIP15}, renormalization-group methods \cite{K16,D21}, or the diagram-free approach of \cite{LOT21,LOTT21,OSSW18,OW19}. In particular, the local well-posedness has been established for the following models\footnote{For simplicity, we omit all renormalization terms below.} with additive space-time white noise: 
\begin{enumerate}[label=(\roman*)]
\item The parabolic $\Phi^4_3$-model \cite{CC18,H14} which, for $\phi\colon [0,\infty) \times \T^3 \rightarrow \R$, is given by
\begin{equs}\label{intro:eq-Phi43}
\partial_t \phi = \Delta \phi - \phi^3 + \xi.
\end{equs}
\item The stochastic Navier-Stokes equations \cite{DPD02,ZZ15} in dimension $d=2$ and $d=3$ which, for $u\colon [0,\infty) \times \T^d \rightarrow \R^3$, is given by 
\begin{equs}\label{intro:eq-SNS}
\partial_t u &= \Delta u + \leray \big( u \cdot \nabla u \big) + \leray \xi.
\end{equs}
In the above, $\leray$ denotes the Leray-projection, which is defined in \eqref{prelim:eq-leray} below.
\item The stochastic Yang-Mills-Higgs equations \cite{BC23,CCHS22,CCHS22+,S21} in dimension $d=2$ and $d=3$ which, for $A\colon [0,\infty) \times \T^d \rightarrow \frkg^d$ and $\phi\colon [0,\infty) \times \T^d \rightarrow V$, are given by 
\begin{equation}\label{intro:eq-SYMH}
\begin{cases}
\begin{aligned}
\partial_t A &= - \covd_A^\ast F_A - 
B (\covd_A \phi, \phi) + \xi, \\ 
\partial_t \phi &= - \covd_A^\ast \covd_A \phi  - |\phi|^2 \phi + \zeta.     
\end{aligned}
\end{cases}
\end{equation}
In the above, $\frkg$ is a Lie algebra, $V$ is a vector space, and $B\colon V \times V \rightarrow \frkg$ is a certain bilinear map, see e.g. \cite[(1.5)]{CCHS22+}. Furthermore, $\covd_A$ is the covariant derivative and $F_A$ is the curvature tensor.
\end{enumerate}

In contrast to local well-posedness, global well-posedness of singular SPDEs cannot be established using general methods, and instead relies heavily on the structure of the equation. For the parabolic $\Phi^4_3$-model, global well-posedness was first shown in \cite{HM18} using discrete approximations of \eqref{intro:eq-Phi43} and Bourgain's globalization argument \cite{B94}. Due to the use of Bourgain's globalization argument, the proof in \cite{HM18} relies on information about the $\Phi^4_3$-measure \cite{GJ73}, which concerns its discrete approximations, rigorous construction, and moments. In stochastic quantization \cite{PW81}, singular SPDEs such as \eqref{intro:eq-Phi43} are meant to be used to construct Gibbs measures (or Euclidean quantum field theories). From a stochastic quantization perspective, it is therefore of significant interest to prove the global well-posedness of \eqref{intro:eq-Phi43} directly from the evolution equation. Without using the $\Phi^4_3$-measure, the global well-posedness of \eqref{intro:eq-Phi43} was obtained in \cite{MW17,MW2020,GH21}. The arguments in \cite{MW17,MW2020,GH21} rely heavily on the damping effect from the $-\phi^3$-term in \eqref{intro:eq-Phi43}, which makes the solution come down from infinity (see e.g. \cite[Section 1.4]{MW17} or Remark \ref{decay:rem-ODE}). In particular, it is shown that the solution of \eqref{intro:eq-Phi43} obeys uniform-in-time estimates.  \\ 

For the stochastic Navier-Stokes equation in dimension $d=2$, global well-posedness can be proven by using the invariant Gibbs measure \cite{DPD02}. The Gibbs measure corresponding to \eqref{intro:eq-SNS} is given by spatial white noise, which can be constructed using elementary methods. Despite the simplicity of the Gibbs measure corresponding to \eqref{intro:eq-SNS}, it is still of interest to prove the global well-posedness of \eqref{intro:eq-SNS} directly from the evolution equation, since such an argument can likely be used for a more general class of stochastic forcing terms.  In \cite{HR23}, the authors prove the global well-posedness of \eqref{intro:eq-SNS} in dimension $d=2$ with stochastic forcing of the form $\xi+\zeta$, where $\xi$ is a space-time white noise and $\zeta$ is a perturbation which is slightly smoother than $\xi$. Since the nonlinearity in \eqref{intro:eq-SNS} has no damping effect, the proof of global well-posedness is then more involved than for the parabolic $\Phi^4_2$ and $\Phi^4_3$-models. The argument in \cite{HR23} relies on energy estimates, a dynamic high-low decomposition, and techniques from para-controlled calculus. In addition to global well-posedness, the argument also yields a double exponential growth estimate of the solution. The $d=3$ case has been studied using quite different techniques in \cite{HZZ2022, HZZ2023a, HZZ2023b},  where the authors show global existence but non-uniqueness. \\ 

In the setting of the stochastic Yang-Mills-Higgs equation \eqref{intro:eq-SYMH}, global well-posedness is an important open problem. In a recent breakthrough \cite{CS23}, Chevyrev and Shen proved the global well-posedness of the pure stochastic Yang-Mills equations, i.e., 
\begin{equation}
\partial_t A = - \covd_A^\ast F_A + \xi,
\end{equation}
in two dimensions, as well as the invariance of the 2D pure Yang-Mills measure under the dynamics. Their approach is to show that the lattice dynamics converge as the mesh size tends to zero to the continuum dynamics, which was previously constructed in \cite{CCHS22}. They apply Bourgain's globalization argument to the lattice dynamics, and a key input they use is the fact that the 2D pure Yang-Mills lattice measure is exactly solvable. This enables them to obtain the moment estimates needed for Bourgain's argument. As mentioned in \cite[Section 1.1]{CS23}, a key barrier in extending their work to more general models such as 2D or 3D Yang-Mills-Higgs is the fact that the analogous moment estimates are still open for these other models\footnote{In addition, neither the 2D Yang-Mills-Higgs, the 3D pure Yang-Mills, nor the 3D Yang-Mills-Higgs continuum measures have been constructed.}. \\

In this work, we consider the stochastic Yang-Mills-Higgs equations for the Abelian gauge group $G=U(1)$, which are often called the stochastic Abelian-Higgs equations. Using the gauge symmetry of the Abelian-Higgs equations, we can\footnote{For more details regarding the gauge symmetry of the Abelian-Higgs equations, we refer the reader to Subsection \ref{section:introduction-setting} below.} impose the Coulomb gauge condition $\partial_j A^j=0$. 
On the torus $\T^2$, the Coulomb gauge condition does not determine the gauge of $A$ entirely, but the remaining group of gauge transformations is discrete and finite-dimensional, see e.g. \eqref{intro:eq-group-action-Zd} below. 
In the Coulomb gauge, the stochastic Abelian-Higgs equations can be written as\footnote{As in \eqref{intro:eq-Phi43}, \eqref{intro:eq-SNS}, and \eqref{intro:eq-SYMH}, we omit the renormalization terms in \eqref{intro:eq-SAH}. For a more detailed formulation of the stochastic Abelian-Higgs equations, which includes the renormalization terms, we refer the reader to Subsection \ref{section:introduction-setting}.}
\begin{equation}\tag{SAH}\label{intro:eq-SAH}
\begin{cases}
\begin{aligned}
\ptl_t A &= \Delta A  -\leray \Im(\bar{\phi} \covd_A \phi) + \leray \xi, \\
\ptl_t \phi &= \covd_A^j \covd_{A, j}\phi  -|\phi|^2 \phi + \zeta. \\
\end{aligned}
\end{cases}
\end{equation}
The Gibbs measure corresponding to \eqref{intro:eq-SAH} for $d = 2, 3$ has been constructed in \cite{BFS-I,BFS-II,BFS-III,King-I,King-II}. Additionally, moment estimates for $d=2$ have been proven in the recent work \cite{CC2022}. Thus, it may be possible to apply Bourgain's globalization argument (either in the continuum or on the lattice) in order to show global well-posedness of \eqref{intro:eq-SAH}. On the other hand, in the spirit of stochastic quantization, it is interesting to show global well-posedness without using any knowledge of the continuum or lattice measures. This is the main result of the present article, which we state informally as follows (see Theorem \ref{intro:thm-abelian-higgs} for the precise version).

\begin{theorem}[Informal version]\label{intro:thm-informal}
The stochastic Abelian-Higgs equations in two dimensions is globally well-posed. Furthermore, the solution obeys gauge invariant, uniform-in-time bounds. 
\end{theorem}

The main novelty of our article is a covariant approach, which utilizes the geometric nature of \eqref{intro:eq-SAH}. In our covariant approach, all estimates of the scalar field $\phi$ will be based on the covariant Laplacian $\covd_A^j \covd_{A,j}$ instead of the standard Laplacian $\partial^j \partial_j$. In Section \ref{section:cshe}, we first control covariant stochastic objects using estimates for the covariant heat-kernel. In Sections \ref{section:Abelian-Higgs} and \ref{section:decay}, we then control nonlinear remainders using covariant monotonicity formulas. It seems likely that our covariant approach may not only be useful for the stochastic Abelian-Higgs equations \eqref{intro:eq-SAH}, but may also be useful for other stochastic geometric evolution equations. \\

We believe that Theorem \ref{intro:thm-informal} is the main step towards a PDE-construction of the Abelian-Higgs measure in two dimensions which, as mentioned above, has previously been studied in \cite{BFS-I,BFS-II,BFS-III,King-I,King-II}. However, in order to combine our gauge invariant, uniform-in-time bounds (see Theorem \ref{intro:thm-abelian-higgs}.\ref{intro:item-AH-2}) with a Krylov-Bogoliubov argument, one first needs to
construct a Markov process on the orbit space determined by the discrete group of gauge transformations from \eqref{intro:eq-group-action-Zd} below. Since the construction of such a Markov process would be quite separate from our PDE-methods, we do not pursue it in this article. See \cite{CCHS22, CCHS22+} for such a construction for 2D and 3D Yang--Mills--Higgs.

\subsection{Setting}\label{section:introduction-setting}

In the following, we restrict ourselves to the spatial dimension $d=2$. We recall that \eqref{intro:eq-SAH} contains the cubic nonlinearity $-|\phi|^2 \phi$. Instead of treating only the cubic nonlinearity, however, this article covers all power-type nonlinearities $-|\phi|^{q-1} \phi$, where $q\geq 3$ is an odd integer. We believe that this makes it easier to follow the estimates in Section \ref{section:decay}, since the meaning of many exponents and parameters is more clear when expressed in terms of $q$ (see also Remark \ref{decay:rem-result}.\ref{decay:item-q}).

\subsubsection{Covariant derivatives, curvature tensor, and energy}
To ensure that all expressions are well-defined, we temporarily restrict ourselves to smooth connection one-forms and scalar fields. That is, we restrict ourselves  to the smooth state space
\begin{equation}\label{intro:eq-state}
\state = C^\infty(\T^2 \rightarrow \R^2) \times C^\infty(\T^2 \rightarrow \C). 
\end{equation}
The state space can be equipped with the standard $L^2$-inner product, which turns $\state$ into a pre-Hilbert space. For all $(A,\phi) \in \state$, we define the covariant derivatives $(\covd_A^j \phi)_{1\leq j\leq 2}$ and curvature tensor $(F_A^{jk})_{1\leq j,k\leq 2}$ as 
\begin{equation}\label{intro:eq-covariant-and-curvature}
\covd_A^j \phi := \partial^j \phi + \icomplex A^j \phi \qquad \text{and} \qquad 
F_A^{jk} = \partial^j \hspace{-0.2ex} A^k - \partial^k \hspace{-0.2ex}  A^j. 
\end{equation}
From the definitions in \eqref{intro:eq-covariant-and-curvature}, we directly obtain that
\begin{equation}\label{intro:eq-covariant-curvature-identity}
\big( \covd_A^j \covd_A^k - \covd_A^k \covd_A^j \big) \phi = \icomplex F_A^{jk} \phi. 
\end{equation}
From the identity in \eqref{intro:eq-covariant-curvature-identity}, it follows that the curvature tensor  $(F_A^{jk})$ captures the non-commutativity of the covariant derivative operators $(\covd_A^j)_{j \in [2]}$. Using the covariant derivatives and curvature tensor, we can then define the energy of the Abelian-Higgs model, which is given by 
\begin{equation}\label{intro:eq-energy}
E(A,\phi) := \int_{\T^2} \dx \bigg( \frac{|F_A|^2}{4} + \frac{|\covd_A \phi|^2}{2} + \frac{|\phi|^{q+1}}{q+1} \bigg),
\end{equation}
where $|F_A|^2 = F_{A,jk} F_A^{jk}$ and $|\covd_A \phi|^2 = \overline{\covd_{A,j} \phi} \,\covd_A^j \phi$.

\subsubsection{Gauge transformations and Coulomb gauge}
The (additive) group of smooth gauge transformations for the Abelian-Higgs equations is given by
\begin{equation}
\gaugegroup = \big\{ g \in C^\infty(\R^2 \rightarrow \R)\colon \, \nabla g \textup{ and } e^{-ig} \textup{ are periodic }\big\}.
\end{equation}
The addition on $\gaugegroup$ is defined using the standard addition of functions, i.e., $(g_1+g_2)(x)=g_1(x)+g_2(x)$. The action of the group $\gaugegroup$ on the smooth state space $\state$ is defined as 
\begin{equation}\label{intro:eq-group-action}
(A,\phi) \mapsto \big( A^g, \phi^g \big) := \big( A + \nabla g, e^{-\icomplex g} \phi \big).
\end{equation}
We note that, for any $g\in \gaugegroup$ and $(A,\phi)\in\state$, it then holds that
\begin{equation}\label{intro:eq-gauge-covariance}
\covd_{A^g} \phi^g = \big( \covd_A \phi \big)^g, \quad F_{A^g} = F_{A}, \quad \text{and} \quad 
E(A^g, \phi^g) = E(A,\phi). 
\end{equation}
Due to \eqref{intro:eq-gauge-covariance}, the group action \eqref{intro:eq-group-action} is called the gauge symmetry of the Abelian-Higgs equations. Similar as the energy, all physical observables are invariant under the gauge symmetry \eqref{intro:eq-group-action}. For this reason, one is often not interested in any individual state $(A,\phi) \in \state$, but rather its gauge orbit $\{ (A^g,\phi^g) \colon g \in \gaugegroup \}$. 

In order to study dynamics corresponding to the energy \eqref{intro:eq-energy}, such as gradient flows or Langevin equations, it is helpful to impose a gauge condition\footnote{A different approach is to first derive the dynamics corresponding to \eqref{intro:eq-energy} and then use DeTurck's trick, see e.g. the deterministic literature \cite{DeT83,D85} or the stochastic literature \cite{BC23,CCHS22,CCHS22+}.}. In this article, we impose the Coulomb gauge condition 
\begin{equation}\label{intro:eq-Coulomb-condition}
\partial_j A^j =0. 
\end{equation}
We recall that, since $\partial_j (A^g)^j=\partial_j A^j +\Delta g$, any connection one-form $A$ is gauge equivalent to a connection one-form $A^g$ satisfying the Coulomb condition \eqref{intro:eq-Coulomb-condition}. Our reason for imposing the Coulomb condition \eqref{intro:eq-Coulomb-condition} is that, in the Coulomb gauge, the nonlinearity in the Abelian-Higgs equations exhibits a null-structure (see e.g. Lemma \ref{prelim:lem-null}). Instead of working on the smooth state space $\state$ from \eqref{intro:eq-state}, we may now work on 
\begin{equation}\label{intro:eq-state-coulomb}
\statecoulomb := \Big\{ (A,\phi) \in \state \colon \partial_j A^j =0 \Big\}. 
\end{equation}
The Coulomb condition in \eqref{intro:eq-Coulomb-condition}, however, does not determine the gauge entirely. To explain this, we first introduce the group homomorphism $\Z^2 \rightarrow \gaugegroup, n \mapsto n\cdot x$, which allows us to identify $\Z^2$ with a subgroup of $\gaugegroup$. Then, we define a group action of $\Z^2$ on the state space $\state$ as 
\begin{equation}\label{intro:eq-group-action-Zd}
(A,\phi) \mapsto \big( A^n, \phi^n \big) := ( A+ n , e^{-\icomplex n \cdot x}\phi). 
\end{equation}
Since $n\in \Z^2$ is constant in $x$, it clearly holds that $\partial_j n^j=0$, and thus the discrete gauge symmetry \eqref{intro:eq-group-action-Zd} does not affect the Coulomb condition \eqref{intro:eq-Coulomb-condition}. Since $\Z^2$ is discrete, it is difficult to eliminate the gauge freedom from \eqref{intro:eq-group-action-Zd} by imposing further gauge conditions. However, since $\Z^2$ is finite-dimensional, the gauge freedom from \eqref{intro:eq-group-action-Zd} does not prevent us from proving well-posedness, and therefore does not need to be eliminated. To address it, we instead state our main estimates (from Theorem \ref{intro:thm-abelian-higgs}.\ref{intro:item-AH-2}) using gauge invariant norms. The gauge invariant norms are defined as
\begin{equation}\label{intro:eq-gauge-invariant-norms}
\begin{aligned}
\big\| A \big\|_{\GCs^{-\kappa}(\T^2 \rightarrow \R^2)}
&:= \inf_{n\in \Z^2} \big\| A +n\big\|_{\Cs_x^{-\kappa}}, 
\hspace{5ex}
\big\| A \big\|_{\GBeta(\T^2 \rightarrow \R^2)}
:= \inf_{n\in \Z^2} \big\| A + n \big\|_{\Beta}, \\ 
\big\| \phi \big\|_{\GCs^{-\kappa}(\T^2\rightarrow \C)}
&:= \sup_{n\in \Z^2} \big\| e^{-\icomplex n \cdot x} \phi \big\|_{\Cs_x^{-\kappa}},
\end{aligned}
\end{equation}
where the $\Cs_x^{-\kappa}$ and $\Beta$-norms are as in Definition \ref{prelim:def-besov} below.

\subsubsection{The stochastic Abelian-Higgs equations}
We now restrict the energy \eqref{intro:eq-energy} to the state space $\statecoulomb$ from \eqref{intro:eq-state-coulomb}. Our main interest then lies in the Langevin equation corresponding to the energy \eqref{intro:eq-energy}, which is formally given by the stochastic Abelian-Higgs equations in \eqref{intro:eq-SAH}. 
Due to the low regularity of the space-time white noises $\xi$ and $\zeta$, however, 
the rigorous formulation of the stochastic Abelian-Higgs equations \eqref{intro:eq-SAH} has to be based on smooth approximations of $\xi$ and $\zeta$. To this end, we let $N\in \dyadic$ be a truncation parameter and define 
\begin{equs}\label{eq:smoothed-noise}
\xi_{\leq N}(t,x) := \int_{\T^2} \dy \chi_{\leq N}(x-y) \xi(t,y)
\qquad \text{and} \qquad 
\zeta_{\leq N}(t,x) := \int_{\T^2} \dy \chi_{\leq N}(x-y) \zeta(t,x),
\end{equs}
where the convolution kernel is as in Definition \ref{def:mollifiers}. We emphasize that, while $\xi_{\leq N}$ and $\zeta_{\leq N}$ are smooth in the spatial variables, they are both white in time.
We can now introduce our smooth approximation of \eqref{intro:eq-SAH}, in which we also include necessary renormalization terms. It is given by 
\begin{equation}\tag{$\text{SAH}_{N}$}\label{intro:eq-SAH-smooth}
\begin{cases}
\begin{aligned}
\partial_t A_{\leq N} &= \Delta A_{\leq N} - \leray \Im \big( \overline{\phi_{\leq N}} \covd_{A_{\leq N}}\phi_{\leq N} \big) + \Cgauge A_{\leq N}+ \leray \xi_{\leq N}, \\ 
\partial_t \phi_{\leq N} &= \big( \covd_{A_{\leq N}}^j \covd_{A_{\leq N},j} + 2 \sigma_{\leq N}^2 \big) \phi_{\leq N} - \biglcol |\phi_{\leq N}|^{q-1} \phi_{\leq N} \bigrcol + \zeta_{\leq N}, \\ 
A_{\leq N}(0) &= A_0, \quad \phi_{\leq N}(0)=\phi_0.
\end{aligned}
\end{cases}
\end{equation}
The constant $\Cgauge$ in the $\Cgauge A_{\leq N}$-term is explicitly given by $\Cgauge =\frac{1}{8\pi}$. This term is needed to cancel a resonance in the derivative nonlinearity. Without the $\Cgauge A_{\leq N}$-term, the limit would not be gauge covariant with respect to the discrete gauge transformation \eqref{intro:eq-state-coulomb}.
The renormalization constant $\sigma_{\leq N}^2$ and Wick-ordered nonlinearity $ \biglcol |\phi_{\leq N}|^{q-1} \phi_{\leq N} \bigrcol$ in \eqref{intro:eq-SAH-smooth} are defined as
\begin{equs}
\sigma_{\leq N}^2 &= \big\|e^s  (\moll_{\leq N} \ast_y p)(s, y; 0, 0)\big\|_{L_s^2 L_y^2((-\infty, 0] \times \T^2)}^2, \label{eq:sigma-N-squared}\\
\biglcol |\phi_{\leq N}|^{q-1} \phi_{\leq N} \bigrcol 
&= H_{\frac{q+1}{2}, \frac{q-1}{2}}\big( \phi_{\leq N}, \overline{\phi_{\leq N}}, \sigma_{\leq N}^2 \big), \label{eq:wick-ordered-products} 
\end{equs}
where $\moll_{\leq N}$ is a mollifier as in Definition \ref{def:mollifiers}, $\ast_y$ denotes convolution in the $y$ variable, $p(s, y; 0, 0)$ is the heat kernel on $\T^2$, and the Hermite polynomials $(H_{m, n})_{m, n \geq 0}$ are as in Section \ref{section:preliminary-hermite}.

\subsubsection{The covariant stochastic heat equation} In our analysis of the stochastic Abelian-Higgs equations \eqref{intro:eq-SAH}, we use covariant stochastic objects, which are defined using the covariant stochastic heat equation. To state the covariant stochastic heat equation, we first let $\Blin \colon [0,\infty) \times \T^2 \rightarrow \R^2$ be a connection one-form. At this point, we neither assume that $\Blin$ satisfies an evolution equation nor that $\Blin$ satisfies the Coulomb gauge condition (see, however, Theorem \ref{intro:thm-cshe} and Remark \ref{intro:rem-cshe}). For all $N\in \dyadic$, we then define the covariant stochastic object $\philinear[\Blin,\leqN]$ as the solution of 
\begin{equation}\tag{CSHE}\label{intro:eq-cshe}
\begin{cases}
\begin{aligned}
\big( \partial_t - \covd_B^j \covd_{B,j} + 1 \big) \philinear[\Blin,\leqN] &=  \zeta_{\leq N}, \\ 
\philinear[\Blin,\leqN] (0) &=0. 
\end{aligned}
\end{cases}
\end{equation}
The massive $\philinear[\Blin,\leq N]$-term has been included in \eqref{intro:eq-cshe} since it will be convenient in Section \ref{section:decay} below, but this is not too important for our argument\footnote{The reason for this is that we can include the massive term $m^2\phi$ in the linear operator by replacing $-|\phi|^{q-1} \phi$ with $-|\phi|^{q-1} \phi+m^2 \phi$. For any fixed $m\geq 0$, the damping effect of $-|\phi|^{q-1} \phi+m^2 \phi$ is essentially as strong as the damping effect of $-|\phi|^{q-1} \phi$, and thus this replacement makes no difference in our nonlinear estimates.}.

\subsection{Main results}\label{section:introduction-results}

In Theorem \ref{intro:thm-informal} above, we stated a formal version of our main theorem, which concerns the global well-posedness of the stochastic Abelian-Higgs equations in two dimensions. Equipped with the definitions from Subsection \ref{section:introduction-setting}, we can now state a rigorous version of our main theorem.

\begin{theorem}[Stochastic Abelian-Higgs equations]\label{intro:thm-abelian-higgs}
Let $q\geq 3$ be an odd integer, let $\kappa,\eta\in (0,1)$ and $r\geq 1$ be as in  
\eqref{prelim:eq-parameter-new-eta-nu}-\eqref{prelim:eq-parameter-new-r}, and let $\gamma=\frac{2}{7}$. Furthermore, let $A_0 \in \Beta(\T^2 \rightarrow \R^2)$ and let $\phi_0 \in L_x^r(\T^2\rightarrow \C)$. 
Then, the stochastic Abelian-Higgs equations almost surely have a global solution $(A,\phi)$, which is defined as the unique global limit of the solutions $(A_{\leq N},\phi_{\leq N})$ of \eqref{intro:eq-SAH-smooth} in $C_{t,\textup{loc}} \Cs_x^{-\kappa}([0,\infty)\times \T^2 \rightarrow \R^2 \times \C)$. Furthermore, we have the following properties:
\begin{enumerate}[label=(\roman*)]
\item \label{intro:item-AH-1} (Exponential bounds)  There exists a constant $C=C(\kappa,\eta,r,q)>0$ such that, for all $p\geq 1$ and $t\geq 0$, it holds that 
\begin{equation*}
 \E \Big[ \max \Big( \big\| A(t) \big\|_{\Cs_x^{-\kappa}}^\gamma, 
\big\| \phi(t) \big\|_{\Cs_x^{-\kappa}} \Big)^p \Big]^{\frac{1}{p}}  
\leq C e^{Ct}  \Big( \E \Big[ \max \Big( \big\| A(0) \big\|_{\Beta}^\gamma, 
\big\| \phi(0)\big\|_{L_x^r} \Big)^p \Big]^{\frac{1}{p}} 
+  p^{\frac{1}{\kappa}} \Big).   
\end{equation*} 
\item\label{intro:item-AH-2} (Gauge invariant, uniform-in-time bounds) There exist constants $C=C(\kappa,\eta,r,q)>0$ and $c=c(\kappa,\eta,r,q)$ such that, for all $p\geq 1$ and $t\geq $0, it holds that
\begin{align*}
&\, \E \Big[ \max \Big( \big\| A(t) \big\|_{\GCs^{-\kappa}(\T^2 \rightarrow \R^2)}^\gamma, 
\big\| \phi(t) \big\|_{\GCs^{-\kappa}(\T^2 \rightarrow \C)} \Big)^p \Big]^{\frac{1}{p}}  \\
\leq&\,  C e^{-ct}  \E \Big[ \max \Big( \big\| A(0) \big\|_{\GBeta}^\gamma, 
\big\| \phi(0)\big\|_{L_x^r} \Big)^p \Big]^{\frac{1}{p}} 
+ C p^{\frac{1}{\kappa}}. 
\end{align*}
\item\label{intro:item-AH-3} (Gauge covariance) The solution is covariant with respect to the discrete gauge transformation \eqref{intro:eq-group-action-Zd}. 
\end{enumerate}
\end{theorem}

\begin{remark}\label{remark:intro-thm-abelian-higgs} We make the following remarks regarding Theorem \ref{intro:thm-abelian-higgs}.
\begin{enumerate}[label=(\alph*)]
\item  (Regularity of initial data) 
In \ref{intro:item-AH-1}, the solution is controlled in the weaker $\Cs_x^{-\kappa}\times \Cs_x^{-\kappa}$-norm while the initial data is controlled in the stronger $\Beta \times L_x^r$-norm. Using local theory and (our estimates from Section \ref{section:decay}), this discrepancy can likely be removed, but the exponent of the resulting $\max \big( \big\| A(0) \big\|_{\Cs_x^{-\kappa}}^\gamma, 
\big\| \phi(0)\big\|_{\Cs_x^{-\kappa}} \big)$-term then needs to be slightly larger than $p$. 
\item (Polynomial growth) While \ref{intro:item-AH-1} only contains an exponential growth estimate, one can obtain a polynomial growth estimate from our argument (see Remark \ref{decay:rem-polynomial}). However, since the exponential bound makes the proofs of \ref{intro:item-AH-1} and \ref{intro:item-AH-2} more similar, we did not pursue this in detail. 
\item (Tail-behavior) From \ref{intro:item-AH-2}, it follows that the tails of linear, gauge invariant observables in $A$ or $\phi$ can be controlled by $\exp(-c\lambda^{\gamma \kappa})$ or $\exp(-c\lambda^\kappa)$, respectively.  However, the energy \eqref{intro:eq-energy} suggests that these tails should be even controlled by $\exp(-c\lambda^2)$ or $\exp(-c\lambda^{q+1})$. The analogous tail estimates for the $\Phi^4_3$-measure were recently obtained by Hairer and Steele \cite{HS22}, and it would be interesting to see whether their method can be applied to \eqref{intro:eq-SAH}. 
\item\label{item:gauge-covariance} (Gauge covariance) The gauge covariance statement in  \ref{intro:item-AH-3} is made more precise in Section \ref{section:gauge-covariance}. Since the discrete gauge transformation from \eqref{intro:eq-group-action-Zd} is rather simple, the proof of gauge covariance is much simpler than in our earlier work \cite{BC23} on the stochastic Yang-Mills equations. For this reason, gauge covariance will not be discussed further in this introduction. 
\end{enumerate}
\end{remark}

As will be discussed in Subsection \ref{section:introduction-argument}, our main theorem (Theorem \ref{intro:thm-abelian-higgs}) is obtained using a covariant approach to the stochastic Abelian-Higgs equations. As part of our covariant approach, we require estimates of the covariant stochastic object $\philinear[\Blin]$, which was defined in \eqref{intro:eq-cshe}. Since we believe that our estimates of the covariant stochastic object $\philinear[\Blin]$ are of independent interest, we record them in a separate theorem. To this end, we first make the following definition.

\begin{definition}[Norm on $\Blin$]\label{intro:def-norm}
For any time $T>0$ and connection one-form $\Blin\colon [0,T]\times \T^2 \rightarrow \R^2$, we define 
\begin{equs}\label{intro:eq-norm}
\|\Blin\|_{T, \resnorm} :=&~~ \| \Blin\|_{L_t^\infty L_x^\infty([0, T] \times \T^2)} + \|\nabla \Blin\|_{L_t^2 L_x^\infty([0, T] \times \T^2)} +\|\ptl_j \Blin^j \|_{L_t^\infty L_x^\infty([0, T] \times \T^2)} \\
&+ \|\ptl_t \Blin\|_{L_t^1 L_x^\infty([0, T] \times \T^2)} + \|(\ptl_j (F_\Blin)^{kj})_{k \in [2]} \|_{L_t^1 L_x^\infty([0, T] \times \T^2)} .
\end{equs}
In the case $T=1$, we also write 
\begin{equs}
\| \Blin \|_{\resnorm} := \| \Blin \|_{1,\resnorm}.
\end{equs}
\end{definition}

We note that if $\Blin$ satisfies the Coulomb gauge condition $\partial_j \Blin^j=0$, then the third summand in \eqref{intro:eq-norm} is equal to zero and the argument in the fifth summand in \eqref{intro:eq-norm} is given by 
$\ptl_j (F_\Blin)^{kj} = -\Delta \Blin^k$. Equipped with Definition \ref{intro:def-norm}, we can now state our estimates for the covariant stochastic heat equation. We emphasize here that for this theorem, we do not require that $\Blin$ is in the Coulomb gauge.

\begin{theorem}[Covariant stochastic heat equation]\label{intro:thm-cshe}
Let $\kappa,\eta\in (0,1)$ be as in \eqref{prelim:eq-parameter-new-eta-nu} and \eqref{prelim:eq-parameter-new-kappa-kappa-j} and let $\Blin\in C_t^0 \Cs_x^\eta([0,1]\times \T^2\rightarrow \R^2)$ be a connection one-form such that $\ptl_j \Blin^j \in C_t^0 \Cs_x^\eta([0, 1] \times \T^2 \ra \R)$. Then, for all $N \in \dyadic$, there exists a unique solution $\philinear[\Blin,\leqN][r][]$ of \eqref{intro:eq-cshe} in $C_t^0 \Cs_x^{-\kappa}([0,1]\times \T^2 \rightarrow \C)$. Furthermore, we have the following estimates: 
\begin{enumerate}[label=(\roman*)]
\item\label{item:thm-cshe-polynomial} (Power-type nonlinearity) 
For $\power \geq 1$, it holds that
\begin{equs}\label{intro:eq-cshe-linear}
\E \bigg[ \sup_{N \in \dyadic} \Big\|\philinear[\Blin,\leqN] \Big\|_{C_t^0 \Cs_x^{-\kappa}([0,1]\times \T^2)}^{\power} \bigg]^{\frac{1}{\power}} 
\lesssim \power^{\frac{1}{2}} \big( 1+ \|\Blin\|_{\resnorm}^{2\kappa}\big).
\end{equs}
For $k\geq 2$, $\power\geq 1$, it holds that
\begin{equs}\label{intro:eq-cshe-power}
\E \bigg[ \sup_{N \in \dyadic} \bigg\| t^\kappa \biglcol\,  \Big| \philinear[\Blin,\leqN] \Big|^{2k-2} \philinear[\Blin,\leqN] \bigrcol \bigg\|_{C_t^0 \Cs_x^{-\kappa}([0,1]\times \T^2)}^\power \bigg]^{\frac{1}{\power}} 
\lesssim \power^{k-\frac{1}{2}} \big( 1+ \|\Blin\|_{\resnorm}^{4(k-1)\kappa}\big).
\end{equs}
\item\label{item:thm:-difference-in-linear-objects}(Difference in linear objects) For all $\power \geq 1$ and $\alpha \in [0,1-\kappa]$, it holds that 
\begin{equs}\label{intro:eq-cshe-difference}
\E\Bigg[ \sup \Big\|  \philinear[\Blin, \leqN] - \philinear[\leqN] \Big\|_{C_t^0 \Cs_x^{\alpha}([0, 1] \times \T^2)}^\power \Bigg]^{1/\power} \lesssim \power^{1/2} \|\Blin\|_{\resnorm}^{\alpha} \big(1 + \|\Blin\|_{\resnorm}^{\kappa}\big). 
\end{equs}
\item\label{item:cshe-derivative-nonlinearity} (Derivative nonlinearity) For all $\power \geq 1$ and $\alpha \in [0,1-\kappa]$, it holds that 
\begin{equs}\label{intro:eq-cshe-derivative}
\E \bigg[ \sup_{\substack{N \in \dyadic \\ N \geq \|B\|_{\resnorm}^3}} \bigg\|  \Im \Duh \Big[\,  \overline{\philinear[\Blin,\leqN]} \covd_\Blin \philinear[\Blin,\leqN][r][] - \gaugerenorm B \Big] \bigg\|_{C_t^0 \Cs_x^{\alpha} \cap C_t^{\alpha/2} \Cs_x^0([0, 1] \times \T^2)}^\power \bigg]^{\frac{1}{\power}} \lesssim \power \big(1 + \|B\|_{\resnorm}^{\frac{1}{2}+\kappa}\big).
\end{equs}
\end{enumerate}
\end{theorem}

\begin{remark}\label{intro:rem-cshe}
We make the following remarks regarding Theorem \ref{intro:thm-cshe}. 
\begin{enumerate}[label=(\alph*)]
\item\label{intro:item-choice-B} (Choice of $\Blin$) 
In our proof of Theorem \ref{intro:thm-abelian-higgs}, it seems tempting to use Theorem \ref{intro:thm-cshe} with $\Blin:= A$, where $A$ is as in the stochastic Abelian-Higgs equations \eqref{intro:eq-SAH}. However, this choice is not allowed, since the connection one-form $A$ and space-time white noise $\zeta$ are probabilistically dependent. To circumvent this, we set up an iteration procedure over small time-intervals and, on each time-interval, choose $B$ as the linear heat evolution of the initial data of $A$. For more details, we refer the reader to Sections \ref{section:introduction-argument}, \ref{section:Abelian-Higgs}, and  \ref{section:decay}. We also remark that this use of independence is inspired by recent works on random dispersive equations, see e.g. \cite{B21,DNY19}. 
\item (Dependence on $\Blin$) The most important aspect of Theorem \ref{intro:thm-cshe} is that the estimates \eqref{intro:eq-cshe-linear}-\eqref{intro:eq-cshe-derivative} exhibit good dependence on $\Blin$, i.e., only involve small powers of $\| \Blin \|_{\resnorm}$. Otherwise, it would not be possible to use the estimates \eqref{intro:eq-cshe-linear}-\eqref{intro:eq-cshe-derivative} in the proof of Theorem \ref{intro:thm-abelian-higgs}. The proof of this good dependence on $\Blin$ relies heavily on our covariant estimates, and cannot be obtained if the contributions of the $ \icomplex\Blin$-term in $\covd_\Blin$ are treated perturbatively.
\item As a follow up to the previous item, we note that in item \ref{item:cshe-derivative-nonlinearity}, the restriction $N \geq \|B\|_{\resnorm}^3$ is only so that the resulting bound is of a nice form. As a consequence of the arguments of Section \ref{section:cshe}, we could also take a supremum over all $N \in \dyadic$, at a cost of a more complicated bound -- see in particular Proposition \ref{cshe:prop-resonant}. Moreover, as a result of that proposition, we have that if we restrict to a smaller time interval $[0, \tau]$, then we may also gain a power like $\tau^{\frac{3}{4}-}$, which will be crucial for our globalization arguments in Sections \ref{section:Abelian-Higgs} and \ref{section:decay}.
\item (Counterterm in the derivative nonlinearity) The $\gaugerenorm B$ counterterm in the derivative nonlinearity is crucial for obtaining an estimate with good dependence on $B$. Without it, we would only expect to obtain a bound that depends on one whole power of $B$, which, as mentioned in the previous item, would be fatal for our global existence proof strategy. Interestingly, the constant $\gaugerenorm$ is precisely the one which ensures gauge covariance of \eqref{intro:eq-SAH} (recall Remark \ref{remark:intro-thm-abelian-higgs}\ref{item:gauge-covariance}). Ultimately, this is due to the fact that the same calculation arises both when proving gauge covariance and bounding the resonant part of the derivative nonlinearity (see Remark \ref{remark:two-calculations}). We believe that there is some underlying geometric explanation for this.
\item (Convergence) The estimates in \eqref{intro:eq-cshe-power}, \eqref{intro:eq-cshe-difference}, and \eqref{intro:eq-cshe-derivative} 
 yield uniform control in the truncation parameter $N \in \dyadic$. From our argument, we can also extract the convergence of the power-type and derivative nonlinearities as $N \toinf$, see e.g. Lemmas \ref{lemma:difference-linear-object-N-regularity} and \ref{lemma:derivative-nonlinearity-high-frequency-estimate}.
\item\label{item:time-weights} (Time weights in item \ref{item:thm-cshe-polynomial}) We separated the cases $k = 1$ (linear object) and $k \geq 2$ (general polynomial object) in item \ref{item:thm-cshe-polynomial}, because in the latter case, we need to renormalize. However, there is a slight discrepancy with how we defined the Wick-ordered products in \eqref{eq:wick-ordered-products}, in that we chose to use the stationary variance $\sigma^2_{\leq N}$ for the non-stationary objects (recall $\philinear[B, \leqN](0) = 0$). This discrepancy results in the time weight in the norm on the higher degree polynomial objects.
\end{enumerate}
\end{remark}

Equipped with the statements of our main theorems (Theorem \ref{intro:thm-abelian-higgs} and \ref{intro:thm-cshe}), we now turn to a discussion of the ideas used in their proofs.

\subsection{Overview of the argument}\label{section:introduction-argument}

The core of our argument consists of a covariant monotonicity formula, which is inspired by Hamilton's monotonicity formulas for geometric parabolic flows \cite{Ham93}. To state a version of Hamilton's monotonicity formula in our setting, we let $A\colon [0,1]\times \T^2 \rightarrow \R^2$ be a connection one-form, let $G\colon [0,1] \times \T^2 \rightarrow \R^2$ be a forcing term, and let $\varphi\colon [0,1]\times \T^2 \rightarrow \C$ 
be a solution of the covariant heat equation
\begin{equation}\label{intro:eq-argument-che}
\big( \partial_t - \covd_A^j \covd_{A,j} \big) \varphi = G. 
\end{equation}
Furthermore, we let $K\colon [0,1]\times \T^2 \rightarrow \R$ be a non-negative solution of  $(\partial_t + \partial^j \partial_j) K=0$, i.e., a backwards heat equation. In Proposition \ref{prop:monotonicity} below, it is then shown that 
\begin{equation}\label{intro:eq-argument-che-phisquare}
 \partial_t \big( K |\varphi|^2 \big) =
  K \Delta |\varphi|^2  -  (\Delta K) |\varphi|^2 - 2 K  | \covd_A \varphi |^2 + 2 K \Re ( \overline{\varphi} G ). 
\end{equation}
We note that the first and second summand on the right-hand side of \eqref{intro:eq-argument-che-phisquare} contain the Laplacian $\Delta$, which does not depend on $A$. By integrating \eqref{intro:eq-argument-che-phisquare} and using Hölder's inequality, one can obtain the estimate 
\begin{equation}\label{intro:eq-argument-estimate-example}
\big\| K^{\frac{1}{2}} \varphi \big\|_{L_t^\infty L_x^2([0,1]\times \T^2)}
+ \big\| K^{\frac{1}{2}} \covd_A \varphi \big\|_{L_t^2 L_x^2([0,1]\times \T^2)}
\lesssim \big\| \big( K^{\frac{1}{2}}\varphi \big)(0) \big\|_{L_x^2(\T^2)}
+ \big\| K^{\frac{1}{2}} G \big\|_{L_t^1 L_x^2([0,1]\times \T^2)}. 
\end{equation}
The important aspect of \eqref{intro:eq-argument-estimate-example} is that it can be used to obtain bounds on $\varphi$ and $\covd_A \varphi$ which do not depend on the size of $A$. We strongly emphasize that \eqref{intro:eq-argument-estimate-example} is only one of many covariant estimates for solutions of \eqref{intro:eq-argument-che} 
and that we need several variants of \eqref{intro:eq-argument-estimate-example}. For example, we need variants of \eqref{intro:eq-argument-estimate-example}  which use different norms of $G$, include additional time-weights, involve higher powers of $\varphi$, or make use of the damping term in \eqref{intro:eq-SAH}.
In the following, we discuss how covariant monotonicity formulas, together with additional ingredients, can be used to  first prove Theorem \ref{intro:thm-cshe} and then prove Theorem \ref{intro:thm-abelian-higgs}. 

\subsubsection{The covariant stochastic heat equation}
As in Theorem \ref{intro:thm-cshe}, we let $B\in C_t^0 \Cs_x^\eta([0,1]\times \T^2)$ be a connection one-form. Then, we let $p_B(s,y;t,x)$ be the kernel of the covariant heat equation
\begin{equation}\label{intro:eq-cshe-covariant-heat}
\big( \partial_t - \covd_B^j \covd_{B,j} \big) \varphi = G. 
\end{equation}
We further let $\massp_B(s,y;t,x)$ be the kernel of the massive variant of \eqref{intro:eq-cshe-covariant-heat}, which is explicitly given by $\massp_B(s,y;t,x)=e^{-(t-s)} p_B(s,y;t,x)$. The additional $e^{-(t-s)}$-factor can safely be ignored in the proof of Theorem \ref{intro:thm-cshe}, but it will be convenient in our global estimates (see Section \ref{section:decay}). From the definition of the covariant stochastic object $\philinear[\Blin]$, we obtain that
\begin{equation}\label{intro:eq-cshe-covariant-object}
\philinear[\Blin](t,x) = \iint \ds \dy \,  \massp_B(s,y;t,x) \zeta(s,y). 
\end{equation}
Due to the representation in \eqref{intro:eq-cshe-covariant-object}, all estimates in Theorem \ref{intro:thm-cshe} can be phrased in terms of the covariant heat kernel $p_B$. For this, it is necessary that the covariant heat kernel $p_B$ and space-time white noise $\zeta$ are probabilistically independent, which is tied to Remark \ref{intro:rem-cshe}.\ref{intro:item-choice-B}. \\

In our proofs of \ref{item:thm-cshe-polynomial} and \ref{item:thm:-difference-in-linear-objects} in Theorem \ref{intro:thm-cshe}, i.e., our estimates of $\scalebox{0.9}{$\philinear[\Blin]$}$,
$\scalebox{0.9}{$\biglcol \, |\philinear[\Blin]|^{2k-2}\,  \philinear[\Blin]\bigrcol$}$, and 
$\scalebox{0.9}{$\philinear[\Blin]-\philinear[]$}$,
the most important ingredient is the diamagnetic inequality (see Lemmas \ref{prelim:lem-diamagnetic} and \ref{kernel:lem-diamagnetic-heat-kernel}). The diamagnetic inequality implies the pointwise estimate
\begin{equation}\label{intro:eq-cshe-diagmagnetic}
\big| p_B(s,y;t,x) \big| \leq p(s,y;t,x), 
\end{equation}
where $p$ is the kernel of the standard heat equation. We emphasize that \eqref{intro:eq-cshe-diagmagnetic} is uniform in the connection one-form $B$. Another ingredient is a resolvent identity for the difference $p_B-p$, which allows us to trade powers of $\| B \|_{\resnorm}$ to obtain more detailed information on $p_B$. Together with Gaussian hypercontractivity, the diamagnetic inequality and resolvent identity then yield both \ref{item:thm-cshe-polynomial} and \ref{item:thm:-difference-in-linear-objects} in Theorem \ref{intro:thm-cshe}. \\

In order to prove \ref{item:cshe-derivative-nonlinearity} in Theorem \ref{intro:thm-cshe}, i.e., our estimates of the derivative nonlinearity $\Duh \Im (\, \scalebox{0.9}{$\overline{\philinear[\Blin]} \covd_\Blin \philinear[\Blin]$})$, we need to work harder. To this end, we first note that \eqref{intro:eq-argument-estimate-example} implies the operator estimate
\begin{equation}\label{intro:eq-cshe-operator}
\bigg\| K^{\frac{1}{2}}(t,x) 
\iint \ds \dy \, \big( \covd_B p_B \big)(s,y;t,x) G(s,y) \bigg\|_{L_t^2 L_x^2([0,1]\times \T^2)}
\lesssim \big\| K^{\frac{1}{2}} G \big\|_{L_s^1 L_y^2([0,1]\times \T^2)}. 
\end{equation}
In \eqref{intro:eq-cshe-operator}, the covariant derivative $\covd_\Blin$ acts on the $x$-variable of $p_\Blin$. Similar as for \eqref{intro:eq-argument-estimate-example}, we emphasize that \eqref{intro:eq-cshe-operator} is only one of many estimates involving $\covd_\Blin p_B$, and different variants of \eqref{intro:eq-cshe-operator} are proven in Corollary \ref{kernel:cor-energy-estimate-dual}, 
Lemma \ref{lemma:two-derivative-energy-estimate}, and Corollary \ref{cor:dual-two-sided-time-weighted-estimate} below. In order to control $\Duh \Im (\, \scalebox{0.9}{$\overline{\philinear[\Blin]} \covd_\Blin \philinear[\Blin]$})$, we then use a decomposition into non-resonant and resonant parts. For the non-resonant part, we need to prove estimates of the form\footnote{The expression in \eqref{intro:eq-cshe-non-resonant} is not exactly the expression needed in the proof of Theorem \ref{intro:thm-cshe}.\ref{item:cshe-derivative-nonlinearity}, since it can only be used to control $C_t^0 \Cs_x^{-\kappa}$-norms and does not include the mollification kernels from the definition of $\zeta_{\leq N}$. For the correct expression, we refer the reader to Section \ref{section:derivative-nonlinearity-non-resonant}.}
\begin{equation}\label{intro:eq-cshe-non-resonant}
\sup_{t_0\in [0,1]} \sup_{x_0\in \T^2}
\bigg\| \iint \dt \dx \, p(t,x;t_0,x_0) \overline{p_B(s_1,y_1;t,x)} \big( \covd_B p_B \big)(s_2,y_2;t,x) \bigg\|_{L_{s_1}^2 L_{s_2}^2 L_{y_1}^2 L_{y_2}^2} \lesssim 1.
\end{equation}
Our estimate of \eqref{intro:eq-cshe-non-resonant} primarily relies on a dual version of \eqref{intro:eq-cshe-operator}. In using \eqref{intro:eq-cshe-operator}, we often choose $K(t,x)$ as $p(t,x;t_0,x_0)$, which is used to address the singular integral kernel of the Duhamel integral. In particular, we make use of the backwards heat flow in \eqref{intro:eq-argument-che-phisquare}. In order to treat the resonant part of $\Duh \Im (\, \scalebox{0.9}{$\overline{\philinear[\Blin]} \covd_\Blin \philinear[\Blin]$})$, we need to prove estimates of the form\footnote{Similar as \eqref{intro:eq-cshe-non-resonant}, \eqref{intro:eq-cshe-resonant} is not exactly the expression needed in the proof of Theorem \ref{intro:thm-cshe}.\ref{item:cshe-derivative-nonlinearity}. In the correct replacement of \eqref{intro:eq-cshe-resonant}, the mollification kernels resulting from the definition of $\zeta_{\leq N}$ play an important role, since they are responsible for the $\Cgauge\Blin$-term.}
\begin{equation}\label{intro:eq-cshe-resonant}
\sup_{t_0\in[0,1]} \sup_{x\in \T^2}
\bigg| \iiiint \ds \dy \dt \dx \, p(t,x;t_0,x_0) \overline{p_B(s,y;t,x)} \big( \covd_B p_B \big)(s,y;t,x) \bigg| \lesssim \big( 1+ \| B \|_{\resnorm} \big)^{\frac{1}{2}+\varepsilon}. 
\end{equation}
In our treatment of \eqref{intro:eq-cshe-resonant}, we introduce a time-scale $\rho>0$ and restrict the integral in \eqref{intro:eq-cshe-resonant} to the region $|t-s|\sim \rho$. Using covariant estimates (see Lemma \ref{lemma:non-A-dependent-estimate}), which are derived from variants of \eqref{intro:eq-cshe-operator},
we can estimate \eqref{intro:eq-cshe-resonant} by $\rho^{-\frac{1}{2}}$. Despite the poor dependence on $\rho$, this estimate is incredibly useful, since it does not depend on the size of $\Blin$. Using perturbative estimates (see Lemma \ref{lemma:resonant-part-A-dependent-estimate}), we can also estimate \eqref{intro:eq-cshe-resonant} by $\rho^{\frac{1}{2}} (1+\| \Blin\|_{\resnorm})$. While the latter estimate has a good dependence on $\rho$, it costs one power of $\| \Blin \|_{\resnorm}$. By interpolating the covariant and perturbative estimates, one then obtains the desired estimate \eqref{intro:eq-cshe-resonant}.

\begin{remark}[Non-resonant and resonant parts]
Using a refined version of \eqref{intro:eq-cshe-non-resonant}, one can prove for all $\alpha \in [0,1-\kappa]$, all $\varepsilon>0$, and all $\power\geq 2$ that 
\begin{equation}\label{intro:eq-cshe-non-resonant-final}
\E \bigg[ \sup_{N\in \dyadic} \bigg\| \leray \Im \Duh \Big[  \biglcol \, \overline{\philinear[\Blin,\leqN]} \covd_\Blin \philinear[\Blin,\leqN][r][] \bigrcol \hspace{-0.2ex} \Big] \bigg\|_{C_t^0 \Cs_x^{\alpha} \cap C_t^{\alpha/2} \Cs_x^0([0, 1] \times \T^2)}^\power \bigg]^{\frac{1}{\power}} 
\lesssim \power \big( 1 + \| \Blin \|_{\resnorm} \big)^\varepsilon,
\end{equation}
where $\biglcol\, \scalebox{0.9}{$\overline{\philinear[\Blin,\leqN]} \covd_\Blin \philinear[\Blin,\leqN][r][]$} \bigrcol \hspace{-0.2ex}$ denotes the non-resonant part of $ \scalebox{0.9}{$\overline{\philinear[\Blin,\leqN]} \covd_\Blin \philinear[\Blin,\leqN][r][]$}$. The $\varepsilon$-loss in \eqref{intro:eq-cshe-non-resonant-final} is due to the supremum over $N\in \dyadic$. By comparing Theorem \ref{intro:thm-cshe}.\ref{item:cshe-derivative-nonlinearity} and \eqref{intro:eq-cshe-non-resonant-final}, one sees that the non-resonant part obeys a better estimate than the resonant part. It is possible that this is just an artifact of our proof for the resonant part.
\end{remark}

\begin{remark}[Heat kernel estimates]
While heat kernels and their estimates have been extensively studied (see e.g. the survey \cite{SC10}), it seems that estimates of the form \eqref{intro:eq-cshe-non-resonant} and \eqref{intro:eq-cshe-resonant} have not previously appeared in the literature, which may have two reasons: First, while the expressions in \eqref{intro:eq-cshe-non-resonant} and \eqref{intro:eq-cshe-resonant} appear naturally in the proof of Theorem \ref{intro:thm-cshe}.\ref{item:cshe-derivative-nonlinearity}, the same expressions do not seem to appear in other problems involving heat kernels. Second, in most problems involving heat kernels, the connection one-form $\Blin$ in $\covd_{\Blin}^j \covd_{\Blin,j}$ (or the Riemannian metric $g$ in the Laplace-Beltrami operator $\Delta_g$) is viewed as given, and the dependence of the estimates on the size of $\Blin$ is therefore not tracked. In order to prove Theorem \ref{intro:thm-abelian-higgs}, however, the good dependence of \eqref{intro:eq-cshe-non-resonant} and \eqref{intro:eq-cshe-resonant} on the size of $\Blin$ is essential. 
\end{remark}

\subsubsection{The stochastic Abelian-Higgs equations}\label{section:globalization-argument-intro}
The main ingredients in the proof of Theorem \ref{intro:thm-abelian-higgs} are 
our estimates for the covariant stochastic heat equation (Theorem \ref{intro:thm-cshe}) and covariant monotonicity formulas. To use the former, we need to isolate a component of the connection one-form $A$ which is probabilistically independent of the space-time white noise $\zeta$. This can be done by first working on small time-intervals and then iterating in time. In the following, we focus on our Ansatz and decay estimates on small time-intervals, which are the most important part of the proof of Theorem \ref{intro:thm-abelian-higgs}. \\

Let $t_0 \in [0,\infty)$ and assume, for simplicity, that 
\begin{equation}\label{intro:eq-ah-initial}
A(t_0) \in \Beta \qquad \text{and} \qquad \phi(t_0)\in L_x^r. 
\end{equation}
While neither of the assumptions in \eqref{intro:eq-ah-initial} will be satisfied for any $t_0>0$, the low-regularity terms in $A$ and $\phi$ are, at least morally, of size $\sim 1$, and can therefore be safely ignored. We now let $\tau$ be a small time scale which will be chosen in \eqref{intro:eq-ah-tau} below. We emphasize that the choice of $\tau$ will not be determined by the local theory of \eqref{intro:eq-SAH}; instead, it will be determined by our non-perturbative estimates of $A$ and $\phi$. On the small time-interval $[t_0,t_0+\tau]$, we decompose the connection one-form $A$ as 
\begin{equation}\label{intro:eq-ah-decomposition-A}
A(t) = \linear\hspace{-1mm}(t) + B(t) + Z(t). 
\end{equation}
The first term $\linear$ is the linear stochastic object determined by $\leray \xi$. The second term $\Blin$ is the linear heat flow of\footnote{In Section \ref{section:Abelian-Higgs} and \ref{section:decay}, the initial data of $\Blin$ will only be the high-regularity part of $A(t_0)$. Similarly, the initial data of $\psi$ below will only be the high-regularity part of $\phi(t_0)$. For the purpose of this introduction, however, this makes no difference.} $A(t_0)$, i.e., $B=e^{(t-t_0)\Delta} A(t_0)$. We note that, since $A$ is adapted, $\Blin$ is probabilistically independent of the stochastic forcing $\zeta|_{[t_0,t_0+\tau]}$. Finally, the third term $Z$ is a nonlinear remainder, which solves a nonlinear heat equation. We decompose\footnote{For technical reasons, the decomposition of $\phi$ is further split into low and high-frequency components, see e.g. Section \ref{section:AH-Ansatz}. For the purpose of this introduction, this makes no difference.} the scalar field $\phi$ as 
\begin{equation}\label{intro:eq-ah-decomposition-phi}
\phi(t) = \philinear [\Blin]\hspace{-0.5mm} (t) + \psi(t). 
\end{equation}
The first term $\philinear[\Blin]$ is the covariant stochastic object corresponding to the connection one-form $\Blin$ and the second term $\psi$ is a nonlinear remainder with initial data $\psi(t_0)=\phi(t_0)$.  \\

In order to prove Theorem \ref{intro:thm-abelian-higgs}, we then want to show that both\footnote{This is actually not true when the sizes of $A(t_0)$ and $\phi(t_0)$ are very different, but this is addressed by the maxima in Theorem \ref{intro:thm-abelian-higgs}. It is also not true when the mean of $A(t_0)$ is very large, but this is addressed by the exponential growth in Theorem \ref{intro:thm-abelian-higgs}.\ref{intro:item-AH-1}.}
$A$ and $\phi$ exhibit decay on the small time-interval $[t_0,t_0+\tau]$.  To this end, we then want to show the following three different statements:  
\begin{enumerate}[leftmargin=12ex,label=(\Roman*)]
\item\label{intro:item-ah-B} $\| B(t) \|_{\Beta}$ decreases significantly on $[t_0,t_0+\tau]$. 
\item\label{intro:item-ah-psi} $\| \psi(t) \|_{L_x^r}$ decreases significantly on $[t_0,t_0+\tau]$. 
\item\label{intro:item-ah-Z} $\| Z(t) \|_{\Beta}$ stays small on $[t_0,t_0+\tau]$.
\end{enumerate}
The phrases ``decreases significantly" and ``stays small" are purposefully kept vague, since the precise formulation is rather technical (see Section \ref{section:decay-short}). 
The decay of $B$ is obtained using standard estimates for the heat flow $e^{(t-t_0)\Delta}$. The decay of $\psi$ and smallness of $Z$ have to be proven together, since their evolution equations are strongly coupled. By using covariant monotonicity formulas involving $\psi$, we obtain estimates of the schematic form 
\begin{align}
\frac{1}{r} \partial_t \big\| \psi(t) \big\|_{L_x^r}^r 
+ \big\| \psi(t) \big\|_{L_x^{r+q-1}}^{r+q-1}
&\lesssim \| Z \|_{L_t^\infty \Beta}  ^{\frac{2}{q} (r+q-1)} + \big\{\, \textup{l.o.t.} \big\}, \label{intro:eq-ah-psi-1} \\
\big\| K^{\frac{1}{2}} \covd_A \psi \big\|_{L_t^2 L_x^2}
&\lesssim \| \psi(t_0) \|_{L_x^r} + \tau^{\frac{1}{2}} 
\| Z \|_{L_t^\infty \Beta}   + \big\{\, \textup{l.o.t.} \big\}, \label{intro:eq-ah-psi-2} 
\end{align}
where $K$ is a certain solution of the backwards heat equation. From the evolution equation for $Z$, we directly obtain that
\begin{equation}\label{intro:eq-ah-Z}
\big\| Z \big\|_{L_t^\infty \Beta}  
\lesssim  \Big\| \leray \Im \Duh \Big[\,  \overline{\philinear[\Blin,\leqN]} \covd_\Blin \philinear[\Blin,\leqN][r][] - \tfrac{1}{8\pi} \Blin \Big] \Big\|_{L_t^\infty \Beta}  
+  \Big\| \leray \Im \Duh \Big[\,  \overline{\psi} \covd_A \psi \Big] \Big\|_{L_t^\infty \Beta}   + \big\{\, \textup{l.o.t.} \big\}.
\end{equation}
The first term on the right-hand side of \eqref{intro:eq-ah-Z} can be estimated using our bounds for the covariant stochastic object $\philinear[\Blin]$, i.e.,  Theorem \ref{intro:thm-cshe}. It then remains to close estimates for $\psi$ and $Z$ using \eqref{intro:eq-ah-psi-1}, \eqref{intro:eq-ah-psi-2}, and \eqref{intro:eq-ah-Z}, which can be done.

In order to obtain \ref{intro:item-ah-B}, \ref{intro:item-ah-psi}, and \ref{intro:item-ah-Z}, the time-scale $\tau$ can be neither too small nor too large. If $\tau$ is too small, then we obtain no significant decay of $\Blin$ or $\psi$, and if $\tau$ is too large, then we cannot close our estimates of $\psi$ and $Z$. The choice of $\tau$ has to be just right, and we choose\footnote{The actual definition of $\tau$ differs slightly from \eqref{intro:eq-ah-tau}, see Definition \ref{decay:def-tau}.}
\begin{equation}\label{intro:eq-ah-tau}
\tau \sim \max \big( \| A(t_0) \|_{\Beta}^{\frac{3}{4}}, \| \phi(t_0) \|_{L_x^r}^{\frac{5}{2}} \big)^{-1}. 
\end{equation}
The choice of the exponents in \eqref{intro:eq-ah-tau} is far from obvious but will be well-motivated in Remark \ref{decay:rem-parameters} below.

\subsection{Further discussion and future directions}
\label{section:introduction-further}

We mention several other relevant works that did not quite fit into the introductory discussion. Although we did not introduce the generalized parabolic Anderson model, we note that there are recent works \cite{CdLFW24,SZZ24} which show global well-posedness for this model. In another direction, the paper \cite{MM2022} manages to define the operator $\covd_{A_0}^j \covd_{A_0, j}$ (with suitable renormalization) when $A_0$ has the distribution of a Gaussian free field. Note that in our ansatz \eqref{intro:eq-ah-decomposition-A} for $A$, the most singular part $\linear$ is precisely distributed as a Gaussian free field at any fixed point in time. Thus, it is natural to wonder if the results of \cite{MM2022} could help towards showing the global well-posedness of \eqref{intro:eq-SAH}. This is not clear to us, because the main aspect of our globalization argument (outlined in Section \ref{section:globalization-argument-intro}) is to understand the effects of $\Blin$, $Z$, $\philinear[\Blin]$, and $\psi$, as opposed to the effects of $\linear[]$. \\

Next, we discuss possible future directions. We consider a general singular SPDE, which is written as 
\begin{equation}\label{intro:eq-general-SPDE}
\partial_t \phi - \mathcal{L} \phi = \mathcal{N}(\phi,\covd \phi) + \xi, \qquad \phi(0)=\phi_0.
\end{equation}
In \eqref{intro:eq-general-SPDE}, $\phi\colon [0,\infty) \times \T^d \rightarrow \R^m$ is the unknown, $\phi_0 \colon \T^d \rightarrow \R^m$ is the initial data, and $d,m\geq 1$ are dimensions. Furthermore, $\mathcal{L}$ is a linear differential operator, $\mathcal{N}=\mathcal{N}(\phi,p)\colon \R^{m} \times \R^{m\times d}\rightarrow \R^m$ is a nonlinearity, and $\xi\colon [0,\infty)\times \T^d \rightarrow \R^m$ is a stochastic forcing. 

To apply the strategy of this article to \eqref{intro:eq-general-SPDE}, we first let $\varphi\colon [0,\infty) \times \T^d\rightarrow \R$ be a deterministic evolution, such as the solution of 
\begin{alignat*}{3}
\partial_t \varphi - \mathcal{L} \varphi &=0, \qquad & \qquad \varphi(0)=\phi_0, \\ 
\text{or} \qquad \partial_t \varphi - \mathcal{L} \varphi &=\mathcal{N}(\varphi,\covd \varphi), \qquad & \qquad \varphi(0)=\phi_0. 
\end{alignat*}
Then, we let $\philinear[\varphi]\colon [0,\infty)\times \T^d\rightarrow \R^m$ be the linear stochastic object corresponding to the linearization\footnote{In our analysis of the stochastic Abelian-Higgs model \eqref{intro:eq-SAH}, we treat the connection one-form $A$ and scalar field $\phi$ differently. For this reason, \eqref{intro:eq-cshe} is only a partial linearization of \eqref{intro:eq-SAH}.} of \eqref{intro:eq-general-SPDE} around $\varphi$, i.e., we let $\philinear[\varphi]$ be the solution of 
\begin{equation}\label{intro:eq-general-object}
\partial_t \philinear[\varphi] - \mathcal{L} \,  \philinear[\varphi] 
- \big( \partial_\phi\hspace{0.05ex} \mathcal{N} \big)(\varphi,\covd \varphi)  \philinear[\varphi]
- \big( \partial_p\hspace{0.05ex}  \mathcal{N} \big) (\varphi,\covd \varphi) \covd  \philinear[\varphi]
=\xi, \qquad \philinear[\varphi](0)=0. 
\end{equation}
Once the linear stochastic object $\philinear[\varphi]$ has been defined, one naturally arrives at the following two questions:
\begin{enumerate}[label=(\roman*)]
\item\label{intro:item-general-1} Are there estimates of $\philinear[\varphi]$ which exhibit good dependence on $\varphi$?
\item\label{intro:item-general-2} If so, can the estimates be used to prove global well-posedness of \eqref{intro:eq-general-SPDE}?
\end{enumerate}
In this article, we gave an affirmative answer to both \ref{intro:item-general-1} and \ref{intro:item-general-2} for the two-dimensional, stochastic Abelian-Higgs model \eqref{intro:eq-SAH}. For other singular SPDEs, such as the two-dimensional  
Yang-Mills-Higgs equations, both \ref{intro:item-general-1} and \ref{intro:item-general-2} are interesting open problems. We thus hope that, in addition to proving Theorem \ref{intro:thm-abelian-higgs}, our article leads to future research on  \eqref{intro:eq-general-SPDE} and \eqref{intro:eq-general-object}. \\ 

\textbf{Acknowledgements:} The authors thank Martin Hairer, Tom Mrowka, Hao Shen, Wenhao Zhao, Rongchan Zhu and Xiangchan Zhu for helpful and interesting discussions. B.B. was partially supported by the NSF under Grant No. DMS-1926686. S.C. was partially supported by a Minerva Research Foundation membership while at IAS, as well as by the NSF under Grant No. DMS-2303165.

\section{Preliminaries}\label{section:preliminaries} 

Before the start of our argument, we need to make preparations. In Section \ref{section:parameters}, we introduce several parameters which will be used for the rest of the article. In Section \ref{section:preliminary-useful}, we recall several useful formulas from differential geometry. In Sections \ref{section:prelimary-harmonic} and \ref{section:preliminary-probability}, we recall basic facts and notation from harmonic analysis and probability theory, respectively. Finally, in Section \ref{section:preliminary-hermite}, we recall basic properties of Hermite polynomials.

\subsection{Parameters}\label{section:parameters}
In the following, all parameters are allowed to be chosen depending on the exponent $q$ from Theorem \ref{intro:thm-abelian-higgs}, which is viewed as given. First, we define parameters $\eta$ and $\nu$ satisfying
\begin{equation}\label{prelim:eq-parameter-new-eta-nu}
0 < \nu \ll 1 \qquad \text{and} \qquad \eta := \nu^{100}.
\end{equation}
Loosely speaking, we let $\nu$ be sufficiently small and then let $\eta$ be much smaller than any reasonable power of $\nu$. Second, we define parameters $(\eta_j)_{j=1}^3$ as
\begin{equation}\label{prelim:eq-parameter-new-eta-j}
\eta_1 := \eta^{10} \qquad \text{and} \qquad \eta_{j+1} := \eta_j^{10} \quad \text{for } j=2,3.
\end{equation}
Third, we define parameters $\kappa$ and $(\kappa_j)_{j=1}^4$ as 
\begin{equation}\label{prelim:eq-parameter-new-kappa-kappa-j}
\kappa:= \nu^{10} \kappaone, \qquad \kappaone := 100 \eta \eta_3, \qquad \text{and} \qquad \kappa_{j+1}= \nu^{-10} \kappa_j \quad \text{for } j=1,2,3. 
\end{equation}
We note that, due to \eqref{prelim:eq-parameter-new-eta-nu}, \eqref{prelim:eq-parameter-new-eta-j}, and \eqref{prelim:eq-parameter-new-kappa-kappa-j}, it holds that
\begin{equation*}
\kappafour = \nu^{-30} \kappaone = 100 \nu^{-30} \eta \eta_3 \ll \eta_3 
\qquad \text{and} \qquad \eta \kappafour = \eta \nu^{-40} \kappa \ll \kappa. 
\end{equation*}
Finally, we define the parameter $r\in [1,\infty)$ as 
\begin{equation}\label{prelim:eq-parameter-new-r}
r := \frac{1}{\kappa^{10}}.
\end{equation}
As a consequence of \eqref{prelim:eq-parameter-new-eta-nu}, \eqref{prelim:eq-parameter-new-eta-j}, \eqref{prelim:eq-parameter-new-kappa-kappa-j}, and \eqref{prelim:eq-parameter-new-r}, we obtain that 
\begin{equation}\label{prelim:eq-parameter-ordering}
0 < \frac{1}{r} \ll \kappa \ll \kappaone \ll \kappatwo \ll \kappathree \ll \kappafour \ll \eta_3 \ll \eta_2 \ll \eta_1 \ll \eta .
\end{equation}

We further introduce constants $(c_j)_{j=0}^7$ such that $c_0$ is sufficiently small depending on the parameters from 
 \eqref{prelim:eq-parameter-new-eta-nu}-\eqref{prelim:eq-parameter-new-r} 
 and such that, for all $1\leq j \leq 7$, $c_j$ is sufficiently small depending on both $c_{j-1}$ and the parameters from  
 \eqref{prelim:eq-parameter-new-eta-nu}-\eqref{prelim:eq-parameter-new-r}. To simplify the notation, we also introduce constants $(C_j)_{j=0}^7$, which are defined as
\begin{equs}
C_j := c_j^{-1}. 
\end{equs}
For any non-negative numbers $A,B\geq 0$, we  write $A\lesssim B$ if there exists an implicit constant $C$, depending only on the parameters in  \eqref{prelim:eq-parameter-new-eta-nu}-\eqref{prelim:eq-parameter-new-r}, such that $A\leq C B$. Unless stated otherwise, the implicit constant $C$ in our definition of the symbol $\lesssim\,$ is never allowed to depend on $(c_j)_{j=0}^7$ and $(C_j)_{j=0}^7$.

\subsection{Formulas from differential geometry}\label{section:preliminary-useful}
In the following lemmas, we recall basic formulas and estimates from differential geometry.

\begin{lemma}[Diamagnetic inequality]\label{prelim:lem-diamagnetic}
Let $A\colon \T^2 \rightarrow \R^2$, let $\phi\colon \T^2 \rightarrow \C$, and let $1\leq j \leq 2$. Then, we have the pointwise estimate
\begin{equation*}
|\ptl_j (|\phi|)| \leq |\covd_{A, j} \phi|.
\end{equation*}
\end{lemma}

The statement in Lemma \ref{prelim:lem-diamagnetic} is standard, see e.g. \cite[Theorem 7.21]{LL97}. 

\begin{lemma}[Formulas for derivatives]\label{prelim:lem-derivatives}
Let $\phi,\psi \colon \T^2 \rightarrow \C$ and let $A,B\colon \T^2 \rightarrow \R^2$. Then, we have the following identities:
\begin{enumerate}[label=(\roman*)]
\item\label{prelim:item-product} (Product formulas) For all $1\leq j \leq 2$, it holds that
\begin{align*}
\partial_j \big( \, \overline{\phi} \psi \big) &= \overline{\covd_{A,j} \phi}\, \psi + \overline{\phi} \, \covd_{A,j} \psi, \\ 
\covd_{A,j} \big( \phi \psi \big) &= (\partial_j \phi)  \psi + \phi \covd_{A,j} \psi. 
\end{align*}
Furthermore, if $\partial_k B^k=0$, then it holds that
\begin{equation*}
\covd_{A,k} \big( B^k \phi \big) = B^k \big( \covd_{A,k} \phi \big).
\end{equation*}
\item\label{prelim:item-difference} (Difference of covariant Laplacians) It holds that
\begin{equs}
(\covd_A^j \covd_{A, j} - \covd_B^j \covd_{B, j}) \phi = 2\icomplex \covd_B^j((A - B)_j \phi) - \icomplex (\ptl^j (A - B)_j) \phi - |A - B|^2 \phi.
\end{equs}
\item\label{prelim:item-Bochner} (Bochner formula)  It holds that
\begin{equs}
\frac{1}{2} \Delta(|\phi|^2) = \mrm{Re}\big( \overline{\phi} \covd_A^j \covd_{A, j} \phi \big) + |\covd_A \phi|^2.
\end{equs}
\end{enumerate}
\end{lemma}
The formulas in Lemma \ref{prelim:lem-derivatives} can be obtained from direct computations and we omit the details.

\begin{lemma}[Evolution of covariant derivatives]\label{lemma:D-A-phi-equation}
Let $T>0$, let $A\colon [0,T]\times \T^2 \rightarrow \R^2$, let $\phi\colon [0,T]\times \T^2 \rightarrow \C$, and let $G\colon [0,T]\times \T^2\rightarrow \C$. Furthermore, assume that the covariant heat equation 
\begin{equs}
(\ptl_t - \covd_A^j \covd_{A, j})\phi = G
\end{equs}
is satisfied. For all $1\leq k\leq 2$, the covariant derivative $\covd_A^k \phi$ then solves the covariant heat equation
\begin{equs}\label{eq:D-A-phi-equation}
(\ptl_t - \covd_A^j \covd_{A, j}) \covd_A^k \phi = 2 \icomplex (F_A)^{kj} \covd_{A, j} \phi + \icomplex \big( \ptl_t A^k + \ptl_j (F_A)^{kj} \big) \phi + \covd_A^k G.
\end{equs}
\end{lemma}
Evolution equations for covariant derivatives such as \eqref{eq:D-A-phi-equation} are standard in the analysis of geometric flows, see e.g. \cite[Chapter 6.1]{Tao06} for similar evolution equations. Since \eqref{eq:D-A-phi-equation} follows directly from \eqref{intro:eq-covariant-curvature-identity} and a direct computation, we omit the proof.

\subsection{Harmonic analysis and Euclidean heat flow}\label{section:prelimary-harmonic} 
In this subsection, we recall basic facts from harmonic analysis and basic estimates for the Euclidean heat flow. 

\subsubsection{Fourier transform and Littlewood-Paley projections} For notational convenience, we define the complex exponential $\e\colon \R \rightarrow \C$ by $\e(y)=\exp(\icomplex y)$ for all $y\in \R$. Furthermore, for all $n\in \Z^2$, we define $\e_{n}\colon \T^2 \rightarrow \C$ by $\e_{n}(x):=  \e(n\cdot x)$, and we let $\langle n \rangle := \sqrt{1 + |n|^2}$, where $|n|$ is the Euclidean norm of $n$.
For any distribution $\phi \colon \T^2\rightarrow \C$, we define the Fourier transform $\widehat{\phi} \colon \Z^2\rightarrow \C$ by 
\begin{equation}\label{eq:fourier-transform}
\widehat{\phi}(n) := \frac{1}{(2\pi)^2} \int_{\T^2} \dx \phi(x) \overline{\e_{n}(x)}
\end{equation}
for all $n\in \Z^2$.
We also define the Fourier transform on $\R^2$. For any Schwartz function $\phi \colon \R^2 \ra \C$, define the Fourier transform $\widehat{\phi} \colon \R^2 \ra \C$ by
\begin{equs}
\widehat{\phi}(\xi) := \frac{1}{(2\pi)^2} \int_{\R^2} \dx \phi(x) e^{-\icomplex \xi \cdot x}.
\end{equs}
Note that with the factor of $\frac{1}{(2\pi)^2}$ in the definition, the Fourier inversion formula reads:
\begin{equs}\label{eq:fourier-inversion}
\phi(x) = \int_{\R^2} d\xi \widehat{\phi}(\xi) e^{\icomplex \xi \cdot x}.
\end{equs}

Let $\rho \colon \R^2 \rightarrow [0,1]$ be a smooth radial cut-off function which is radially non-increasing, and which satisfies 
\begin{equation}\label{eq:rho-normalization}
\rho(x)= 1 \text{ for $|x| \leq 1$} \qquad \text{and} \qquad 
\rho(x) =0 \text{ for $|x| > \frac{9}{8}$.}
\end{equation}
For all $N > 0$ (possibly real), we define $\rho_{\leq N} \colon \R^2 \rightarrow \R$ by 
\begin{equation*}
\rho_{\leq N} (\xi) = \rho \big( \xi / N\big).    
\end{equation*}
Furthermore, we define
\begin{equation*}
\rho_1 = \rho_{\leq 1} \qquad \text{and} \qquad 
\rho_{N} = \rho_{\leq N} - \rho_{\leq N/2} 
\quad \text{for all } N\geq 2.
\end{equation*}
Note that $\rho_N \geq 0$ since $\rho$ is radially non-increasing. 

\begin{definition}[Mollifiers]\label{def:mollifiers}
Define $\moll \colon \R^2 \ra \R$ by 
\begin{equs}
\eucmoll(x) := \widehat{\rho}(x) = \frac{1}{(2\pi)^2} \int_{\R^2} d\xi \rho(\xi) e^{-\icomplex \xi \cdot x}.
\end{equs}
For $N > 0$ (possibly real), define $\eucmoll_{\leq N}(\cdot) := N^2 \eucmoll(N \cdot)$. Define $\moll, \moll_{\leq N} \colon \T^2 \ra \R$ as the periodizations of $\moll, \moll_{\leq N}$:
\begin{equs}
\moll(x) := \sum_{n \in \Z^2} \eucmoll(x + 2\pi n), \quad \moll_{\leq N}(x) := \sum_{n \in \Z^2} \eucmoll_{\leq N}(x + 2\pi n).
\end{equs}
Define also $\moll_N := \moll_{\leq N} - \moll_{\leq N/2}$.
\end{definition}
\begin{remark}[Properties of $\eucmoll$ and $\moll$]
We observe the following properties of $\eucmoll$ and $\moll$.
\begin{enumerate}[label=(\alph*)]
\item Since the Fourier transform on $\R^2$ preserves Schwartz functions and radial functions, $\eucmoll$ is Schwartz and radial. 
\item\label{item:fourier-transform-of-mollifier} As a consequence of the Fourier inversion formula \eqref{eq:fourier-inversion}, we have that
\begin{equs}\label{eq:fourier-inversion-rho}
\rho(\xi) = \int_{\R^2} dx \eucmoll(x) e^{\icomplex \xi \cdot x}.
\end{equs}
Since $\rho(0) = 1$, we have that
$\int \dx \eucmoll(x) = \rho(0) = 1$, and so the same is true for $\eucmoll_{\leq N}$, $\moll$, $\moll_{\leq N}$.
\item From \eqref{eq:fourier-inversion-rho}, and the fact that $\rho$ is radial, we have that $\widehat{\moll_{\leq N}} = \frac{1}{(2\pi)^2} \rho_{\leq N}$. Consequently, for $N \in \dyadic$, $\moll_{\leq N}$ (resp. $\moll_N$) is the convolution kernel of the Littlewood-Paley operator $P_{\leq N}$ (resp. $P_N$), to be defined in \eqref{prelim:eq-LWP}.
\end{enumerate}
\end{remark}

Finally, we define the Littlewood-Paley operators $(P_{\leq N})_{N \in \dyadic}$ and $(P_{N})_{N \in \dyadic}$ by 
\begin{equation}\label{prelim:eq-LWP}
\widehat{P_{\leq N} \phi}(n) = \rho_{\leq N}(n) \widehat{\phi}(n) 
\qquad \text{and} \qquad
\widehat{P_{N} \phi}(n) = \rho_{N}(n) \widehat{\phi}(n)
\end{equation}
for all distributions $\phi \colon \T^2 \rightarrow \C$ and all $n\in \Z^2$. Note that $\moll_{\leq N}$ and $\moll_N$ are precisely the convolution kernels of $P_{\leq N}$ and $P_N$, respectively.

\subsubsection{Function spaces}

We now introduce the function spaces which will be used throughout the article. For reasons that will be discussed below (see Remark \ref{prelim:rem-decay-fail}), we need to treat the mean and mean-zero components of functions separately and we need to be careful with the precise definitions of our norms. 
For any distribution $\phi\colon \T^2 \rightarrow \C$, we define the mean and mean-zero components by 
\begin{equation}\label{prelim:eq-mean-mean-zero}
\sfint \phi := \frac{1}{(2\pi)^2} \int_{\T^2} \dx \phi(x) \qquad \text{and} \qquad 
\bigdot{\phi} := \phi - \sfint \phi,
\end{equation}
respectively. For $1\leq p \leq \infty$, we define the $L_x^p$-norms as
\begin{alignat*}{3}
\| f \|_{L_x^p} &:= \Big( \int_{\T^2} \dx \, |f(x)|^p \Big)^{\frac{1}{p}} \qquad &\qquad \quad \text{if } 1\leq p < \infty, \\
\| f \|_{L_x^p} &:=  \operatorname{ess\hspace{0.2ex}sup}\displaylimits_{x\in \T^2} |f(x)| \qquad &\qquad \quad \text{if } p=\infty.
\end{alignat*}
We note that the $L_x^p$-norms are defined using non-normalized Lebesgue measure and it therefore holds that $\| 1 \|_{L_x^p}=(2\pi)^{\frac{2}{p}}$. In the following, we define the $\Lc^p_x$-norms, which will be equivalent but not necessarily identical to the $L_x^p$-norms. The $\Lc_x^p$-norms and corresponding spaces are purpose-built so that, in the proof of Lemma \ref{prelim:lem-heat-flow-bound-decay}, our interpolation argument yields exact constants. 

\begin{definition}[\protect{$\Lcbigdot^p_x$ and $\Lc^p_x$-norms}]\label{prelim:def-lc}
For $p\in \{1,2,\infty\}$, we define 
\begin{equation}\label{prelim:eq-lc-1-2-inf}
\Lcbigdot^p_x := \Big\{ \psi \in L_x^p\colon \sfint \psi = 0 \Big\} 
\qquad \text{and} \qquad 
\big\| \psi \big\|_{\Lcbigdot^p_x} := \big\| \psi \big\|_{L_x^p}.
\end{equation}
For $1<p<2$, we define 
\begin{equation*}
\Lcbigdot^p_x := \big( \Lcbigdot^1_x , \Lcbigdot^2_x \big)_{\theta_p} \qquad \text{and} \qquad 
\big\| \psi \big\|_{\Lcbigdot^p_x} := \big\| \psi \big\|_{( \Lcbigdot^1_x , \Lcbigdot^2_x )_{\theta_p}},
\end{equation*}
where the space $(\Lcbigdot^1_x , \Lcbigdot^2_x )_\theta$ is the complex interpolation space of $\Lcbigdot^1_x$ and  $\Lcbigdot^2_x$
(see e.g. \cite[Theorem 4.1.2]{BL76}), the norm $\| \cdot \|_{(\Lcbigdot^1_x , \Lcbigdot^2_x)_{\theta}}$ is the corresponding interpolation norm, 
and the parameter $\theta_p\in [0,1]$ is determined by $\frac{1}{p}=\frac{1-\theta_p}{1}+\frac{\theta_p}{2}$. Similarly, for  $2<p<\infty$, we define
\begin{equation*}
\Lcbigdot^p_x := \big( \Lcbigdot^2_x , \Lcbigdot^\infty_x \big)_{\theta_p} \qquad \text{and} \qquad 
\big\| \psi \big\|_{\Lcbigdot^p_x} := \big\| \psi \big\|_{( \Lcbigdot^2_x , \Lcbigdot^\infty_x )_{\theta_p}},
\end{equation*}
where the parameter $\theta_p\in[0,1]$ is determined by $\frac{1}{p}=\frac{1-\theta_p}{2}+\frac{\theta_p}{\infty}$. Furthermore, for all $1\leq p \leq \infty$, we define
\begin{equation}\label{prelim:eq-lc-inhom}
\Lc^p_x := \Big\{ z + \psi \colon z \in \C, \psi \in \Lcbigdot^p_x \Big\} 
\qquad \text{and} \qquad \| z + \psi \|_{\Lc^p_x} := \max \Big( \| z \|_{L_x^p}, \| \psi \|_{\Lcbigdot^p_x} \Big).
\end{equation}
\end{definition}

From the definition, it directly follows that $\Lcbigdot^p_x\subseteq \Lcbigdot^1_x$ for all $p\geq 1$, and hence elements of $\Lcbigdot^p_x$ are mean-zero functions. Using abstract interpolation theory, one readily obtains the following equivalence of norms.

\begin{lemma}[\protect{Equivalence of $L_x^p$ and $\Lc_x^p$-norms}]\label{prelim:lem-lp-equivalence}
For all $1\leq p \leq \infty$, it holds that 
\begin{equation*}
L_x^p = \Lc_x^p \qquad \text{and} \qquad 
\frac{1}{2} \| \cdot \|_{L_x^p} \leq \| \cdot\|_{\Lc_x^p} \leq 2 \| \cdot \|_{L_x^p}.
\end{equation*}
\end{lemma}

\begin{proof}
We first consider $p\in \{1,2,\infty\}$ and define the linear maps
\begin{equation*}
S \colon L^p_x \rightarrow \Lcbigdot^p_x, \, \phi \mapsto \phi - \sfint \phi 
\qquad \text{and} \qquad 
T \colon \Lcbigdot^p_x \rightarrow L^p_x, \, \psi \mapsto \psi. 
\end{equation*}
Using \eqref{prelim:eq-lc-1-2-inf}, one directly obtains that\footnote{We remark that  $\| S \|_{L_x^\infty \rightarrow \Lcbigdot_x^\infty}= 2$, i.e., the above estimate is sharp. To see this, let $\varepsilon>0$ be arbitrarily small and consider any $\phi_\varep\colon \T^2 \rightarrow [-1,1]$ whose mean is $\varepsilon$-close to $1$ and yet still satisfies $\inf_{x\in \T^2} \phi_\varepsilon =-1$.}
\begin{alignat}{5}
\| S \|_{L_x^1 \rightarrow \Lcbigdot_x^1} &\leq 2, \quad & \quad 
\| S \|_{L_x^2 \rightarrow \Lcbigdot_x^2} &\leq 1, \quad & \quad 
\| S \|_{L_x^\infty \rightarrow \Lcbigdot_x^\infty} &\leq 2, \label{prelim:eq-S-bound} \\
\| T \|_{\Lcbigdot_x^1 \rightarrow L_x^1} &\leq 1, \quad & \quad 
\| T \|_{\Lcbigdot_x^2 \rightarrow L_x^2} &\leq 1, \quad & \quad 
\| T \|_{\Lcbigdot_x^\infty \rightarrow L_x^\infty} &\leq 1. \label{prelim:eq-T-bound}
\end{alignat}
Using abstract interpolation theory \cite[Definition 2.4.3, Theorem 4.1.2, and Theorem 5.1.1]{BL76}, it then follows for all $1\leq p \leq \infty$ that $S\colon L_x^p \rightarrow \Lcbigdot_x^p$, $T\colon \Lcbigdot_x^p \rightarrow L_x^p$, 
\begin{equation*}
\|  S \|_{L_x^p \rightarrow \Lcbigdot_x^p} \leq 2, \qquad \text{and} \qquad  \| T \|_{\Lcbigdot_x^p \rightarrow L_x^p} \leq 1. 
\end{equation*}
For all $1\leq p \leq \infty$, it then directly follows that $L_x^p=\C+\Lcbigdot_x^p=\Lc_x^p$ and that, for all $\psi \in L_x^p$ with zero mean, 
\begin{equation*}
 \| \psi \|_{L_x^p} \leq \| \psi \|_{\Lcbigdot_x^p} \leq 2 \| \psi \|_{L_x^p}.
\end{equation*}
Together with \eqref{prelim:eq-lc-inhom} and the triangle inequality, this yields the desired equivalence of norms.
\end{proof}

\begin{definition}[Besov and H\"{o}lder spaces]\label{prelim:def-besov}
For all $\alpha \in \R$, $1\leq p,q\leq \infty$, and $\phi\colon \T^2 \rightarrow \C$, we define the norm 
\begin{equs}\label{prelim:eq-besov}
\big\| \phi \big\|_{\Bc^{\alpha,p,q}_x} 
:= \max \Big( \big| \sfint \phi \big|, \,  \big\| N^\alpha P_N \bigdot{\phi} \big\|_{\ell_N^q \Lcbigdot_x^p} \Big). 
\end{equs}
We define the corresponding function space $\Bc^{\alpha,p,q}_x$ as the closure of $C^\infty_x(\T^2)$ with respect to the $\Bc^{\alpha,p,q}_x$-norm. To simplify the notation, we also define the function spaces $\Bc^{\alpha,p}_x:= \Bc^{\alpha,p,\infty}_x$ and  $\Cs_x^{\alpha}:=\Bc^{\alpha,\infty,\infty}_x$.
\end{definition}

\begin{remark}[\protect{Equivalence of Besov-norms}]
From Lemma \ref{prelim:lem-lp-equivalence} and Definition \ref{prelim:def-besov}, it directly follows that
\begin{equation}\label{prelim:eq-besov-equivalence}
\big\| \phi \big\|_{\Bc^{\alpha,p,q}_x} \sim  \big\| N^\alpha P_N \phi \big\|_{\ell_N^q L_x^p}. 
\end{equation}
The more complicated expression in \eqref{prelim:eq-besov} will only be relevant in Lemma \ref{prelim:lem-heat-flow-bound-decay} and, as a consequence, Proposition \ref{decay:prop-short} below. In all other parts of this article, we simply use the equivalence \eqref{prelim:eq-besov-equivalence}.
\end{remark}

\begin{definition}[H\"{o}lder in time, Besov in space]\label{prelim:lem-hoelder-in-time}
Let $T > 0$, let $\alpha \in [0, 1)$, let $\beta \in \R$, and let $p, q \in [1, \infty]$. For any $\phi\colon [0,T] \times \T^2 \rightarrow \C$, we define the norm 
\begin{equs}\label{prelim:eq-hoelder-in-time}
\|\phi\|_{C_t^\alpha \Bc_x^{\beta, p, q}([0, T] \times \T^2)} := \sup_{t \in [0, T]} \|\phi(t)\|_{\Bc_x^{\beta, p, q}} + \sup_{\substack{s, t \in [0, T] \\ s \neq t}} \frac{\|\phi(t) - \phi(s)\|_{\Bc_x^{\beta, p, q}}}{|t-s|^\alpha}.
\end{equs}
The corresponding function space $C_t^\alpha \Bc_x^{\beta, p, q}([0, T] \times \T^2)$ is defined as the closure of $C^\infty([0,T]\times \T^2)$ with respect to the norm in \eqref{prelim:eq-hoelder-in-time}.
\end{definition}

\begin{remark}[Subscript and superscript notation]
In the literature, Besov-spaces are usually denoted by $\Bc^{\alpha}_{p,q}$ instead of  $\Bc^{\alpha,p,q}_x$. We chose the latter notation because, as in Definition \ref{prelim:lem-hoelder-in-time}, we often encounter expressions such as $C_t^0 \Bc^{\alpha,p,q}_x$
or $L_t^p \Bc^{\alpha,p,q}_x$, in which the subscripts $t$ and $x$ are important.
\end{remark}

\subsubsection{Heat flow estimates}

We now turn to properties involving the Euclidean heat flow $e^{t\Delta}$, where $t\geq 0$. In Section \ref{section:cshe}, we need to estimate various covariant objects. In these estimates, it is convenient to work with a heat-flow characterization of Besov spaces, since it combines nicely with our weighted energy estimates of Section \ref{section:monotonicity}. To this end, we state the following characterization.

\begin{lemma}[Heat kernel characterization of Besov spaces]\label{lemma:besov-space-heat-kernel-characterization}
We have that
\begin{equs}
\| \phi \|_{\Cs_x^{-\alpha}} &\lesssim \sup_{u \in (0, 1]} u^{\frac{\alpha}{2}} \|e^{u \Delta} \phi \|_{L_x^\infty}, ~~ \alpha > 0 \\
\|\phi \|_{\Cs_x^{1-\beta}} &\lesssim 
\|e^{\Delta} \phi \|_{L_x^\infty} + \sup_{u \in (0, 1]} u^{\frac{\beta}{2}} \|\nabla e^{u \Delta} \phi \|_{L_x^\infty}, ~~ \beta \in (0, 1).
\end{equs}
\end{lemma}

We omit the proof of Lemma \ref{lemma:besov-space-heat-kernel-characterization}; see e.g. \cite[Theorem 2.34]{BCD2011} for the proof of a similar estimate. In the next lemma, we state bounds and decay estimates for the Euclidean heat flow. 

\begin{lemma}[Boundedness and decay of heat flow]\label{prelim:lem-heat-flow-bound-decay}
Let $\alpha\in \R$, let $1\leq p\leq \infty$, and let $t\in [0,\infty)$. Furthermore, let $\phi,\psi\colon \T^2 \rightarrow \C$
and assume that $\psi$ has mean-zero, i.e., $\sfint \psi=0$. Then, it holds that
\begin{equation}\label{prelim:eq-heat-flow-bound}
\big\| e^{t\Delta} \phi \big\|_{\Balphap} 
\leq \big\| \phi \big\|_{\Balphap}. 
\end{equation}
Furthermore,  if $1<p<\infty$ and $c_p>0$ is a sufficiently small constant, then it holds that 
\begin{equation}\label{prelim:eq-heat-flow-decay}
\big\| e^{t\Delta} \psi \big\|_{\Balphap} \leq e^{-c_p t} \big\| \psi \big\|_{\Balphap}.
\end{equation} 
Finally, it holds that\footnote{In \eqref{prelim:eq-heat-flow-bound-gc}, we can always assume that $\alpha\leq 0$, since otherwise the $\GCs^\alpha$-norm is only finite for $\phi=0$.}
\begin{equation}\label{prelim:eq-heat-flow-bound-gc}
\big\| e^{t\Delta} \phi \big\|_{\GCs^\alpha} \lesssim \big\| \phi \big\|_{\GCs^\alpha}, 
\end{equation}
where $\GCs^{\alpha}$ is as in \eqref{intro:eq-gauge-invariant-norms}.
\end{lemma}

\begin{proof}
We first prove \eqref{prelim:eq-heat-flow-bound} and \eqref{prelim:eq-heat-flow-decay}. Since $e^{t\Delta}$ preserves the mean and commutes with Littlewood-Paley projections, it holds that
\begin{equation*}
\big\| e^{t\Delta} \phi \big\|_{\Bc^{\alpha,p}_x}
= \max \Big( \big| \sfint \phi \big|, \big\| N^\alpha \big\| e^{t\Delta} P_N \bigdot{\phi} \big\|_{\Lcbigdot^p_x} \big\|_{\ell^\infty_N} \Big).
\end{equation*}
It therefore suffices to show that, for all mean-zero $\psi$, we have 
\begin{equation*}
\big\| e^{t\Delta} \psi \big\|_{\Lcbigdot^p_x} \leq \big\|  \psi \big\|_{\Lcbigdot^p_x} 
\qquad \text{and, if $1<p<\infty$, that } \quad \big\| e^{t\Delta} \psi \big\|_{\Lcbigdot^p_x} \leq e^{-c_p t} \big\|  \psi \big\|_{\Lcbigdot^p_x}.
\end{equation*}
By abstract interpolation theory \cite[Theorem 4.1.2]{BL76} and Definition \ref{prelim:def-lc}, it then suffices to prove that, for all mean-zero $\psi$, 
\begin{equation*}
\big\| e^{t\Delta} \psi \big\|_{\Lcbigdot^1_x} \leq \big\|  \psi \big\|_{\Lcbigdot^1_x}, \qquad      
\big\| e^{t\Delta} \psi \big\|_{\Lcbigdot^2_x} \leq e^{-ct} \big\|  \psi \big\|_{\Lcbigdot^2_x}, \qquad      \text{and} \qquad 
\big\| e^{t\Delta} \psi \big\|_{\Lcbigdot^\infty_x} \leq \big\|  \psi \big\|_{\Lcbigdot^\infty_x},
\end{equation*}
where $c>0$ is an absolute constant. By Definition \ref{prelim:def-lc}, the $\Lcbigdot_x^p$-norms are equal to $L_x^p$-norms for $p=1,2,\infty$. It therefore suffices to prove that, for all mean-zero $\psi$,
\begin{equation}\label{prelim:eq-heat-p1}
\big\| e^{t\Delta} \psi \big\|_{L^1_x} \leq \big\|  \psi \big\|_{L^1_x}, \qquad      
\big\| e^{t\Delta} \psi \big\|_{L^2_x} \leq e^{-ct} \big\|  \psi \big\|_{L^2_x}, \qquad      \text{and} \qquad 
\big\| e^{t\Delta} \psi \big\|_{L^\infty_x} \leq \big\|  \psi \big\|_{L^\infty_x}.
\end{equation}
The first and third estimate in \eqref{prelim:eq-heat-p1} are satisfied since the heat kernel is $L_x^1$-normalized and both estimates do not rely on the mean-zero condition. In contrast, the second estimate in \eqref{prelim:eq-heat-p1} relies on the mean-zero condition. It can be easily obtained using the Fourier-transform and Plancherel's formula.

It remains to prove the $\GCs^\alpha$-estimate \eqref{prelim:eq-heat-flow-bound-gc}. Since \eqref{prelim:eq-heat-flow-bound-gc} involves an implicit constant, the precise definition of the $\Bc^{\alpha,p}_x$-norms is less important than in the proof of \eqref{prelim:eq-heat-flow-bound} and \eqref{prelim:eq-heat-flow-decay}. In order to prove \eqref{prelim:eq-heat-flow-bound-gc}, we let $P_{N;n}$ be the Littlewood-Paley projection to frequencies at a distance $\sim N$ from $n\in \Z^2$, i.e., we let $P_{N;n}$ be the Fourier-multiplier with symbol $\rho_N(\cdot-n)$. For any $n\in \Z^2$ and $N\in \dyadic$, it holds that 
\begin{align*}
    &\, \big\| P_N \big( e^{-\icomplex n x} e^{t\Delta}\phi \big) \big\|_{L_x^\infty}
    = \big\| e^{\icomplex n x} P_N \big( e^{-\icomplex n x} e^{t\Delta}\phi \big) \big\|_{L_x^\infty} \\ 
    =&\,  \big\| P_{N;n} e^{t\Delta}\phi \big\|_{L_x^\infty}
    = \big\|  e^{t\Delta} P_{N;n} \phi \big\|_{L_x^\infty}
    \leq \big\|  P_{N;n} \phi \big\|_{L_x^\infty}
    \leq N^{-\alpha} \big\| \phi \big\|_{\GCs^{\alpha}}.
\end{align*}
By taking a supremum over $n\in \Z^2$ and $N\in \dyadic$ and using \eqref{prelim:eq-besov-equivalence}, this directly implies \eqref{prelim:eq-heat-flow-bound-gc}. 
\end{proof}

\begin{remark}\label{prelim:rem-decay-fail} 
In the proof of Proposition \ref{decay:prop-short} below, the absolute constant $C=1$ in \eqref{prelim:eq-heat-flow-decay} is crucial, i.e., we cannot make use of the estimate $\| e^{t\Delta} \psi\|_{\Balphap}\leq C_p e^{-c_p t} \| \psi\|_{\Balphap}$ for constants $C_p>1$. The reason is that we use \eqref{prelim:eq-heat-flow-decay} to obtain non-trivial decay over short time scales, and therefore the $C_p$-loss would 
outweigh the $e^{-c_pt}$-gain. This is the reason for introducing the $\Lcbigdot^p_x$-norms from Definition \ref{prelim:def-lc}. 

The definition of the $\Lcbigdot^p_x$-norms could be avoided if, for $1<p<\infty$ and mean-zero $\psi\colon \T^2 \rightarrow \C$, it holds that
\begin{equation}\label{prelim:eq-lp-decay-question}
\big\| e^{t\Delta} \psi \big\|_{L_x^p} \leq e^{-c_p t} \big\| \psi\big\|_{L_x^p}.
\end{equation}
Unfortunately, it is unclear to the authors whether \eqref{prelim:eq-lp-decay-question} holds and, if it holds, how it can be proven. We were unable to obtain \eqref{prelim:eq-lp-decay-question} from the Riesz-Thorin interpolation theorem since, if it is applied to the operator $e^{t\Delta}-\sfint e^{t\Delta}$, then \eqref{prelim:eq-S-bound} leads to a loss of a constant factor. We were also unable to obtain \eqref{prelim:eq-lp-decay-question} from a monotonicity formula for $\| e^{t\Delta} \psi\|_{L_x^p}^p$ since it is unclear to us whether $\int_{\T^2}  |\psi|^{p-2} |\nabla \psi|^2 \dx$ controls $\int_{\T^2}  |\psi|^{p}\dx $. 
We also remark that \eqref{prelim:eq-lp-decay-question} fails for $p=\infty$.  To see this, 
let $\psi\colon \T^2 \rightarrow [-1,1]$ be any mean-zero function satisfying $\psi(x)=1$ for all $|x|\leq\frac{\pi}{8}$. Using the explicit formula for the Euclidean heat kernel, it is then easy to show for all $0\leq t \leq 1$ that
\begin{equation}\label{prelim:eq-step-function-heat}
\big( e^{t \Delta} \psi \big)(0) \geq 1 - C \exp \big( - c \, t^{-1} \big), 
\end{equation}
where $C\geq 1$ and $c>0$ are absolute constants. From \eqref{prelim:eq-step-function-heat}, one then sees that $\| e^{t\Delta} \psi\|_{L_x^\infty}$ initially decays much slower than $e^{-c^\prime t}$ for any constant $c^\prime>0$.  
\end{remark}

Next, we state a classical smoothing property of the heat flow.

\begin{lemma}[Smoothing of the heat flow]\label{lemma:heat-flow-smoothing}
There is a constant $C$ such that for all $\alpha \leq \beta \in \R$, $p, q\in [1, \infty]$, $t \in (0, 1]$, we have that
\begin{equs}
\|e^{t\Delta} \phi\|_{\mc{B}_x^{\beta, p, q}} \leq C (\beta-\alpha)^{\frac{\beta-\alpha}{2}} t^{-\frac{(\beta-\alpha)}{2}} \|\phi\|_{\mc{B}_x^{\alpha, p, q}}.
\end{equs}
\end{lemma}

We omit the proof of Lemma \ref{lemma:heat-flow-smoothing}; it follows from \cite[Lemma 2.4]{BCD2011} and \eqref{prelim:eq-besov-equivalence}. We now turn to estimates of the Duhamel integral corresponding to the Euclidean heat flow. For any $F\colon [0,\infty) \times \T^2 \rightarrow \C$, we define its Duhamel integral
\begin{equation}\label{prelim:eq-Duhamel}
\Duh \big[ F \big](t) := \int_0^t \ds \,  e^{(t-s) \Delta} F(s). 
\end{equation}
In order to state our estimate of \eqref{prelim:eq-Duhamel}, we also let $p(s,y;t,x)$ be the Euclidean heat kernel, i.e., the integral kernel of $e^{(t-s)\Delta}$. Using the Euclidean heat kernel, the Duhamel integral in \eqref{prelim:eq-Duhamel} can also be written as 
\begin{equation*}
\Duh \big[ F \big](t,x) 
:= \int_0^t \int_{\T^2} \ds \dy \, p(s,y;t,x) F(s,y). 
\end{equation*}

\begin{lemma}[Duhamel integral estimate with heat-kernel weights]\label{prelim:lem-Duhamel-weighted}
Let $I \subseteq [0,1]$ be a compact interval. Furthermore, let 
$\alpha \in [0,1]$, $\theta \in [0,1)$, and $q,r \in [0,1]$ satisfy
\begin{equation}\label{prelim:eq-Duhamel-weighted-assumption}
\frac{\alpha}{2} + \frac{1}{q} + \frac{1}{r} < 1 + \theta. 
\end{equation}
Then, with $\gamma = 1+\theta - (\frac{\alpha}{2} + \frac{1}{q} + \frac{1}{r})$, it holds that 
\begin{equation}\label{prelim:eq-Duhamel-weighted}
\big\| \Duh \big[ F \big] \big\|_{C_t^0 C_x^\alpha(I\times \T^2)} + \big\| \Duh \big[ F \big] \big\|_{C_t^{\frac{\alpha}{2}} C_x^0(I\times \T^2)}
\lesssim |I|^{\gamma} \big\| p^\theta(s,y;t,x) F(s,y) \big\|_{L_t^\infty L_x^\infty L_s^q L_y^r(I\times \T^2 \times I \times \T^2)}.
\end{equation}
\end{lemma}

The weighted norms on the right-hand side of \eqref{prelim:eq-Duhamel-weighted} are used since they also appear in our covariant monotonicity formula for the covariant heat equation (see Proposition \ref{prop:monotonicity}). Since the proof of Lemma \ref{prelim:lem-Duhamel-weighted} cannot easily be found in the literature, we provide it in Appendix \ref{appendix:misc}. 

\subsubsection{Leray projection and para-product operators}
As discussed in the introduction, we primarily work with connection one-forms in the Coulomb gauge, which leads to the Leray-projection $\leray$ in \eqref{intro:eq-SAH} and \eqref{intro:eq-SAH-smooth}. For any $A\colon \T^2 \rightarrow \R^2$, we define its Leray-projection $\leray A$ as 
\begin{equation}\label{prelim:eq-leray}
\leray A := A - \nabla \big( \Delta^{-1} \nabla \cdot A \big) 
\qquad \text{or, equivalently,} \qquad 
\big( \leray A \big)^j := A^j - \Delta^{-1} \partial^j \partial_k A^k. 
\end{equation}
In the following lemma, we study the mapping properties of the Leray-projection $\leray$. 

\begin{lemma}[Mapping properties of Leray-projection]\label{prelim:lem-leray}
Let $p\in (1,\infty)$ and let $\alpha \in \R$. Then, it holds that 
\begin{equation*}
\big\| \leray \big\|_{L^p_x \rightarrow L^p_x} \lesssim 1 
\qquad \text{and} \qquad \big\| \leray \big\|_{\Cs_x^{\alpha}\rightarrow \Cs_x^\alpha}
\lesssim 1. 
\end{equation*}
\end{lemma}

\begin{proof}
The first estimate $\| \leray \|_{L^p_x \rightarrow L^p_x} \lesssim 1$  follows directly from Mikhlin's multiplier theorem. Due to the definition of the $\Cs_x^\alpha$-norm, the second estimate can be reduced to proving that, for all $N\in \dyadic$, $\| \leray P_N \|_{L_x^\infty \rightarrow L_x^\infty}\lesssim 1$. The latter estimate holds trivially since $\leray P_N$ can be written as a Fourier multiplier with symbol $\widetilde{\rho}(\cdot/N)$, where $\widetilde{\rho}\colon \R^2 \rightarrow \R$ is a smooth, compactly supported function.
\end{proof}

In the next definition, we introduce para-product operators. The para-product operators will primarily be needed in Section \ref{section:gauge-covariance} and Appendix \ref{section:high}, which rely on the local well-posedness theory of \eqref{intro:eq-SAH}. 

\begin{definition}[Para-product operators]\label{prelim:def-para}
For all smooth functions $\phi,\psi \colon \T^2 \rightarrow \C$, we define
\begin{alignat}{2}
\phi \parall \psi &:= \sum_{\substack{M,N\colon \\ M \ll N}} 
P_M \phi \, P_N \psi, \hspace{10ex}
\phi \parasim \psi &:= \sum_{\substack{M,N\colon \\ M \sim N}} 
P_M \phi \, P_N \psi, \\  
\phi \paragg \psi &:= \sum_{\substack{M,N\colon \\ M \gg N}} 
P_M \phi \, P_N \psi, \hspace{10ex}
\phi \paransim \psi &:= \sum_{\substack{M,N\colon \\ M \not \sim N}} 
P_M \phi \, P_N \psi.
\end{alignat}
\end{definition}

We now state several standard estimates for the para-product operators from the previous definition.

\begin{lemma}[Para-product estimates]\label{prelim:lem-para-besov}
Let $\alpha,\alpha_1,\alpha_2\in \R$ and let $p,p_1,p_2\in [1,\infty]$ satisfy $\frac{1}{p}\geq \frac{1}{p_1}+\frac{1}{p_2}$. Furthermore, let $f,g\colon \T^2 \rightarrow \C$ and let $\Pi\in \big\{ \parall, \parasim, \paragg, \times \big\}$ be a para-product. Then, the estimate 
\begin{equs}
\big\| \Pi(f,g) \big\|_{\Bc^{\alpha,p}_x}\lesssim \big\| f \big\|_{\Bc^{\alpha_1,p_1}_x}
\big\| g \big\|_{\Bc^{\alpha_2,p_2}}
\end{equs}
is satisfied under either of the following four conditions:
\begin{enumerate}[label=(\roman*)]
\item (Low$\times$high-estimate) If $\Pi=\parall$, 
\begin{equs}
\alpha\leq \alpha_2, \qquad \text{and} \qquad \alpha_1+\alpha_2 > \alpha. 
\end{equs}
\item (high$\times$high-estimate) If $\Pi=\parasim$, 
\begin{equs}
\alpha\leq \alpha_1+\alpha_2, \qquad \text{and} \qquad \alpha_1+\alpha_2 > 0. 
\end{equs}
\item (High$\times$low-estimate) If $\Pi=\paragg$, 
\begin{equs}
\alpha\leq \alpha_1, \qquad \text{and} \qquad \alpha_1+\alpha_2 > \alpha. 
\end{equs}
\item\label{prelim:item-product-estimate} (Product-estimate) If $\Pi=\times$, 
\begin{equs}
\alpha\leq \min(\alpha_1,\alpha_2) \qquad \text{and} \qquad \alpha_1+\alpha_2 >0. 
\end{equs}
\end{enumerate}
\end{lemma}

For a proof of (similar estimates as in) Lemma \ref{prelim:lem-para-besov}, see e.g. \cite[Lemma 2.1]{GIP15}. 

\subsection{Probability theory}\label{section:preliminary-probability}

The $\R^2$-valued space-time white noise $\xi = (\xi_1, \xi_2)$ and $\C$-valued space-time white noise $\zeta$
are random distributions such that for any $A \in C^\infty(\R \times \T^2 \rightarrow \R^2)$ and $\phi \in C^\infty(\R\times \T^2 \rightarrow~\C)$ with compact support, we have that 
\begin{equs}\label{eq:white-noise}
\iint \dt \dx A^j(t,x) \xi_j(t,x) \sim \mathsf{N}(0, \|A\|_{L_x^2}^2)
\quad \text{and} \quad 
\iint \dt \dx \phi(t,x) \zeta(t,x) \sim \mathsf{N}_\C\big(0, \|\phi\|_{L_x^2}^2\big).
\end{equs}
Here, $\mathsf{N}(0, \sigma^2)$ is the law of a real-valued Gaussian with mean zero and variance $\sigma^2$ and $\mathsf{N}_\C(0, \sigma^2)$ is the law of $X + \icomplex Y$, where $X, Y \stackrel{i.i.d.}{\sim} \mathsf{N}(0, \sigma^2/2)$. Explicitly, in Fourier space, $\xi$ and $\zeta$ may be represented as:
\begin{equs}\label{eq:space-time-white-noise-fourier}
\xi(t, x) = \frac{1}{2\pi} \sum_{n \in \Z^2} \e_n(x) dW_\xi(t, n) \qquad \text{and} \qquad  \zeta(t, x) = \frac{1}{2\pi} \sum_{n \in \Z^2} \e_n(x) dW_{\zeta}(t, n),
\end{equs}
where $((W_\xi(t, n))_{t \in \R}, n \in \Z^2)$ and $((W_\zeta(t, n))_{t \in \R}, n \in \Z^2)$ are independent of each other and distributed as follows. For each $n \in \Z^2$, $(W_\xi(t, n))_{t \in \R}$ is a standard two-sided $\C^2$-valued Brownian motion. The $(W_\xi(t, n))_{n \in \Z^2}$ are i.i.d., modulo the condition that $W_\xi(t, -n) = \ovl{W_\xi(t, n)}$ (which reflects the fact that $\xi$ is $\R^2$-valued). The processes $(W_\zeta(t, n))_{t \in \R}$ are i.i.d. $\C$-valued Brownian motions.

\begin{remark}
The factor of $\frac{1}{2\pi}$ in \eqref{eq:space-time-white-noise-fourier} is a normalization arising because (1) by definition, the random variable obtained by testing $\e_n$ against $\R$-valued or $\C$-valued spatial white noise has variance $\|\e_n\|_{L_x^2}^2 = (2\pi)^2$ (2) in our definition of the Fourier transform \eqref{eq:fourier-transform}, there is the $\frac{1}{(2\pi)^2}$ factor. Combining these two facts, we see that any Fourier coefficient of $\R$-valued or $\C$-valued spatial white noise has variance $\frac{1}{(2\pi)^2}$.
\end{remark}

We let $(\mc{F}_t)_{t \geq 0}$ be the filtration generated by $\xi, \zeta$, which is the same as the filtration generated by the Brownian motions $((W_\xi(t, n))_{t \in \R}, n \in \Z^2)$ and  $((W_\zeta(t, n)_{t \in \R}, n \in \Z^2)$.

In the next lemma, we state a simple estimate involving moments, which will be used in the proofs of Lemma~\ref{decay:lem-exponential-growth} and Lemma~\ref{decay:lem-decay-unit}.

\begin{lemma}[Comparison of moments]\label{prelim:lem-comparison}
Let $X$ and $Y$ be two non-negative random variables, let $D>0$, and let  $k\geq 1$. Furthermore, assume that there exists events $(E_\lambda)_{\lambda\in \dyadic}$ such that, for each $\lambda\in\dyadic$, 
\begin{equs}\label{prelim:eq-comparison-condition}
\ind_{E_\lambda} X \leq \ind_{E_\lambda} Y + D \lambda^k
\qquad \text{and} \qquad \bP \big( \Omega \backslash E_\lambda \big) \leq \exp \big( - \lambda^2\big). 
\end{equs}
Then, there exists a constant $C_k\geq 1$ such that, for all $p\geq 1$, we have the estimate
\begin{equs}\label{prelim:eq-comparison}
\E \big[ X^p \big]^{\frac{1}{p}} \leq \E \big[ Y^p \big]^{\frac{1}{p}} + C_k D p^{\frac{k}{2}}.
\end{equs}
\end{lemma}

\begin{proof} We define the events $(E_{<\lambda})_{\lambda\in \dyadic}$ by 
\begin{equs}
E_{<1} := \emptyset \qquad \text{and} \qquad E_{<\lambda} = \bigcup_{\substack{ \mu < \lambda\,\, }} E_{\mu}\quad \text{for all } \lambda \geq 2. 
\end{equs}
Using that $\sum_{\lambda\in\dyadic} \ind_{E_{\lambda}\backslash E_{<\lambda}}=1$ almost surely and using the upper bound on $X$ from \eqref{prelim:eq-comparison-condition}, we obtain that 
\begin{equs}\label{prelim:eq-comparison-p1}
  \E \big[ X^p \big]^{\frac{1}{p}}
  = \E \Big[ \Big( \sum_{\lambda \in \dyadic} \ind_{E_{\lambda}\backslash E_{<\lambda}} X \Big)^{p} \Big]^{\frac{1}{p}} 
  \leq \E \Big[ \Big( \sum_{\lambda \in \dyadic} \ind_{E_{\lambda}\backslash E_{<\lambda}} Y  \Big)^{p} \Big]^{\frac{1}{p}}
  + D \E \Big[ \Big( \sum_{\lambda \in \dyadic} \ind_{E_{\lambda}\backslash E_{<\lambda}} \lambda^k  \Big)^{p} \Big]^{\frac{1}{p}}.  
\end{equs}
The first summand in \eqref{prelim:eq-comparison-p1} is equal to $\E[Y^p]^{\frac{1}{p}}$. 
Using the probability estimate from \eqref{prelim:eq-comparison-condition}, the second summand in \eqref{prelim:eq-comparison-p1} can be estimated using 
\begin{equation*}
\E \Big[ \Big( \sum_{\lambda \in \dyadic} \ind_{E_{\lambda}\backslash E_{<\lambda}} \lambda^k  \Big)^{p} \Big]^{\frac{1}{p}} 
\leq 
  \sum_{\lambda\in \dyadic} \lambda^k \bP \big( E_{\lambda}\backslash E_{<\lambda} \big)^{\frac{1}{p}}
  \leq 1 + \sum_{\substack{\lambda\in \dyadic \colon \\ \lambda \geq 2}}
  \lambda^k \exp \Big( - \frac{1}{p} \Big( \frac{\lambda}{2} \Big)^2 \Big) \lesssim_k p^{\frac{k}{2}}. \qedhere
\end{equation*}
\end{proof}

\subsection{Complex Hermite polynomials}\label{section:preliminary-hermite}
In the following, we consider the complex Hermite polynomials  $(H_{m, n}(z, \bar{z}; \sigma^2))_{m, n \geq 0}$. Here, $\sigma^2\geq 0$ is a parameter, which will now be omitted from the notation. 
The complex Hermite polynomials are defined inductively via
\begin{equs}
H_{m+1, n}(z, \bar{z}) &= z H_{m, n}(z, \bar{z}) - n \sigma^2 H_{m, n-1}(z, \bar{z}), \\
H_{m, n+1}(z, \bar{z}) &= \bar{z} H_{m, n}(z, \bar{z}) - m \sigma^2 H_{m-1, n}(z, \bar{z}), \\
H_{0, 0} &\equiv 1, H_{m, 0}(z, \bar{z}) = z^m, H_{0, n}(z, \bar{z}) = \bar{z}^n.
\end{equs}
Note that this recursive definition arises directly from the product formula for multiple stochastic integrals. The recursion has appeared for instance in \cite{Ghanmi2013} which studies the complex Hermite polynomials, see in particular equations \cite[(3.5) and (3.7)]{Ghanmi2013}.  

We will need the following binomial expansion property of the complex Hermite polynomials. Since this exact result may not be so easy to find in the literature, we provide a proof in Appendix \ref{appendix:misc}.

\begin{lemma}\label{lemma:complex-hermite-polynomial-expansion}
For $m, n \geq 0$, we have that
\begin{equs}\label{eq:complex-hermite-polynomial-expansion}
H_{m, n}(z+w, \ovl{z+w}) = \sum_{k_1=0}^m \sum_{k_2 = 0}^n \binom{m}{k_1} \binom{n}{k_2} w^{m-k_1} \bar{w}^{n-k_2} H_{k_1, k_2}(z, \bar{z}).
\end{equs}
\end{lemma}

\section{Covariant monotonicity formula and the covariant heat kernel}\label{section:monotonicity}

This section forms the backbone of the paper. As outlined in Section \ref{section:introduction-argument}, we prove a covariant monotonicity formula (Proposition \ref{prop:monotonicity}). We then derive many consequences of this formula in Section \ref{section:energy-estimates} which are later used in Section \ref{section:cshe} to prove Theorem \ref{intro:thm-cshe}. The covariant monotonicity formula will also be crucially used later in Section \ref{section:Abelian-Higgs} to obtain covariant estimates for the stochastic Abelian-Higgs equations, which will be the basis of our globalization argument in Section \ref{section:decay}  to prove Theorem \ref{intro:thm-abelian-higgs}. Besides the covariant monotonicity formula, in this section we also prove various properties of the covariant heat kernel (to be defined in Section~\ref{section:energy-estimates}).

\subsection{Covariant monotonicity formula}

The following covariant monotonicity formula and its proof are inspired by Hamilton's monotonicity formulas \cite{Ham93}. 

\begin{proposition}[Covariant monotonicity formula]\label{prop:monotonicity}
Let $T > 0$. Let $A \colon [0, T] \times \T^2 \ra \R^2$ be a time-varying connection one-form, let $G \colon [0, T] \times \T^2 \ra \C$, and let $\phi\colon [0,T]\times \T^2 \rightarrow \C$ be a solution to
\begin{equs}
(\ptl_t - \covd_A^j \covd_{A, j}) \phi = G.
\end{equs}
Furthermore, let $K\colon [0,T]\times \T^2 \rightarrow \R$ be a solution to the backwards heat equation, i.e. 
\begin{equs}
(\ptl_t + \Delta) K = 0.
\end{equs}
Then, it holds that 
\begin{equs}
\ptl_t \Big(K \frac{|\phi|^2}{2}\Big) &= -(\Delta K) \frac{|\phi|^2}{2} + K \Delta \Big(\frac{|\phi|^2}{2} \Big) - K |\covd_A\phi|^2 + K \mrm{Re}(\bar{\phi} G).
\end{equs}
More generally, for any $p \geq 2$,
\begin{equs}
\ptl_t \Big(K \frac{|\phi|^{p}}{p} \Big) = -(\Delta K) \frac{|\phi|^{p}}{p} &+ K \Delta\Big(\frac{|\phi|^{p}}{p}\Big) \\
&-K \Big(\frac{p-2}{4} \Big| |\phi|^{\frac{p-4}{2}} \nabla (|\phi|^2)\Big|^2 + |\phi|^{p-2} |\covd_A\phi|^2 - |\phi|^{p-2} \mrm{Re}(\bar{\phi} G)\Big).
\end{equs}
\end{proposition}
\begin{proof}
We first compute (using the Bochner formula from Lemma \ref{prelim:lem-derivatives}.\ref{prelim:item-Bochner})
\begin{equs}
\ptl_t \Big(\frac{|\phi|^2}{2} \Big) = \mrm{Re}\big( \bar{\phi} \ptl_t \phi \big) = \Delta \Big(\frac{|\phi|^2}{2}\Big)  -|\covd_A \phi|^2 + \mrm{Re}(\bar{\phi} G).
\end{equs}
From this, we obtain
\begin{equs}
\ptl_t \Big(\frac{|\phi|^{p}}{p} \Big) = \frac{1}{p} \ptl_t ((|\phi|^2)^{\frac{p}{2}}) &= (|\phi|^2)^{\frac{p-2}{2}} \ptl_t \Big(\frac{|\phi|^2}{2}\Big) = |\phi|^{p-2} \Big(\Delta \Big(\frac{|\phi|^2}{2}\Big)  -|\covd_A \phi|^2 + \mrm{Re}(\bar{\phi} G)\Big) \\
&= \Delta \Big(\frac{|\phi|^{p}}{p}\Big) - \frac{p-2}{4} \Big| |\phi|^{\frac{p-4}{2}} \nabla (|\phi|^2)\Big|^2 - |\phi|^{p-2} |\covd_A\phi|^2 + |\phi|^{p-2} \mrm{Re}(\bar{\phi} G).
\end{equs}
Here, we used that
\begin{equs}
\Delta \Big(\frac{|\phi|^{p}}{p}\Big) = |\phi|^{p-2} \Delta\Big(\frac{|\phi|^2}{2} \Big) + \frac{p-2}{4}\Big| |\phi|^{\frac{p-4}{2}} \nabla (|\phi|^2)\Big|^2.
\end{equs}
The formula for $\ptl_t(K |\phi|^{p}/p)$ now follows by the product rule and the assumption on $K$.
\end{proof}

\subsection{Energy estimates for the covariant heat kernel}\label{section:energy-estimates}

In this subsection, we collect the various energy estimates that will be needed when estimating the covariant stochastic objects in Section \ref{section:cshe}. These energy estimates are the key to obtaining bounds on the covariant stochastic objects which have minimal dependence on the connection.

\begin{notation}[Space-time variables]\label{notation:space-time-variables}
For brevity, we will often write $z = (t, x)$, $w = (s, y)$. For an additional spacetime variable, we will write $v = (t_v, x_v)$. We will always assume that $s < t_v < t$, even if this is not explicitly stated.
\end{notation}

\begin{notation}[Space-time norms]
In this paper, we will use the usual notation $L_t^p L_x^q$ or $L_s^p L_y^q$ to denote space-time norms. When $p = q$, we will often shorten this to $L_z^p$ or $L_w^p$.
\end{notation}

\begin{definition}[Covariant heat kernel]
Given $A \colon [0, T] \times \T^2 \ra \R^2$, let $p_A(w; z)$ be the covariant heat kernel on $[0, T] \times \T^2$, i.e. the fundamental solution to the covariant heat equation $(\ptl_t - \covd_A^j \covd_{A, j}) \phi = 0$. In other words, $p_A(w; z)$ is defined by the following property. For $s \in (0, T)$, and $\phi_0 \in C^\infty(\T^2 \rightarrow \C)$ , let $\phi(z) = \int p_A((s, y); z) \phi_0(y)$ for $z \in [s, T] \times \T^2$. Then $(\ptl_t - \covd_A^j \covd_{A, j}) \phi = 0$ on $(s, T] \times \T^2$, and $\phi(s) = \phi_0$.

We will write $p = p_0$ to be the usual heat kernel, i.e. the covariant heat kernel when $A \equiv 0$. 

We will write covariant derivatives of the covariant heat kernel as $\covd_{A(z)} p_A(w; z)$, which really means that the derivative acts on the $x$ variable in $z$. Similarly, when we write $\covd_{-A(w)} p_A(w; z)$, we mean that the derivative acts on the $y$ variable in $w$.
\end{definition}

\begin{remark}
For continuous $A \colon [0, T] \times \T^2 \ra \R^2$, the existence of the covariant heat kernel $p_A$ follows directly by the Feynman-Kac-It\^{o} formula (Lemma \ref{lemma:feynman-kac-ito-formula}) below.
\end{remark}

We next state the Feyman-Kac-It\^{o} formula, which gives a probabilistic path-integral representation of the covariant heat kernel. This representation will not play a huge role in the paper, but it may help give a concrete feel for the kernel, and we find it intrinsically interesting. This formula is well known in certain parts of the literature, see e.g. \cite{Guneysu2010, Norris1992, Simon2005}. We provide a self-contained proof in \mbox{Appendix \ref{appendix:misc}}, assuming only knowledge of It\^{o} calculus, because (1) the reader unfamiliar with stochastic analysis on manifolds may find it hard to read these papers, (2) these papers technically do not consider the case where the underlying connection is time-varying, though the proof directly extends to this case\footnote{The fact that the proof works even for time-varying connections was also observed in \cite{HasNab2018} in the setting of Ricci flow.}.

\begin{lemma}[Feynman-Kac-It\^{o} formula]\label{lemma:feynman-kac-ito-formula}
Let $T > 0$, $A \colon [0, T] \times \T^2 \ra \R^2$ be continuous, and further suppose that $\ptl_j A^j : [0, T] \times \T^2 \ra \R^2$ is continuous. For $w, z \in [0, T] \times \T^2$, let $\E_{w \ra z}$ be expectation with respect to a Brownian bridge $W$ in $\T^2$ of rate\footnote{I.e. $dW_t^i dW_t^i = 2 dt$ for $i \in [2]$.} $2$ which goes from $y$ to $x$ on the interval $[s, t]$. Then, we have that 
\begin{equs}
p_A(w; z) = p(w; z) \E_{w \ra z} \bigg[ \exp\bigg(-\icomplex \int_{s}^t A(u, W_{u}) \cdot dW_{u} - \icomplex \int_s^t (\ptl_j A^j)(u, W_u) du \bigg)\bigg],
\end{equs}
where $\int_s^t A(u, W_u) \cdot dW_u = \int_s^t A_j(u, W_u) \cdot dW_u^j$ is an It\^{o} integral.
\end{lemma}

\begin{remark}\label{remark:covariant-heat-kernel-ito-stratonovich}
In the case that $A$ is sufficiently regular (say, $C^1$ in both space and time), we have by It\^{o}-Stratonovich conversion that
\begin{equs}
\int_{s}^t A(u, W_{u}) \cdot dW_{u} + \int_s^t (\ptl_j A^j)(u, W_u) du = \int_s^t A(u, W_u) \circ dW_u,
\end{equs}
where the right hand side is a Stratonovich integral. (Note the lack of the usual $\frac{1}{2}$ factor in the finite-variation term on the left hand side since we took $W$ to be rate-$2$.) In some sense, the Stratonovich integral is the better stochastic integral to use here, because it is more geometric. On the other hand, given our minimal regularity assumptions on $A$, it is less clear whether the Stratonovich integral is well-defined, while the It\^{o} integral is well-defined by the usual $L^2$ isometry arguments. We note here for later use that if $A \equiv a \in \R^2$ is constant, then by the properties of the Stratonovich integral, we have that
\begin{equs}
\int_s^t A(u, W_u) \circ dW_u = a \cdot (W_t - W_s).
\end{equs}
\end{remark}

\begin{definition}[Massive covariant heat kernel]
Given $A \colon [0, T] \times \T^2 \ra \R^2$, let $\massp_A(w; z)$ be the massive covariant heat kernel on $[0, T] \times \T^2$, i.e. the fundamental solution to the massive covariant heat equation $(\ptl_t - \covd_A^j \covd_{A, j} + 1) \phi = 0$. 
\end{definition}

\begin{remark}\label{remark:massive-kernel}
We have the basic relation
\begin{equs}\label{eq:massive-heat-kernel-formula}
\massp_A(w; z) = e^{-(t-s)} p_A(w; z).
\end{equs}
In this section, most estimates we prove will be stated with the non-massive kernel $p_A$. Due to relation \eqref{eq:massive-heat-kernel-formula}, all results will transfer over in a trivial way to $\massp_A$. We prefer to prove the results with $p_A$ to make it clear that the massive term is not needed for such estimates. 

The massive kernel $\massp_A$ will only really enter in Section \ref{section:cshe}, since the covariant linear object $\philinear[A, \leqN]$ was defined as a solution to the massive covariant stochastic heat equation (recall equation \eqref{intro:eq-cshe}). The fact that the results of this section are proven with $p_A$ means that the estimates that we will later prove for $\philinear[A, \leqN]$ and other objects in Section \ref{section:cshe} would still hold if we had defined $\philinear[A, \leqN]$ using the non-massive equation. The massive term is only needed in Sections \ref{section:Abelian-Higgs} and \ref{section:decay}, where it becomes technically convenient.
\end{remark}

\begin{lemma}[Diamagnetic inequality for the heat kernel]\label{kernel:lem-diamagnetic-heat-kernel}
Let $T > 0$ and $A\colon [0,T] \times \R^2 \rightarrow \R^2$ be continuous and $C^1$ on $(0, T] \times \T^2$. For all $w, z \in [0, T] \times \T^2$, it holds that
\begin{equs}
\big| p_A(w;z) \big| \leq p(w;z). 
\end{equs}
\end{lemma}
\begin{proof}
This follows directly from the Feynman-Kac-It\^{o} formula (Lemma \ref{lemma:feynman-kac-ito-formula}). Alternatively, one may use the monotonicity formula (Proposition \ref{prop:monotonicity}) and then standard arguments, see e.g. the proof of \cite[Proposition 2.7]{CG2015}.
\end{proof}

\begin{notation}
Let $I \sse \R$ and $J := I \times \T^2$. For space-time functions $\phi_1, \psi_1 \colon J \ra \C$, and spatial functions $\phi_2, \psi_2 \colon \T^2 \ra \C$, we write 
\begin{equs}
(\phi_1, \psi_1)_{L_z^2(J)} &:= \int_J \dz \phi_1(z) \ovl{\psi_1(z)}, \quad \quad (\phi_2, \psi_2)_{L_x^2} := \int_{\T^2} \dx \phi_2(x) \ovl{\psi_2(x)}.
\end{equs}
\end{notation}

\begin{lemma}[Time reversal and the backwards equation]\label{lemma:heat-kernel-time-reversal}
We have that $\ovl{p_A(\cdot; z)}$ is the fundamental solution to the backwards covariant heat equation, i.e.
\begin{equs}
(\ptl_s + \covd_{A(w)}^j \covd_{A(w), j}) \ovl{p_A(w; z)} = 0.
\end{equs}
\end{lemma}
\begin{proof}
Let $\phi, \psi \colon [s, t] \times \T^2 \ra \C$, and integrate by parts to obtain
\begin{equs}
((\ptl_t - \covd_A^j \covd_{A, j}) \phi, \psi) = (\phi, -(\ptl_t + \covd_A^j \covd_{A, j}) \psi) + (\phi(t), \psi(t))_{L_x^2} - (\phi(s), \psi(s))_{L_x^2}.
\end{equs}
Now, let $K_A(w; z)$ denote the fundamental solution to the backwards covariant heat equation, so that with $z$ fixed, 
\begin{equs}
-(\ptl_s + \covd_{A(w)}^j \covd_{A(w), j}) K_A(w; z) = 0.
\end{equs}
Fix $w = (s, y)$, $z = (t, x)$ with $s < t$. Let $\phi(\cdot) = p_A(w; \cdot)$ and $\psi(\cdot) = K_A(\cdot; z)$. By the integration by parts identity, we formally have that
\begin{equs}
0 = (\phi(t), \psi(t))_{L_x^2} - (\phi(s), \psi(s))_{L_x^2} = p_A(w; z) - \ovl{K_A(w; z)},
\end{equs}
which is the desired result. To make this argument rigorous, one needs a limiting argument to handle the singularities of $p_A$ and $K_A$ at their endpoints, see e.g. the proof of \cite[Lemma 2.9(2)]{HeinNab2014}.
\end{proof}

\begin{remark}
An alternative probabilistic way to prove Lemma \ref{lemma:heat-kernel-time-reversal} would be to start from the Feynman-Kac-It\^{o} formula and use time-reversal properties of the Stratonovich integral. 
\end{remark}

Let $I = [t_0, t_1] \sse \R$ be a compact interval. For brevity, let $\spacetime := I \times \T^2$. In the following, suppose that $\phi \colon J \ra \C$ is a solution to
\begin{equs}
(\ptl_t - \covd_A^j \covd_{A, j}) \phi  = G, ~~ \phi(t_0) = 0.
\end{equs}

We first prove a relatively simple energy estimate, and then later turn to more refined energy estimates.

\begin{lemma}\label{kernel:lem-energy-estimate}
Let $K$ be a positive solution to the backwards heat equation on $I \times \T^2$. We have that
\begin{equs}
\|K^{1/2}\phi\|_{L_t^\infty L_x^2(\spacetime)} + \|K^{1/2} \covd_A \phi\|_{L_t^2 L_x^2(\spacetime)} \lesssim \|K^{1/2} G\|_{L_t^1L_x^2(\spacetime)}.
\end{equs}
\end{lemma}
\begin{proof}
By the monotonicity formula (Proposition \ref{prop:monotonicity}), we have that
\begin{equs}
\frac{1}{2} \|K^{1/2} \phi\|_{L_t^\infty L_x^2(J)}^2 + \|K^{1/2} \covd_A \phi \|_{L_t^2L_x^2(J)}^2 \leq 2 \|K \mrm{Re}(\bar{\phi} G)\|_{L_t^1 L_x^1(J)}.
\end{equs}
The result follows by estimating
\begin{equs}
\|K \mrm{Re}(\bar{\phi} G)\|_{L_t^1 L_x^1(J)} \leq \frac{1}{8} \|K^{1/2} \phi\|_{L_t^\infty L_x^2(J)}^2 + 2 \|K^{1/2} G\|_{L_t^1 L_x^2(J)}^2
\end{equs}
and using a kick-back argument. 
\end{proof}

To estimate the various covariant stochastic objects in Section \ref{section:cshe}, we will not directly use the energy estimate from Lemma \ref{kernel:lem-energy-estimate}, but rather we will need the dual version of this estimate, which we state next. Similarly, we will need to dualize the more refined energy estimates that we later prove in this subsection.

\begin{corollary}[Dual version of energy estimate]\label{kernel:cor-energy-estimate-dual}
Let $I\subseteq [0,\infty)$ be a compact interval, let $J:= I \times \T^2$, and let $K\colon J  \rightarrow (0,\infty)$ be a positive solution of the backwards heat equation. Then, it holds for all $H\colon J \rightarrow \C$ that
\begin{equs}\label{kernel:eq-energy-estimate-dual}
\bigg\| K^{-\frac{1}{2}}(w) \int dz \,  \covd_A p_A(w;z) H(z) \bigg\|_{L_s^\infty L_y^2(J)} 
\lesssim \Big\| K^{-\frac{1}{2}}(z) H(z) \Big\|_{L_z^2(J)}.
\end{equs}
\end{corollary}

\begin{proof}
The estimate \eqref{kernel:eq-energy-estimate-dual} follows from Lemma \ref{kernel:lem-energy-estimate} and duality. To be more precise, 
let $G\in L_s^1 L_y^2(J \rightarrow \C)$. Then, using Cauchy-Schwarz and Lemma \ref{kernel:lem-energy-estimate}, it holds that 
\begin{align*}
&\,\Big| \int \dw \, G(w) \Big( K^{-\frac{1}{2}}(w) \int \dz \, \covd_A p_A(w;z) H(z) \Big) \Big| \\
=&\, \Big| \int \dz\,  H(z) \Big( \int \dw\,  \covd_A p_A(w;z) K^{-\frac{1}{2}}(w) G(w) \Big| \\ 
\leq&\, \big\| K^{-\frac{1}{2}}(z) H(z) \big\|_{L_z^2(J)} \Big\| K^{\frac{1}{2}}(z) \int \dw \covd_A p_A(w;z) \, K^{-\frac{1}{2}}(w) G(w) \Big\|_{L_z^2} \\
\lesssim&\, \big\| K^{-\frac{1}{2}}(z) H(z) \big\|_{L_z^2(J)} \big\| G(w) \big\|_{L_s^1 L_y^2(J)}.
\end{align*}
By taking a supremum over all $G\in L_s^1 L_y^2(J\rightarrow \C)$ satisfying $\| G\|_{L_s^1 L_y^2(J)}\leq 1$, we then obtain the desired estimate \eqref{kernel:eq-energy-estimate-dual}.
\end{proof}

Next, we state and prove an energy estimate involving two covariant derivatives of the covariant heat kernel.

\begin{lemma}\label{lemma:two-derivative-energy-estimate}
Let $I\subseteq [0,\infty)$ be a compact interval, let $J:= I \times \T^2$, 
and let $H\colon J \rightarrow \C$. Then, we have that 
\begin{equs}
\bigg\| \int dz H(z) \covd_{-A(w)} \covd_{A(z)} p_A(w; z)\bigg\|_{L_w^2(J)} \lesssim \|H(z)\|_{L_z^2(J)}.
\end{equs}
\end{lemma}
\begin{proof}
Fix $k \in [2]$. Let $G\in L_w^2(J)$ and let $\phi \colon J \rightarrow \C^2$ be the solution of the covariant heat equation
$(\partial_t - \covd_A^j \covd_{A,j}) \phi = \covd_{A}^k G$. Using the monotonicity formula (Proposition \ref{prop:monotonicity}), 
we have that 
\begin{equs}
\frac{1}{2} \|\phi \|_{L_t^\infty L_x^2(J)}^2 + \|\covd_A \phi\|_{L_z^2(J)}^2 \leq 2 \big| (\phi, \covd_{A}^k G)_{L_z^2(J)} \big| = 2 \big| (\covd_{A}^k \phi, G)_{L_z^2(J)} \big| \leq \frac{1}{2}\|\covd_A^k \phi\|_{L_z^2(J)}^2 + 2 \|G\|_{L_w^2(J)}^2.
\end{equs}
Using a kick-back argument, this implies that $\| \covd_A \phi \|_{L_z^2(J)} \lesssim \| G \|_{L_w^2(J)}$.
It then follows that 
\begin{align*}
&\, \Big| \int \dw G(w) \int \dz H(z) \covd_{-A(w)}^k \covd_{A(z)} p_A(w; z)  \Big| \\ 
=&\, \Big| \int \dw \covd_{A(w)}^k G(w) \int \dz H(z) \covd_{A(z)} p_A(w;z) \Big| 
= \Big| \int \dz H(z) \covd_{A(z)} \phi(z) \Big| \\
\lesssim&\, \big\| H(z) \big\|_{L_z^2(J)} \big\| \covd_A \phi (z) \big\|_{L_z^2(J)}
\lesssim \big\| H(z) \big\|_{L_z^2(J)} \big\| G \big\|_{L_w^2(J)}. 
\end{align*}
After taking a supremum over all $G\in L_w^2(J)$ satisfying $\| G \|_{L_w^2(J)}\leq 1$, we obtain the desired estimate. 
\end{proof}

The following result is a more refined version of Lemma \ref{kernel:lem-energy-estimate}, which includes time weights at both endpoints. These time weights provide more flexibility in the estimates, which will be crucial for obtaining the right bounds in Section \ref{section:cshe}.

\begin{lemma}[Two-sided time-weighted estimate]\label{lemma:two-sided-time-weighted-estimate}
Let $\alpha, \beta > 0$, let $\lambda \in [0, \frac{1}{2}]$, and let $J = [t_0, t_1] \times \T^2$.  Suppose $K \in C^\infty([t_0, t_1) \times \T^2)$ is a positive solution to the backwards heat equation on $J$, i.e. that $(\ptl_t + \Delta)K = 0$ and $K > 0$. Then for any $G \colon J \ra \C$, we have that
\begin{equs}\label{eq:two-sided-time-weighted-estimate}
\bigg\| (t - t_0)^{-\frac{\alpha}{2}} (t_1 - t)^{\frac{\beta}{2}} K^{1/2}(z) &\int dw \covd_{A(z)} p_A(w; z) G(w) \bigg\|_{L_z^2(J)} \leq \\
&\min(\alpha, \beta)^{-\frac{1}{2}} \bigg\| (s - t_0)^{-\frac{\alpha}{2} + \lambda} (t_1 - s)^{\frac{\beta}{2} + \frac{1}{2}-\lambda} K^{1/2}(w) G(w) \bigg\|_{L_w^2(J)}.
\end{equs}
\end{lemma}
\begin{proof}
Let $(\ptl_t - \covd_A^j \covd_{A, j}) \phi = G$, $\phi(t_0) = 0$. By the monotonicity formula (Proposition \ref{prop:monotonicity}), we have that
\begin{equs}
\ptl_t \big(\|K^{1/2} \phi(t)\|_{L_x^2}^2\big) = -2 \|K^{1/2} \covd_A \phi(t)\|_{L_x^2}^2 + (K^{1/2} G(t), K^{1/2} \phi(t) )_{L_x^2}, ~~ t \in (t_0, t_1).
\end{equs}
Define the weight $v(t) := (t - t_0)^{-\alpha} (t_1 - t)^{\beta}$. We may then compute
\begin{equs}
\ptl_t \Big( v(t) \|K^{1/2} \phi(t)\|_{L_x^2}^2\Big) = &-2 v(t) \|K^{1/2} \covd_A \phi(t)\|_{L_x^2}^2 + v(t) (K^{1/2} G(t), K^{1/2} \phi(t) )_{L_x^2} \\
&- v(t) \|K^{1/2} \phi(t)\|_{L_x^2}^2 \big(\alpha (t-t_0)^{-1} + \beta (t_1 - t)^{-1}\big).
\end{equs}
Let $\varep > 0$ and let $J_\varep := [t_0, t_1 - \varep] \times \T^2$. Upon integrating over $t \in (t_0, t_1 - \varep)$, we obtain
\begin{equs}
2 \Big\|v^{1/2} K^{1/2} \covd_A \phi\Big\|_{L_z^2(J_\varep)}^2 + \Big\| (\alpha(t-t_0)^{-1} + \beta (t_1 - t)^{-1})^{\frac{1}{2}} v^{1/2} K^{1/2} \phi\Big\|_{L_z^2(J_\varep)}^2 = \Big(v^{1/2} K^{1/2} G, v^{1/2} K^{1/2} \phi\Big)_{L_z^2(J_\varep)}.
\end{equs}
Here, we assumed that 
\begin{equs}\label{eq:singular-time-weight-initial-time-is-still-zero}
\lim_{t \downarrow t_0} v(t) \|K^{1/2} \phi(t)\|_{L_x^2}^2 = 0,
\end{equs}
which will be shown at the end of this proof.  Moving on, we estimate
\begin{equs}
\Big(v^{1/2}& K^{1/2} G, v^{1/2} K^{1/2} \phi\Big)_{L_z^2(J_\varep)}  \\
&\leq\Big\| v^{1/2} (s - t_0)^{\lambda} (t_1 - s)^{\frac{1}{2}-\lambda} K^{1/2} G\Big\|_{L_w^2(J_\varep)} \Big\| v^{1/2} (t - t_0)^{-\lambda} (t_1 - t)^{-(\frac{1}{2}-\lambda)} K^{1/2} \phi\Big\|_{L_z^2(J_\varep)} \\
&\leq \frac{1}{2 \min(\alpha, \beta)} \Big\| v^{1/2} (s - t_0)^{\lambda} (t_1 - s)^{\frac{1}{2}-\lambda} K^{1/2} G\Big\|_{L_w^2(J_\varep)}^2 \\
&~~~~~~~~+ \frac{1}{2}\min(\alpha, \beta) \Big\| v^{1/2} (t - t_0)^{-\lambda} (t_1 - t)^{-(\frac{1}{2}-\lambda)} K^{1/2} \phi\Big\|_{L_z^2(J_\varep)}^2.
\end{equs}
By estimating
\begin{equs}
(t - t_0)^{-2\lambda} (t_1 - t)^{-(1-2\lambda)} \leq 
(t - t_0)^{-1} + (t_1 - t)^{-1}, 
\end{equs}
we may bound
\begin{equs}
\min(\alpha, \beta) \Big\| v^{1/2} (t - t_0)^{-\lambda} (t_1 - t)^{-(\frac{1}{2}-\lambda)} K^{1/2} \phi\Big\|_{L_z^2(J_\varep)}^2 \leq  \Big\| (\alpha(t-t_0)^{-1} + \beta (t_1 - t)^{-1})^{\frac{1}{2}} v^{1/2} K^{1/2} \phi\Big\|_{L_z^2(J_\varep)}^2.
\end{equs}
We thus obtain, for all $\varep > 0$,
\begin{equs}
\Big\| v^{1/2} K^{1/2} \covd_A \phi\Big\|_{L_z^2(J_\varep)}^2 &\leq \frac{1}{4 \min(\alpha, \beta)} \Big\| v^{1/2} (s - t_0)^{\lambda} (t_1 - s)^{\frac{1}{2}-\lambda} K^{1/2} G\Big\|_{L_w^2(J)}^2 .
\end{equs}
(Note the norm on the RHS is over $J$ instead of $J_\varep$.) Taking $\varep \downarrow 0$ and applying monotone convergence, we obtain the desired result.

It remains to show the claim \eqref{eq:singular-time-weight-initial-time-is-still-zero}. It suffices to assume that the norm on $G$ appearing in the RHS of \eqref{eq:singular-time-weight-initial-time-is-still-zero} is finite, because otherwise the lemma is trivial. We note that by the diamagnetic inequality (Lemma \ref{kernel:lem-diamagnetic-heat-kernel}), we have that for $t \in [t_0, t_1]$,
\begin{equs}
|\phi(t,x)| \leq \int_{t_0}^t \ds \int_{\T^2}  \hspace{-0.5ex} \dy \,  \big| p_A(s,y;t,x) \big| \big| G(s,y) \big| 
\leq \int_{t_0}^t \ds \int_{\T^2} \hspace{-0.5ex} \dy \, p(s,y;t,x) \big| G(s,y) \big| 
= \int_{t_0}^t \ds  \Big( e^{(t-s)\Delta} \big|G(s)\big| \Big)(x)
\end{equs}
and thus
\begin{equs}
\|\phi(t)\|_{L_x^2} \leq \int_{t_0}^t \|G(s)\|_{L_x^2} ds &\lesssim \bigg(\int_{t_0}^t (s - t_0)^{\alpha - 2\lambda} ds\bigg)^{\frac{1}{2}} \Big\|(s-t_0)^{-\frac{\alpha}{2} + \lambda} G\Big\|_{L_w^2([t_0, t] \times \T^2)} \\
&\lesssim_\alpha (t-t_0)^{\frac{\alpha}{2}} \Big\|(s-t_0)^{-\frac{\alpha}{2} + \lambda} G\Big\|_{L_w^2([t_0, t] \times \T^2)} .
\end{equs}
In the last inequality, we used that $\alpha > 0$ and $\lambda \in [0, \frac{1}{2}]$. To finish, observe that by our assumption on $K$, there is some uniform lower bound $K(z) \geq c > 0$ for all $z = (t, x)$ with $t$ bounded away from $t_1$, and thus we have the bound
\begin{equs}
\Big\|(s-t_0)^{-\frac{\alpha}{2} + \lambda} G\Big\|_{L_w^2([t_0, t] \times \T^2)}  \lesssim \Big\|(s - t_0)^{-\frac{\alpha}{2} + \lambda} (t_1 - s)^{\frac{\beta}{2} + \frac{1}{2}-\lambda} K^{\frac{1}{2}}(w) G\Big\|_{L_w^2([t_0, t] \times \T^2)},
\end{equs}
for all $t$ sufficiently close to $t_0$. Since we assumed that the norm on $G$ appearing in the RHS \eqref{eq:two-sided-time-weighted-estimate} is finite, the RHS above goes to zero as $t \downarrow t_0$ by dominated convergence. The claim \eqref{eq:singular-time-weight-initial-time-is-still-zero} now follows.
\end{proof}

We state the dual version of the two-sided time-weighted estimate.

\begin{corollary}[Dual two-sided time-weighted estimate]\label{cor:dual-two-sided-time-weighted-estimate}
Let $\alpha, \beta > 0$ and let $\lambda \in [0, \frac{1}{2}]$. Furthermore, let $J = [t_0, t_1] \times \T^2$ and let $K \in C^\infty([t_0, t_1) \times \T^2)$ be a positive solution to the backwards heat equation on $J$, i.e. that $(\ptl_t + \Delta)K = 0$ and $K > 0$. Then for any $H \colon J \ra \C$, we have that
\begin{equs}
\bigg\| (s - t_0)^{\frac{\alpha}{2} - \lambda}& (t_1 - s)^{-\frac{\beta}{2} - (\frac{1}{2}-\lambda)} K(w)^{-\frac{1}{2}} \int dz H(z) \covd_{A(z)} p_A(w; z) \bigg\|_{L_w^2(J)} \leq \\
&\min(\alpha, \beta)^{-\frac{1}{2}} \bigg\|(t - t_0)^{\frac{\alpha}{2}} (t_1 - t)^{-\frac{\beta}{2}} K(z)^{-\frac{1}{2}} H(z)\bigg\|_{L_z^2(J)}.
\end{equs}
\end{corollary}
\begin{proof} 
To simplify the notation, we define $\eta(z) := (t - t_0)^{-\frac{\alpha}{2}} (t_1 - t)^{\frac{\beta}{2}} K(z)^{\frac{1}{2}}$. Given $G \colon  J \ra \C$, we test
\begin{equs}
\Big( &(s - t_0)^{\frac{\alpha}{2} - \lambda} (t_1 - s)^{-\frac{\beta}{2} - (\frac{1}{2}-\lambda)} K(w)^{-\frac{1}{2}} \int dz H(z) \covd_{A(z)} p_A(w; z), G(w)\Big)_{L_w^2(J)}  \\
&= \int dz H(z) \eta(z)^{-1}  \eta(z)  \int dw \eta(w)^{-1} (s - t_0)^{-\lambda} (t_1 - s)^{-(\frac{1}{2}-\lambda)} G(w) \covd_{A(z)} p_A(w; z) \\
&\leq \bigg\| \eta^{-1} H\bigg\|_{L_z^2(J)} \bigg\| \eta(z) \int dw \eta(w)^{-1}  (s - t_0)^{-\lambda} (t_1 - s)^{-(\frac{1}{2}-\lambda)} G(w) \covd_{A(z)} p_A(w; z)\bigg\|_{L_z^2(J)} \\
&\leq  \| \eta^{-1} H\|_{L_z^2(J)}  \|G\|_{L_w^2(J)},
\end{equs}
where we applied the two-sided time-weighted estimate (Lemma \ref{lemma:two-sided-time-weighted-estimate}) in the final inequality. After taking the supremum over all $G\in L^2_w(J)$ satisfying $\| G \|_{L_w^2(J)}\leq 1$, this yields the desired result.
\end{proof}

\begin{corollary}\label{cor:dual-two-sided-weighted-estimate}
Let $J = [t_0, t_1] \times \T^2$ and let $\lambda \in [0, \frac{1}{2}]$. Let $z_1 = (t_1, x)$. Then for any $H \colon J \ra \C$, we have that
\begin{equs}
\bigg\|\int dz H(z) \covd_{A(z)} p_A(w; z)\bigg\|_{L_w^2(J)} \lesssim \lambda^{-\frac{1}{2}} \bigg\| (t - t_0)^{\lambda} (t_1 - t)^{-\lambda} p(z; z_1)^{-\frac{1}{2}} H(z)\bigg\|_{L_z^2(J)}.
\end{equs}
\end{corollary}
\begin{proof}
Observe that $p(w; z_1) \lesssim (t_1 - s)^{-1}$, so that we may bound $1 \lesssim (t_1 - s)^{-\frac{1}{2}} p(w; z_1)^{-\frac{1}{2}}$. Then apply the dual two-sided time-weighted estimate (Corollary \ref{cor:dual-two-sided-time-weighted-estimate}) in the case $\alpha = \beta = 2\lambda$ and $K(w) = p(w; z_1)$. 
\end{proof}

We will also need a weighted estimate where we only include an initial time weight or endpoint time weight. We thus state the following sequence of results. The proofs are omitted as they are very similar to before.

\begin{lemma}[Initial time-weighted estimate]\label{kernel:lem-time-weighted-dual}
Let $J = [t_0, t_1] \times \T^2$. For $G \colon J \ra \C$, we have that
\begin{equs}
\bigg\|(t - t_0)^{-\frac{1}{2}} \int dw G(w) \covd_{A(z)} p_A(w; z) \bigg\|_{L_z^2(J)} \lesssim \|G\|_{L_w^2(J)}.
\end{equs}
For $H \colon J \ra \C$, we also have the dual estimate
\begin{equs}
\bigg\|\int dz H(z) \covd_{A(z)} p_A(w; z)\bigg\|_{L_w^2(J)} \lesssim \Big\|(t - t_0)^{\frac{1}{2}} H\Big\|_{L_z^2(J)}.
\end{equs}
\end{lemma}

\begin{lemma}[Endpoint time-weighted estimate]\label{lemma:endpoint-time-weighted-estimate}
Let $J = [t_0, t_1] \times \T^2$. Suppose $K \in C^\infty([t_0, t_1) \times \T^2)$ is a positive solution to the backwards heat equation on $J$, i.e. that $(\ptl_t + \Delta)K = 0$ and $K > 0$. Suppose $\phi\colon J \rightarrow \C$ satisfies
\begin{equs}
(\ptl_t - \covd_A^j \covd_{A, j}) \phi = G, ~~ \phi(0) = 0.
\end{equs}
Then for $\alpha > 0$, $t_2 \geq t_1$, we have that
\begin{equs}
\Big\|(t_2 - t)^{\frac{\alpha}{2}} K^{\frac{1}{2}} \covd_A \phi \Big\|_{L_z^2(J)} \lesssim \alpha^{-\frac{1}{2}} \Big\|(t_2 - s)^{\frac{\alpha+1}{2}} K^{\frac{1}{2}} G\Big\|_{L_w^2(J)}.
\end{equs}
For $H \colon J \ra \C$, we also have the dual estimate
\begin{equs}\label{eq:dual-endpoint-time-weighted-estimate}
\bigg\|(t_2 - s)^{-\frac{\alpha+1}{2}}  K(w)^{-\frac{1}{2}}\int dz  H(z) \covd_{A(z)} p_A(w; z)  \bigg\|_{L_w^2(J)} \lesssim \alpha^{-\frac{1}{2}}\Big\|(t_2 - t)^{-\frac{\alpha}{2}} K(z)^{-\frac{1}{2}}H(z)\Big\|_{L_z^2(J)}.
\end{equs}
\end{lemma}

\begin{remark}
We chose to state the endpoint time-weighted estimate in a slightly more general form, in that instead of $t_1 - t$ for the time-weight, we have $t_2 - t$ for some $t_2 \geq t_1$. The proof of this slightly more general statement is exactly the same. This more general statement is needed in the proof of Lemma \ref{lemma:difference-linear-object-spatial-regularity} later on.
\end{remark}

By combining the dual estimate \eqref{eq:dual-endpoint-time-weighted-estimate} and the proof of Corollary \ref{cor:dual-two-sided-time-weighted-estimate}, we directly obtain the following estimate.

\begin{corollary}\label{cor:dual-endpoint-time-weighted-estimate}
Let $\alpha > 0$, $J = [t_0, t_1] \times \T^2$, $t_2 \geq t_1$, and $z_2 = (t_2, x_2)$. Then, we have that 
\begin{equs}
\bigg\|(t_2 - s)^{-\frac{\alpha}{2}} \int dz H(z) \covd_{A(z)} p_A(w; z)\bigg\|_{L_w^2(J)} \lesssim \alpha^{-\frac{1}{2}} \bigg\|(t_2 - t)^{-\frac{\alpha}{2}} p(z; z_2)^{-\frac{1}{2}}H(z)\bigg\|_{L_z^2(J)}.
\end{equs}
\end{corollary}

Next, we prove an energy estimate which involves the difference between the covariant and non-covariant heat kernels $p_A$ and  $p$.

\begin{lemma}[Difference between $p_A$ and $p$]\label{lemma:difference-covariant-and-euclidean-heat-kernel}
Let $J = [t_0, t_1] \times \T^2$, $G \colon J \ra \C$. Let $\phi, \psi \colon J \ra \C$ be respective solutions to the covariant and non-covariant heat equations with forcing $G$:
\begin{equs}
(\ptl_t - \covd_A^j \covd_{A, j}) \phi = G, ~~ (\ptl_t - \Delta) \psi = G.
\end{equs}
Then 
\begin{equs}
\|\phi - \psi\|_{L_t^\infty L_x^2(J)} &+ \|\nabla(\phi - \psi)\|_{L_z^2(J)}  \lesssim \|\phi(t_0) - \psi(t_0)\|_{L_x^2(\T^2)} + \|A \phi\|_{L_z^2(J)} + \|A\psi\|_{L_z^2(J)} + \|(\ptl_j A^j) \psi\|_{L_t^1 L_x^2(J)}.
\end{equs}
\end{lemma}
\begin{proof}
The difference satisfies the following forced heat equation
\begin{equs}
(\ptl_t - \covd_A^j \covd_{A, j})(\phi - \psi) = (\covd_{A}^j \covd_{A, j} - \Delta) \psi .
\end{equs}
By the monotonicity formula (Proposition \ref{prop:monotonicity}) and integration by parts, we obtain
\begin{equs}
\frac{1}{2}\ptl_t \|\phi - \psi \|_{L_x^2}^2 &= - \|\covd_A(\phi - \psi)\|_{L_x^2}^2 + \big(\phi - \psi, (\covd_A^j \covd_{A, j} - \Delta)\psi \big)_{L_x^2} \\
&= -\|\covd_A (\phi - \psi)\|_{L_x^2}^2 + \big(\phi - \psi, 2\icomplex \ptl_j (A^j \psi) - \icomplex (\ptl_j A^j) \psi - |A|^2 \psi\big)_{L_x^2} \\
&= -\|\covd_A (\phi - \psi)\|_{L_x^2}^2 - \big(\ptl_j(\phi - \psi), 2\icomplex A^j \psi\big)_{L_x^2} - \big(\phi - \psi, \icomplex (\ptl_j A^j) \psi\big)_{L_x^2} - \big(\phi - \psi, |A|^2 \psi\big)_{L_x^2}.
\end{equs}
By integrating over time and applying H\"{o}lder and Young, we obtain
\begin{equs}
\|\phi - \psi\|_{L_t^\infty L_x^2}^2 + \|\covd_A(&\phi - \psi)\|_{L_z^2}^2 \leq \\
&\frac{1}{4} \|\nabla (\phi - \psi)\|_{L_z^2}^2 + \frac{1}{4} \|\phi - \psi\|_{L_t^\infty L_x^2}^2 + C \Big(\|(\ptl_j A^j) \psi\|_{L_t^1 L_x^2}^2 + \|A \phi\|_{L_z^2}^2 + \|A\psi\|_{L_z^2}^2\Big).
\end{equs}
To finish, use that $\|\nabla (\phi - \psi)\|_{L_z^2} \leq \|\covd_A(\phi - \psi)\|_{L_z^2} + \|A(\phi - \psi)\|_{L_z^2}$, and apply a kick-back argument. 
\end{proof}

Lemma \ref{lemma:difference-covariant-and-euclidean-heat-kernel} leads to the following corollary.

\begin{corollary}\label{cor:gradient-pA-p-l2-spacetime-bound}
Let $z = (t, x) \in [0,1] \times \T^2$, $u \in (0, 1]$. Let $p_u(x, x') := p((0, x); (u, x'))$ be the spatial heat kernel at time $u$. We have that (here $z' = (t, x')$)
\begin{equs}
\bigg\| \int dx' p_u(x, x') \nabla_{y} (p_A(w; z') - p(w; z'))\bigg\|_{L_w^2([0, t] \times\T^2)} \lesssim \big(\|A\|_{L_z^\infty([0, t] \times \T^2)} + \|\ptl_j A^j\|_{L_t^2 L_x^\infty([0, t] \times \T^2)}\big)\sqrt{\log u^{-1}}.
\end{equs}
\end{corollary}
\begin{proof}
Define $\tilde{\phi}, \tilde{\psi} \colon [0,t] \times \T^2\ra \C$ by $\tilde{\phi}(w) := \int dx' p_u(x, x') \ovl{p_A(w; z')}$, $\tilde{\psi}(w):= \int dx' p_u(x, x') p(w;z')$. Since $\ovl{p_A(\cdot;z)}$ is the fundamental solution to the covariant backwards heat equation (Lemma \ref{lemma:heat-kernel-time-reversal}), we have that
\begin{equs}
(\ptl_t + \covd_A^j \covd_{A, j}) \tilde{\phi} = 0.
\end{equs}
We also have that $(\ptl_t + \Delta) \tilde{\psi} = 0$, and $\tilde{\phi}(t) = \tilde{\psi}(t) = p_u(x, \cdot)$. 
Now define the time-reversals $\phi(\tau) := \tilde{\phi}(t - \tau)$, $\psi(\tau) := \tilde{\psi}(t - \tau)$, $B := A(t - \tau)$. One may check that $\phi, \psi$ respectively solve the covariant and usual heat equations. Thus by Lemma \ref{lemma:difference-covariant-and-euclidean-heat-kernel}, 
we have that
\begin{equs}
\|\nabla_y (\phi - \psi)\|_{L_w^2} \lesssim \|B \phi\|_{L_z^2} + \|B\psi\|_{L_z^2} + \|(\ptl_j B^j) \psi\|_{L_t^1 L_x^2}.
\end{equs}
By the diamagnetic inequality (Lemma \ref{kernel:lem-diamagnetic-heat-kernel}), we have the pointwise bound $|\phi| \leq \psi$, and so
\begin{equs}
\max \big( \|B\phi\|_{L_z^2}, \|B\psi\|_{L_z^2} \big) \leq \|B\|_{L_z^\infty} \|\psi\|_{L_z^2} \leq \|B\|_{L_z^\infty} \sqrt{\log u^{-1}},
\end{equs}
where in the final inequality, we used that the initial data is $\psi(0) = p_u(x, \cdot)$. The term $\|(\ptl_j B^j) \psi\|_{L_t^1 L_x^2}$ may be bounded similarly. To finish, simply observe that 
\begin{equs}
\|\nabla_y (\phi - \psi)\|_{L_w^2} = \|\nabla_y (\ovl{\phi} - \ovl{\psi})\|_{L_w^2} = \bigg\| \int dx' p_u(x, x') \nabla_{y} (p_A(w; z') - p(w; z'))\bigg\|_{L_w^2([0, t] \times\T^2)},
\end{equs}
as well as $\|B\|_{L_z^\infty} = \|A\|_{L_z^\infty}$ and $\|\ptl_j B^j\|_{L_t^2 L_x^\infty} = \|\ptl_j A^j\|_{L_t^2 L_x^\infty}$.
\end{proof}

\begin{lemma}\label{lemma:covd-pa-l2-spacetime-norm} 
Let $\alpha > 0$. For $w = (s, y)$ and $s < t_0 < s+1$, we have that
\begin{equs}
\bigg\|(t-s)^{\frac{1+\alpha}{2}} \covd_{A(z)} p_A(w; z)\bigg\|_{L_z^2([s, t_0] \times \T^2)} \lesssim_\alpha (t_0 - s)^{\frac{\alpha}{2}}.
\end{equs}
Additionally, for $\rho > 0$, we have that
\begin{equs}
\bigg\|\ind(|t-s| \in [\rho/10, 10\rho]) \covd_{A(z)} p_A(w; z)\bigg\|_{L_z^2([s, t_0] \times \T^2)} \lesssim \rho^{-\frac{1}{2}}. 
\end{equs}
Dually, for $z = (t,x)$, and $s_0 < t$, we have that
\begin{equs}
\Big\|(t-s)^{\frac{1+\alpha}{2}} \covd_{-A(w)} p_A(w; z)\Big\|_{L_w^2([s_0, t]\times\T^2)} &\lesssim_\alpha (t-s_0)^{\frac{\alpha}{2}}, \\
\bigg\|\ind(|t - s| \in [\rho/10, 10\rho])\covd_{-A(w)} p_A(w; z)\bigg\|_{L_w^2([s_0, t] \times \T^2)} &\lesssim \rho^{-\frac{1}{2}}.
\end{equs}
\end{lemma}
\begin{proof}
For the first estimate, we apply the monotonicity formula (Proposition \ref{prop:monotonicity}) with $K\equiv 1$ to $p_A$ itself and obtain
\begin{equs}
\frac{1}{2} \ptl_t \big((t-s)^{1+\alpha}|p_A(w;z)|^2\big) = -(t-s)^{1+\alpha} |\covd_{A(z)} p_A(w; z)|^2 + (1+\alpha)(t-s)^\alpha |p_A(w; z)|^2.
\end{equs}
Upon integrating over $z$, we obtain
\begin{equs}
\bigg\|(t-s)^{\frac{1+\alpha}{2}} \covd_{A(z)} p_A(w; z)\bigg\|_{L_z^2([s, t_0] \times \T^2)} \lesssim \Big\|(t-s)^{\frac{\alpha}{2}} p_A(w; z)\Big\|_{L_z^2([s, t_0] \times \T^2)}  \lesssim (t_0 - s)^{\frac{\alpha}{2}}.
\end{equs}
This shows the first estimate. For the second estimate, apply the first estimate with $\alpha = 1$ (say) to obtain
\begin{equs}
\rho \bigg\|\ind(|t-s| \in [\rho/10, 10\rho]) \covd_{A(z)} p_A(w; z)\bigg\|_{L_z^2([s, s+1] \times \T^2)} \lesssim \bigg\|(t-s) \covd_{A(z)} p_A(w; z)\bigg\|_{L_z^2([s, s + 10\rho] \times \T^2)} \lesssim \rho^{\frac{1}{2}}.
\end{equs}
For the last two estimates, we use the fact that $\ovl{p_A(\cdot; z)}$ is the fundamental solution to the backwards covariant heat equation (Lemma \ref{lemma:heat-kernel-time-reversal}) and a time-reversal argument as in the proof of Corollary \ref{cor:gradient-pA-p-l2-spacetime-bound}. 
\end{proof}

\subsection{Covariant derivatives of the covariant heat kernel}

In this subsection, we relate covariant derivatives of $p_A(w; z)$ in $z$ and $w$.

\begin{lemma}[Expansion of $\covd_A p_A$]\label{lemma:covd-pA-expansion}
For $s < t$, $w = (s, y)$, $z = (t, x)$, and $k \in [2]$, we have that
\begin{equs}
\covd^k_{A(z)} p_A(w; z) = - \covd^k_{-A(w)} p_A(w; z) &+ 2\icomplex \int dv p_A(v; z) (F_A)^{kj}(v) \covd_{A(v), j} p_A(w; v) \\
&+\icomplex \int dv p_A(v; z) (\ptl_t A^k + \ptl_j F_A^{kj})(v) p_A(w; v).
\end{equs}
\end{lemma}
\begin{remark}
If $A$ satisfies the Coulomb condition, then $\Delta A^k = -\ptl_j F_A^{kj}$, and so $\ptl_t A^k + \ptl_j F_A^{kj} = (\ptl_t - \Delta) A^k$. Thus, if $A$ also solves the heat equation, then the last term above drops out.
\end{remark} 

\begin{proof}
Let $G \colon [0, T] \times \T^2 \ra \R$ be smooth. Let $\phi$ be the solution to
\begin{equs}
(\ptl_t - \covd_A^j \covd_{A, j}) \phi = G, ~~ \phi(0) = 0.
\end{equs}
By Lemma \ref{lemma:D-A-phi-equation}, we have that
\begin{equs}
(\ptl_t - \covd_A^j \covd_{A, j}) \covd_A^k \phi = \covd_A^k G + 2\icomplex (F_A)^{kj} \covd_{A, j}  \phi + \icomplex (\ptl_t A^k + \ptl_j F_A^{kj}) \phi. 
\end{equs}
For brevity, let $H^k(w) := \ptl_t A^k(w) + \ptl_j F_A^{kj}(w)$. The above identity implies that
\begin{equs}
\int dw ~\covd_{A(z)}^k p_A(w; z) G(w) = ( \covd_A^k \phi)(z) = \int dw ~  p_A(w; z) \covd_{A(w)}^k G(w) &+ \int dv ~ p_A(v; z)  2\icomplex (F_A)^{kj}(v) \covd_{A(v), j} \phi (v) \\
&+ \int dv ~ p_A(v; z) \icomplex H^k(v) \phi(v).
\end{equs}
By integration by parts in the spatial variable, we have that
\begin{equs}
 \int dw ~  p_A(w; z) \covd_{A(w)}^k G(w) = - \int dw~ \Big( \covd_{-A(w)}^k p_A(w; z)\Big) G(w).
\end{equs}
The second term may be written
\begin{equs}
\int dv ~ p_A(v; z)  2\icomplex (F_A)^{kj}(v) \covd_{A, j} \phi (v) &= 2\icomplex \int dv ~ p_A(v; z) (F_A)^{kj}(v) \int dw ~ \covd_{A(v), j} p_A(w; v) G(w) \\
&= 2\icomplex \int dw ~ G(w) \int dv ~ p_A(v; z) (F_A)^{kj}(v) \covd_{A(v), j} p_A(w; v).
\end{equs}
Similarly, the third term may be written
\begin{equs}
\int dv ~ p_A(v; z) \icomplex H^k(v) \phi(v) =  \icomplex \int dw~ G(w) \int dv p_A(v; z) H^k(v) p_A(w; v).
\end{equs}
The desired result now follows since $G$ was arbitrary.
\end{proof}

\subsection{Perturbations of the covariant heat kernel}

In this subsection, we derive some formulas for $\ptl_u p_{A_u}(w; z)$ when the connection $A_u$ depends on an auxiliary parameter $u$. These formulas will be used in Section \ref{section:cshe} when we estimate the covariant stochastic objects.

\begin{lemma}\label{lemma:ptl-u-phi-u-evolution} 
Let $(A_u)_{u \in [0, 1]}\colon [0,\infty)\times \T^2 \rightarrow \R^2$ be a family of connection one-forms indexed by $u\in [0,1]$ and let $G\colon [0,\infty) \times \T^2 \rightarrow \C$. For each $u$, let $\phi_u\colon [0,\infty) \times \T^2 \rightarrow \C$ be the solution to the covariant heat equation corresponding to $A_u$, i.e.
\begin{equs}
(\ptl_t - \covd_{A_u}^j \covd_{A_u, j}) \phi_u = G. 
\end{equs} 
Then, the time-evolution of $\ptl_u \phi_u$ is given by
\begin{equs}
(\ptl_t - \covd_{A_u}^j \covd_{A_u, j}) (\ptl_u \phi_u) =  2 \icomplex (\ptl_u A_u^j) \covd_{A_u, j} \phi_u + \icomplex (\ptl_u \ptl_j A_{u}^j) \phi_u.
\end{equs}
\end{lemma}
\begin{proof}
For each $u$, by assumption $\phi_u$ solves
\begin{equs}
\ptl_t \phi_u = \covd_{A_u}^j \covd_{A_u, j} \phi_u  + G.
\end{equs}
Taking $u$-derivatives on both sides, we obtain
\begin{equs}
\ptl_t (\ptl_u \phi_u) &= \ptl_u \big( \covd_{A_u}^j \covd_{A_u, j} \phi_u\big)   \\
&= \covd_{A_u}^j \covd_{A_u, j} (\ptl_u \phi_u) + (\ptl_u \covd_{A_u}^j) \covd_{A_u, j} \phi_u + \covd_{A_u}^j (\ptl_u \covd_{A_u, j}) \phi_u  \\
&= \covd_{A_u}^j \covd_{A_u, j} (\ptl_u \phi_u) + \icomplex (\ptl_u A_u^j) \covd_{A_u, j} \phi_u + \covd_{A_u}^j (\icomplex (\ptl_u A_{u, j}) \phi_u) \\
&=\covd_{A_u}^j \covd_{A_u, j} (\ptl_u \phi_u) + 2 \icomplex (\ptl_u A_u^j) \covd_{A_u, j} \phi_u + \icomplex (\ptl_u \ptl_j A_{u}^j) \phi_u ,
\end{equs}
as desired.
\end{proof}

\begin{corollary}\label{cor:covariant-heat-kernel-derivative-varying-connection}
Let $(A_u)_{u \in [0, 1]}\colon [0,\infty)\times \T^2 \rightarrow \R^2$ be a family of connection one-forms indexed by $u\in [0,1]$. For any $w,z\in [0,\infty)\times \T^2$, we then have that 
\begin{equs}
\ptl_u p_{A_u}(w; z) = &2\icomplex \int dv \,  p_{A_u}(v; z) (\ptl_u A_u^j)(v) \covd_{A_u(v), j} p_{A_u}(w; v) +\icomplex \int dv \,  p_{A_u}(v; z) (\ptl_u \ptl_j A_u^j(v)) p_{A_u}(w; v).
\end{equs}
The same identity holds with $p_{A_u}$ replaced by $\massp_{A_u}$ everywhere.
\end{corollary}
\begin{proof}
Let $G$ be arbitrary, and let $\phi_u$ be the solution to 
\begin{equs}
(\ptl_t - \covd_{A_u}^j \covd_{A_u, j}) \phi_u = G, ~~ \phi_u(0) = 0.
\end{equs}
By Lemma \ref{lemma:ptl-u-phi-u-evolution}, we have that
\begin{equs}\label{eq:covariant-heat-kernel-derivative-varying-connection}
\ptl_u \phi_u(z) = 2\icomplex \int dv \, p_{A_u}(v; z) (\ptl_u A_u^j)(v) \covd_{A_u(v), j} \phi_u(v) + \icomplex \int dv \, p_{A_u}(v; z) (\ptl_u \ptl_j A_u^j(v)) \phi_u(v). 
\end{equs}
By inserting into both sides of \eqref{eq:covariant-heat-kernel-derivative-varying-connection} the identity 
\begin{equs}
\phi_u(v) = \int dw ~p_{A_u}(w; v) G(w),
\end{equs}
we obtain
\begin{equs}
\int dw  ~\ptl_u p_{A_u}(w; z) G(w) = &~2 \icomplex \int dw  ~G(w) \int dv  ~ p_{A_u}(v; z) (\ptl_u A_u^j)(v) \covd_{A_u(v), j} p_{A_u}(w; v) \\
&~+ \icomplex \int dw ~G(w) \int dv ~ p_{A_u}(v; z) (\ptl_u \ptl_j A_u^j(v)) p_{A_u}(w; v).
\end{equs}
The desired result now follows since $G$ was arbitrary. The identity for $\massp_{A_u}$ follows by Remark \ref{remark:massive-kernel}.
\end{proof}

The following result is a direct consequence of Corollary \ref{cor:covariant-heat-kernel-derivative-varying-connection}.

\begin{corollary}\label{cor:p-A-u-expansion}
Let $A\colon [0,\infty) \times \T^2 \rightarrow \R^2$ be a connection one-form and let 
$z_0 = (t_0,x_0) \in [0,\infty) \times \T^2$ be fixed. For $u \in [0, 1]$, define
\begin{equs}
A_u(w) := A_{u, z_0}(w) := (1-u) A(z_0) + u A(w).
\end{equs}
Then, it holds that 
\begin{equs}
\ptl_u p_{A_u}(w; z) = 2\icomplex \int dv~ p_{A_u}(v; z) (A_j(v) - A_j(z_0))\covd_{A_u(v)}^j p_{A_u}(w; v) + \icomplex \int dv p_{A_u}(v; z) \ptl_j A^j(v) p_{A_u}(w; v).
\end{equs}
The same identity holds with $p_{A_u}$ replaced by $\massp_{A_u}$ everywhere.
\end{corollary}

We finish this section off with the following corollary which will be needed in Section \ref{section:cshe}.

\begin{corollary}\label{cor:pA-p-L2-spacetime-norm}
Let $z = (t, x)$. We have that
\begin{equs}
\Big\|p_A(w; z) - p(w; z)\Big\|_{L_w^2([0, t] \times \T^2)} \lesssim t^{\frac{1}{2}} \big( \|A\|_{L_z^\infty([0, t] \times \T^2)} + \|\ptl_j A^j\|_{L_t^2 L_x^\infty([0, t] \times \T^2)}\big).
\end{equs}
\end{corollary}
\begin{proof}
For brevity, let $J := [0, t] \times \T^2$. For $u \in [0, 1]$, define $A_u := u A$. We then have by Corollary \ref{cor:covariant-heat-kernel-derivative-varying-connection} that
\begin{equs}
p_{A}(w; z) - p(w; z) &= \int_0^1 du \ptl_u p_{A_u}(w; z) \\
&= \int_0^1 du \bigg(2\icomplex \int dv p_{A_u}(v; z) A^j(v) \covd_{A_u(v), j} p_{A_u}(w; v) + \icomplex \int dv p_{A_u}(v; z) (\ptl_j A^j(v)) p_{A_u}(w; v)\bigg).
\end{equs}
It follows that
\begin{equs}
\Big\|p_A(w; z) &- p(w; z)\Big\|_{L_w^2(J)} \lesssim \\
& \int_0^1 du \bigg\|\int dv p_{A_u}(v; z) A^j(v) \covd_{A_u(v), j} \label{eq:diff-heat-kernel-intermediate-term-1} p_{A_u}(w; v)\bigg\|_{L_w^2(J)} \\
&+ \int_0^1 du \bigg\|\int dv p_{A_u}(v; z) (\ptl_j A^j(v)) p_{A_u}(w; v)\bigg\|_{L_w^2(J)} \label{eq:diff-heat-kernel-intermediate-term-2}.
\end{equs}
We bound each term separately. By Corollary \ref{cor:dual-endpoint-time-weighted-estimate}, we have that
\begin{equs}
\eqref{eq:diff-heat-kernel-intermediate-term-1} &\lesssim  \int_0^1 du \|(t - s)^{\frac{1}{4}}\|_{L_w^\infty(J)} \bigg\|(t-s)^{-\frac{1}{4}}\int dv p_{A_u}(v; z) A^j(v) \covd_{A_u(v), j} p_{A_u}(w; v)\bigg\|_{L_w^2(J)} \\
&\lesssim t^{\frac{1}{4}}\int_0^1 du \Big\|(t-t_v)^{-\frac{1}{4}} p(v; z)^{-\frac{1}{2}} p_{A_u}(v; z) A(v) \Big\|_{L_v^2(J)} \\
&\lesssim t^{\frac{1}{4}} \|A\|_{L_z^\infty(J)} \Big\|(t-t_v)^{-\frac{1}{4}} p(v; z)^{\frac{1}{2}} \Big\|_{L_v^2(J)} \lesssim t^{\frac{1}{2}} \|A\|_{L_z^\infty(J)},
\end{equs}
which is acceptable. In the second-to-last inequality, we used the diamagnetic inequality (Lemma \ref{kernel:lem-diamagnetic-heat-kernel}) to bound $|p_{A_u}(v; z)| \leq p(v; z)$ for all $u$. For the other term, first note that for fixed $u, w$, we have the pointwise bound 
\begin{equs}
\bigg|\int dv p_{A_u}(v; z) (\ptl_j A^j(v)) p_{A_u}(w; v)\bigg| \leq \int_s^t dt_v \|\ptl_j A^j(t_v)\|_{L_x^\infty} p(w; z) \lesssim (t-s)^{\frac{1}{2}} p(w; z) \|\ptl_j A^j\|_{L_t^2 L_x^\infty(J)},
\end{equs}
where we applied the diamagnetic inequality $|p_{A_u}(v; z)| \leq p(v; z)$ and $|p_{A_u}(w; v)| \leq p(w; v)$, as well as the semigroup property of the heat kernel $p$. We thus obtain
\begin{equs}
\eqref{eq:diff-heat-kernel-intermediate-term-2} \lesssim \|\ptl_j A^j\|_{L_t^2 L_x^\infty(J)} \Big\|(t-s)^{\frac{1}{2}} p(w; z)\Big\|_{L_w^2(J)} \lesssim t^{\frac{1}{2}} \|\ptl_j A^j\|_{L_t^2 L_x^\infty(J)},
\end{equs}
which is acceptable. The desired result now follows.
\end{proof}

\section{Gauge covariance}\label{section:gauge-covariance}

In this section, we show that solutions to \eqref{intro:eq-SAH} are gauge covariant, when the $A$-counterterm $\Cgauge A_{\leq N}$ appearing in \eqref{intro:eq-SAH-smooth} is precisely $\Cgauge = \gaugerenorm$. Moreover, this constant is uniquely specified in a certain sense -- see Remark \ref{remark:gauge-renorm-uniquely-determined}. Because many of the results of this section have analogs in our previous paper \cite{BC23}, we will often just state the results and refer to the locations of \cite{BC23} which contain the proofs.

\begin{definition}[Data to solution map]
For $\Cgauge \in \R$, let $\mbb{S}(A_0, \phi_0, \xi, \zeta; \Cgauge)$ be the data-to-solution map for \eqref{intro:eq-SAH}, which is defined by
\begin{equs}
\mbb{S}(A_0, \phi_0, \xi, \zeta; \Cgauge) := \lim_{N \toinf} \mbb{S}(A_0, \phi_0, \xi_{\leq N}, \zeta_{\leq N}; \Cgauge),
\end{equs}
where $\mbb{S}(A_0, \phi_0, \xi_{\leq N}, \zeta_{\leq N}; \Cgauge)$ is the solution to \eqref{intro:eq-SAH-smooth}. Here, the limit is in $C_t^0 \Cs_x^{-\kappa}([0, T] \times \T^2 \rightarrow \R^2) \times C_t^0 \Cs_x^{-\kappa}([0, T] \times \T^2 \rightarrow \C)$, for the maximal $T > 0$ for which the limit exists. 
\end{definition}

In the following, recall from \eqref{intro:eq-group-action-Zd} that for $n \in \Z^2$, we write $(A^n, \phi^n) = (A + n, e^{-\icomplex n \cdot x} \phi)$. We will also write $(A, \phi)^n := (A^n, \phi^n)$.

\begin{theorem}[Gauge covariance of \eqref{intro:eq-SAH}]\label{thm:gauge-covariance}
Let $n_0 \in \Z^2$. We have that 
\begin{equs}
\solutionmap\Big(A_0, \phi_0, \xi, \zeta; \gaugerenorm\Big)^{n_0} \stackrel{a.s.}{=} \solutionmap\Big(A_0^{n_0}, \phi_0^{n_0}, \xi, \zeta^{n_0}; \gaugerenorm\Big).
\end{equs}
In particular, if $\mbb{S}(A_0, \phi_0, \xi, \zeta; \gaugerenorm)$ exists on some interval $[0, T]$, then so does $\mbb{S}(A_0^{n_0}, \phi_0^{n_0}, \xi, \zeta^{n_0}; \gaugerenorm)$, and vice-versa.
\end{theorem}

To begin, we review the local theory for the stochastic Abelian-Higgs equation. We make the following solution ansatz:
\begin{equs}
A_{\leq N} &= \linear[\leqN][r][\st] + \Aquadratic[\leqN][r][\st] + X_{\leq N} , \\
\phi_{\leq N} &= \philinear[\leqN][r][\st] + 2\icomplex \mixedquadratic[\leqN][r][\st] + \eta_{\leq N} + \psi_{\leq N},
\end{equs}
where we define the stationary objects
\begin{align}
\linear[\leqN][r][\st](t) &:= \int_{-\infty}^t ds  e^{(t-s)(\Delta-1)} \leray \xi_{\leq N}(s) , \quad \philinear[\leqN][r][\st](t) := \int_{-\infty}^t ds e^{(t-s)(\Delta-1)} \zeta_{\leq N}(s), \label{eq:linear-object-stationary}\\
\Aquadraticnlst[\leqN] &:= \big|\,\linear[\leqN][r][\st]\big|^2 - 2 \sigma^2_{\leq N}, \\
\Aquadratic[\leqN][r][\st](t) &:= - \leray 
 \int_{-\infty}^t ds e^{(t-s)(\Delta - 1)} \Im\Big( \philinear[\leqN][r][\st](s) \label{eq:Aquadratic} \covd \philinear[\leqN][r][\st](s)\Big), \\
\mixedquadratic[\leqN][r][\st](t)  &:= \int_{-\infty}^t ds e^{(t-s)(\Delta - 1)} \linear[\leqN][r][\st, j](s) \ptl_j \philinear[\leqN][r][\st](s), \label{eq:Aquadratic-stationary}
\end{align}
and where $X_{\leq N}, \eta_{\leq N}, \psi_{\leq N}$ solve the following system. The nonlinear remainder $X_{\leq N}$ solves
\begingroup
\allowdisplaybreaks
\begin{align}\label{eq:para-sah}
(\ptl_t - &\Delta) X_{\leq N} =
-  \leray\Big( 4 \Im\big(\icomplex \ovl{\philinear[\leqN][r][\st]} \covd \mixedquadratic[\leqN][r][\st]\big) +  \sigma_{\leq N}^2 \linear[\leqN][r][\st] \Big) - \leray\Big(  \linear[\leqN][r][\st] \big(\big|\philinear[\leqN][r][\st]\big|^2 - \sigma_{\leq N}^2\big)\Big)  \\
&- \leray \bigg(4 (A_{\leq N} - \linear[\leqN][r][\st])^j \Big(\Im\Big( \icomplex \ovl{\philinear[\leqN][r][\st]} \covd \Duh\big(\ptl_j \philinear[\leqN][r][\st]\big)\Big) + \sigma_{\leq N}^2 \Big)\bigg) \label{eq:para-sah-e2} \\
&- \leray \Big((A_{\leq N}- \linear[\leqN][r][\st])(\big|\philinear[\leqN][r][\st]\big|^2 - \sigma_{\leq N}^2)\Big) 
\label{eq:para-sah-e3}\\
&+ 4\leray  \Im\Big( \icomplex \ovl{\philinear[\leqN][r][\st]}  \Big((A_{\leq N} - \linear[\leqN][r][\st])^j \paragtrsim \covd  \Duh(\ptl_j \philinear[\leqN][r][\st])\Big)\Big)  \\
&- 4 \leray \Im\Big(\icomplex \ovl{\philinear[\leqN][r][\st]} \Big((\covd (A_{\leq N} - \linear[\leqN][r][\st])^j) \parall \Duh(\ptl_j \philinear[\leqN][r][\st])\Big)\Big)    \\
&-4 \leray \Im\bigg(\icomplex \ovl{\philinear[\leqN][r][\st]} \covd  \Big(\Duh \Big((A_{\leq N} - \linear[\leqN][r][\st])^j \parall \ptl_j \philinear[\leqN][r][\st]\Big) - (A_{\leq N} - \linear[\leqN][r][\st])^j \parall \Duh(\ptl_j \philinear[\leqN][r][\st])\Big)\bigg) \label{gauge:eq-commutator-term} \\
&- 2 \leray \Im\Big(\ovl{\philinear[\leqN][r][\st]} \covd \psi_{\leq N}\Big) - \leray \Im\Big(\ovl{(\phi_{\leq N} - \philinear[\leqN][r][\st])} \covd (\phi_{\leq N} - \philinear[\leqN][r][\st])\Big) \\
& - \leray \Big(\ovl{\philinear[\leqN][r][\st]}\linear[\leqN][r][\st] (\phi_{\leq N} - \philinear[\leqN][r][\st]) + \philinear[\leqN][r][\st] \linear[\leqN][r][\st] \ovl{(\phi_{\leq N} - \philinear[\leqN][r][\st])} \Big) \\
&-2\leray\Big(\Re\Big(\ovl{\philinear[\leqN][r][\st]} (\phi_{\leq N} - \philinear[\leqN][r][\st])\Big) (A_{\leq N} - \linear[\leqN][r][\st]) \Big)\\
&- \leray \Big(\big|\phi_{\leq N} - \philinear[\leqN][r][\st] \big|^2 A_{\leq N}\Big) \\
&+ \linear[\leqN][r][\st] + \Aquadratic[\leqN][r][\st] + \Cgauge A_{\leq N}, 
\end{align}
the para-controlled component $\eta_{\leq N}$ solves
\begin{align}
(\ptl_t - \Delta ) \eta_{\leq N} = &~2\icomplex (A_{\leq N} - \linear[\leqN][r][\st])^j \parall \ptl_j \philinear[\leqN][r][\st], 
\end{align}
and the nonlinear remainder $\psi_{\leq N}$ solves 
\begin{align}
(\ptl_t - \Delta)\psi_{\leq N} = &~2\icomplex (A_{\leq N}- \linear[\leqN][r][\st])^j \paragtrsim \ptl_j \philinear[\leqN][r][\st] + 2\icomplex \ptl_j (A_{\leq N}^j (\phi_{\leq N} - \philinear[\leqN][r][\st])) \\
&- \Aquadraticnlst[\leqN] \philinear[\leqN][r][\st] - \Aquadraticnlst[\leqN](\phi_{\leq N} - \philinear[\leqN][r][\st]) - 2  \, 
\linear[\leqN][r][\st,\scalebox{1.2}{$j$}]\philinear[\leqN][r][\st] \big( A_{\leq N} - \linear[\leqN][r][\st] \big)_j  \\
&- 2 \,  \linear[\leqN][r][\st,\scalebox{1.2}{$j$}] \big( A_{\leq N} - \linear[\leqN][r][\st]  \big)_j  (\phi_{\leq N} - \philinear[\leqN][r][\st]) - |A_{\leq N} - \linear[\leqN][r][\st]|^2 \phi_{\leq N} \\
&- \sum_{j=0}^{\frac{q+1}{2}} \sum_{k=0}^{\frac{q-1}{2}} \binom{\frac{q+1}{2}}{j} \binom{\frac{q-1}{2}}{k} \biglcol\, \philinear[\leqN][r][\st, j] \ovl{\philinear[\leqN][r][\st, k]}\bigrcol\, (\phi_{\leq N} - \philinear[\leqN][r][\st])^{\frac{q+1}{2} - j} \ovl{(\phi_{\leq N} - \philinear[\leqN][r][\st])}^{\frac{q-1}{2} - k} \\
&+ \philinear[\leqN][r][\st] + \mixedquadratic[\leqN][r][\st].
\end{align}
Finally, the initial data is:
\begin{equs}\label{eq:local-theory-initial-data}
(X_{\leq N}(0), \eta_{\leq N}(0), \psi_{\leq N}(0)) = ~ \Big(A_0 - \linear[\leqN][r][\st](0) - \Aquadratic[\leqN][r][\st](0), 0, \phi_0 - \philinear[\leqN][r][\st](0) - \mixedquadratic[\leqN][r][\st](0)\Big).
\end{equs}
\endgroup

\begin{remark}\label{remark:local-theory-remarks}
We make the following remarks about the local theory.
\begin{enumerate}[label=(\alph*)]
    \item\label{item:A-equation-no-diverging-counterterm} Observe that there is in fact no divergent renormalization in the $X$-equation, because the counterterms appearing in \eqref{eq:para-sah}, \eqref{eq:para-sah-e2}, and \eqref{eq:para-sah-e3} add up to zero. This cancellation is essentially the same as the cancellation observed in \cite{BC23, CCHS22}. We note that this only holds in two dimensions. In three dimensions, the $A$-equation has a diverging counterterm -- see \cite{CCHS22+}.
    \item\label{item:null-form-estimate} By the null-form estimate from Lemma \ref{prelim:lem-null} below, if $A_{\leq N} \in \Cs_x^{-\kappa}$ is in the Coulomb gauge and $\phi_{\leq N} - \philinear[\leqN][r][\st] \in \Cs_x^{1-2\kappa}$, then $\ptl_j (A_{\leq N}^j (\phi_{\leq N} - \philinear[\leqN][r][\st])) \in \Cs_x^{-5\kappa}$. This justifies placing this term in the equation for the smooth remainder $\psi$ (which will essentially be in $\Cs_x^{2-}$).
    \item\label{item:stationary-objects} In \eqref{eq:linear-object-stationary}-\eqref{eq:Aquadratic-stationary}, we defined the stochastic objects to be stationary, which matches with our previous paper \cite{BC23}, and which is also convenient because then the counterterms are constant in time. However, in Sections \ref{section:cshe}, \ref{section:Abelian-Higgs}, and \ref{section:decay}, we primarily work with non-stationary objects. The reason is that, since the connection one-form $\Blin$ is time-dependent, it would be difficult to introduce a stationary version of the covariant linear stochastic object $\philinear[\Blin]$ from \eqref{intro:eq-cshe-covariant-object}.
\end{enumerate}
\end{remark}

\begin{lemma}[Null-form estimate]\label{prelim:lem-null}
Let $\alpha,\alpha_1,\alpha_2 \in \R$ and $p,p_1,p_2\in [1,\infty]$ satisfy $\frac{1}{p} \geq \frac{1}{p_1}+\frac{1}{p_2}$, 
\begin{equs}\label{prelim:eq-null-assumption}
\alpha \leq \min(\alpha_1,\alpha_2-1), \quad \alpha_1 + \alpha_2 >0, \quad 
\text{and} \quad \alpha_1 + \alpha_2 - 1 > \alpha.  
\end{equs}
Furthermore, let $A\colon \T^2 \rightarrow \R^2$ satisfy the Coulomb gauge condition $\partial_{j} A^j =0$ and let $\phi \colon \T^2 \rightarrow \C$. Then, it holds that
\begin{equs}\label{prelim:eq-null}
\big\| A^j \partial_j \phi \big\|_{\Bc^{\alpha,p}_x}
\lesssim \big\| A \big\|_{\Bc^{\alpha_1,p_1}_x} \big\| \phi \big\|_{\Bc^{\alpha_2,p_2}_x}.
\end{equs}
\end{lemma}

\begin{remark}
The Coulomb gauge condition $\partial_j A^j=0$ in Lemma \ref{prelim:lem-null} is necessary. Without the Coulomb gauge condition, the condition $\alpha_1 + \alpha_2 >0$ in \eqref{prelim:eq-null-assumption} has to be replaced by the stronger condition $\alpha_1+\alpha_2>1$. 
\end{remark}
\begin{remark}[Null-forms] 
Since $A$ is in the Coulomb gauge, the nonlinearity $A^j \partial_j \phi$ can be written as 
\begin{equation*}
A^j \partial_j \phi = \big( \Delta^{-1} \partial_k F_A^{kj} \big) \partial_j \phi = \frac{1}{2} Q_{jk}\big(\phi, \Delta^{-1} F_A^{jk} \big),
\end{equation*}
where $Q_{jk}(\varphi,\psi):= \partial_j \varphi \partial_k \psi - \partial_k \varphi \partial_j \psi$. The $Q_{jk}$ are called null-forms and have been  extensively studied in the context of wave equations, see e.g. \cite{KM93,KS02}. For instance, it is well-known that if $\varphi$ and $\psi$ are solutions of the linear wave equation on Euclidean space, then the null-forms $Q_{jk}(\varphi,\psi)$ satisfy better Strichartz estimates than the products $\partial_j \varphi \partial_k \psi$. Since Lemma \ref{prelim:lem-null} is better than the product estimate from Lemma \ref{prelim:lem-para-besov}.\ref{prelim:item-product-estimate}, we therefore also refer to Lemma \ref{prelim:lem-null} as a null-form estimate. 
\end{remark}

\begin{proof}
We first estimate
\begin{equs}\label{prelim:eq-null-p1}
\big\| A^j \partial_j \phi \big\|_{\Bc^{\alpha,p}_x} 
\leq \big\| A^j \parall \partial_j \phi \big\|_{\Bc^{\alpha,p}_x}
+ \big\| A^j \parasim \partial_j \phi \big\|_{\Bc^{\alpha,p}_x} 
+ \big\| A^j \paragg \partial_j \phi \big\|_{\Bc^{\alpha,p}_x}. 
\end{equs}
We now estimate the three summands in \eqref{prelim:eq-null-p1} using Lemma \ref{prelim:lem-para-besov}. Using $\alpha\leq \alpha_2 -1$ and $\alpha_1+(\alpha_2-1)>\alpha$, we obtain
\begin{equs}
\big\| A^j \parall \partial_j \phi \big\|_{\Bc^{\alpha,p}_x}
\lesssim 
\big\| A \big\|_{\Bc^{\alpha_1,p_1}_x} \big\|  \nabla \phi \big\|_{\Bc^{\alpha_2-1,p_2}_x}
\lesssim \big\| A \big\|_{\Bc^{\alpha_1,p_1}_x} \big\| \phi \big\|_{\Bc^{\alpha_2,p_2}_x}. 
\end{equs}
Using the Coulomb gauge condition $\partial_j A^j=0$, $\alpha_1+\alpha_2>0$, and $\alpha_1 + \alpha_2 > \alpha+1$, we obtain 
\begin{equs}
\big\| A^j \parasim \partial_j \phi \big\|_{\Bc^{\alpha,p}_x}
= \big\|\partial_j \big( A^j \parasim \phi \big) \big\|_{\Bc^{\alpha,p}_x}
\lesssim \big\| A \parasim \phi \big\|_{\Bc^{\alpha+1,p}_x} 
\lesssim \big\| A \big\|_{\Bc^{\alpha_1,p_1}_x} \big\| \phi \big\|_{\Bc^{\alpha_2,p_2}_x}.
\end{equs}
Finally, using $\alpha\leq \alpha_1$ and $\alpha_1 + (\alpha_2-1)>\alpha$, we obtain
\begin{equation*}
 \big\| A^j \paragg \partial_j \phi \big\|_{\Bc^{\alpha,p}_x}
 \lesssim \big\| A \big\|_{\Bc^{\alpha_1,p_1}_x} \big\|  \nabla \phi \big\|_{\Bc^{\alpha_2-1,p_2}_x}
\lesssim \big\| A \big\|_{\Bc^{\alpha_1,p_1}_x} \big\| \phi \big\|_{\Bc^{\alpha_2,p_2}_x}. \qedhere
\end{equation*} 
\end{proof}

The following discussion is very similar to the beginning of \cite[Section 5]{BC23}. Define the enhanced data set
\begin{equation}\label{eq:enhanced-data-set}
\begin{aligned}
\Xi_{\leq N} := \bigg(\linear[\leqN][r][\st], &\Aquadraticnlst[\leqN],  \linear[\leqN][r][\st] \philinear[\leqN][r][\st],  \Aquadratic[\leqN][r][\st], \mixedquadratic[\leqN][r][\st],  \Aquadraticnlst[\leqN] \philinear[\leqN][r][\st], \\
& \bigg(\Im\Big(\icomplex \ovl{\philinear[\leqN][r][\st]} \ptl_i \ptl_j \Duh\big(\philinear[\leqN][r][\st]\big)\Big) + \frac{\delta_{ij}}{4} \sigma_{\leq N}^2\bigg)_{i \in [2]}, \\
& \bigg(\Im\Big(\icomplex \ovl{\philinear[\leqN][r][\st]} \ptl^i \mixedquadratic[\leqN][r][\st]\Big) + \frac{1}{4} \sigma_{\leq N}^2 \linear[\leqN][r][\st, i]\bigg)_{i \in [2]} \\
&:\philinear[\leqN][r][\st, j] \ovl{\philinear[\leqN][r][\st, k]}\colon 0 \leq j \leq \frac{q+1}{2}, 0 \leq k \leq \frac{q-1}{2}
\bigg) .
\end{aligned}
\end{equation}

\begin{notation}
Let $n_q := 8 + \big(\frac{q+1}{2}\big)\big(\frac{q-1}{2}\big)$, which is the number of stochastic objects in $\Xi$ (we count different components of a connection as a single stochastic object, so for instance, $\linear[\leqN][r][\st] = (\linear[\leqN][r][\st, j])_{j \in [2]}$ counts as one object.)
\end{notation}

Define $\dg_{\mrm{poly}} \in \N^{(q+1)(q-1)/4}$ and
$\dg \in \N^{n_q}$ by
\begin{equs}
\dg_{\mrm{poly}} := \Big( j+k,  ~~ 0\leq j \leq \frac{q+1}{2}, 0 \leq k \leq \frac{q-1}{2}\Big),~~
\dg := \Big(1, 2, 2, 2, 2, 3, 2, 3, \dg_{\mrm{poly}}\Big).
\end{equs}
Define also $\reg_{\mrm{poly}} \in \R^{(q+1)(q-1)/4}$ and $\reg \in \R^{n_q}$ by
\begin{equs}
\reg_{\mrm{poly}} &:= \Big(-(j+k)\kappa, ~~ 0\leq j \leq \frac{q+1}{2}, 0 \leq k \leq \frac{q-1}{2}\Big),  \\
 \reg &:= \Big(-\kappa, -2\kappa, -2\kappa, 1-2\kappa, 1-2\kappa, -3\kappa, -2\kappa, -3\kappa, \reg_{\mrm{poly}}\Big).
\end{equs}
Define also $\spc \in \{\R^2, \C^2, \C\}^{n_q}$ to index the space that each object takes values in. So for instance, $\spc(1) = \R^2, \spc(2) = \R^2, \spc(3) = \C^2$, etc. Recall the weighted space $\Sc^{1-2\kappa}$ from \cite[Definition 2.11]{BC23}. Define the data space
\begin{equs}
\mc{D}([0, T]) := &\bigg(\prod_{j =1}^3 C_t^0 \Cs_x^{\reg(j)}([0, T] \times \T^2, \spc(j))\bigg) \times \Big(  C_t^0 \Cs_x^{\reg(4)}([0, T] \times \T^2, \spc(4)) \cap \Sc^{1-2\kappa}([0, T])\Big) \\
&\times \prod_{j=5}^{n_q} C_t^0 \Cs_x^{\reg(j)}([0, T] \times \T^2, \spc(j)).
\end{equs}
\begin{remark}
Note that the index $j = 4$ corresponds to the object $\scalebox{0.9}{$\Aquadratic[][r][\st]$}$. The reason why we need the extra $\Sc^{1-2\kappa}$-norm for this object is because we need time regularity of $A_{\leq N} - \linear[\leqN][r][\st]$ to estimate the commutator
\begin{equs}
\Duh\Big( (A_{\leq N} - \linear[\leqN][r][\st]) \parall \ptl_j \philinear[\leqN][r][\st]\Big) - (A_{\leq N} - \linear[\leqN][r][\st]) \parall \Duh\big(\ptl_j \philinear[\leqN][r][\st]\big)
\end{equs}
appearing in \eqref{gauge:eq-commutator-term}, and hence we need time regularity of $\scalebox{0.9}{$\Aquadratic[][r][\st]$}$. 
For a similar remark, see \cite[Remark~5.3]{BC23}. 
\end{remark}

Define the metric $d_T$ on $\mc{D}([0, T])$ by
\begin{equs}
d_T((S_j, j \in [n_q]), (\tilde{S}_j, j \in [n_q])) := \sum_{j \in [n_q]} \Big\|S_j - \tilde{S}_j\Big\|_{C_t^0 \Cs_x^{\reg(j)}}^{1/\dg(j)} + \Big\|S_4 - \tilde{S}_4\Big\|_{\Sc^{1-2\kappa}}^{\frac{1}{2}},
\end{equs}
Here, the purpose of the degrees is only so that $d_1(\Xi_{\leq N}, 0)$ satisfies a sub-Gaussian tail bound, instead of a more general stretched-exponential tail bound (see also \cite[Remark 5.3]{BC23}). For $R \geq 1$, define also the ball
\begin{equs}
\mc{D}_R([0, T]) := \{S \in \mc{D}([0, T]) \colon d_T(0, S) \leq R\}.
\end{equs}

\begin{example}[Fourier space representations of objects]\label{example:objects-fourier-representation}
To help the reader relate back to our previous work \cite{BC23}, we briefly describe how we would write some of the stochastic objects in $\Xi_{\leq N}$ in the notation of \cite{BC23}. Recalling the harmonic analysis and probability theory preliminaries from Subsections \ref{section:prelimary-harmonic} and \ref{section:preliminary-probability}, we would write the linear object (recall that $\langle n \rangle^2 = |n|^2 + 1$) 
\begin{equs}\label{eq:philinear-fourier-representation}
\philinear[\leqN][r][\st](z) = \frac{1}{2\pi}\sum_{n \in \Z^2} \rho_{\leq N}(n) \e_n(x) \int_{-\infty}^t e^{-(t-s)\langle n \rangle^2} dW_\zeta(s, n),
\end{equs}
where $((W_\zeta(t, n))_{t \in \R})_{n \in \Z^2}$ is as in \eqref{eq:space-time-white-noise-fourier}. 

For a quadratic object, we would for instance write (in the following, $n_{12} = n_1 + n_2$) 
\begin{equs}
\Aquadratic[\leqN][r][\st](z_0) = \frac{1}{(2\pi)^2}
&\sum_{n_1, n_2 \in \Z^2} \rho_{\leq N}(n_1) \rho_{\leq N}(n_2) \e_{n_{12}}(x_0) ~\times\\
&\int_{-\infty}^{t_0} dt e^{-(t_0 - t) \langle n_{12} \rangle^2} \int_{-\infty}^t \int_{-\infty}^t ds_1 ds_2 e^{-(t-s_1) \langle n_1 \rangle^2} e^{-(t-s_2) \langle n_2 \rangle^2} dW_\zeta(s, n_1) dW_\zeta(s, n_2).
\end{equs}
This is very similar to the representation given by \cite[Lemma 5.15]{BC23}.
\end{example}

We next state the analog of \cite[Proposition 5.4]{BC23}.

\begin{proposition}[Control of enhanced data set]\label{prop:control-of-enhanced-data-set}
Let $0<c\ll 1$ be a sufficiently small absolute constant and let $R \geq 1$. Then, there exists an event $E_R$ satisfying
\begin{equation*}
\mathbb{P} \big( E_R \big) \geq 1 - c^{-1} \exp \Big( - c R^2\Big), 
\end{equation*}
and such that, on this event,
\begin{align}\label{objects:eq-enhanced-main-estimate}
\sup_{N \in \dyadic} \dc_1 \big( \Xi_{\leq N}, 0 \big) \leq R \qquad \text{and} \qquad 
\lim_{M,N\rightarrow \infty} 
\dc_1 \big( \Xi_{\leq M}, \Xi_{\leq N} \big) =0. 
\end{align}
\end{proposition}
\begin{proof}[Proof sketch]
We briefly map out the results of \cite[Section 5]{BC23} which give control on the stochastic objects in $\Xi_{\leq N}$. 
\begin{enumerate}[itemsep=3mm]
    \item Control of the linear objects $\linear[][r][\st], \philinear[][r][\st]$ follows by \cite[Lemma 5.6]{BC23}.
    \item The quadratic objects $\Aquadraticnlst, \linear[][r][\st] \philinear[][r][\st]$ are similar to \cite[Lemma 5.8]{BC23}.
    \item The quadratic objects $\Aquadratic[][r][\st], \mixedquadratic[][r][\st]$ are similar to \cite[Lemma 5.13]{BC23}.
    \item The cubic object $\Aquadraticnlst \philinear[][r][\st]$ is simlar to to \cite[Lemma 5.11]{BC23}. 
    \item The object $\bigg(\Im\Big(\icomplex \ovl{\philinear[\leqN][r][\st]} \ptl_i \ptl_j \Duh\big(\philinear[\leqN][r][\st]\big)\Big) + \frac{\delta_{ij}}{4} \sigma_{\leq N}^2\bigg)_{i \in [2]}$ is similar to \cite[Lemma 5.18]{BC23}. 
    \item The object $\bigg(\Im\Big(\ovl{\philinear[\leqN][r][\st]} \ptl^i \mixedquadratic[\leqN][r][\st]\Big) + \frac{1}{2} \sigma_{\leq N}^2 \linear[\leqN][r][\st, i]\bigg)_{i \in [2]}$ is similar to \cite[Lemma 5.21]{BC23}.
    \item Finally, while the general degree polynomial objects $\biglcol\,\philinear[][r][\st, j] \ovl{\philinear[][r][\st, k]}\bigrcol$ are not handled in \cite{BC23}, control of these objects is by now classical, see e.g. \cite[Lemma 3.2]{DPD02}\footnote{Technically, this reference looks at real-valued objects, but the extension to the complex case is straightforward.}.
\end{enumerate}
\end{proof}

By small variations of the nonlinear estimates in \cite[Section 6]{BC23}, we have the following result.

\begin{proposition}[Local well-posedness of the para-controlled stochastic Abelian Higgs equations]\label{prop:lwp-para-sah}
Let $\Cgauge \in \R$. Let $C\geq 1$ be a sufficiently large absolute constant, let $R,S\geq 1$, and let $0<\tau\leq C^{-1} (RS)^{-C}$. Furthermore, suppose $\|A_0\|_{\Cs_x^{-\kappa}}, \|\phi_0\|_{\Cs_x^{-\kappa}} \leq S$ and suppose  $\Xi_{\leq N}\in \Dc_R([0,\tau])$. Then, there exists a unique solution 
\begin{equation*}
(X_{\leq N}, \eta_{\leq N}, \psi_{\leq N}) \in (\Sc^{2-5\kappa} \times \Sc^{1-2\kappa}\times \Sc^{2-5\kappa}) ([0,\tau])
\end{equation*}
of the para-controlled stochastic Abelian Higgs equations \eqref{eq:para-sah}-\eqref{eq:local-theory-initial-data}. Furthermore, it can be written as
\begin{equation*}
(X_{\leq N}, \eta_{\leq N}, \psi_{\leq N}) = \mbb{S}_{R, S, \tau}(\Xi_{\leq N}, A_0, \phi_0),
\end{equation*}
where 
\begin{equation*}
\mbb{S}_{R, S, \tau} \colon \mc{D}_R([0, \tau]) \times \mc{B}_{\leq S}^{-\kappa} \ra (\Sc^{2-5\kappa} \times \Sc^{1-2\kappa}\times \Sc^{2-5\kappa}) ([0,\tau])
\end{equation*}
is Lipschitz continuous. Here, 
\begin{equation*}
\mc{B}_{\leq S}^{-\kappa} = \big\{(A_0, \phi_0) \colon A_0 \in \Cs_x^{-\kappa}(\T^2 \rightarrow \R^2), \phi_0 \in \Cs_x^{-\kappa}(\T^2 \rightarrow \C), \|A_0\|_{\Cs_x^{-\kappa}}, \|\phi_0\|_{\Cs_x^{-\kappa}} \leq S \big\}    
\end{equation*}
is the ball of radius $S$ for the space of initial data.
\end{proposition}
\begin{proof}[Sketch]
As usual, Proposition \ref{prop:lwp-para-sah} is proven by a contraction mapping argument. From the para-controlled stochastic Abelian Higgs equations \eqref{eq:para-sah}-\eqref{eq:local-theory-initial-data}, one defines a map $M(\cdot, \cdot, \cdot; \Xi_{\leq N}, A_0, \phi_0)$, which is shown via standard estimates to be a self-map and $\frac{1}{2}$-contraction on $(\Sc^{2-5\kappa} \times \Sc^{1-2\kappa} \times \Sc^{2-5\kappa})([0, \tau])$. See \cite[Section 6]{BC23} for more details. The only difference is that we use the null-form estimate Lemma \ref{prelim:lem-null} in order to estimate $\ptl_j\big(A^j_{\leq N}(\phi_{\leq N} - \philinear[\leqN][r][\st])\big)$ -- recall Remark \ref{remark:local-theory-remarks}\ref{item:null-form-estimate}.
\end{proof}

We state the following corollary of Proposition \ref{prop:lwp-para-sah}, whose proof is omitted. This will allow us to work directly with the limiting solution in our globalization arguments of Sections \ref{section:Abelian-Higgs} and \ref{section:decay}.

\begin{corollary}\label{cor:blowup-characterization}
Let $\Cgauge \in \R$. Let $(A, \phi) = \mbb{S}(A_0, \phi_0, \xi, \zeta; \Cgauge)$ and $(A_{\leq N}, \phi_{\leq N}) = \mbb{S}(A_0, \phi_0, \xi_{\leq N}, \zeta_{\leq N};\Cgauge)$. If, for some $T > 0$, $(A, \phi)$ exists in $C_t^0 \Cs_x^{-\kappa}([0, T] \times \T^2)$, then for $N$ large enough, $(A_{\leq N}, \phi_{\leq N})$ also exists in $C_t^0 \Cs_x^{-\kappa}([0, T]\times T^2)$, and furthermore,
\begin{equs}
\lim_{N \toinf} \Big( \|A_{\leq N} - A\|_{C_t^0 \Cs_x^{-\kappa}([0, T] \times \T^2) } + \|\phi_{\leq N} - \phi\|_{C_t^0 \Cs_x^{-\kappa}([0, T] \times \T^2)} \Big) = 0.
\end{equs}
\end{corollary}

Now to prove gauge covariance, it suffices to show the following. First, for $N \in \dyadic$, $n_0 \in \Z^2$, define the following modified noise:
\begin{equs}
\tilde{\zeta}_{\leq N}(w) := e^{\icomplex n_0 \cdot y} (e^{-\icomplex n_0 \cdot } \zeta)_{\leq N}(w) =  e^{\icomplex n_0 \cdot y} \int dy' e^{-\icomplex n_0 \cdot y'} \zeta(w')   \moll_{\leq N}(y - y') = \int dy' \zeta(w')  \moll_{\leq N}(y - y') e^{\icomplex n_0 \cdot (y - y')}.
\end{equs}
Here, $w = (s, y)$ and $w' = (s, y')$.

\begin{remark}\label{remark:gauge-transformation-noise-shifts-symbol}
In Fourier space, the modified noise $\tilde{\zeta}_{\leq N}$ can be represented as follows. Recall from equation \eqref{eq:space-time-white-noise-fourier} that $\zeta(t) = \frac{1}{2\pi} \sum_n \e_n dW_\zeta(t, n)$. Then
\begin{equs}
\zeta_{\leq N}(t) = \frac{1}{2\pi} \sum_n \rho_{\leq N}(n) \e_n dW_\zeta(t, n) , \qquad \text{ while } \qquad \tilde{\zeta}_{\leq N}(t) = \frac{1}{2\pi} \sum_n \rho_{\leq N}(n - n_0) \e_n dW_\zeta (t, n)  .
\end{equs}
That is, the argument of the symbol $\rho_{\leq N}$ is simply shifted by $n_0$.
\end{remark}

Let $\tilde{A}_{\leq N}, \tilde{\phi}_{\leq N}$ be the solution to the following modified version \eqref{intro:eq-SAH-smooth}:
\begin{align}\label{eq:modified-sah-smoothed}
(\ptl_t - \Delta) \tilde{A}_{\leq N} &= -\leray \Im(\ovl{\tilde{\phi}_{\leq N}} \covd_{\tilde{A}_{\leq N}} \tilde{\phi}_{\leq N}) + \Cgauge \big( \tilde{A}_{\leq N} + n_0 \big) + \leray \xi_{\leq N},  \\
(\ptl_t - \covd_{\tilde{A}_{\leq N}}^j \covd_{\tilde{A}_{\leq N}, j} - 2\sigma_{\leq N}^2) \tilde{\phi}_{\leq N} &= -\biglcol\,|\tilde{\phi}_{\leq N}|^{q-1} \tilde{\phi}_{\leq N}\bigrcol\, + \tilde{\zeta}_{\leq N}, \\
(\tilde{A}(0), \tilde{\phi}(0)) &= (A_0, \phi_0),
\end{align}
where $\biglcol\,|\tilde{\phi}_{\leq N}|^{q-1} \tilde{\phi}_{\leq N}\bigrcol\, = H_{\frac{q+1}{2}, \frac{q-1}{2}}(\tilde{\phi}_{\leq N}, \ovl{\tilde{\phi}_{\leq N}}, \sigma^2_{\leq N})$ and $\sigma^2_{\leq N}$ is as in \eqref{eq:wick-ordered-products}. Here, the only differences are the additional $\Cgauge n_0$ term in the $A$-equation and the modified noise $\tilde{\zeta}_{\leq N}$ in the $\phi$-equation. This equation is chosen so that $(\tilde{A}_{\leq N}, \tilde{\phi}_{\leq N})$ has the following property.

\begin{lemma}\label{lemma:gauge-covariance-sah-computation}
Let $\Cgauge \in \R$. Then, the gauge transformed pair $(\tilde{A}_{\leq N}, \tilde{\phi}_{\leq N})^{n_0}$ is a solution to \eqref{intro:eq-SAH-smooth} with initial data $(A_0, \phi_0)^{n_0}$ and noise $(\xi_{\leq N}, (\zeta^{n_0})_{\leq N})$. Put more succinctly,
\begin{equs}
(\tilde{A}_{\leq N}, \tilde{\phi}_{\leq N})^{n_0} = \solutionmap\Big(A_0^{n_0}, \phi_0^{n_0}, \xi_{\leq N}, ( \zeta^{n_0})_{\leq N}; \Cgauge \Big).
\end{equs}
\end{lemma}
\begin{proof}
Since $\tilde{A}_{\leq N}^{n_0} = \tilde{A}_{\leq N} + n_0$, we have that
\begin{equs}
(\ptl_t - \Delta) \tilde{A}_{\leq N}^{n_0} = (\ptl_t - \Delta) \big( \tilde{A}_{\leq N} + n_0 \big) &= -\leray \Im(\ovl{\tilde{\phi}_{\leq N}} \covd_{\tilde{A}_{\leq N}} \tilde{\phi}_{\leq N}) + \Cgauge \big(  \tilde{A}_{\leq N} + n_0 \big) + \leray \xi_{\leq N} \\
&= -\leray \Im(\ovl{\tilde{\phi}_{\leq N}^{n_0}} \covd_{\tilde{A}_{\leq N}^{n_0}} \tilde{\phi}_{\leq N}^{n_0}) + \Cgauge \tilde{A}_{\leq N}^{n_0} + \leray \xi_{\leq N} .
\end{equs}
Next, by covariance of covariant derivatives, we may compute
\begin{equs}
(\ptl_t - \covd_{\tilde{A}_{\leq N}^{n_0}}^j \covd_{\tilde{A}_{\leq N}^{n_0}, j} - 2\sigma^2_{\leq N}) \tilde{\phi}_{\leq N}^{n_0} &= \Big((\ptl_t - \covd_{\tilde{A}_{\leq N}}^j \covd_{\tilde{A}_{\leq N}, j} - 2\sigma^2_{\leq N}) \tilde{\phi}_{\leq n}\Big)^{n_0} \\
&= - \biglcol\,|\tilde{\phi}_{\leq N}^{n_0}|^{q-1} \tilde{\phi}_{\leq N}^{n_0}\bigrcol\, + \tilde{\zeta}_{\leq N}^{n_0}.
\end{equs}
To finish, note that by the definition of $\tilde{\zeta}_{\leq N}$, we have that $\tilde{\zeta}_{\leq N}^{n_0} = (\zeta^{n_0})_{\leq N}$. 
\end{proof}

Lemma \ref{lemma:gauge-covariance-sah-computation} leads directly to the following corollary, which gives gauge covariance of \eqref{intro:eq-SAH} provided one can show that $(A_{\leq N}, \phi_{\leq N})$ and $(\tilde{A}_{\leq N}, \tilde{\phi}_{\leq N})$ approach each other as $N \toinf$.

\begin{corollary}\label{cor:A-tilde-A-converge-implies-gauge-covariant}
Let $\tau_0 > 0$ be an interval on which
\begin{equs}\label{eq:A-phi-tilde-A-phi-zero}
\lim_{N \toinf} \Big\|(A_{\leq N}, \phi_{\leq N}) - (\tilde{A}_{\leq N}, \tilde{\phi}_{\leq N})\Big\|_{C_t^0 \Cs_x^{-\kappa}([0, \tau_0] \times \T^2)} = 0,
\end{equs}
and such that $\lim_{N \toinf} (A_{\leq N}, \phi_{\leq N})$ exists in $C_t^0 \Cs_x^{-\kappa}([0, \tau_0] \times \T^2)$. Then on the same interval,
\begin{equs}
\mbb{S}\Big(A_0, \phi_0, \xi, \zeta; \Cgauge\Big)^{n_0} = \mbb{S}\Big(A_0^{n_0}, \phi_0^{n_0}, \xi, \zeta^{n_0}; \Cgauge\Big).
\end{equs}
\end{corollary}
\begin{proof}
We have that
\begin{equs}
\mbb{S}\Big(A_0^{n_0}, \phi_0^{n_0}, \xi, \zeta^{n_0}; \Cgauge \Big) &= \lim_{N \toinf}\mbb{S}\Big( A_0^{n_0}, \phi_0^{n_0}, \xi_{\leq N}, (\zeta^{n_0})_{\leq N}; \Cgauge\Big) = \lim_{N \toinf} ~(\tilde{A}_{\leq N}, \tilde{\phi}_{\leq N})^{n_0} \\
&= \lim_{N \toinf} ~(A_{\leq N}, \phi_{\leq N})^{n_0} = \mbb{S}\Big(A_0, \phi_0, \xi, \zeta; \Cgauge\Big)^{n_0},
\end{equs}
as desired.
\end{proof}

Clearly, to prove gauge covariance of \eqref{intro:eq-SAH} (Theorem \ref{thm:gauge-covariance}), it now just suffices to show \eqref{eq:A-phi-tilde-A-phi-zero}. To begin towards this, define the following modified linear object:
\begin{equs}
\glinear[\leqN][r][\st](t) := \int_{-\infty}^t ds e^{(t-s)(\Delta-1)} \tilde{\zeta}_{\leq N}(s).
\end{equs}
Define the modified quadratic objects
\begin{equs}\label{eq:gAquadratic-modified}
\gAquadratic[\leqN][r][\st](t) &:=  
\int_{-\infty}^t ds e^{(t-s)(\Delta - 1)} \Big( - \leray\Im\Big( \glinear[\leqN][r][\st](s) \covd \glinear[\leqN][r][\st](s)\Big) + \Cgauge n_0\Big), \\
\gphiquadratic[\leqN][r][\st](t)  &:= \int_{-\infty}^t ds e^{(t-s)(\Delta - 1)} \linear[\leqN][r][\st, j](s) \ptl_j \glinear[\leqN][r][\st](s), 
\end{equs}
Note here that we place the extra $\Cgauge n_0$-term appearing in \eqref{eq:modified-sah-smoothed} in the quadratic object $\scalebox{0.9}{$\gAquadratic[\leqN][r][\st]$}$. By including the $\Cgauge$-term in \eqref{eq:gAquadratic-modified}, as well as choosing $\Cgauge=\gaugerenorm$ below, we will be able to prove Lemma \ref{lemma:convergence-of-enhanced-data-sets} below.  
Define the gauge transformed enhanced data set
\begin{equs}
\Xi^g_{\leq N} := \bigg(\linear[\leqN][r][\st], &\Aquadraticnlst[\leqN], \linear[\leqN][r][\st] \glinear[\leqN][r][\st], \gAquadratic[\leqN][r][\st],  \gphiquadratic[\leqN][r][\st], \Aquadraticnlst[\leqN]\glinear[\leqN][r][\st],    \\
& \bigg(\Im\Big(\icomplex \ovl{\glinear[\leqN][r][\st]} \ptl_i \ptl_j \Duh\big(\glinear[\leqN][r][\st]\big)\Big) + \frac{\delta^{ij}}{4} \sigma_{\leq N}^2\bigg)_{i \in [2]}, \\
& \bigg(\Im\Big(\ovl{\glinear[\leqN][r][\st]} \ptl^i \gphiquadratic[\leqN][r][\st]\Big) + \frac{1}{4} \sigma_{\leq N}^2 \linear[\leqN][r][\st, i]\bigg)_{i \in [2]}, \\
&\biglcol\, \glinear[\leqN][r][\st, j] \ovl{\glinear[\leqN][r][\st, k]}\bigrcol 0 \leq j \leq \frac{q+1}{2}, 0 \leq k \leq \frac{q-1}{2}\bigg) .
\end{equs}
\begin{remark}\label{remark:same-ansatz}
Due to our definition of the modified quadratic object from \eqref{eq:gAquadratic-modified}, we can then make the exact same ansatz for $(\tilde{A}_{\leq N}, \tilde{\phi}_{\leq N})$ as we did for $(A_{\leq N}, \phi_{\leq N})$, where now we use the enhanced data set $\Xi_{\leq N}^g$ in place of $\Xi_{\leq N}$ to define the para-controlled system for $(\tilde{X}_{\leq N}, \tilde{\eta}_{\leq N}, \tilde{\psi}_{\leq N})$.
\end{remark}

The crux of the argument for showing \eqref{eq:A-phi-tilde-A-phi-zero} is the following result about enhanced data sets.

\begin{lemma}[Convergence of the enhanced data sets]\label{lemma:convergence-of-enhanced-data-sets}
Let $\Cgauge=\gaugerenorm$. Then, we have that 
\begin{equs}
\lim_{N \toinf} d_1(\Xi_{\leq N}, \Xi_{{\leq N}}^g) \stackrel{a.s.}{=} 0.
\end{equs}
\end{lemma}

\begin{remark}\label{remark:gauge-renorm-uniquely-determined}
From \eqref{eq:gAquadratic-modified}, we see that there can be at most one constant $\Cgauge$ for which Lemma \ref{lemma:convergence-of-enhanced-data-sets} holds, and thus $\Cgauge=\frac{1}{8\pi}$ is uniquely specified in this sense. Of course, it would be interesting to have a more intrinsic characterization of the uniqueness of $\gaugerenorm$, similar to what is shown in \cite{CS23} for 2D pure Yang--Mills, but we do not pursue that here.
\end{remark}

\begin{proof}
Recall the representation \eqref{eq:philinear-fourier-representation} of the linear object. Then by Remark \ref{remark:gauge-transformation-noise-shifts-symbol}, we have that
\begin{equs}
\philinear[\leqN][r][\st](z) - \glinear[\leqN][r][\st](z) = \frac{1}{2\pi} \sum_n \big(\rho_{\leq N}(n) - \rho_{\leq N}(n - n_0)\big) \e_n(x) \int_{-\infty}^t e^{-(t-s)\langle n\rangle^2} dW(s, n).
\end{equs}
By Taylor expanding the symbol $\rho_{\leq N}$, we may write $\rho_{\leq N}(n) - \rho_{\leq N}(n - n_0) = n_0 \nabla \rho_{\leq N}(\tilde{n}) = \frac{n_0}{N} N \nabla \rho_{\leq N}(\tilde{n})$. Here, we group $N \nabla \rho_{\leq N}(\tilde{n})$, because this defines a multiplier operator of order 1 (for instance, the corresponding convolution kernel has $L^1$ norm $\sim 1$). Thus, we have gained a power of $N$. Now by slightly modifying the proof of control on the linear object from \cite[Lemma 5.6]{BC23}, this extra $N$ factor can be simultaneously used to gain almost one derivative, as well as show convergence to zero, i.e. to show that:
\begin{equs}\label{eq:linear-objects-converge}
\lim_{N \toinf} \Big\|\philinear[\leqN][r][\st] - \glinear[\leqN][r][\st]\Big\|_{C_t^0 \Cs_x^{1-\kappa}([0, 1] \times \T^2)} \stackrel{a.s.}{=} 0.
\end{equs}
Given this, the convergence of all polynomial objects in $C_t^0 \Cs_x^{-\kappa}([0, 1] \times \T^2)$ just follows from the identity \eqref{eq:complex-hermite-polynomial-expansion}, which gives
\begin{equs}
H_{m, n}\Big(\glinear[\leqN][r][\st], \ovl{\glinear[\leqN][r][\st]}; \sigma_{\leq N}^2\Big) -& H_{m, n}\Big(\philinear[\leqN][r][\st], \ovl{\philinear[\leqN][r][\st]}; \sigma_{\leq N}^2\Big) = \\
&\sum_{\substack{0 \leq j \leq m \\ 0 \leq k \leq n \\ (j, k) \neq (0, 0)}} \binom{m}{j} \binom{n}{k} \Big( \glinear[\leqN][r][\st] - \philinear[\leqN][r][\st] \Big)^{j} \ovl{\Big(\glinear[\leqN][r][\st] - \philinear[\leqN][r][\st] \Big)}^k H_{j, k}\Big(\philinear[\leqN][r][\st], \ovl{\philinear[\leqN][r][\st]}; \sigma_{\leq N}^2\Big).
\end{equs}
The convergence to zero of all other objects (besides $\scalebox{0.9}{$\Aquadratic[][r][\st] - \gAquadratic[][r][\st]$}$) can be shown by similar considerations: we simply Taylor expand $\rho_{\leq N}(n) - \rho_{\leq N}(n - n_0)$ and gain a factor of $N$, which then directly gives convergence to zero (for the other objects, we do not worry about gaining derivatives, which makes seeing the convergence even easier).

It remains to discuss the difference of quadratic objects $\scalebox{0.9}{$\Aquadratic[][r][\st] - \gAquadratic[][r][\st]$}$, which does not automatically follow by the same considerations, due to the slightly modified definition of $\scalebox{0.9}{$\gAquadratic[][r][\st]$}$ (equation \eqref{eq:gAquadratic-modified}). However, since the modified definition only affects the resonant part of 
$\scalebox{0.9}{$\gAquadratic[][r][\st]$}$, the non-resonant parts of these objects may be shown to converge in $C_t^0 \Cs_x^{1-\kappa}$ by the same Taylor expansion considerations as before. Thus, we look at the resonant parts. We separate this argument out into Lemma \ref{lemma:resonant-parts-converge-fourier-calculation}. Conditional on this lemma, which we prove next, the desired result follows.
\end{proof}

\begin{lemma}\label{lemma:resonant-parts-converge-fourier-calculation}
Let $\Cgauge=\gaugerenorm$. Then, we have that
\begin{equs}
\lim_{N \toinf} \bigg\| \E\Big[ \Aquadratic[\leqN][r][\st] - \gAquadratic[\leqN][r][\st]\Big] \bigg\|_{C_t^0 \Cs_x^{1-\kappa}([0, 1] \times \T^2)} = 0.
\end{equs} 
\end{lemma}
\begin{proof} 
By the definitions  of the quadratic objects (\eqref{eq:Aquadratic} and \eqref{eq:gAquadratic-modified}), and recalling Example \ref{example:objects-fourier-representation}, we have that
\begin{equs}\label{eq:resonant-part-representation-intermediate}
\E\Big[ \Big( \Aquadratic[\leqN][r][\st]- \gAquadratic[\leqN][r][\st]\Big)(z_0)\Big] = \mDuh \Big(-\leray  \Big( \E\Big[\Im\Big(\ovl{\philinear[\leqN][r][\st]} \covd \philinear[\leqN][r][\st]\Big)\Big] - \E\Big[\Im\Big(\ovl{\glinear[\leqN][r][\st]} \covd \glinear[\leqN][r][\st]\Big) \Big]\Big) - \gaugerenorm n_0\Big)(z_0), \qquad 
\end{equs}
where $\mDuh$ denotes the Duhamel operator corresponding to the massive Laplacian $\Delta - 1$. We have that
\begin{equs}
R_{\leq N} &:= \E\Big[\Im\Big(\ovl{\philinear[\leqN][r][\st]} \covd \philinear[\leqN][r][\st]\Big)(z)\Big] - \E\Big[\Im\Big(\ovl{\glinear[\leqN][r][\st]} \covd \glinear[\leqN][r][\st]\Big)(z)\Big]  \\
&= \frac{1}{(2\pi)^2} \sum_{n \in \Z^2} \big(\rho_{\leq N}(n)^2 - \rho_{\leq N}(n - n_0)^2 \big) n \int_{-\infty}^t  e^{-2(t-s)\langle n \rangle^2} ds \\
&= \frac{1}{(2\pi)^2} \sum_{n \in \Z^2} \big(\rho_{\leq N}(n)^2 - \rho_{\leq N}(n - n_0)^2 \big) \frac{n}{2\langle n \rangle^2} . 
\end{equs}
In particular, note that $R_{\leq N}$ is constant in $x$. From Taylor expansion, we obtain that
\begin{equation*}
\rho_{\leq N}(n)^2 - \rho_{\leq N}(n - n_0)^2 = 2 \rho_{\leq N}(n) \nabla \rho_{\leq N}(n) \cdot n_0 + O(N^{-2}). 
\end{equation*}
By also using that $\nabla \rho_{\leq N}(n)$ is supported on $|n| \sim N$ and $\langle n \rangle^{-2} - |n|^{-2} = O(|n|^{-4})$, we further obtain that  (here $o(1)$ is with respect to $N \toinf$)
\begin{equs}
R_{\leq N} = \frac{1}{(2\pi)^2} \sum_{n\in \Z^2} \frac{\rho_{\leq N}(n) n \nabla \rho_{\leq N}(n) \cdot n_0}{\langle n \rangle^2} + o(1) 
&= \frac{1}{(2\pi)^2} \bigg(\sum_{n\in \Z^2} \frac{\rho_{\leq N}(n) n (\nabla \rho_{\leq N}(n))^T}{|n|^2}\bigg) n_0 + o(1).
\end{equs}
Next, recall (from Section \ref{section:prelimary-harmonic}) that $\rho_{\leq N}(n) = \rho(\frac{n}{N})$, so that $\nabla \rho_{\leq N} = N^{-1} (\nabla \rho) (\frac{n}{N})$. Moreover, note that $\rho \nabla \rho = \frac{1}{2} \nabla (\rho^2)$. Since $\rho$ is radial, so is $\rho^2$, and so let $f \colon\R \ra \R$ be such that $\rho^2(x) = f(|x|)$. Then $\nabla (\rho^2)(x) = f'(|x|) \frac{x}{|x|}$. From this, we obtain that
\begin{equs}
\frac{1}{(2\pi)^2} \sum_{n\in \Z^2} \frac{\rho_{\leq N}(n) n (\nabla \rho_{\leq N}(n))^T}{|n|^2} &= \frac{1}{8\pi^2} \sum_{n \in \Z^2} \frac{n (\nabla (\rho^2) (\frac{n}{N}))^T }{N |n|^2} = \frac{1}{8\pi^2} \sum_{n \in \Z^2} f'\Big(\frac{|n|}{N}\Big) \frac{nn^T}{N |n|^3} \\
&=\frac{1}{8\pi^2} N^{-2} \sum_{n \in N^{-1} \Z^2} \frac{f'(|n|)}{|n|} \frac{nn^T}{|n|^2} \\
&\ra 
\frac{1}{8\pi^2} \int_{\R^2} du \frac{f'(|u|)}{|u|} \frac{uu^T}{|u|^2} , \label{eq:limiting-resonance-gauge-covariance-proof}
\end{equs}
where the final limit is as $N \toinf$. Note that \eqref{eq:limiting-resonance-gauge-covariance-proof} commutes with any orthogonal matrix, and thus we have that $\eqref{eq:limiting-resonance-gauge-covariance-proof} = c I$ for some constant $c$. To determine the constant, we may take trace and compute
\begin{equs}
\Tr(\eqref{eq:limiting-resonance-gauge-covariance-proof}) = \frac{1}{8\pi^2} \int_{\R^2} du \frac{f'(|u|)}{|u|} = \frac{2\pi}{8\pi^2} \int_0^\infty f'(r) dr = - \frac{1}{4\pi} f(0).
\end{equs}
From \eqref{eq:rho-normalization}, we have that $f(0) = \rho^2(0) = 1$, and thus we obtain $\Tr(\eqref{eq:limiting-resonance-gauge-covariance-proof}) = - \frac{1}{4\pi}$. From this, we obtain $\eqref{eq:limiting-resonance-gauge-covariance-proof} = -\gaugerenorm I$. Combining the previous few results, and recalling that $R_{\leq N}$ is constant, we have that
\begin{equs}
\mDuh \Big(-\leray R_{\leq N} - \gaugerenorm n_0\Big) = \mDuh \Big(- R_{\leq N} - \gaugerenorm n_0 \Big) &= \mDuh \Big( o(1)\Big) \\
&\ra 0 \text{ in $C_t^0 \Cs_x^{1-\kappa}([0, 1] \times \T^2)$ as $N \toinf$.}
\end{equs}
Recalling \eqref{eq:resonant-part-representation-intermediate}, this shows the desired result.
\end{proof}

\begin{remark}\label{remark:two-calculations}
In Section \ref{section:cshe}, we will essentially redo the calculation in the proof of Lemma \ref{lemma:resonant-parts-converge-fourier-calculation}, except in real space -- see Lemma \ref{lemma:constant-A-limit-resonance}. We decided to keep both calculations, because in the present section, we are relying heavily on our previous work \cite{BC23}, where all the analysis was done in Fourier space. Thus for the sake of consistency, we also perform the calculation in Fourier space. By contrast, in Section \ref{section:cshe}, we will exclusively work in real space (see the discussion at the beginning of that section for why), so again for the sake of consistency there, we later perform the calculation in real space.
\end{remark}

\begin{proof}[Proof of Theorem \ref{thm:gauge-covariance}]
To wrap up the proof of gauge covariance, we proceed very similarly to the arguments in \cite[Section 7.5]{BC23} (in fact, the present situation is much simpler since we never need to consider time-dependent gauge transformations). As previously remarked, by Corollary \ref{cor:A-tilde-A-converge-implies-gauge-covariant} we reduce to showing \eqref{eq:A-phi-tilde-A-phi-zero}. Let $(X_{\leq N}, \eta_{\leq N}, \psi_{\leq N})$ be the solution to the para-controlled stochastic Abelian Higgs equation (as in \eqref{eq:para-sah}-\eqref{eq:local-theory-initial-data}). Let $(\tilde{X}_{\leq N}, \tilde{\eta}_{\leq N}, \tilde{\psi}_{\leq N})$ be the solution corresponding to $(\tilde{A}_{\leq N}, \tilde{\phi}_{\leq N})$. By convergence of the enhanced data sets (Lemma \ref{lemma:convergence-of-enhanced-data-sets}), it suffices to show that 
\begin{equs}
\lim_{N \toinf} \Big\| (X_{\leq N}, \eta_{\leq N}, \psi_{\leq N}) - (\tilde{X}_{\leq N}, \tilde{\eta}_{\leq N}, \tilde{\psi}_{\leq N})\Big\|_{\Sc^{2-5\kappa} \times \Sc^{1-2\kappa} \times \Sc^{2-5\kappa}} = 0.
\end{equs}
As in the proof sketch of Proposition \ref{prop:lwp-para-sah}, let $M(\cdot, \cdot, \cdot; \Xi_{\leq N}, A_0, \phi_0)$ be the contraction map for which $S_{\leq N} := (X_{\leq N}, \eta_{\leq N}, \psi_{\leq N})$ is a fixed point. Since we made the same ansatz for $(\tilde{A}_{\leq N}, \tilde{\phi}_{\leq N})$ as we did for $(A_{\leq N}, \phi_{\leq N})$ (Remark \ref{remark:same-ansatz}), we have that $\tilde{S}_{\leq N} := (\tilde{X}_{\leq N}, \tilde{\eta}_{\leq N}, \tilde{\psi}_{\leq N})$ is a fixed point of the map $M(\cdot, \cdot, \cdot; \Xi_{\leq N}^g, A_0, \phi_0)$. We thus have that (for brevity, we omit the subscript $\Sc^{2-5\kappa} \times \Sc^{1-2\kappa} \times \Sc^{2-5\kappa}$ in the norms, as well as the initial data $A_0, \phi_0$ in the map $M$)
\begin{equs}
\Big\|S_{\leq N} - \tilde{S}_{\leq N}\Big\| &= \Big\|M(S_{\leq N}; \Xi_{\leq N}) - M(\tilde{S}_{\leq N}; \Xi_{\leq N}^g)\Big\|\\
&\leq  \Big\|M(S_{\leq N}; \Xi_{\leq N}) - M(\tilde{S}_{\leq N}; \Xi_{\leq N})\Big\| + \Big\|M(\tilde{S}_{\leq N}; \Xi_{\leq N}) - M(\tilde{S}_{\leq N}; \Xi_{\leq N}^g)\Big\| \\
&\leq \frac{1}{2} \big\|S_{\leq N} - \tilde{S}_{\leq N}\big\|+ C \big\|\tilde{S}_{\leq N}\big\| d_1(\Xi_{\leq N}, \Xi_{\leq N}^g).
\end{equs}
Here the final inequality follows by standard estimates which shows that $M$ is Lipschitz in all of its variables. Taking $N \toinf$, the desired result follows.
\end{proof}

\section{Covariant stochastic heat equation}\label{section:cshe}

In this section, we analyze the covariant stochastic heat equation \eqref{intro:eq-cshe}, and in particular prove Theorem~\ref{intro:thm-cshe}. As emphasized in Remark \ref{intro:rem-cshe}, the main point will be to prove estimates with as little dependence on $\Blin$ as possible. The quantities we need to estimate are geometric in nature, and thus it makes sense to try to estimate them in a way that respects the geometry. This is the main philosophy.  

In particular, because we are trying to be as geometric as possible, we will go away from the Fourier methods used in the local theory (Section \ref{section:gauge-covariance}), and instead work in real space.

We start by setting some notation which will be used throughout this section.

\begin{notation}[Spatial kernels]\label{notation:heat-kernel-space-only}
Our notation for heat kernels is $p(w; z)$, because we often deal with spacetime convolutions. When we just want to denote a spatial convolution, it will be helpful to have the notation $p_u(x, x') := p((0, x); (u, x'))$ which is the convolution kernel for the operator $e^{u \Delta}$.

Additionally, we will write $p_u(w; z) := \int dx' p_u(x, x') p_u(w; z')$, which is a smoothed space-time heat kernel. This is a slight clash of notation, but we hope it is clear: whenever we use variables $v, w, z$, we mean the space-time kernel, while whenever we use variables $x, x', y$, we mean the spatial kernel.
\end{notation}

\begin{notation}[Moments]
In this section, we will write $\|\cdot\|_{L^\power(\Omega)}$ instead of $\E[|\cdot|^\power]^{\frac{1}{\power}}$. We find that this leads to shorter and nicer displays, and it combines well when we take expectations of norms.
\end{notation}

\begin{notation}[Additional space-time variables]
Recall Notation \ref{notation:space-time-variables} for our notation involving space-time variables. In this section, we will often need to deal with variables at the same time but with different spatial coordinates. Thus, we will often write $z' = (t, x')$, $w' = (s, y')$, $w_1 = (s, y_1)$, $w_2 = (s, y_2)$, etc. We stress that the additional space-time variables always have the same time variable as the corresponding original space-time variable, e.g. the time variable in $z'$ is the same as the one in $z$.
\end{notation}

Next, we define smoothed versions of the covariant heat kernel. Recall the mollifiers $\moll_{\leq N}$ and $\moll_N$ from Definition \ref{def:mollifiers}.

\begin{definition}\label{def:smoothed-covariant-heat-kernel}
For $N \in \dyadic$, $w, z \in [0, 1] \times \T^2$, define
\begin{equs}
p_{\Blin, \leq N}(w; z) := \int dy' \chi_{\leq N}(y - y') p_{\Blin}(w'; z), ~~ p_{\Blin, N}(w; z) := \int dy' \chi_N(y - y') p_{\Blin}(w'; z),
\end{equs}
and similarly define $\massp_{\Blin, \leq N}, \massp_{\Blin, N}$ using the massive kernel $\massp_{\Blin}$.
\end{definition}

\begin{remark}
Recalling Remark \ref{remark:massive-kernel}, we will freely apply the results of Section \ref{section:monotonicity} which hold for the kernel $p_{\Blin}$, even when our expressions involve $\massp_{\Blin}$. This is because $\massp_{\Blin} = e^{-(t-s)} p_{\Blin}$ (equation \eqref{eq:massive-heat-kernel-formula}), so we can always upper bound $\massp_{\Blin}$ by $p_{\Blin}$. Similarly, we have $\covd_{{\Blin}(z)} \massp_{\Blin}(w; z) = e^{-(t-s)} \covd_{{\Blin}(z)} p_{\Blin}(w; z)$, $\covd_{-{\Blin}(w)} \massp_{\Blin}(w; z) = e^{-(t-s)} \covd_{-{\Blin}(w)} p_{\Blin}(w; z)$.

Additionally, by the diamagnetic inequality (Lemma \ref{kernel:lem-diamagnetic-heat-kernel}), we will freely upper bound $|\massp_{\Blin}(w; z)| \leq |p_{\Blin}(w; z)| \leq p(w; z)$. We may not always explicitly say this.
\end{remark}

\begin{remark}[Young's convolution inequality]
Throughout Section \ref{section:cshe}, we will apply Young's convolution inequality at many points, without explicitly saying so. One example where we apply this inequality is to bound, for a given kernel $K(w)$, and any $N \in \dyadic$,
\begin{equs}
\bigg\|\int dw \int dy' \moll_{\leq N}(y - y') K(w')\bigg\|_{L_w^2} \lesssim \|K(w)\|_{L_w^2}.
\end{equs}
(Here we used that  $\|\moll_{\leq N}\|_{L^1} \lesssim 1$.) The point here is that in cases where we can expect a bound which is uniform in $N$, we can just apply Young's inequality and bound the corresponding quantity at $N = \infty$ (i.e. without noise mollification). In summary, if there is ever a point where the $N$ parameter disappears from line to line, just know that we applied Young's convolution inequality.
\end{remark}



Recall the small parameter $\kappa$ that we fixed in Section \ref{section:parameters}.

\subsection{Power-type nonlinearity}

In this subsection, we will prove Theorem \ref{intro:thm-cshe}\ref{item:thm-cshe-polynomial} and \ref{item:thm:-difference-in-linear-objects}. To begin, we state the following estimate for the usual (non-covariant) polynomial objects. 

\begin{proposition}[Classical estimates for non-covariant objects]\label{prop:classical-estimate-non-covariant-object}
For $j, k \geq 0$, $\power \in [1, \infty)$, we have that 
\begin{equs}
\bigg\| \sup_{N \in \dyadic} \Big\| t^\kappa \biglcol \, \philinear[\leqN][r][j] \ovl{\philinear[\leqN][r][k]}\bigrcol\,\Big\|_{C_t^0 \Cs_x^{-\kappa}([0, 1] \times \T^2)}  \bigg\|_{L^\power(\Omega)} \lesssim \power^{\frac{j+k}{2}}.
\end{equs}
\end{proposition}
\begin{proof}
Using an auxiliary random variable $\Phi$, independent of everything else, define the stationary object
\begin{equs}
(\ptl_t - \Delta + 1)\philinear[\leqN][r][\mrm{st}] = \zeta_{\leq N}, ~~ \philinear[\leqN][r][\mrm{st}](0) = \Phi_{\leq N},
\end{equs}
where $\Phi_{\leq N} = \moll_{\leq N} * \Phi$,
and the law of $\Phi$ is chosen so that the law of $\philinear[\leqN][r][\mrm{st}](t)$ is stationary in time (so $\Phi$ has the law of a complex Gaussian free field). With the stationary object, we have the classical estimate (see e.g. \cite[Lemma 3.2]{DPD02})
\begin{equs}
\bigg\| \sup_{N \in \dyadic} \Big\|\biglcol \, \Big(\philinear[\leqN][r][\mrm{st}]\Big)^j \ovl{\Big(\philinear[\leqN][r][\mrm{st}]\Big)^k}\bigrcol\,\Big\|_{C_t^0 \Cs_x^{-\kappa}([0, 1] \times \T^2)}  \bigg\|_{L^\power(\Omega)} \lesssim  \power^{\frac{j+k}{2}}.
\end{equs}
Note that $\philinear[\leqN](t) - \philinear[\leqN][r][\mrm{st}](t) = e^{t(\Delta - 1)} \Phi_{\leq N} =: D_{\leq N} $. Using the complex Hermite polynomial expansion from Lemma~\ref{lemma:complex-hermite-polynomial-expansion}, we have that
\begin{equs}
\biglcol\, \philinear[\leqN][r][j] \ovl{\philinear[\leqN][r][k]}\bigrcol\, - \biglcol \, \Big(\philinear[\leqN][r][\mrm{st}]\Big)^j \ovl{\Big(\philinear[\leqN][r][\mrm{st}]\Big)^k}\bigrcol\, = \sum_{\substack{0 \leq \ell_1 \leq j \\ 0 \leq \ell_2 \leq k \\ (\ell_1, \ell_2) \neq (0, 0)}} \binom{j}{\ell_1} \binom{k}{\ell_2} \biglcol \, \Big(\philinear[\leqN][r][\mrm{st}]\Big)^{\ell_1}\ovl{\Big(\philinear[\leqN][r][\mrm{st}]\Big)^{\ell_2}}\bigrcol\, D_{\leq N}^{j - \ell_1} \ovl{D_{\leq N}}^{k - \ell_2}.
\end{equs}
Using that $D_{\leq N}$ is the linear heat flow of a (mollified) Gaussian free field and using Lemma \ref{lemma:heat-flow-smoothing}, 
we obtain for all $\kappa>0$ and $\power\geq 1$ that 
\begin{equs}
\bigg\| \sup_{N \in \dyadic} \Big\|t^\kappa D_{\leq N}\Big\|_{C_t^0 \Cs_x^{3\kappa/2}([0, 1] \times \T^2)} \bigg\|_{L^\power(\Omega)} \lesssim \Big\| \|\Phi\|_{\Cs_x^{-\kappa/2}(\T^2)} \Big\|_{L^\power(\Omega)} \lesssim \power^{\frac{1}{2}}.
\end{equs}
The desired result now follows by combining the two estimates with the expansion formula.
\end{proof}

We next note that Theorem \ref{intro:thm-cshe}\ref{item:thm-cshe-polynomial} follows directly from Proposition \ref{prop:classical-estimate-non-covariant-object} and Theorem \ref{intro:thm-cshe}\ref{item:thm:-difference-in-linear-objects}.

\begin{proof}[Proof of Theorem \ref{intro:thm-cshe}\ref{item:thm-cshe-polynomial}]
By the complex Hermite polynomial expansion (Lemma \ref{lemma:complex-hermite-polynomial-expansion}), we have that
\begin{equs}
\biglcol\,\Big|\philinear[{\Blin},\leqN]\Big|^{2k-2} \philinear[{\Blin}, \leqN]\bigrcol\, =  \sum_{j_1=0}^{k} \sum_{j_2 = 0}^{k-1} \binom{k}{j_1} \binom{k-1}{j_2} \Big(\biglcol\,\philinear[\leqN][r][j_1] \ovl{\philinear[\leqN][r][j_2]}\bigrcol\,\Big)  \Big(\philinear[{\Blin},\leqN] - \philinear[\leqN]\Big)^{k - j_1} \Big(\ovl{\philinear[{\Blin},\leqN] - \philinear[\leqN]}\Big)^{k-1-j_2}.
\end{equs}
The desired result now follows by combining the estimates on the non-covariant objects (Proposition \ref{prop:classical-estimate-non-covariant-object}) with Theorem \ref{intro:thm-cshe}\ref{item:thm:-difference-in-linear-objects} in the case $\alpha = 2\kappa$.
\end{proof}

We next turn towards the proof of Theorem \ref{intro:thm-cshe}\ref{item:thm:-difference-in-linear-objects}, whose proof occupies most of the remainder of this subsection. We begin with our analysis of the covariant linear object $\philinear[{\Blin}]$, defined by equation \eqref{intro:eq-cshe}. We first show the following estimate of the object at negative time and spatial regularity, which is completely independent of ${\Blin}$. In the following, note that (recall Definition \ref{def:smoothed-covariant-heat-kernel})
\begin{equs}\label{eq:philinear-formula}
\philinear[{\Blin},\leqN](z) = \int dw~ \zeta(w) \massp_{{\Blin}, \leq N}(w; z).
\end{equs}

\begin{lemma}[${\Blin}$-independent bound for the covariant linear object]\label{lemma:A-independent-bound-covariant-linear-object}
For $\power, \power_1 \in [1, \infty)$, $N \in \dyadic$, we have that 
\begin{equs}
\Big\|\Big\| \philinear[{\Blin}, \leqN]\Big\|_{L_t^{\power_1} \Cs_x^{-\kappa}([0, 1] \times \T^2)}\Big\|_{L^\power(\Omega)} \lesssim_{p_1} \power^{\frac{1}{2}}.
\end{equs}
\end{lemma}
\begin{proof}
We may assume that $\power_1 \leq \power$, in which case we have by H\"{o}lder
\begin{equs}
\Big\|\Big\| \philinear[{\Blin}, \leqN]\Big\|_{L_t^{\power_1} \Cs_x^{-\kappa}([0, 1] \times \T^2)}\Big\|_{L^{\power}(\Omega)} &\leq \Big\|\Big\| \philinear[{\Blin}, \leqN]\Big\|_{L_t^{\power} \Cs_x^{-\kappa}([0, 1] \times \T^2)}\Big\|_{L^\power(\Omega)} \\
&= \bigg(\int_0^1 dt \Big\| \Big\|\philinear[{\Blin}, \leqN](t)\Big\|_{\Cs_x^{-\kappa}} \Big\|_{L^\power(\Omega)}^\power  \bigg)^{\frac{1}{\power}}.
\end{equs}
Fix $z = (t, x) \in [0, 1] \times \T^2$ and let $u > 0$. By the formula \eqref{eq:philinear-formula} for the covariant linear object, the diamagnetic inequality, and Young's inequality, we have the second moment bound (recall the two uses of $p_u$ from Notation~\ref{notation:heat-kernel-space-only})
\begin{equs}
\Big\| \Big(e^{u \Delta} \philinear[{\Blin}, \leqN]\Big)(z)\Big\|_{L^2(\Omega)} &= \bigg\| \int dx' p_u(x, x') \massp_{{\Blin}, \leq N}(w; z')\bigg\|_{L_w^2([0, t] \times \T^2)} \lesssim \Big\| p_u(w; z)\Big\|_{L_w^2([0, t] \times \T^2)} \lesssim \sqrt{\log u^{-1}}.
\end{equs}
Then by Gaussian hypercontractivity, the heat kernel characterization of Besov spaces (Lemma \ref{lemma:besov-space-heat-kernel-characterization}), and standard reductions (see e.g. \cite[Appendix A]{BC23}), we obtain for any $t \in [0, 1]$
\begin{equs}
\Big\|\Big\|\philinear[{\Blin}, \leqN](t)\Big\|_{\Cs_x^{-\kappa}} \Big\|_{L^\power(\Omega)} \lesssim \power^{\frac{1}{2}}.
\end{equs}
The desired result now follows.
\end{proof}

We proceed to estimate the difference of linear objects at parabolic regularity $1-$, but using one power of ${\Blin}$. For $N \in \dyadic$, $N \geq 2$, define the high-frequency objects
\begin{equs}
\philinear[{\Blin}, N] := \philinear[{\Blin}, \leqN] - \philinear[{\Blin}, \leq \frac{N}{2}],~~ \philinear[N] := \philinear[\leqN] - \philinear[\leq \frac{N}{2}].
\end{equs}
When $N = 1$, define $\philinear[{\Blin}, N] := \philinear[{\Blin},\leqN]$, and similarly for $\philinear[N]$. We state the following proposition, and then work towards its proof. 

\begin{proposition}[Difference of linear objects]\label{prop:difference-objects-bound}
For all $N \in \dyadic$ and $\power \in [1, \infty)$, we have that
\begin{equs}
\bigg\|\Big\|\philinear[{\Blin}, N] - \philinear[N]\Big\|_{C_t^{\kappa} \Cs_x^{1-10\kappa}([0, 1] \times \T^2)}\bigg\|_{L^\power(\Omega)} \lesssim N^{-\kappa} \power^{\frac{1}{2}} \|{\Blin}\|_{\resnorm} \Big(1 + \|{\Blin}\|_{\resnorm}^{10 \kappa}\Big).
\end{equs}
\end{proposition}

Before beginning towards the proof of Proposition \ref{prop:difference-objects-bound}, we first note that Theorem \ref{intro:thm-cshe}\ref{item:thm:-difference-in-linear-objects} follows from Lemma~\ref{lemma:A-independent-bound-covariant-linear-object} and Proposition \ref{prop:difference-objects-bound}.

\begin{proof}[Proof of Theorem \ref{intro:thm-cshe}\ref{item:thm:-difference-in-linear-objects}]
By the triangle inequality and combining the classical estimates on non-covariant objects (Proposition \ref{prop:classical-estimate-non-covariant-object}) and the ${\Blin}$-independent estimate on the covariant linear object (Lemma \ref{lemma:A-independent-bound-covariant-linear-object}), we have that for all $N \in \dyadic$, $\kappa > 0$,
\begin{equs}
\bigg\| \Big\|\philinear[{\Blin}, N]-\philinear[N]\Big\|_{L_t^{\frac{1}{\kappa}} \Cs_x^{-\kappa}([0, 1] \times \T^2)}\bigg\|_{L^\power(\Omega)} \lesssim \power^{\frac{1}{2}}.
\end{equs}
The desired result now follows by interpolating this with the estimate given by Proposition \ref{prop:difference-objects-bound}.
\end{proof}

Next, we begin towards the proof of Proposition \ref{prop:difference-objects-bound}. We split the proof into three main lemmas: a spatial regularity estimate, a time regularity estimate, and an $N$-regularity estimate (that is, a bound which allows us to sum in $N$).

\begin{lemma}[Spatial regularity]\label{lemma:difference-linear-object-spatial-regularity}
For all $N \in \dyadic$, $\power, \power_1 \in [1, \infty)$, we have that
\begin{equs}
\bigg\| \Big\|\philinear[{\Blin}, \leqN] - \philinear[\leqN]\Big\|_{L_t^{\power_1} \Cs_x^{1-\kappa}([0, 1] \times \T^2)}\bigg\|_{L^\power(\Omega)} \lesssim_{p_1} \power^{\frac{1}{2}} \|{\Blin}\|_{\resnorm}
\end{equs}
\end{lemma}
\begin{proof}
By standard reductions (see e.g. \cite[Appendix A]{BC23}), it suffices to show the following second moment bounds given $z = (t, x) \in [0, 1] \times \T^2$, $u > 0$:
\begin{equs}
\Big\| \philinear[{\Blin}, \leqN](z) - \philinear[\leqN](z)\Big\|_{L^2(\Omega)} &\lesssim \|{\Blin}\|_{L_z^\infty([0, t] \times \T^2)} + \|\ptl_j {\Blin}^j \|_{L_t^2 L_x^\infty([0, t] \times \T^2)}, \label{eq:difference-l2-bound} \\
\Big\| \nabla e^{u \Delta} \Big( \philinear[{\Blin}, \leqN] - \philinear[\leqN]\Big)(z)\Big\|_{L^2(\Omega)} &\lesssim \sqrt{\log u^{-1}} \|{\Blin}\|_{\resnorm} \label{eq:gradient-difference-l2-bound}.
\end{equs}
The inequality \eqref{eq:difference-l2-bound} follows directly from Corollary \ref{cor:pA-p-L2-spacetime-norm} and the fact that 
\begin{equs}
\Big\| \philinear[{\Blin}, \leqN](z) - \philinear[\leqN](z)\Big\|_{L^2(\Omega)}  = \Big\| \massp_{{\Blin}, \leq N}(w; z) - \massp_{\leq N}(w; z)\Big\|_{L_w^2(J)}. 
\end{equs}

Next, we show \eqref{eq:gradient-difference-l2-bound}. We have that
\begin{equs}
\Big\| \nabla e^{u \Delta} \Big(\philinear[{\Blin}, \leqN] - \philinear[\leqN]\Big)(z)\Big\|_{L^2(\Omega)} = \bigg\| \int dx' p_u(x, x') \nabla_{x'} \big(\massp_{\Blin}(w; z') - \massp(w; z')\big)\bigg\|_{L_w^2(J)}.
\end{equs}
By Lemma \ref{lemma:covd-pA-expansion}, we have that
\begin{equs}
\nabla_{x'} \massp_{\Blin}(w; z') = -\nabla_{y'} \massp_{\Blin}(w; z') + \icomplex ({\Blin}(w) - {\Blin}(z')) \massp_{\Blin}(w; z') &+ 2\icomplex e^{-(t-s)}\int dv p_{\Blin}(v; z') F_{\Blin}(v) \covd_{{\Blin}(v)} p_{\Blin}(w; v) \\
&+ \icomplex e^{-(t-s)} \int dv p_{\Blin}(v; z') H(v) p_{\Blin}(w; v),
\end{equs}
where $H^k = \ptl_t {\Blin}^k + \ptl_j F_{\Blin}^{kj}$. We bound the contribution coming from the four terms separately. For the first term, by Corollary \ref{cor:gradient-pA-p-l2-spacetime-bound}, we have that
\begin{equs}
\bigg\| \int dx' p_u(x, x') \nabla_{y'} \big(\massp_{\Blin}(w; z') - \massp(w; z')\big)\bigg\|_{L_w^2} \lesssim \big(\|{\Blin}\|_{L_z^\infty} + \|\ptl_j {\Blin}^j\|_{L_t^2 L_x^\infty}\big) \sqrt{\log u^{-1}},
\end{equs}
which is acceptable. For the second term, we may bound
\begin{equs}
\bigg\| \int dx' p_u(x, x') ({\Blin}(w) - {\Blin}(z')) \massp_{\Blin}(w; z')\bigg\|_{L_w^2} \lesssim \|{\Blin}\|_{L_z^\infty} \big\| p_u(w; z)\big\|_{L_z^2} \lesssim \|{\Blin}\|_{L_z^\infty} \sqrt{\log u^{-1}}.
\end{equs}
For the third term, by the dual endpoint time-weighted estimate (Corollary \ref{cor:dual-endpoint-time-weighted-estimate}) with $K(w) = p_u(w; z)$, $t_2 = t + u$, and $\alpha = (-\log u)^{-1}$, we have that (using that $u^{(\log u)^{-1}} \lesssim 1$ in the third inequality)
\begin{equs}
\bigg\| \int dv &\int dx' p_u(x, x') p_{\Blin}(v; z') F_{\Blin}(v) \covd_{{\Blin}(v)} p_{\Blin}(w; v)\bigg\|_{L_w^2}  \\
&\lesssim \bigg\| (t - s + u)^{-\alpha} \int dv \int dx' p_u(x, x') p_{\Blin}(v; z') F_{\Blin}(v) \covd_{{\Blin}(v)} p_{\Blin}(w; v)\bigg\|_{L_w^2} \\
&\lesssim \alpha^{-\frac{1}{2}} \bigg\|(t-t(v) + u)^{-\alpha} p_u(v; z)^{-\frac{1}{2}} p_u(v; z) F_{\Blin}(v) \bigg\|_{L_v^2} \\
&\lesssim \alpha^{-\frac{1}{2}} \Big\|p_u(v; z)^{\frac{1}{2}} F_{\Blin}(v)\Big\|_{L_v^2} \lesssim \alpha^{-\frac{1}{2}} \Big\|F_{\Blin}\Big\|_{L_t^2 L_x^\infty} \Big\|p_u(v; z)^{\frac{1}{2}} \Big\|_{L_t^\infty L_x^2} \lesssim \sqrt{\log u^{-1}} \Big\|F_{\Blin}\Big\|_{L_t^2 L_x^\infty},
\end{equs}
which is acceptable. For the final term, first note by the semigroup property of the heat kernel, we have the pointwise bound for fixed $w, z$,
\begin{equs}
\bigg|\int dv \int dx' p_u(x, x') p_{\Blin}(v; z') H(v) p_{\Blin}(w; v)\bigg| \leq p_u(w; z) \|H\|_{L_t^1 L_x^\infty([s, t] \times \T^2)} \leq p_u(w; z)\|H\|_{L_t^1 L_x^\infty([0, t] \times \T^2)}.
\end{equs}
Since $\|p_u(w; z)\|_{L_w^2} \lesssim \sqrt{\log u^{-1}}$, we obtain an acceptable bound upon taking the $L_w^2$ norm of the LHS above.
The estimate \eqref{eq:gradient-difference-l2-bound} now follows by combining the previous few estimates.
\end{proof}

Next, we discuss the time regularity estimate. Because we will only need an epsilon of time regularity, we find it most straightforward to just work non-covariantly for this estimate. As a consequence, our estimate will ultimately involve two powers of ${\Blin}$ instead of just one power, but this is good enough for our purposes. Before getting to the time regularity estimate (Lemma \ref{lemma:difference-linear-object-time-regularity}), we prove two preliminary lemmas.

\begin{lemma}\label{lemma:f-philinear-spacetime-norm-bound}
Let $f \colon [0, 1] \times \T^2 \ra \C$ be a deterministic function. For all $N \in \dyadic$, $\power, \power_1 \in [1, \infty)$, we have that
\begin{equs}
\bigg\|  \Big\|f\philinear[{\Blin}, \leqN] \Big\|_{L_t^{\power_1}\Cs_x^{-\kappa}([0, 1] \times \T^2)}\bigg\|_{L^\power(\Omega)} \lesssim_{p_1} \power^{\frac{1}{2}} \|f\|_{L_z^{\infty}([0, 1] \times \T^2)}.
\end{equs}
\end{lemma}
\begin{proof}
For $z = (t, x) \in [0, 1] \times \T^2$, $u > 0$, we have that
\begin{equs}
\Big\|e^{u \Delta} \big(f\philinear[{\Blin},\leqN]\big) (z)\Big\|_{L^2(\Omega)} &= \bigg\| \int dx' p_u(x, x') f(z') \massp_{\Blin, \leq N}(w; z')\bigg\|_{L_w^2} \\
&\lesssim \|f\|_{L_z^\infty} \|p_u(w; z)\|_{L_w^2} \lesssim \|f\|_{L_z^\infty} \sqrt{\log u^{-1}}.
\end{equs}
From this, it follows by standard arguments that for any $t \in [0, 1]$,
\begin{equs}
\bigg\| \Big\|f(t) \philinear[{\Blin},\leqN](t)\Big\|_{\Cs_x^{-\kappa}} \bigg\|_{L^\power(\Omega)} \lesssim \power^{\frac{1}{2}} \|f\|_{L_z^\infty}.
\end{equs}
Now, we may assume that $\power_1 \leq \power$, in which case by H\"{o}lder, we have that
\begin{equs}
\bigg\|  \Big\|f\philinear[{\Blin}, \leqN] \Big\|_{L_t^{\power_1}\Cs_x^{-\kappa}([0, 1] \times \T^2)}\bigg\|_{L^{\power}(\Omega)} &\leq \bigg(\int_0^1 dt \bigg\|\Big\|f(t) \philinear[{\Blin},\leqN](t)\Big\|_{\Cs_x^{-\kappa}}\bigg\|_{L^{\power}(\Omega)}^\power\bigg)^{\frac{1}{\power}} \lesssim {\power}^{\frac{1}{2}} \|f\|_{L_z^\infty},
\end{equs}
as desired.
\end{proof}

\begin{lemma}\label{lemma:f-times-philinear-stochastic-estimates}
Let $f \colon [0, 1] \times \T^2 \ra \C$ be a deterministic function. For all $N \in \dyadic$, $\power \in [1, \infty)$, we have that 
\begin{equs}
\bigg\| \Big\| \mDuh\Big(f \philinear[{\Blin},\leqN] \Big)\Big\|_{C_t^{1- \kappa} \Cs_x^{-\kappa}([0, 1] \times \T^2)} \bigg\|_{L^{\power}(\Omega)} \lesssim {\power}^{\frac{1}{2}} \|f\|_{L_z^{\infty}([0, 1] \times \T^2)},
\end{equs}
where $\mDuh$ is the Duhamel operator corresponding to the massive Laplacian $\Delta - 1$.
\end{lemma}
\begin{proof}
By standard Schauder estimates, we have that for $\power_1 = \frac{2}{\kappa}$,
\begin{equs}
\Big\|\mDuh(f \philinear[{\Blin}, \leqN])\Big\|_{C_t^{1-\kappa} \Cs_x^{-\kappa}} \lesssim \Big\|f \philinear[{\Blin},\leqN]\Big\|_{L_t^{\power_1} \Cs_x^{-\kappa}}.
\end{equs}
We now finish by Lemma \ref{lemma:f-philinear-spacetime-norm-bound}.
\end{proof}

\begin{lemma}[Time regularity]\label{lemma:difference-linear-object-time-regularity}
For all $N \in \dyadic$, $\power \in [1, \infty)$, we have that
\begin{equs}
\bigg\| \Big\| \philinear[{\Blin}, \leqN] - \philinear[\leqN]\Big\|_{C_t^{1-\kappa} \Cs_x^{-1-\kappa}([0, 1] \times \T^2)}\bigg\|_{L^{\power}(\Omega)} \lesssim {\power}^{\frac{1}{2}} \big(\|{\Blin}\|_{L_z^\infty([0, 1] \times \T^2)} + \|\ptl_j \Blin^j\|_{L_z^\infty([0, 1] \times \T^2)} + \|{\Blin}\|_{L_z^\infty([0, 1] \times \T^2)}^2\big).
\end{equs}
\end{lemma}
\begin{proof}
Observe that
\begin{equs}
(\ptl_t - \Delta + 1) (\philinear[{\Blin},\leqN] - \philinear[\leqN]) = (\covd_{\Blin}^j \covd_{{\Blin}, j} -\Delta)\philinear[{\Blin},\leqN] = 2\icomplex \ptl_j ({\Blin}^j \philinear[{\Blin},\leqN]) - \icomplex (\ptl_j \Blin^j) \philinear[B, \leqN] - |{\Blin}|^2 \philinear[{\Blin},\leqN] .
\end{equs}
Thus
\begin{equs}
\philinear[{\Blin}, \leqN] - \philinear[\leqN] =2\icomplex \ptl_j \mDuh ({\Blin}^j \philinear[{\Blin}, \leqN]) - \icomplex \mDuh ((\ptl_j \Blin^j) \philinear[B, \leqN]) - \mDuh (|{\Blin}|^2 \philinear[{\Blin}, \leqN]).
\end{equs}
The desired result now follows by applying Lemma \ref{lemma:f-times-philinear-stochastic-estimates} to each of the three terms on the right hand side above.
\end{proof}

Finally, we discuss the $N$-regularity estimate. Recall from Definition \ref{def:mollifiers} that our mollifier $\moll_{\leq N}$ is of the form $\moll_{\leq N}(y) = N^2 \moll(N y)$. Thinking of $N$ as a continuous parameter, note that
\begin{equs}
\nabla \moll_{\leq N}(y) &= N^3 \nabla \moll(Ny), \\
\ptl_N \moll_{\leq N}(y) &= 2N \moll(Ny) + N^2 \nabla \moll(Ny) \cdot y = 2 N^{-1} \moll_{\leq N}(y) + N^{-1} \nabla \moll_{\leq N}(y) \cdot y.
\end{equs}
Thus,
\begin{equs}
\moll_N(y) = \moll_{\leq N}(y) - \moll_{\leq N/2}(y) = \int_{\frac{N}{2}}^N dM \Big( 2M^{-1} \moll_{\leq M}(y) + M^{-1} \nabla \moll_{\leq M}(y) \cdot y\Big) .
\end{equs}
Then by integration by parts, given a kernel $K(y)$, we have the following high-frequency convolution identity:
\begin{equs}\label{eq:high-frequency-kernel-convolution-identity}
(\moll_N * K)(y) = \int_{\frac{N}{2}}^N dM M^{-1} \int dy' \moll_{\leq M}(y - y') (y - y') \cdot \nabla K(y').
\end{equs}

\begin{lemma}[$N$-regularity]\label{lemma:difference-linear-object-N-regularity}
For all $N \in \dyadic$, $\power, \power_1 \in [1, \infty)$, we have that
\begin{equs}
\bigg\| \Big\| \philinear[{\Blin}, N] - \philinear[N] \Big\|_{L_t^{\power_1} \Cs_x^{-\kappa}([0, 1] \times \T^2)} \bigg\|_{L^{\power}(\Omega)} \lesssim_{\power_1} {\power}^{\frac{1}{2}} N^{-1} \big(\|{\Blin}\|_{L_z^\infty([0, 1] \times \T^2)} + \|\ptl_j {\Blin}^j\|_{L_t^2 L_x^\infty([0, 1] \times \T^2)}\big).
\end{equs}
\end{lemma}
\begin{proof}
From standard reductions (see e.g. \cite[Appendix A]{BC23}), it suffices to show the following second moment estimate for $u > 0$, $z \in [0, 1] \times \T^2$:
\begin{equs}
\Big\| e^{u\Delta}\Big(\philinear[{\Blin}, N] - \philinear[N]\Big)(z)\Big\|_{L^2(\Omega)} \lesssim N^{-1} \sqrt{\log u^{-1}} \big(\|{\Blin}\|_{L_z^\infty([0, 1] \times \T^2)} + \|\ptl_j {\Blin}^j\|_{L_t^2 L_x^\infty([0, 1] \times \T^2)}\big).
\end{equs}
Let $\tilde{\moll}_{\leq N}(y) := \moll_{\leq N}(y) Ny$. By \eqref{eq:high-frequency-kernel-convolution-identity}, we may write
\begin{equs}
e^{u\Delta}\Big(\philinear[{\Blin}, N] - \philinear[N]\Big)(z) = \int_{\frac{N}{2}}^N dM M^{-2} \int dw \zeta(w) \int dy' \int dx' p_u(x, x') \tilde{\moll}_{\leq M}(y-y') \cdot \nabla_{y'} \big(\massp_{\Blin}(w'; z') - \massp(w'; z')\big)
\end{equs}
Note that $\|\tilde{\moll}_{\leq N}\|_{L^1} \lesssim 1$. Thus by Young's inequality, we have that (and applying Corollary \ref{cor:gradient-pA-p-l2-spacetime-bound} in the second inequality)
\begin{equs}
\Big\| e^{u\Delta}\Big(\philinear[{\Blin}, N] - \philinear[N]\Big)(z)\Big\|_{L^2(\Omega)} &\lesssim N^{-2} \int_{\frac{N}{2}}^N dM \bigg\|\int dx' p_u(x, x') \nabla_y \big( \massp_{\Blin}(w; z') - \massp(w; z')\big)\bigg\|_{L_w^2} \\
&\lesssim (\|{\Blin}\|_{L_z^\infty} + \|\ptl_j {\Blin}^j \|_{L_t^2 L_x^\infty})
N^{-2} \int_{\frac{N}{2}}^N dM \sqrt{\log u^{-1}} \\
&\lesssim N^{-1} \sqrt{\log u^{-1}}  (\|{\Blin}\|_{L_z^\infty} + \|\ptl_j {\Blin}^j \|_{L_t^2 L_x^\infty}),
\end{equs}
as desired.
\end{proof}

Finally, we note that Proposition \ref{prop:difference-objects-bound} follows from the three estimates by interpolation.

\begin{proof}[Proof of Proposition \ref{prop:difference-objects-bound}]
Let $\lambda_1 = 1-3\kappa, \lambda_2 = 2\kappa, \lambda_3 = \kappa$, so that $\lambda_1 + \lambda_2 = \lambda_3 = 1$. By interpolation, we have that for a general space-time function $f$,
\begin{equs}
\|f\|_{C_t^\kappa \Cs_x^{1-10\kappa}} \leq \|f\|_{L_t^{1/\kappa} \Cs_x^{1-\kappa}}^{\lambda_1} \|f\|_{C_t^{1-\kappa} \Cs_x^{-1-\kappa}}^{\lambda_2} \|f\|_{L_t^{1/\kappa^2} \Cs_x^{-\kappa^2}}^{\lambda_3}.
\end{equs}
By combining this and Lemmas \ref{lemma:difference-linear-object-spatial-regularity}, \ref{lemma:difference-linear-object-time-regularity}, and \ref{lemma:difference-linear-object-N-regularity}, we obtain that
\begin{equs}
\bigg\| \Big\|\philinear[{\Blin}, N] - \philinear[N]\Big\|_{C_t^{\kappa} \Cs_x^{1-10\kappa}([0, 1] \times \T^2)}\bigg\|_{L^{\power}(\Omega)} \lesssim N^{-\kappa} {\power}^{\frac{1}{2}} \|{\Blin}\|_{\resnorm} \Big(1 + \|{\Blin}\|_{\resnorm}^{10\kappa}\Big),
\end{equs}
as desired.
\end{proof}

To close off this subsection, we state and prove the following variant of Theorem \ref{intro:thm-cshe}\ref{item:thm-cshe-polynomial} for the linear object, which will be needed in Sections \ref{section:Abelian-Higgs} and \ref{section:decay}.

\begin{lemma}\label{lemma:covariant-linear-object-mollified-noise-L-infty-bound}
For all $N \in \dyadic$, $\power \in [1, \infty)$, we have that
\begin{equs}
\bigg\|\Big\| \philinear[\Blin,\leqN]\Big\|_{L_t^\infty L_x^\infty([0, 1] \times \T^2)} \bigg\|_{L^{\power}(\Omega)} \lesssim {\power}^{\frac{1}{2}} N^{\frac{\kappa}{2}} (1 + \|\Blin\|_{\resnorm}^\frac{\kappa}{2}).
\end{equs}
\end{lemma}
\begin{proof}
By classical arguments, we may show that
\begin{equs}
\bigg\|\Big\| \philinear[\leqN]\Big\|_{L_t^\infty L_x^\infty([0, 1] \times \T^2)} \bigg\|_{L^{\power}(\Omega)} \lesssim {\power}^{\frac{1}{2}}N^{\frac{\kappa}{2}} .
\end{equs}
We now finish by combining this with Theorem \ref{intro:thm-cshe}\ref{item:thm:-difference-in-linear-objects}.
\end{proof}

\subsection{Non-resonant part of derivative nonlinearity}\label{section:derivative-nonlinearity-non-resonant}

In this subsection, we analyze the non-resonant part of the derivative nonlinearity, which will form part of the proof of Theorem \ref{intro:thm-cshe}\ref{item:cshe-derivative-nonlinearity}. To be precise, by the product formula for stochastic integrals (see e.g. \cite[Proposition 1.1.3]{Nua06}), we have that
\begin{equs}\label{eq:resonant-non-resonant-decomp}
\Duh\Big[ \ovl{\philinear[\Blin, \leqN]} \covd_{\Blin(z)} \philinear[\Blin, \leq N]\Big] = \Duh\Big[ \biglcol\,\ovl{\philinear[\Blin, \leqN]} \covd_{B(z)} \philinear[\Blin, \leq N]\bigrcol\Big] + \Duh\Big[ \E\Big[\ovl{\philinear[\Blin, \leqN]} \covd_{\Blin(z)} \philinear[\Blin, \leq N]\Big] \Big],
\end{equs}
where by definition the first term on the right-hand side (which we refer to as the non-resonant part) is a multiple stochastic integral given by
\begin{equs}\label{eq:non-resonant-part-multiple-stochastic-integral}
\hspace{-5mm}\Duh\Big[ \biglcol\,\ovl{\philinear[\Blin, \leqN]} \covd_{B(z)} \philinear[\Blin, \leq N]\bigrcol\Big](z_0) = \int dw_1 dw_2 \ovl{\zeta(w_1)} \zeta(w_2) \int dz p(z; z_0) \ovl{\massp_{\Blin, \leq N}(w_1; z)} \covd_{\Blin(z)} \massp_{\Blin, \leq N}(w_2; z).
\end{equs}
Here, and in the following, integrals such as $\int dw_1 dw_2 \ovl{\zeta(w_1)} \zeta(w_2) (\cdots)$ denote multiple stochastic integrals with respect the white noise $\zeta$ (see \cite[Chapter 1]{Nua06}). The second term on the right-hand side of \eqref{eq:resonant-non-resonant-decomp} is what we refer to as the resonant part, and will be studied in Subsection \ref{subsection:resonant-derivative-nonlinearity}.

We now state the main result of this subsection, which gives an estimate of the non-resonant part at regularity $1-$, and which essentially does not depend on ${\Blin}$. (As we will see in the proof, the ${\Blin}$-dependence only arises when we want to gain a tiny power of $N$, so that we can sum in $N$.) 

\begin{proposition}[Estimate of non-resonant part of derivative nonlinearity]\label{prop:derivative-nonlinearity-non-resonant}
For all $\power \in [1, \infty)$, we have that
\begin{equs}
\Bigg\|\sup_{N \in \dyadic} \bigg\| \Im \Duh\Big[\biglcol\,\ovl{\philinear[{\Blin}, \leqN]} \covd_{\Blin} \philinear[{\Blin}, \leqN]\bigrcol\,\Big]\bigg\|_{C_t^{\kappa} \Cs_x^{1-10\kappa} \cap C_t^{\frac{1}{2}-5\kappa} \Cs_x^{2\kappa}([0, 1] \times \T^2)} \Bigg\|_{L^{\power}(\Omega)}\lesssim \power \big(1 + \|{\Blin}\|_{\resnorm}^{\kappa}\big).
\end{equs}
\end{proposition}

We begin towards the proof of Proposition \ref{prop:derivative-nonlinearity-non-resonant}. For $z_0 \in [0, 1] \times \T^2$, $N \in \dyadic$, define the kernels
\begin{equs}
\qkernel_{z_0} (w_1, w_2) &:= \int dz p(z; z_0) \ovl{\massp_{\Blin}(w_1; z)} \covd_{{\Blin}(z)} \massp_{\Blin}(w_2; z), \label{eq:def-qkernel} \\
\qkernel_{z_0, \leq N}(w_1, w_2) &:= \int dy_1' dy_2' \moll_{\leq N}(y_1 - y_1') \moll_{\leq N}(y_2 - y_2') \qkernel_{z_0}(w_1', w_2').
\end{equs}
It follows from \eqref{eq:non-resonant-part-multiple-stochastic-integral} that
\begin{equs}
\Duh\Big(\mrm{Im}\Big(\biglcol\,&\ovl{\philinear[{\Blin}, \leqN] }\covd_{\Blin}\philinear[{\Blin}, \leqN]\bigrcol\,\Big)\Big)(z_0) = \Im\bigg( \int dw_1 dw_2 \ovl{\zeta(w_1)} \zeta(w_2) \qkernel_{z_0, \leq N} (w_1, w_2)\bigg).
\end{equs}
Next, we note some preliminary reductions.

\begin{lemma}[Integration by parts identity]\label{lemma:qkernel-ibp-identity}
We have that
\begin{equs}
Q_{z_0}(w_1, w_2) = -\ovl{Q_{z_0}(w_2, w_1)} + \nabla_{x_0} \int p(z; z_0) \ovl{\massp_{\Blin}(w_1; z)} \massp_{\Blin}(w_2; z).
\end{equs}
\end{lemma}
\begin{proof}
This follows directly from integration by parts and the identities 
\begin{align*}
\covd_{{\Blin}(z)} (p(z; z_0) \massp_{\Blin}(w_1; z)) &= (\nabla_x p(z; z_0)) \massp_{\Blin}(w_1; z) + p(z; z_0) \covd_{{\Blin}(z)} \massp_{\Blin}(w_1; z),\\
\nabla_x p(z; z_0) &= - \nabla_{x_0} p(z; z_0). \qedhere
\end{align*}
\end{proof}

\begin{lemma}\label{lemma:integration-by-parts-identity}
We have that
\begin{equs}
\Duh\Big(\mrm{Im}\Big(\biglcol\,\ovl{\philinear[{\Blin}, \leqN] }&\covd_{\Blin}\philinear[{\Blin}, \leqN]\bigrcol\,\Big)\Big)(z_0) = 2 \Im\bigg( \int dw_1 dw_2 \ind(s_1 < s_2) \ovl{\zeta(w_1)} \zeta(w_2) \qkernel_{z_0, \leq N}(w_1, w_2)\bigg)  \\
& +\nabla_{x_0} \Im\bigg(\int dw_1 dw_2 \ind(s_1 > s_2) \ovl{\zeta(w_1)} \zeta(w_2) \int dz p(z; z_0) \ovl{\massp_{{\Blin}, \leq N} (w_1; z)} \massp_{{\Blin}, \leq N}(w_2; z)\bigg) .
\end{equs}
\end{lemma}
\begin{proof}
Applying the integration by parts identity for $\qkernel_{z_0}$ (Lemma \ref{lemma:qkernel-ibp-identity}), we have that
\begin{equs}
\int dw_1 dw_2 \ind(s_1 > s_2) \ovl{\zeta(w_1)} \zeta(w_2) \qkernel_{z_0, \leq N}(w_1, w_2) = &- \int dw_1 dw_2 \ind(s_1 > s_2) \ovl{\zeta(w_1)} \zeta(w_2) \ovl{\qkernel_{z_0, \leq N}(w_2, w_1)} \\
&+ \int dw_1 dw_2 \ind(s_1 > s_2) \ovl{\zeta(w_1)} \zeta(w_2) \nabla_{x_0} F_{z_0, \leq N}(w_1, w_2),
\end{equs}
where 
\begin{equs}
F_{z_0, \leq N}(w_1, w_2) := \int dz p(z; z_0) \ovl{\massp_{{\Blin}, \leq N}(w_1; z)} \massp_{{\Blin}, \leq N}(w_2; z).
\end{equs}
Taking imaginary parts and swapping $w_1 \leftrightarrow w_2$, we have that
\begin{equs}
&\, - \Im\bigg(\int dw_1 dw_2 \ind(s_1 > s_2) \ovl{\zeta(w_1)} \zeta(w_2) \ovl{Q_{z_0, \leq N}(w_2, w_1)}\bigg) \\
=&\, -\Im\bigg(\int dw_1 dw_2 \ind(s_2 > s_1) \ovl{\zeta(w_2)} \zeta(w_1) \ovl{\qkernel_{z_0, \leq N}(w_1, w_2)}\bigg) \\
=&\, \Im\bigg(\int dw_1 dw_2 \ind(s_1 < s_2) \ovl{\zeta(w_1)} \zeta(w_2) \qkernel_{z_0, \leq N}(w_1, w_2)\bigg).
\end{equs}
To finish, observe that
\begin{equation*}
\int dw_1 dw_2 \ind(s_1 > s_2) \ovl{\zeta(w_1)} \zeta(w_2) \nabla_{x_0} F_{z_0, \leq N}(w_1, w_2) = \nabla_{x_0} \int dw_1 dw_2 \ind(s_1 > s_2) \ovl{\zeta(w_1)} \zeta(w_2) F_{z_0, \leq N}(w_1, w_2) . \qedhere
\end{equation*}
\end{proof}

By the integration by parts identity from Lemma \ref{lemma:integration-by-parts-identity}, we have that
\begin{equs}\label{eq:derivative-nonlinearity-non-resonant-decomposition}
\Im \Duh\Big[\biglcol\,\ovl{\philinear[{\Blin}, \leqN]} \covd_{\Blin} \philinear[{\Blin}, \leqN]\bigrcol\,\Big] = D_{1, \leq N} + D_{2, \leq N},
\end{equs}
where
\begin{align}
D_{1, \leq N}(z_0) &:= 2 \Im\bigg( \int dw_1 dw_2 \ind(s_1 < s_2) \ovl{\zeta(w_1)} \zeta(w_2) \qkernel_{z_0, \leq N}(w_1, w_2)\bigg) \label{eq:D-1-N-def}\\
D_{2, \leq N}(z_0) &:= \nabla_{x_0} \Im\bigg(\int dw_1 dw_2 \ind(s_1 > s_2) \ovl{\zeta(w_1)} \zeta(w_2) \int dz p(z; z_0) \ovl{\massp_{{\Blin}, \leq N}(w_1; z)} \massp_{{\Blin}, \leq N}(w_2; z)\bigg) .
\end{align}
Observe that $D_{2, \leq N}$ is almost precisely $\nabla \Duh\big(\biglcol\, \big|\,\philinear[{\Blin}, \leq N]\big|^2\bigrcol\,\big)$, except there is the condition $\ind(s_1 > s_2)$ in the multiple stochastic integral defining $D_{2, \leq N}$. However, for obtaining probabilistic bounds, this difference does not matter, as one would proceed to estimate in exactly the same way. We thus have the following result, whose proof is omitted.

\begin{lemma}\label{lemma:derivative-nonlinearity-non-resonant-pure-gradient-estimate}
For all $\power \in [1, \infty)$, we have that
\begin{equs}
\bigg\| \sup_{N \in \dyadic} \Big\| D_{2, \leq N}\Big\|_{C_t^{\kappa} \Cs_x^{1-10\kappa} \cap C_t^{\frac{1}{2}-5\kappa} \Cs_x^{2\kappa}([0, 1] \times \T^2)} \bigg\|_{L^{\power}(\Omega)} \lesssim \power (1 + \|{\Blin}\|_{\resnorm}^\kappa).
\end{equs}
\end{lemma}

Define
\begin{equs}
D_{1, N} := D_{1, \leq N} - D_{1, \leq \frac{N}{2}}.
\end{equs}

Next, we state the bound for $D_{1, N}$ whose proof we will work towards in the remainder of this subsection.

\begin{proposition}\label{prop:D-1-N-bound}
For all $N \in \dyadic$, $\power \in [1, \infty)$, we have that
\begin{equs}
\bigg\| \Big\|D_{1, N}\Big\|_{C_t^{\kappa} \Cs_x^{1-10\kappa} \cap C_t^{\frac{1}{2}-5\kappa} \Cs_x^{2\kappa}([0, 1] \times \T^2)} \bigg\|_{L^{\power}(\Omega)} \lesssim \power N^{-\kappa} \|{\Blin}\|_{L_z^\infty([0, 1] \times \T^2)}^{\kappa} .
\end{equs}
\end{proposition}

Before proceeding with the proof of Proposition \ref{prop:D-1-N-bound}, we first note that Proposition \ref{prop:derivative-nonlinearity-non-resonant} follows directly from Lemma \ref{lemma:derivative-nonlinearity-non-resonant-pure-gradient-estimate} and Proposition \ref{prop:D-1-N-bound}.
\begin{proof}[Proof of Proposition \ref{prop:derivative-nonlinearity-non-resonant}]
This follows directly by combining the decomposition \eqref{eq:derivative-nonlinearity-non-resonant-decomposition} with Lemma \ref{lemma:derivative-nonlinearity-non-resonant-pure-gradient-estimate} and Proposition \ref{prop:D-1-N-bound}.
\end{proof}

Similar to our final bound on the difference of linear objects (Proposition \ref{prop:difference-objects-bound}), in order to prove Proposition~\ref{prop:D-1-N-bound}, we will separately prove spatial regularity, time regularity, and $N$-regularity estimates for $D_{1, \leq N}$, and then combine them all in the end by interpolation. The key result underlying all of these results is the following technical estimate, which gives a completely ${\Blin}$-independent bound. As we will see, this is possible because of the covariant energy estimates of Section \ref{section:monotonicity}.

\begin{lemma}\label{lemma:derivative-nonlinearity-non-resonant-part-L2-estimates}
Let $t_0 > 0$, let $x_0 \in \T^2$, and let $z_0 = (t_0, x_0)$. Then, we have that
\begin{equs}
\Big\|\qkernel_{z_0}(w_1, w_2)\Big\|_{L_{w_1}^2 L_{w_2}^2(0 < s_1 < s_2 < t_0)} \lesssim t_0^{\frac{1}{2}}.
\end{equs}
Additionally, for any $j \geq 1$, $u \in (0, 1]$, $\varep \in (0, 1)$, we have that
\begin{equs}
\bigg\| \nabla_{x_0}^j \int dx_1 p_u(x_0, x_1) Q_{(t_0, x_1)}(w_1, w_2) \bigg\|_{L_{w_1}^2 L_{w_2}^2(0 < s_1 < s_2 < t_0)} &\lesssim_{\varep, j} t_0^{\varep} u^{\frac{1-j}{2}-\varep} .
\end{equs}
\end{lemma}
\begin{proof}
Fix $w_1$ and define $J_{w_1} := [s_1, t_0] \times \T^2$. Recalling the definition of $\qkernel_{z_0}$ (equation \eqref{eq:def-qkernel}), by the dual two-sided weighted estimate (Corollary \ref{cor:dual-two-sided-weighted-estimate} with say $\lambda = \frac{1}{4}$), we have that
\begin{equs}
\Big\|Q_{z_0}(w_1, w_2)\Big\|_{L_{w_2}^2(J_{w_1})}  &\lesssim \Big\|(t - s_1)^{\frac{1}{4}} (t_0 - t)^{-\frac{1}{4}} p(z; z_0)^{-\frac{1}{2}} p(z; z_0) \ovl{p_{\Blin}(w_1; z)} \Big\|_{L_z^2(J_{w_1})} .
\end{equs}
From this, it follows that
\begin{equs}
\Big\|Q_{z_0}(w_1, w_2)\Big\|_{L_{w_1}^2 L_{w_2}^2(s_1 < s_2)} &\lesssim \Big\|(t - s_1)^{\frac{1}{4}} (t_0 - t)^{-\frac{1}{4}} p(z; z_0)^{-\frac{1}{2}} p(z; z_0) p(w_1; z)\Big\|_{L_z^2 L_{w_1}^2(0 < s_1 < t < t_0)} \\
&\lesssim \Big\| t^{\frac{1}{4}} (t_0 - t)^{-\frac{1}{4}} p(z; z_0)^{\frac{1}{2}} \Big\|_{L_z^2} \lesssim t_0^{\frac{1}{2}}.
\end{equs}
The first desired result now follows.

For the second desired result, first consider the case $j = 1$. We again apply the dual two-sided weighted estimate but now with $2\lambda = \varep$, to obtain
\begin{equs}
\bigg\|\nabla_{x_0} &\int dx_1 p_u(x_0, x_1) Q_{(t_0, x_1)} (w_1, w_2)\bigg\|_{L_{w_1}^2 L_{w_2}^2(s_1 < s_2)} \\
&\lesssim \bigg\|(t-s_1)^{\lambda} (t_0 - t)^{-\lambda}  p_u(z; z_0)^{-\frac{1}{2}} \nabla_{x_0}p_u (z; z_0) p(w_1; z)\bigg\|_{L_z^2 L_{w_1}^2(s_1 < t)} \\
&\lesssim \Big\|t^\lambda (t_0 - t)^{-\lambda} p_u(z; z_0)^{-\frac{1}{2}} \nabla_{x_0} p_u(z; z_0)\Big\|_{L_z^2} \lesssim \bigg(\int_0^t t^{2\lambda} (t_0 - t)^{-2\lambda} (t_0 - t + u)^{-1}\bigg)^{\frac{1}{2}} \lesssim u^{-2\lambda} t_0^{2\lambda}.
\end{equs}
Here, we used that, at least morally, $|\nabla_{x_0} p_u(z; z_0)|$ is bounded by $(t_0 - t + u)^{-\frac{1}{2}} p_u(z; z_0)$.
The second desired result when $j=1$ now follows.

The result for general $j \geq 1$ follows similarly, except we give up an additional $u^{-\frac{1}{2}}$ for each additional derivative.
\end{proof}

\begin{remark}\label{remark:modified-Q-estimates}
As is clear from the proof, the estimates of Lemma \ref{lemma:derivative-nonlinearity-non-resonant-part-L2-estimates} hold if instead of $Q_{z_0}(w_1, w_2)$, we had the slightly modified version
\begin{equs}
\tilde{Q}_{z_0}(w_1, w_2) := \int dz \ind(t < t_0') p(z; z_0) \ovl{\massp_{\Blin}(w_1; z)} \covd_{\Blin(z)} \massp_{\Blin} (w_2; z),
\end{equs} 
where $t_0'$ is some additional parameter. The additional $t < t_0'$ constraint would not affect any of the steps in the proof, as we would just bound $\ind(t < t_0') \leq 1$.
\end{remark}

The spatial regularity estimate now directly follows from Lemma \ref{lemma:derivative-nonlinearity-non-resonant-part-L2-estimates}.

\begin{lemma}[Spatial regularity estimate]\label{lemma:derivative-nonlinearity-nonresonant-spatial-regularity-estimate}
For any $N \in \dyadic$, $t \in [0, 1]$, $\power \in [1, \infty)$, we have that
\begin{equs}
\bigg\| \Big\| D_{1, \leq N}(t)\Big\|_{L_t^p\Cs_x^{1-\kappa}([0, 1] \times \T^2)} \bigg\|_{L^{\power}(\Omega)} \lesssim \power.
\end{equs}
\end{lemma}
\begin{proof}
From the definition \eqref{eq:D-1-N-def} of $D_{1, \leq N}$, note that
\begin{equs}
D_{1, \leq N}(z_0) &= 2\Im\int dw_1 dw_2 \ind(s_1 < s_2) \ovl{\zeta(w_1)} \zeta(w_2) \qkernel_{z_0, \leq N}(w_1, w_2), \\
\nabla e^{u \Delta} D_{1, \leq N}(z_0) &= 2\Im\int dw_1 dw_2 \ind(s_1 < s_2) \ovl{\zeta(w_1)} \zeta(w_2) \nabla_{x_0} \int dx_0' p_u(x_0, x_0') \qkernel_{z_0', \leq N}(w_1, w_2),
\end{equs}
The result now follows by standard reductions (see e.g. \cite[Appendix A]{BC23}), the heat kernel characterization of Besov spaces (Lemma \ref{lemma:besov-space-heat-kernel-characterization}), the first estimate of Lemma \ref{lemma:derivative-nonlinearity-non-resonant-part-L2-estimates}, and the second estimate of the same lemma with $j = 1$. 
\end{proof}

Next, we proceed to show the following time regularity estimate for $D_{1, \leq N}$.

\begin{lemma}[Time regularity estimate]\label{lemma:derivative-nonlinearity-nonresonant-time-regularity-estimate}
Let $0 < t_0' < t_0 \leq 1$,  $x_0 \in \T^2$, $z_0 = (t_0, x_0)$, and $z_0' = (t_0', x_0)$. Then, we have that for all $\power \in [1, \infty)$ and $N \in \dyadic$ that
\begin{equs}
\Big\|D_{1, \leq N}(z_0) - D_{1, \leq N}(z_0')\Big\|_{L^{\power}(\Omega)} \lesssim \power |t_0 - t_0'|^{\frac{1}{2}-\kappa}.
\end{equs}
\end{lemma}
\begin{proof}
By Gaussian hypercontractivity, it suffices to take $\power = 2$. Recalling the definition \eqref{eq:D-1-N-def} of $D_{1, \leq N}$, we have that
\begin{equs}
D_{1, \leq N}(z_0) - D_{1, \leq N}(z_0') = 2 \Im \int dw_1 dw_2 \ind(s_1 < s_2) \ovl{\zeta(w_1)} \zeta(w_2) (Q_{z_0, \leq N}(w_1, w_2) - Q_{z_0', \leq N}(w_1, w_2)).
\end{equs}
Thus, it suffices to show
\begin{equs}\label{eq:time-difference-second-moment-intermediate-estimate}
\bigg\|Q_{z_0, \leq N}(w_1, w_2) - Q_{z_0', \leq N}(w_1, w_2)\bigg\|_{L_{w_1}^2 L_{w_2}^2(s_1 < s_2)} \lesssim |t_0 - t_0'|^{\frac{1}{2}-\kappa}.
\end{equs}
We split
\begin{equs}
Q_{z_0, \leq N}(w_1, w_2) - Q_{z_0', \leq N}(w_1, w_2) &= Q^1_{t_0, t_0', x_0
, \leq N}(w_1, w_2) + Q^2_{t_0, t_0', x_0, \leq N}(w_1, w_2) \\
&:= \int dz \ind(t < t_0') (p(z; z_0) - p(z; z_0')) \ovl{\massp_{{\Blin}, \leq N}(w_1; z)} \covd_{{\Blin}(z)} \massp_{{\Blin}, \leq N}(w_2; z) \\
&~~~+\int dz \ind(t_0' < t < t_0) p(z; z_0) \ovl{\massp_{{\Blin}, \leq N}(w_1; z)} \covd_{{\Blin}(z)} \massp_{{\Blin}, \leq N}(w_2; z).
\end{equs}
We first bound $Q^2_{t_0, t_0', x_0 \leq N}$. Fix $w_1 = (s_1, y_1)$ with $s_1 \in (0, t_0)$. Let $J_{w_1} := [s_1, t_0] \times \T^2$. By the dual two-sided time-weighted estimate Corollary \ref{cor:dual-two-sided-time-weighted-estimate}, we have that for $\lambda \in (0, \frac{1}{2})$,
\begin{equs}
\Big\| Q^2_{t_0, t_0', x_0 \leq N}(w_1, w_2)\Big\|_{L_{w_2}^2(J_{w_1})} &\lesssim \Big\|\ind(t_0' < t < t_0) (t - s_1)^\lambda (t_0 - t)^{-\lambda} p(z; z_0)^{-\frac{1}{2}} p(z; z_0) \ovl{p_{{\Blin}, \leq N}(w_1; z)} \Big\|_{L_z^2(J_{w_1})} 
\end{equs}
Taking $\lambda = \kappa$, we obtain 
\begin{equs}\label{eq:Q-2-intermediate-bound}
\Big\| Q^2_{t_0, t_0', x_0 \leq N}(w_1, w_2)\Big\|_{L_{w_1}^2 L_{w_2}^2(s_1 < s_2)} \lesssim \Big\|\ind(t_0' < t < t_0) t^\kappa (t_0 - t)^{-\kappa} p(z; z_0)^{\frac{1}{2}} \Big\|_{L_z^2} \lesssim  (t_0 - t_0')^{\frac{1}{2}-\kappa} t_0^\kappa,
\end{equs}
which is acceptable. Next, we bound $Q_{t_0, t_0', x_0, \leq N}^1$. Using that $\ptl_{t_0} p(z; z_0) = \Delta_{x_0} p(z; z_0)$, we may express
\begin{equs}\label{eq:Q-t-time-decomposition}
Q_{t_0, t_0', x_0, \leq N}^1(w_1, w_2) = \int_{t_0'}^{t_0} d \tau  \int dz \ind(t < t_0')\Delta_{x_0} p (z; (\tau, x_0)) \ovl{\massp_{{\Blin}, \leq N}(w_1; z)} \covd_{{\Blin}(z)} \massp_{{\Blin}, \leq N}(w_2; z).
\end{equs}
Next, observe that for a spatial function $f \colon\T^2 \ra \C$, we have that $f - e^\Delta f = -\int_0^1 du \ptl_u e^{u\Delta} f = -\int_0^1 du \Delta e^{u \Delta}f$. We may thus express (in the following the heat kernel acts on the $x_0$ variable and we omit the $w_1, w_2$ variables)
\begin{equs}\label{eq:Q-1-heat-kernel-decomposition}
Q_{t_0, t_0', x_0, \leq N}^1 = e^{\Delta} Q_{t_0, t_0', x_0, \leq N}^1 - \int_0^1 du \Delta e^{u \Delta} Q_{t_0, t_0', x_0, \leq N}^1.
\end{equs}
Fix $u \in (0, 1)$. We have by \eqref{eq:Q-t-time-decomposition} that
\begin{equs}
\Delta e^{u \Delta} Q_{t_0, t_0', x_0, \leq N}^1(w_1, w_2) &=  \int_{t_0'}^{t_0} d \tau  \int dz \ind(t < t_0')\Delta_{x_0}^2 p_u (z; (\tau, x_0)) \ovl{\massp_{{\Blin}, \leq N}(w_1; z)} \covd_{{\Blin}(z)} \massp_{{\Blin}, \leq N}(w_2; z) .
\end{equs}
By Minkowski's integral inequality and the second estimate of Lemma \ref{lemma:derivative-nonlinearity-non-resonant-part-L2-estimates} with $j = 4$ (and Remark \ref{remark:modified-Q-estimates}), we have that
\begin{equs}
\bigg\| &\int_{t_0'}^{t_0} d \tau  \int dz \ind(t < t_0')\Delta_{x_0}^2 p_u (z; (\tau, x_0)) \ovl{\massp_{{\Blin}, \leq N}(w_1; z)} \covd_{{\Blin}(z)} \massp_{{\Blin}, \leq N}(w_2; z) \bigg\|_{L_{w_1}^2 L_{w_2}^2(s_1 < s_2)} \\
&\leq  \int_{t_0'}^{t_)} d\tau \bigg\| \int dz \ind(t < t_0')\Delta_{x_0}^2 p_u (z; (\tau, x_0)) \ovl{\massp_{{\Blin}, \leq N}(w_1; z)} \covd_{{\Blin}(z)} \massp_{{\Blin}, \leq N}(w_2; z)  \bigg\|_{L_{w_1}^2 L_{w_2}^2(s_1 < s_2)} \\
&\lesssim \int_{t_0'}^{t_0} d\tau \tau^\kappa u^{-\frac{3}{2}-\kappa} \leq (t_0 - t_0') t_0^\kappa u^{-\frac{3}{2}-\kappa}.
\end{equs}
On the other hand, by the definition of $Q_{t_0, t_0', x_0, \leq N}^1$, the triangle inequality, and the second estimate of Lemma \ref{lemma:derivative-nonlinearity-non-resonant-part-L2-estimates} with $j = 2$ (and Remark \ref{remark:modified-Q-estimates}), we may also bound
\begin{equs}
\bigg\|\Delta e^{u \Delta} Q_{t_0, t_0', x_0, \leq N}^1(w_1, w_2)\bigg\|_{L_{w_1}^2 L_{w_2}^2(s_1 < s_2)} \lesssim t_0^{\frac{\kappa}{2}} u^{-\frac{1}{2}-\frac{\kappa}{2}}.
\end{equs}
Interpolating the two bounds, we obtain (for some $0 < \kappa_1 < \kappa$)
\begin{equs}
\bigg\|\Delta e^{u \Delta} Q_{t_0, t_0', x_0, \leq N}^1\bigg\|_{L_{w_1}^2 L_{w_2}^2(s_1 < s_2)} \lesssim (t_0 - t_0')^{\frac{1}{2}-\kappa} u^{-(1-\kappa_1)}.
\end{equs}
From this, it follows that
\begin{equs}\label{eq:Q-1-first-intermediate-bound}
\bigg\| \int_0^1 du \Delta e^{u \Delta} Q_{t_0, t_0', x_0, \leq N}^1\bigg\|_{L_{w_1}^2 L_{w_2}^2(s_1 < s_2)} &\leq \int_0^1 du \bigg\| \Delta e^{u \Delta} Q_{t_0, t_0', x_0, \leq N}^1\bigg\|_{L_{w_1}^2 L_{w_2}^2(s_1 < s_2)} \\
&\lesssim (t_0 - t_0')^{\frac{1}{2}-\kappa} .
\end{equs}
Now, by a similar argument, we may show that
\begin{equs}\label{eq:Q-1-second-intermediate-bound}
\bigg\| e^{\Delta} Q_{t_0, t_0', x_0, \leq N}^1\bigg\|_{L_{w_1}^2 L_{w_2}^2(s_1 < s_2)} \lesssim t_0 - t_0'.
\end{equs}
The estimate \eqref{eq:time-difference-second-moment-intermediate-estimate} now follows by combining \eqref{eq:Q-2-intermediate-bound}, \eqref{eq:Q-1-heat-kernel-decomposition}, \eqref{eq:Q-1-first-intermediate-bound}, and \eqref{eq:Q-1-second-intermediate-bound}.
\end{proof}

\begin{lemma}[$N$-regularity estimate]\label{lemma:derivative-nonlinearity-high-frequency-estimate}
For all $z_0 = (t_0, x_0)$, $\power \in [1, \infty)$, $N \in \dyadic$, $u > 0$, we have that
\begin{equs}
\Big\| (e^{u \Delta}D_{1, N})(z_0)\Big\|_{L^{\power}(\Omega)} \lesssim \power \big(1 + \|{\Blin}\|_{L_z^\infty([0, t_0] \times \T^2)}\big) N^{-1}\sqrt{\log N} \sqrt{\log u^{-1}}.
\end{equs}
\end{lemma}
\begin{proof}
Recall that $D_{1, N} = D_{1, \leq N} - D_{1, \frac{N}{2}}$. When we take this difference of $D_{1, \leq N}$ and $D_{1, \leq \frac{N}{2}}$, one of the linear objects must enter at high frequency. In particular, we have that
\begin{equs}\label{eq:D-1-N-decomp}
&D_{1, N}(z_0)  = \\
&\,\,\,\int dw_1 dw_{2} \ind(s_1 < s_2) \ovl{\zeta(w_1)} \zeta(w_2) \int dz p (z; z_0)\int dy_1' \moll_N(y_1 - y_1') \ovl{\massp_{\Blin}(w_1'; z)} \covd_{{\Blin}(z)} \massp_{{\Blin}, \leq N}(w_2; z) \\
&+\int dw_1 dw_{2} \ind(s_1 < s_2) \ovl{\zeta(w_1)} \zeta(w_2) \int dz p(z; z_0) \ovl{\massp_{{\Blin}, \leq\frac{N}{2}}(w_1; z)} \int dy_2' \moll_N(y_2 - y_2') \covd_{{\Blin}(z)} \massp_{{\Blin}}(w_2'; z),
\end{equs}
This also implies that the smoothed version $(e^{u \Delta} D_{1, N})(z_0)$  is given by almost the same formula above, except the $p(z; z_0)$ is replaced by $p_u(z; z_0)$. We bound the $L_{w_1}^2 L_{w_2}^2$ norm of each term separately. Recalling the high frequency convolution identity \eqref{eq:high-frequency-kernel-convolution-identity}, we have that
\begin{equs}
\bigg\| \int d&z p_u(z; z_0)\int dy_1' \moll_N(y_1 - y_1') \ovl{\massp_{\Blin}(w_1'; z)} \covd_{{\Blin}(z)} \massp_{{\Blin}, \leq N}(w_2; z)\bigg\|_{L_{w_1}^2 L_{w_2}^2(s_1 < s_2)} \\
&\leq \int_{\frac{N}{2}}^N dM M^{-2} \bigg\| \int dz p_u(z; z_0) \int dy_1' \tilde{\moll}_M(y_1 - y_1') \cdot \ovl{\nabla_{y_1'} p_{\Blin}(w_1'; z)} \covd_{{\Blin}(z)} p_{{\Blin}, \leq N}(w_2; z)\bigg\|_{L_{w_1}^2 L_{w_2}^2(s_1 < s_2)},
\end{equs}
where $\tilde{\moll}_M(y) = \moll_{\leq M}(y) M y$. Fix $M$ and $w_1$. By the dual time-weighted estimate Corollary \ref{kernel:lem-time-weighted-dual}, we have that
\begin{equs}
\bigg\| \int dz p_u(z; z_0) \int dy_1' \tilde{\moll}_M(y_1 - y_1') &\cdot \ovl{\nabla_{y_1'} p_{\Blin}(w_1'; z)} \covd_{{\Blin}(z)} p_{{\Blin}, \leq N}(w_2; z)\bigg\|_{L_{w_2}^2([s_1, t_0] \times \T^2)} \\
&\lesssim \bigg\| p_u(z; z_0) (t-s_1)^{\frac{1}{2}}  \int dy_1' \tilde{\moll}_M(y_1 - y_1') \cdot \ovl{\nabla_{y'} p_{\Blin}(w_1'; z)} \bigg\|_{L_z^2([s_1, t_0] \times \T^2)}.
\end{equs}
Taking the $L_{w_1}^2$ norm and then exchanging integration, we reduce to bounding
\begin{equs}
\bigg\| p_u(z; z_0)  (t-s)^{\frac{1}{2}} \int dy_1' &\tilde{\moll}_M(y_1 - y_1') \cdot \ovl{\nabla_{y'} p_{\Blin}(w_1'; z)} \bigg\|_{L_z^2 L_{w_1}^2(s_1 < t)} \leq \\
&\bigg\| p_u(z; z_0)  (t-s_1)^{\frac{1}{2}} \int dy_1' \tilde{\moll}_M(y_1 - y_1') \cdot \ovl{\covd_{-{\Blin}(w_1')} p_{\Blin}(w_1'; z)} \bigg\|_{L_z^2 L_{w_1}^2(s_1 < t)} \label{eq:intermediate-high-frequency-term-1} \\
&+ \bigg\| p_u(z; z_0) (t-s_1)^{\frac{1}{2}}  \int dy_1' \tilde{\moll}_M(y_1 - y_1') \cdot {\Blin}(w_1')\ovl{p_{\Blin}(w_1'; z)} \bigg\|_{L_z^2 L_{w_1}^2(s_1 < t)}. \label{eq:intermediate-high-frequency-term-2}
\end{equs}
To bound \eqref{eq:intermediate-high-frequency-term-1}, we apply the last estimate of Lemma \ref{lemma:covd-pa-l2-spacetime-norm}, combined with the presence of the mollifier $\tilde{\moll}_M$ which prevents a divergence, at the cost of a $\sqrt{\log M}$ term. To bound \eqref{eq:intermediate-high-frequency-term-2}, we use the diamagnetic inequality and standard estimates. In summary, we may obtain:
\begin{equs}
\eqref{eq:intermediate-high-frequency-term-1} \lesssim \sqrt{\log M} \|p_u(z; z_0)\|_{L_z^2} \lesssim \sqrt{\log M} \sqrt{\log u^{-1}}, ~~~~~ \eqref{eq:intermediate-high-frequency-term-2} \lesssim \|{\Blin}(z)\|_{L_z^\infty}\sqrt{\log u^{-1}}.
\end{equs}
Since both terms are acceptable, this takes care of the first term arising from the decomposition \eqref{eq:D-1-N-decomp} of $D_{1, N}$.

Next, we bound the second term arising from \eqref{eq:D-1-N-decomp}. Again, it suffices to fix $M$ and bound
\begin{equs}
\bigg\|\int dz& p_u(z; z_0) \ovl{p_{{\Blin}, \frac{N}{2}}(w_1; z)} \nabla_{y_2} \covd_{{\Blin}(z)} p_{\Blin}(w_2; z)\bigg\|_{L_{w_1}^2 L_{w_2}^2(s_1 < s_2)} \\
&\leq \bigg\|\int dz p_u(z; z_0) \ovl{p_{{\Blin}, \frac{N}{2}}(w_1; z)} \covd_{-{\Blin}(w_2)} \covd_{{\Blin}(z)} p_{\Blin}(w_2; z)\bigg\|_{L_{w_1}^2 L_{w_2}^2(s_1 < s_2)} \\
&+ \bigg\|\int dz p_u(z; z_0) \ovl{p_{{\Blin}, \frac{N}{2}}(w_1; z)}  {\Blin}(w_2) \covd_{{\Blin}(z)} p_{\Blin}(w_2; z)\bigg\|_{L_{w_1}^2 L_{w_2}^2(s_1 < s_2)}. 
\end{equs}
By the two derivative energy estimate Lemma \ref{lemma:two-derivative-energy-estimate}, the first term is bounded by
\begin{equs}
\bigg\| p_u(z; z_0) \ovl{p_{{\Blin}, \frac{N}{2}}(w_1; z)}\bigg\|_{L_{w_1}^2 L_z^2(s_1 < t)} &= \bigg\| p_u(z; z_0) \ovl{p_{{\Blin}, \frac{N}{2}}(w_1; z)}\bigg\|_{L_{z}^2 L_{w_1}^2(s_1 < t)} \lesssim \sqrt{\log N} \sqrt{\log u^{-1}},
\end{equs}
which is acceptable. The second term may be handled by using $\|{\Blin}\|_{L_z^\infty}$ and applying energy estimates in a similar manner.
\end{proof}

Finally, we can prove Proposition \ref{prop:D-1-N-bound}.

\begin{proof}[Proof of Proposition \ref{prop:D-1-N-bound}]
There is some choice $\lambda_1, \lambda_2 \in [0, 1-\kappa]$ with $\lambda_1 + \lambda_2 = 1-\kappa$ such that for functions $f \colon [0, 1] \times \T^2 \ra \C$, we have that
\begin{equs}
\|f\|_{C_t^{\kappa} \Cs_x^{1-10\kappa}} \leq \|f\|_{L_t^{\frac{2}{\kappa}}\Cs_x^{1-\frac{\kappa}{2}}}^{\lambda_1} \|f\|_{C_t^{\frac{1}{2}-\frac{\kappa}{4}} \Cs_x^{-\frac{\kappa}{4}}}^{\lambda_2} \|f\|_{L_t^{-\frac{2}{\kappa}} \Cs_x^{-\frac{\kappa}{2}}}^{\kappa}.
\end{equs}
This is because the left-hand side has parabolic regularity $1-8\kappa$, while the three norms on the right-hand side respectively have parabolic regularity $1-\frac{3\kappa}{2}$, $1 - \frac{3\kappa}{4}$, $-\frac{3\kappa}{2}$. We may also bound $\|f\|_{C_t^{\frac{1}{2}-5\kappa} \Cs_x^{2\kappa}}$ with a possibly different choice of $\lambda_1, \lambda_2$, due to similar considerations. The desired result now follows by combining this with Lemmas \ref{lemma:derivative-nonlinearity-nonresonant-spatial-regularity-estimate}, \ref{lemma:derivative-nonlinearity-nonresonant-time-regularity-estimate}, and \ref{lemma:derivative-nonlinearity-high-frequency-estimate}.
\end{proof}

\subsection{Resonant part of derivative nonlinearity}\label{subsection:resonant-derivative-nonlinearity}

In this section, we control the resonant part of the derivative nonlinearity. Combined with Proposition \ref{prop:derivative-nonlinearity-non-resonant}, this will complete the proof of Theorem \ref{intro:thm-cshe}\ref{item:cshe-derivative-nonlinearity}.

The following central idea underlies the arguments of this subsection. First, in the case when ${\Blin}$ is constant, we can actually compute the limit\footnote{Moreover, this limit is exactly matching the counterterm needed to ensure gauge covariance (Theorem \ref{thm:gauge-covariance}).} of the resonant part as the noise smoothing parameter $N \toinf$ (Lemma~\ref{lemma:constant-A-limit-resonance}). On the other hand, the whole difficulty with estimating the resonant part is in understanding the behavior at short distances. If we assume control on derivatives of $\Blin$, then  ${\Blin}$ is approximately constant at these scales. Naturally, we will thus proceed by a perturbation argument based on Corollary \ref{cor:p-A-u-expansion} in order to control the difference in resonant parts with general ${\Blin}$ and constant ${\Blin}$. This allows us to reduce to the case of constant ${\Blin}$, which we understand perfectly well.

The argument outlined in the preceding paragraph will ultimately give an estimate which depends on one power of ${\Blin}$ (Lemma \ref{lemma:resonant-part-A-dependent-estimate}). In order to reduce the ${\Blin}$-dependence as much as possible (recall that this is crucial for us to control the growth of solutions to \eqref{intro:eq-SAH}), we will interpolate this estimate with one which essentially does not depend on ${\Blin}$ (Lemma \ref{lemma:non-A-dependent-estimate}).

\begin{notation}[Localization]
In this section, whenever we write $|t-s| \sim \rho$, we mean that $|t-s| \in [\frac{\rho}{2}, \rho]$. When we write $|t-s| \lesssim \rho$, we mean that $|t-s| \in [0, 10\rho]$. In particular, $|t-s| \sim \rho$ implies $|t-s| \lesssim \rho$.
\end{notation}

\begin{remark}
We are able to use sharp cutoffs in our localization here because we will never need control on derivatives of the cutoff function $\ind(|t-s| \sim \rho)$.
\end{remark}

\begin{definition} 
For $N \in \dyadic$, define $R_{{\Blin}, \leq N}$ to be the resonant part of $\Im\big(\ovl{\philinear[{\Blin}, \leqN]} \covd_{\Blin} \philinear[{\Blin}, \leqN]\big)$, i.e.
\begin{equs}
R_{{\Blin}, \leq N}(z) := \int dw \int dy_1 dy_2 \moll_{\leq N}(y - y_1) \chi_{\leq N}(y - y_2) \Im\Big(\ovl{\massp_{\Blin}(w_1; z)} \covd_{{\Blin}(z)} \massp_{\Blin}(w_2; z)\Big) .
\end{equs}
For $\rho > 0$, define also the localized version
\begin{equs}
R_{{\Blin}, \leq N}^\rho(z) :=  \int dw \ind(|t-s| \sim \rho) \int dy_1 dy_2 \moll_{\leq N}(y - y_1) \moll_{\leq N}(y - y_2)  \Im \Big( \ovl{\massp_{\Blin}(w_1; z)} \covd_{{\Blin}(z)} \massp_{\Blin}(w_2; z)\Big) .
\end{equs}
\end{definition}

Observe that we may integrate out the $y$ variable and write
\begin{equs}\label{eq:resonant-part-y-integrated-out}
R_{{\Blin}, \leq N}^\rho(z) = \int_0^t ds \ind(|t-s| \sim \rho) \int dy_1 dy_2 \moll_{\leq N}^{*(2)}(y_1 - y_2) \Im\Big(\ovl{\massp_{\Blin}(w_1; z)} \covd_{{\Blin}(z)} \massp_{\Blin}(w_2; z)\Big). 
\end{equs}

We will gradually build towards the main result of this section, which we state as follows. Recall the norm $\|\cdot\|_{T, \resnorm}$ from Definition \ref{intro:def-norm}.

\begin{proposition}[Resonant part of derivative nonlinearity]\label{cshe:prop-resonant}
Let $T \in [0, 1]$, $\alpha \in [0, 1)$, $N \in \dyadic$. For any $\lambda \in (\frac{1}{2}, 1]$, we have that
\begin{align*}
\Big\| \Duh\Big(R_{{\Blin}, \leq N}\Big) - &\Duh\Big(\gaugerenorm {\Blin}\Big)\Big\|_{C_t^0 \Cs_x^\alpha \cap C_t^{\alpha/2} \Cs_x^0([0, T] \times \T^2)} \lesssim_{\alpha, \lambda} \allowdisplaybreaks[0]\\
&T^{\frac{1+\lambda-\alpha}{2}}(1 + \|{\Blin}\|_{T, \resnorm} + N^{-1} \|{\Blin}\|_{T, \resnorm}^2)^\lambda (1 + N^{-1} \|{\Blin}\|_{T, \resnorm})^{1-\lambda} + N^{-1} T^{\frac{1-\alpha}{2}}\|{\Blin}\|_{T, \resnorm}^3.
\end{align*}
\end{proposition}
\begin{remark}
\begin{enumerate}[label=(\roman*)]
\item In Section \ref{section:Abelian-Higgs}, $N$ will be chosen\footnote{To be precise, the frequency-truncation parameter $L$ in \eqref{AH:eq-L} will be chosen much larger than the size of the linear heat flow $\Blin$ in \eqref{AH:eq-Blin}.} to be much larger than ${\Blin}$, and one should therefore regard any term in the estimate with an $N^{-1}$-factor as negligible. With this heuristic, the estimate simplifies to
$ T^{\frac{1+\lambda-\alpha}{2}} (1 + \|{\Blin}\|_{T, \resnorm})^\lambda$. 
We will later take $\lambda = \frac{1}{2}+$ and $\alpha = 0+$, for which the estimate becomes 
\begin{equs}\label{eq:resonant-part-estimate-moral-bound}
T^{\frac{3}{4}-} (1 + \|{\Blin}\|_{T, \resnorm}^{\frac{1}{2}+}).
\end{equs}
The precise powers here will play a key role in the proof of Proposition \ref{decay:prop-short}.
\item The counterterm $\Duh(\gaugerenorm {\Blin})$ is crucial here, because if we were to just estimate $\Duh(R_{{\Blin}, \leq N})$, we would not be able to obtain a bound with powers as in \eqref{eq:resonant-part-estimate-moral-bound}. This is because at best we could hope to bound $\Duh(\gaugerenorm {\Blin})$ by $T \|{\Blin}\|$ for some norm of ${\Blin}$. Thus $\Duh(\gaugerenorm {\Blin})$ plays two roles for us in this paper: (1) it ensures gauge covariance (Theorem \ref{thm:gauge-covariance}) (2) it allows us to obtain the estimate \eqref{eq:resonant-part-estimate-moral-bound} which has good (enough) dependence on ${\Blin}$.
\end{enumerate}
\end{remark}

Before starting with the proof of Proposition \ref{cshe:prop-resonant}, we note that Theorem \ref{intro:thm-cshe}\ref{item:cshe-derivative-nonlinearity} now immediately follows.

\begin{proof}[Proof of Theorem \ref{intro:thm-cshe}\ref{item:cshe-derivative-nonlinearity}]
This follows by combining Propositions \ref{prop:derivative-nonlinearity-non-resonant} and \ref{cshe:prop-resonant}.
\end{proof}

We begin towards the proof of Proposition \ref{cshe:prop-resonant}. We first proceeds towards Definition \ref{def:resgauge}, which we will later see is precisely the formula for the resonant part when ${\Blin}$ is constant. Recall from Definition \ref{def:mollifiers} the Euclidean mollifiers $\eucmoll$ and $\eucmoll_{\leq N} = N^2 \eucmoll(N^2 \cdot)$, the latter of which was used to define $\moll_{\leq N}$ via periodization.

\begin{definition}[Euclidean heat 
 kernel]
Let $\eucheat_t(x)$ be the Euclidean heat kernel, i.e.
\begin{equs}
\eucheat_t(x) = \frac{1}{4\pi t} 
\exp\bigg(-\frac{|x|^2}{4\pi t}\bigg), ~~ x \in \R^2, t > 0.
\end{equs}
\end{definition}

\begin{definition}\label{def:resgauge}
For $b \in \R^2$, $t > 0$, and $N \in \dyadic > 0$, define 
\begin{equs}
\resgauge{\leq N}(b, t) := -\int_0^t ds e^{-2(t-s)} \int_{(\R^2)^2} dy_1 dy_2 (\eucmoll_{\leq N})^{*(2)} (y_1 - y_2) \eucheat_{t-s}(y_1) (\nabla \eucheat_{t-s})(y_2) \Im(e^{\icomplex (y_2 - y_1) \cdot b}).
\end{equs}
For $\rho > 0$, define also the localized version
\begin{equs}
&\resgauge{\leq N}^\rho(b, t) := \\
&-\int_0^t ds \ind(|t-s| \sim \rho) e^{-2(t-s)} \int_{(\R^2)^2} dy_1 dy_2 (\eucmoll_{\leq N})^{*(2)} (y_1 - y_2) \eucheat_{t-s}(y_1) (\nabla \eucheat_{t-s})(y_2) \Im(e^{\icomplex (y_2 - y_1) \cdot b}).
\end{equs}
\end{definition}

Proposition \ref{cshe:prop-resonant} will follow from combining three lemmas: (1) a convergence result for $\resgauge{\leq N}$ (Lemma \ref{lemma:constant-A-limit-resonance}) (2) an estimate with minimal ${\Blin}$-dependence (Lemma \ref{lemma:non-A-dependent-estimate}), (3) a ${\Blin}$-dependent estimate with good dependence on the localization parameter $\rho$ (Lemma \ref{lemma:resonant-part-A-dependent-estimate}). We next state these three lemmas and then discuss how Proposition \ref{cshe:prop-resonant} follows. After that, we prove the three lemmas.

\begin{lemma}[Limiting resonance]\label{lemma:constant-A-limit-resonance}
For all $b \in \R^2$ and $t > 0$, we have that
\begin{equs}
\lim_{N \toinf} \resgauge{\leq N}(b, t) = \gaugerenorm b.
\end{equs}
Moreover, the convergence happens at rate $O(N^{-1} |b| t^{-\frac{1}{2}} + N^{-2}|b|^3) = O(N^{-1} |b|^3 t^{-\frac{1}{2}})$. 
\end{lemma}

\begin{lemma}[Estimate with minimal ${\Blin}$-dependence]\label{lemma:non-A-dependent-estimate}
Let $T \in [0, 1]$, $\alpha \in [0, 1)$, $\rho \in (0, 1]$, and $N \in \dyadic$. Then, we have that 
\begin{equs}
\Big\| \Duh\Big(R_{{\Blin}, \leq N}^\rho\Big) - \Duh\Big(\resgauge{\leq N}^\rho({\Blin}(z), t)\Big)\Big\|_{C_t^0 \Cs_x^\alpha \cap C_t^{\alpha/2} \Cs_x^0([0, T] \times \T^2)} \lesssim \rho^{-\frac{1}{2}} T^{1-\frac{\alpha}{2}}\big(1 + N^{-1} \|{\Blin}\|_{L_t^\infty L_x^\infty}\big).
\end{equs}
\end{lemma}

\begin{lemma}[Estimate using one power of ${\Blin}$]\label{lemma:resonant-part-A-dependent-estimate}
Let $T \in [0, 1]$, $\alpha \in [0, 1)$, $\rho \in (0, 1]$, and $N \in \dyadic$. Then, we have that 
\begin{align*}
&\, \Big\|\Duh\Big(R_{{\Blin}, \leq N}^\rho\Big) - \Duh\Big(\resgauge{\leq N}^\rho({\Blin}(z), t)\Big)\Big\|_{C_t^0 \Cs_x^\alpha \cap C_t^{\alpha/2} \Cs_x^0([0, T] \times \T^2)}  \\
\lesssim&\, \rho^{\frac{1}{2}} T^{\frac{1-\alpha}{2}}  \big(1 + \|{\Blin}\|_{T, \resnorm} + N^{-1} \|{\Blin}\|_{T, \resnorm}^2\big).
\end{align*}
\end{lemma}

Assuming these three lemmas, we now prove Proposition \ref{cshe:prop-resonant}.

\begin{proof}[Proof of Proposition \ref{cshe:prop-resonant}]
By Lemma \ref{lemma:constant-A-limit-resonance}, we have the pointwise bound
\begin{equs}
\Big| \resgauge{\leq N}({\Blin}(z), t) - \frac{{\Blin}(z)}{8\pi}\Big| \lesssim \frac{1}{N}|{\Blin}(z)|^3 t^{-\frac{1}{2}},
\end{equs}
which leads to an acceptable contribution after applying the Duhamel integral and using classical Schauder estimates. Thus it suffices to bound $\Duh(R_{{\Blin}, \leq N}) - \Duh(\resgauge{\leq N}({\Blin}, t))$. By summing in $\rho$ and interpolating the estimate with minimal ${\Blin}$-dependence (Lemma \ref{lemma:non-A-dependent-estimate}) and the ${\Blin}$-dependent estimate (Lemma \ref{lemma:resonant-part-A-dependent-estimate}), we have that 
\begin{equs}
\Big\|\Duh&(R_{B, \leq N}) - \Duh(\resgauge{\leq N}({\Blin}, t))\Big\|_{C_t^0 \Cs_x^\alpha \cap C_t^{\alpha/2} \Cs_x^0} \\
&\lesssim \sum_{\substack{\rho^{-1} \in \dyadic \\ \rho \leq T}} \Big(\rho^{-\frac{1}{2}} T^{1-\frac{\alpha}{2}}(1 + N^{-1} \|{\Blin}\|_{T, \resnorm})\Big)^{1-\lambda} \Big( \rho^{\frac{1}{2}} T^{\frac{1-\alpha}{2}}(1 + \|{\Blin}\|_{T, \resnorm} + N^{-1} \|{\Blin}\|_{T, \resnorm}^2)\Big)^{\lambda} \\
&\lesssim \sum_{\substack{\rho^{-1} \in \dyadic \\ \rho \leq T}} \rho^{\lambda - \frac{1}{2}} T^{1 - \frac{\alpha + \lambda}{2}}  (1 + N^{-1} \|{\Blin}\|_{T, \resnorm})^{1-\lambda} (1 + \|{\Blin}\|_{T, \resnorm} + N^{-1} \|{\Blin}\|_{T, \resnorm}^2)^\lambda \\
&\lesssim T^{\frac{1+\lambda-\alpha}{2}}(1 + \|{\Blin}\|_{T, \resnorm} + N^{-1} \|{\Blin}\|_{T, \resnorm}^2)^\lambda (1 + N^{-1} \|{\Blin}\|_{T, \resnorm})^{1-\lambda}.
\end{equs}
The desired result now follows.
\end{proof}

The remainder of this subsection is devoted to the proofs of Lemmas \ref{lemma:constant-A-limit-resonance}-\ref{lemma:resonant-part-A-dependent-estimate}. First, we prove Lemma \ref{lemma:constant-A-limit-resonance}.

\begin{proof}[Proof of Lemma \ref{lemma:constant-A-limit-resonance}]
For brevity, let $R_{\leq N} = \resgauge{\leq N}(b, t)$. Recalling the definition of $R_{\leq N}$ (Definition~\ref{def:resgauge}) and changing variables $u = y_2 - y_1$, $w = y_1$, we obtain
\begin{equs}
R_{\leq N} &=  -\int_0^t ds e^{-2(t-s)}  \int_{(\R^2)^2} du dw (\eucmoll_{\leq N})^{*(2)}(u) \eucheat_{t-s}(w) (\nabla \eucheat_{t-s})(w + u) \Im(e^{\icomplex u \cdot b}) \\
&= - \int_0^t ds e^{-2(t-s)} \int_{\R^2} du (\eucmoll_{\leq N})^{*(2)}(u) \sin(u \cdot b) \int_{\R^2} dw \eucheat_{t-s}(w) (\nabla \eucheat_{t-s})(w+u) \\
&= -\int_0^t ds e^{-2(t-s)} \int_{\R^2} du (\eucmoll_{\leq N})^{*(2)}(u) \sin(u\cdot b) (\eucheat_{t-s} * \nabla \eucheat_{t-s})(u) \\
&= -\int_0^t ds e^{-2(t-s)}  \int_{\R^2}  du (\eucmoll_{\leq N})^{*(2)}(u) \sin(u \cdot b) (\nabla \eucheat_{2(t-s)})(u) \\
&=\frac{1}{32\pi} \int_0^t ds e^{-2(t-s)} 
 \int_{\R^2} du (\eucmoll_{\leq N})^{*(2)}(u) \sin(u \cdot b) \frac{u}{(t-s)^2} \exp\bigg(-\frac{|u|^2}{8(t-s)}\bigg).
\end{equs}
We note that, in going from the third to the fourth line, we used the semi-group property of the heat kernel. 
Fixing $u$ and performing the $s$-integral, and using the bounds 
\begin{equs}
1 - e^{-2(t-s)} \leq 2(t-s), ~~ \int_1^\infty \frac{dr}{r} e^{-c r} \lesssim \log (1 + c^{-1}), ~~ c \in (0, \infty),
\end{equs}
we may compute
\begin{equs}
\int_0^t ds \frac{e^{-2(t-s)}}{(t-s)^2} \exp\bigg(-\frac{|u|^2}{8(t-s)}\bigg) &= \int_{\frac{1}{t}}^\infty dr \exp\bigg(-\frac{|u|^2 r}{8}\bigg) + O\big(\log(1 + t |u|^{-2})\big)  \\
&= \frac{8}{|u|^2} \exp\bigg(-\frac{|u|^2}{8t}\bigg) +  O\big(\log(1 + t |u|^{-2})\big). 
\end{equs}
Inserting this, we further obtain
\begin{equs}
R_{\leq N} = \frac{1}{4\pi} \int_{\R^2} du (\eucmoll_{\leq N})^{*(2)}(u) \frac{\sin(u \cdot b)} {|u|^2} u \exp\bigg(-\frac{|u|^2}{8t}\bigg) + \int_{\R^2} du (\eucmoll_{\leq N})^{*(2)}(u) O\Big( |\sin(u \cdot b) u| \log(1 + t |u|^{-2})\Big).
\end{equs}
Since $(\eucmoll_{\leq N})^{*(2)}(\cdot) = N^2 (\eucmoll)^{*(2)}(N \cdot)$, we further obtain (using that $1 - \exp(-c N^{-2} |u|^2) \leq c^{\frac{1}{2}} N^{-1} |u|$, $|\sin(x)| \leq |x|$, and $|\sin(x) - x| = O(x^3)$)
\begin{equs}
R_{\leq N} &= \frac{1}{4\pi} \int_{\R^2} du (\eucmoll)^{*(2)} (u) \frac{\sin(N^{-1} u \cdot b)}{N^{-1}  |u|^2} u \exp\bigg(-\frac{N^{-2} |u|^2}{8t}\bigg) + O\big(|b| N^{-2} \log(1 + tN^2)\big) \\
&= \frac{1}{4\pi} \int_{\R^2} du (\eucmoll)^{*(2)}(u) \frac{\sin(N^{-1}u \cdot b)}{N^{-1} |u|^2} u + O\big(N^{-1} |b| t^{-\frac{1}{2}}) +  O(|b| N^{-2} \log(1 + tN^2)\big) \\
&= \frac{1}{4\pi} \int_{\R^2} du (\eucmoll)^{*(2)}(u) \frac{(u \cdot b) u}{|u|^2} + O(N^{-1} |b| t^{-\frac{1}{2}}) + O(N^{-2} |b|^3) +  O\big(N^{-1} |b| t^{-\frac{1}{2}}).
\end{equs}
To finish, note that 
\begin{equs}
\frac{1}{4\pi} \int_{\R^2} du (\eucmoll)^{*(2)}(u) \frac{(u \cdot b) u}{|u|^2} = \frac{1}{4\pi} \int_{\R^2} du (\eucmoll)^{*(2)}(u) \frac{u u^T}{|u|^2} b .
\end{equs}
Since $\eucmoll$ is radial and integrates to $1$, we have that
\begin{equs}
\int_{\R^2} du (\eucmoll)^{*(2)}(u) \frac{uu^T}{|u|^2} = \frac{1}{2} I,
\end{equs}
which follows by the following two observations: (1) since $\eucmoll$ (and thus also $(\eucmoll)^{*(2)}$) is radial, one sees that the above must commute with all rotation matrices, and thus must be a multiple of the identity. (2) To determine the multiple, simply compute its trace to be 1, by using that $\Tr(uu^T) = \Tr(u^T u) = |u|^2$ (and the assumption that $\eucmoll$ integrates to $1$). The desired result now follows.
\end{proof}

Next, we prove the estimate with minimal ${\Blin}$-dependence (Lemma \ref{lemma:non-A-dependent-estimate}). Before doing so, we need the following preliminary lemma.

\begin{lemma}\label{lemma:resgauge-easy-bound}
For $\rho > 0$, $b \in \R^2$, $t > 0$, $N \in \dyadic$, we have that 
\begin{equs}
\Big| \resgauge{\leq N}^\rho(b, t)\Big| \lesssim \frac{1}{N} \rho^{-\frac{1}{2}} |b|.
\end{equs}
\end{lemma}
\begin{proof}
For $|t-s| \sim \rho$, we may bound $\| \eucheat_{t-s}\|_{L_x^\infty} \lesssim \rho^{-1}$ and $\|\nabla \eucheat_{t-s}\|_{L_x^1} \lesssim \rho^{-\frac{1}{2}}$, and thus we obtain
\begin{align*}
\Big| \resgauge{\leq N}^\rho(b, t)\Big| &\lesssim |b| \rho^{-1} \int_0^t ds \ind(|t-s| \sim \rho) \int_{(\R^2)^2} dy_1 dy_2 |(\eucmoll_{\leq N})^{*(2)}(y_1 - y_2)|  |\nabla  \eucheat_{t-s}(y_1)| |y_1 - y_2| \\
&= |b| \rho^{-1} \int_0^t ds \ind(|t-s| \sim \rho) \int dy 
|(\eucmoll_{\leq N})^{*(2)}(y)| |y| \int dy' |\nabla \eucheat_{t-s}(y' - y)| \\
&\lesssim |b| \frac{1}{N} \rho^{-\frac{3}{2}} \int_0^t ds \ind(|t-s| \sim \rho) = \frac{1}{N} \rho^{-\frac{1}{2}} |b|. \qedhere
\end{align*}
\end{proof}

\begin{proof}[Proof of Lemma \ref{lemma:non-A-dependent-estimate}]
By Lemma \ref{lemma:resgauge-easy-bound} and classical Schauder estimates, we have that
\begin{equs}
\Big\|\Duh\Big(\resgauge{\leq N}^\rho({\Blin}(z), t)\Big)\Big\|_{C_t^0 \Cs_x^\alpha \cap C_t^{\alpha/2} \Cs_x^0([0, T] \times \T^2)} \lesssim \frac{1}{N} \rho^{-\frac{1}{2}} T^{1 -\frac{\alpha}{2}} \|{\Blin}\|_{L_t^\infty L_x^\infty}.
\end{equs}
We next handle $\Duh(R_{{\Blin}, \leq N}^\rho)$. By the Duhamel integral estimate (Lemma \ref{prelim:lem-Duhamel-weighted} with $\theta = 0$), it suffices to show that
\begin{equs}
\big\|R_{\Blin, \leq N}^\rho\big\|_{L_t^2 L_x^\infty([0, T] \times \T^2)} \lesssim \rho^{-\frac{1}{2}} T^{\frac{1}{2}} .
\end{equs}
Recalling \eqref{eq:resonant-part-y-integrated-out}, applying Cauchy-Schwarz, and then applying Lemma \ref{lemma:covd-pa-l2-spacetime-norm}, we may obtain the following pointwise bound in $z$:
\begin{equs}
|R_{\Blin, \leq N}^\rho(z)| &\leq \big\| (t-s)^{-1} p_{\Blin}(w; z)\big \|_{L_w^2(|t-s| \sim \rho)} \big\|(t-s) \covd_{\Blin(z)} p_{\Blin}(w; z)\big\|_{L_w^2(|t-s| \sim \rho)} \\
&\lesssim \rho^{-1} \rho^{\frac{1}{2}} = \rho^{-\frac{1}{2}}.
\end{equs}
Taking $L_t^2 L_x^\infty$, we then obtain the desired bound.
\end{proof}

The remainder of this subsection is devoted towards the proof of Lemma \ref{lemma:resonant-part-A-dependent-estimate}. This lemma will follow by combining many intermediate technical results. To help the reader see the structure of the proof, we prefer to directly give the proof of Lemma \ref{lemma:resonant-part-A-dependent-estimate} first, and then state and prove all the intermediate technical results which are needed.

\begin{proof}[Proof of Lemma \ref{lemma:resonant-part-A-dependent-estimate}]
By the Duhamel integral estimate (Lemma \ref{prelim:lem-Duhamel-weighted} with $\theta = 0$), it suffices to show the following:
\begin{equs}
\Big\|R_{{\Blin}, \leq N}^\rho(z) - \resgauge{\leq N}^\rho({\Blin}(z), t)\Big\|_{L_t^2 L_x^\infty} \lesssim \rho^{\frac{1}{2}} \big(1 + \|{\Blin}\|_{T, \resnorm} + N^{-1} \|{\Blin}\|_{T, \resnorm}^2 \big) .
\end{equs}
For brevity, let $H^k := \ptl_t {\Blin}^k + \ptl_j F_{\Blin}^{kj}$.
Inserting the formula for $\covd_{{\Blin}(z)} p_{\Blin}(w; z)$ given by Lemma \ref{lemma:covd-pA-expansion}, we have that
\begin{equs}
&R_{{\Blin}, \leq N}^\rho(z) =\\
&-\int_0^t ds \ind(|t-s| \sim \rho)\int dy_1 dy_2 \moll_{\leq N}^{*(2)}(y_1 - y_2) \Im\Big( \ovl{\massp_{\Blin}(w_1; z)} \covd_{-{\Blin}(w_2)} \massp_{\Blin}(w_2; z)\Big)    \\
&+\int_0^t  ds \ind(|t-s| \sim \rho) \int dy_1 dy_2 \moll_{\leq N}^{*(2)}(y_1 - y_2)\Im\bigg(2\icomplex  \ovl{\massp_{\Blin}(w_1; z)} \int dv \massp_{\Blin}(v; z) F_{\Blin}(v) \covd_{{\Blin}(v)} \massp_{\Blin}(w_2; v)\bigg) \\
&+ \int_0^t ds \ind(|t-s| \sim \rho) \int dy_1 dy_2 \moll_{\leq N}^{*(2)} (y_1 - y_2) \Im\bigg(\icomplex \ovl{\massp_{\Blin}(w_1; z)} \int dv \massp_{\Blin}(v; z) H(v) \massp_{\Blin}(w_2; v)\bigg) \\
&=: R_{{\Blin}, \leq N}^{\rho, 1}(z) + R_{{\Blin}, \leq N}^{\rho, 2}(z) + R_{{\Blin}, \leq N}^{\rho, 3}(z). \label{eq:R-rho-Blin-defs}
\end{equs}
Lemma \ref{lemma:pA-FA-covd-pA-localized-estimate} gives an acceptable bound for $R_{{\Blin}, \leq N}^{\rho, 2}$, while Lemma \ref{lemma:p-H-p-contribution} gives an acceptable bound for $R_{{\Blin}, \leq N}^{\rho, 3}$. To bound the first term, define ${\Blin}_u(w) := {\Blin}_{u, z}(w) := (1-u) {\Blin}(w) + u {\Blin}(z)$. Note that ${\Blin}_0 = {\Blin}$. We decompose
\begin{equs}
R_{{\Blin}_0, \leq N}^{\rho, 1}(z) - \resgauge{\leq N}^{\rho}({\Blin}(z), t) = R_{{\Blin}_1, \leq N}^{\rho, 1}(z) -\resgauge{\leq N}^{\rho}({\Blin}(z), t) - \int_0^1  \ptl_u R_{{\Blin}_u, \leq N}^{\rho, 1}(z) du.
\end{equs}
Note that ${\Blin}_1$ is constant with value $\Blin(z)$. Thus by Lemma \ref{lemma:constant-A-bounded-resonance}, we have that 
\begin{equs}\label{eq:R-Blin-1-estimate}
|R_{{\Blin_1}, \leq N}^{\rho, 1}(z) - \resgauge{\leq N}^\rho({\Blin}(z), t)| \lesssim \rho^{10}\big(1 + N^{-10} |{\Blin}(z)|\big),
\end{equs}
which is acceptable upon taking $L_t^2 L_x^\infty$. Next, using that $\ptl_u {\Blin}_u(w) = {\Blin}(z) - {\Blin}(w)$, we have that 
\begin{equs}
\ptl_u &R_{{\Blin}_u, \leq N}^{\rho, 1}(z) = \\
&-\int_0^t ds \ind(|t-s| \sim \rho) \int dy_1 dy_2 \moll_{\leq N}^{*(2)}(y_1 - y_2) \Im\Big( \ovl{\ptl_u \massp_{{\Blin}_u}(w_1; z)} \covd_{-{\Blin}_u(w_2)} \massp_{{\Blin}_u}(w_2; z) \Big)  \label{eq:perturb-intermediate-1}\\
&- \int_0^t ds \ind(|t-s| \sim \rho) \int dy_1 dy_2 \moll_{\leq N}^{*(2)}(y_1 - y_2) \Im\Big( \ovl{\massp_{{\Blin}_u}(w_1; z)} \covd_{-{\Blin}_u(w_2)} \ptl_u \massp_{{\Blin}_u}(w_2; z) \Big) \label{eq:perturb-intermediate-2} \\
&-\int_0^t ds \ind(|t-s| \sim \rho) \int dy_1 dy_2 \moll_{\leq N}^{*(2)}(y_1 - y_2) \Im\Big(\icomplex \ovl{\massp_{{\Blin}_u}(w_1; z)} \massp_{{\Blin}_u}(w_2; z)) ({\Blin}(w_2) - {\Blin}(z))\Big) \label{eq:perturb-intermediate-3}.
\end{equs}
Lemmas \ref{lemma:ptl-u-p-Au--D-Au-estimate} and \ref{lemma:p-squared-difference-A-estimate} give estimates of \eqref{eq:perturb-intermediate-1} and \eqref{eq:perturb-intermediate-3}, which we then integrate over $u \in (0, 1)$ to obtain an acceptable estimate. Thus, we focus on \eqref{eq:perturb-intermediate-2}. By integration by parts in the $y_2$ variable, we have that
\begin{equs}\label{eq:perturb-intermediate-2-ibp}
\eqref{eq:perturb-intermediate-2} = \int_0^t ds \ind(|t-s| \sim \rho) \int dy_1 dy_2 \covd_{{\Blin}_u(w_2)} \moll_{\leq N}^{*(2)}(y_1 - y_2) \Im\Big(\ovl{\massp_{{\Blin}_u}(w_1; z)} \ptl_u \massp_{{\Blin}_u}(w_2; z)\Big).
\end{equs}
Since $\nabla_{y_2}\moll_{\leq N}^{*(2)}(y_1 - y_2) = - \nabla_{y_1} \moll_{\leq N}^{*(2)}(y_1 - y_2)$, we have that
\begin{equs}
\covd_{{\Blin}_u(w_2)} \moll_{\leq N}^{*(2)}(y_1 - y_2) &= -\covd_{-{\Blin}_u(w_1)} \moll_{\leq N}^{*(2)}(y_1 - y_2) + (\covd_{{\Blin}_u(w_2)} + \covd_{-{\Blin}_u(w_1)})  \moll_{\leq N}^{*(2)}(y_1 - y_2)\\
&= -\covd_{-{\Blin}_u(w_1)} \moll_{\leq N}^{*(2)}(y_1 - y_2) + \icomplex ({\Blin}_u(w_2) - {\Blin}_u(w_1)) \moll_{\leq N}^{*(2)}(y_1 - y_2).
\end{equs}
Inserting this identity into \eqref{eq:perturb-intermediate-2-ibp}, we may write $\eqref{eq:perturb-intermediate-2-ibp} = I_1 + I_2$.
To handle $I_1$, we may integrate by parts again (this time in the $y_1$ variable) to place the derivative $\covd_{-{\Blin}_u(w_1)}$ on $\massp_{{\Blin}_u}(w_1; z)$, which in fact shows that $I_1 = \eqref{eq:perturb-intermediate-1}$, which we have previously already estimated. Finally, Lemma \ref{lemma:perturbation-argument-ibp-error-term} gives an estimate of $I_{2}$, which leads to an acceptable contribution after integrating over $u \in (0, 1)$.
\end{proof}

In the remainder of this subsection, we prove all the intermediate results which were cited in the proof of Lemma \ref{lemma:resonant-part-A-dependent-estimate}. The next lemma gives explicit formulas for the covariant heat kernel and its covariant derivatives when ${\Blin}$ is constant. Recall that $\eucheat$ is the heat kernel on $\R^2$. 

\begin{lemma}\label{lemma:constant-A-heat-kernel-formulas}
Suppose ${\Blin} \equiv b \in \R^2$ is constant. Then for any $w, z$, we have that
\begin{equs}
p_{\Blin}(w; z) &= \sum_{n \in \Z^2} \eucheat_{t-s}(y-x + 2\pi n) e^{\icomplex (y - x + 2\pi n) \cdot b}, \\
\covd_{-{\Blin}(w)} p_{\Blin}(w; z) &= \sum_{n \in \Z^2} (\nabla \eucheat_{t-s})(y-x+ 2\pi n)  e^{\icomplex (y - x + 2\pi n) \cdot b}.
\end{equs}
Here, we identify $x, y \in \T^2$ with elements of $[-\pi, \pi]^2$.
\end{lemma}
\begin{proof}
If ${\Blin} \equiv b \in \R^2$, then there is the explicit formula for the covariant heat kernel 
\begin{equs}
p_{\Blin}(w; z) = \sum_{n \in \Z^2}  \eucheat_{t-s}(y-x + 2\pi n) e^{\icomplex (y - x + 2\pi n)\cdot b},
\end{equs}
which follows directly from the Feynman-Kac-It\^{o} formula (Lemma \ref{lemma:feynman-kac-ito-formula} and Remark \ref{remark:covariant-heat-kernel-ito-stratonovich}), and the fact that $\T^2$-valued Brownian motion may be viewed as the periodiziation of $\R^2$-valued Brownian motion. Given this formula, the second identity follows by explicit computation.
\end{proof}

\begin{lemma}\label{lemma:constant-A-bounded-resonance}
For all $z = (t, x)\in [0, 1] \times \T^2$, $b \in \R^2$, $\rho > 0$, and  $N \in \dyadic$, we have that for constant ${\Blin} \equiv b$,
\begin{equs}
\bigg| \int_0^t ds& \ind(|t-s| \sim \rho) e^{-2(t-s)} \int dy_1 dy_2 \moll_{\leq N}^{*(2)}(y_1 - y_2) \Im\Big(\ovl{p_{\Blin}(w_1; z)} \covd_{-{\Blin}(w_2)} p_{\Blin}(w_2; z)\Big) + \resgauge{\leq N}^\rho(b, t)\bigg| \\
&\lesssim \rho^{10} \big(1 + N^{-10} |b|\big).
\end{equs}
\begin{remark}
To connect back to the proof of Lemma \ref{lemma:resonant-part-A-dependent-estimate}, from \eqref{eq:R-rho-Blin-defs} we have that
\begin{equs}
R_{B_1, \leq N}^{\rho, 1}(z) = -  \int_0^t ds \ind(|t-s| \sim \rho) e^{-2(t-s)} \int dy_1 dy_2 \moll_{\leq N}^{*(2)}(y_1 - y_2) \Im\Big(\ovl{p_{\Blin_1}(w_1; z)} \covd_{-{\Blin_1}(w_2)} p_{\Blin_1}(w_2; z)\Big).
\end{equs}
Thus, Lemma \ref{lemma:constant-A-bounded-resonance} implies the estimate (recall that $\Blin_1$ has constant value $\Blin(z)$)
\begin{equs}
\big|-R_{\Blin_1, \leq N}^{\rho, 1}(z) + \resgauge{\leq N}^{\rho}(\Blin(z), t)\big| \lesssim \rho^{10}\big(1 + N^{-10}|B(z)|\big),
\end{equs}
which is what was claimed in \eqref{eq:R-Blin-1-estimate}.
\end{remark}
\end{lemma}
\begin{proof}
By Lemma \ref{lemma:constant-A-heat-kernel-formulas} and a change of variables argument, we may take $x = 0$ (this is consistent with the fact that $B \equiv b$ is constant, so that one might expect the resonant part to be translation-invariant). Inserting the formulas from Lemma \ref{lemma:constant-A-heat-kernel-formulas} for $x = 0$, and using the fact that $\moll_{\leq N}^{*(2)}$ is the periodization of $(\eucmoll_{\leq N})^{*(2)}$ (recall Definition \ref{def:mollifiers}), we obtain the following sum of integrals over $[0, t] \times ([-\pi, \pi]^2)^2$:
\begin{equs}
&\int_0^t ds \ind(|t-s| \sim \rho) e^{-2(t-s)} \int dy_1 dy_2 \moll_{\leq N}^{*(2)}(y_1 - y_2) \Im\Big(\ovl{p_{\Blin}(w_1; z)} \covd_{-{\Blin}(w_2)} p_{\Blin}(w_2; z)\Big) = \\
&\sum_{n, n_1, n_2 \in \Z^2} \int_0^t ds \ind(|t-s| \sim \rho)  e^{-2(t-s)} \int_{([-\pi, \pi]^2)^2} dy_1 dy_2 (\eucmoll_{\leq N})^{*(2)}(y_1 - y_2 + 2\pi n) ~\times \\
&\eucheat_{t-s}(y_1 + 2\pi n_1) \nabla \eucheat_{t-s}(y_2 + 2\pi n_2) \Im\big(e^{\icomplex (y_2 - y_1) \cdot b} e^{\icomplex 2\pi (n_2 - n_1) \cdot b}\big).
\end{equs}
Let $I(n, n_1, n_2)$ be the summand corresponding to $n, n_1, n_2$ in the above sum. We proceed to argue that the contribution when at least one of the $n, n_1, n_2 \neq 0$ is acceptable. (Intuitively, this must be the case as the whole problem is at short distances.) Then, we will argue that the contribution of $I(0, 0, 0) + \resgauge{\leq N}^\rho(b, t)$ is also acceptable.

First, when one of $n_1, n_2 \neq 0$, we have the simple pointwise bound $|\eucheat_{t-s} (y_1 + 2\pi n_1) \nabla \eucheat_{t-s}(y_2 + 2\pi n_2)| \lesssim \rho^{10} \exp(-c (|n_1|^2 + |n_2|^2) / \rho)$ (say). We thus obtain
\begin{align*}
&\sum_{\substack{n, n_1, n_2 \\ \max(|n_1|, |n_2|) > 0}} |I(n, n_1, n_2)| \\
&\lesssim \rho^{10} \int_0^t ds \ind(|t-s| \sim \rho) \exp\Big(-c \frac{|n_1|^2 + |n_2|^2}{\rho}\Big) \int_{([-\pi, \pi]^2)^2} dy_1 dy_2 \sum_n |(\eucmoll_{\leq N})^{*(2)}(y_1 - y_2 + 2\pi n)|  \lesssim \rho^{10},
\end{align*}
which is acceptable. Next, we focus on the case $n_1 = n_2 = 0$. Consider $n \neq 0$. We split into cases $\max(|y_1|, |y_2|) \leq \rho^{\frac{1}{4}}$, $\max(|y_1|, |y_2|) > \rho^{\frac{1}{4}}$. Split the integral $I(n, 0, 0) =: I_1(n, 0, 0) + I_2(n, 0, 0)$ according to these two cases. In the latter case, due to the bound $\exp(-\rho^{-\frac{1}{2}}) \lesssim \rho^{100}$, we have the simple bound
\begin{equs}
\sum_{n \neq 0} |I_2(n, 0, 0)| \lesssim \rho^{10} \int_0^t ds \ind(|t-s| \sim \rho) \int_{([-\pi, \pi]^2)^2} dy_1 dy_2 \sum_{n \neq 0} |(\eucmoll_{\leq N})^{*(2)} (y_1 - y_2 + 2\pi n)| \lesssim \rho^{10}.
\end{equs}
In the former case, we use that when $\max(|y_1|, |y_2|) \leq \rho^{\frac{1}{4}}$, we have that $|y_1 - y_2 + 2\pi n|\geq 2\pi |n| - 2 \rho^{\frac{1}{4}} \gtrsim |n|$, and thus using the rapid decay of $(\eucmoll_{\leq N})^{*(2)}$ as well as the bound $|\Im(e^{\icomplex (y_2 -y_1) \cdot b})| \leq |y_2 - y_1| |b|$, we obtain
\begin{align*}
&|I_1(n, 0, 0)| \\
&\lesssim |b| N^{-10} |n|^{-10} \int_0^t ds \ind(|t-s| \sim \rho) \int_{([-\pi, \pi]^2)^2} dy_1 dy_2 |\eucheat_{t-s}(y_1)| |\nabla \eucheat_{t-s}(y_2)| |y_2 - y_1| \lesssim |b| \rho N^{-10} |n|^{-10}.
\end{align*}
This yields an acceptable contribution upon summing over $n \neq 0$. Finally, in the case $n = n_1 = n_2 = 0$, observe that with $D = (\R^2)^2 - ([-\pi, \pi]^2)^2$, we have that (recalling the definition of $\resgauge{\leq N}^\rho(b,t)$ (Definition \ref{def:resgauge}), which comes with a minus sign)
\begin{equs}
I(0, &0, 0) + \resgauge{\leq N}^\rho(b, t)  \\
&= -\int_0^t ds \ind(|t-s| \sim \rho) e^{-2(t-s)}  \int_{D} dy_1 dy_2 (\eucmoll_{\leq N})^{*(2)}(y_1 - y_2) \eucheat_{t-s}(y_1) (\nabla \eucheat_{t-s})(y_2) \Im(e^{\icomplex(y_2 - y_1) \cdot b}).
\end{equs}
This gives a contribution bounded by $\rho^{10}$ (say) since one of $|y_1|, |y_2| \gtrsim 1$ on the region $D$, so that one of the heat kernel factors has very rapid decay in $\rho$.
\end{proof}

\begin{lemma}\label{lemma:ptl-u-p-Au-estimate}
Let $t,\rho \in (0, 1]$, $x\in \T^2$, and $z = (t, x)$. For any $u \in (0, 1)$, we have that
\begin{equs}
\bigg\|\ind(|t-s| \lesssim \rho) (t-s)^{-\frac{3}{4}} \ptl_u \massp_{{\Blin}_u}(w; z)\bigg\|_{L_w^2([0, t] \times \T^2)} \lesssim ~~&\bigg(\int_0^t dt_v \ind(|t-t_v| \lesssim \rho) (t-t_v)^{-\frac{3}{2}} \bigg(\int_{t_v}^t d\tau \|\ptl_\tau {\Blin}\|_{L_x^\infty}\bigg)^2\bigg)^{\frac{1}{2}} \\
&+ \rho^{\frac{1}{4}} \big(\|\nabla {\Blin}(t)\|_{L_x^\infty} + \|\ptl_j {\Blin}^j\|_{L_z^\infty([0, t] \times \T^2)}\big). 
\end{equs}
Moreover,
\begin{equs}
\bigg\|\bigg(\int_0^t dt_v \ind(|t-t_v| \lesssim \rho) (t-t_v)^{-\frac{3}{2}} \bigg(\int_{t_v}^t d\tau \|\ptl_\tau {\Blin}\|_{L_x^\infty}\bigg)^2\bigg)^{\frac{1}{2}} \bigg\|_{L_t^2([0, t_0])} \lesssim \rho^{\frac{1}{4}}\|\ptl_t {\Blin}\|_{L_t^1 L_x^\infty([0, t] \times \T^2)}.
\end{equs}
\end{lemma}
\begin{proof}
For brevity, let $J_\rho := \{w \in [0, t] \times \T^2 : |t-s| \lesssim \rho\}$. Recalling the formula for $\ptl_u \massp_{{\Blin}_u}$ from Corollary~\ref{cor:p-A-u-expansion}, we have that
\begin{equs}
\bigg\|(t-s)^{-\frac{3}{4}} \ptl_u &\massp_{{\Blin}_u}(w; z)\bigg\|_{L_w^2(J_\rho)} \\
&\lesssim \bigg\| (t-s)^{-\frac{3}{4}} \int dv p_{{\Blin}_u}(v; z) ({\Blin}(v) - {\Blin}(z)) \covd_{{\Blin}_u(v)} p_{{\Blin}_u}(w; v)\bigg\|_{L_w^2(J_\rho)} \label{eq:ptl-u-p-Au-estimate-intermediate-1} \\
&+ \bigg\|(t-s)^{-\frac{3}{4}} \int dv \massp_{{\Blin}_u}(v; z) (\ptl_j {\Blin}^j)(v)\massp_{{\Blin}_u}(w; v)\bigg\|_{L_w^2(J_\rho)}. \label{eq:ptl-u-p-Au-estimate-intermediate-2}
\end{equs}
To bound \eqref{eq:ptl-u-p-Au-estimate-intermediate-2}, we use the following pointwise estimate for fixed $w, z$ (here we apply the diamagnetic inequality and the semigroup property of the heat kernel):
\begin{equs}
\bigg|\int dv \massp_{{\Blin}_u}(v; z) (\ptl_j {\Blin}^j)(v) \massp_{{\Blin}_u}(w; v)\bigg| &\leq p(w; z) \int_s^t dt_v \|\ptl_j {\Blin}^j(t_v)\|_{L_x^\infty} \leq (t-s) p(w; z)\|\ptl_j {\Blin}^j\|_{L_z^\infty([0, t] \times \T^2)}.
\end{equs}
Note that $\|(t-s)^{\frac{1}{4}} p(w; z)\|_{L_w^2(J_\rho)} \lesssim \rho^{\frac{1}{4}}$. We thus obtain that the contribution from \eqref{eq:ptl-u-p-Au-estimate-intermediate-2} is acceptable.

To bound \eqref{eq:ptl-u-p-Au-estimate-intermediate-1}, we apply the dual endpoint time-weighted estimate \eqref{eq:dual-endpoint-time-weighted-estimate} with $\alpha = \frac{1}{2}$ and $K \equiv 1$, and use that $\|p_{{\Blin}_u}(v; z)\|_{L^\infty_x} \lesssim (t-t_v)^{-1}$, to obtain
\begin{equs}
\eqref{eq:ptl-u-p-Au-estimate-intermediate-1} 
&\lesssim \bigg\|(t - t_v)^{-\frac{1}{4}}p_{{\Blin}_u}(v; z) ({\Blin}(v) - {\Blin}(z))\bigg\|_{L_v^2(J_\rho)} \\
&\lesssim \bigg(\int_0^t dt_v \ind(|t-t_v| \lesssim \rho) (t-t_v)^{-\frac{3}{2}} \int dx_v p(v; z) |{\Blin}(v) - {\Blin}(z)|^2 \bigg)^{\frac{1}{2}} .
\end{equs}
We may bound
\begin{equs}
\int dx_v p(v; z) |B(v) - B(z)|^2 &\lesssim \int dx_v p(v; z) |B(v) - B(t, x_v)|^2 + \int dx_v p(v; z) |B(t, x_v) - B(z)|^2 \\
&\lesssim \bigg(\int_{t_v}^t d\tau \|\ptl_\tau B\|_{L_x^\infty}\bigg)^2 + (t - t_v) \|\nabla B(t)\|_{L_x^\infty}^2.
\end{equs}
Inserting this bound, we thus obtain
\begin{equs}
\eqref{eq:ptl-u-p-Au-estimate-intermediate-1}  \lesssim I_1(t) + I_2(t) := ~~&\bigg(\int_0^t dt_v \ind(|t-t_v| \lesssim \rho) (t-t_v)^{-\frac{3}{2}} \bigg(\int_{t_v}^t d\tau \|\ptl_\tau {\Blin}\|_{L_x^\infty}\bigg)^2\bigg)^{\frac{1}{2}} \\
&+ \bigg(\int_0^t dt_v \ind(|t-t_v| \lesssim \rho) (t-t_v)^{-\frac{1}{2}}\bigg)^{\frac{1}{2}} \|\nabla {\Blin}(t)\|_{L_x^\infty}.
\end{equs}
Observe that $
|I_2(t)| \lesssim \rho^{\frac{1}{4}} \|\nabla {\Blin}(t)\|_{L_x^\infty}$. Combining our bounds for \eqref{eq:ptl-u-p-Au-estimate-intermediate-1} and \ref{eq:ptl-u-p-Au-estimate-intermediate-2}, we obtain the first estimate. For the second, we bound
\begin{equs}
\|I_1\|_{L_t^2([0, t_0])}^2 &\lesssim \|\ptl_t {\Blin}\|_{L_t^1 L_x^\infty([0, t_0] \times \T^2)} \int_0^{t_0} dt \int dt_v \ind(|t-t_v| \lesssim \rho) (t-t_v)^{-\frac{3}{2}} \int_{t_v}^t d\tau \|\ptl_\tau {\Blin}\|_{L_x^\infty} \\
&\lesssim  \|\ptl_t {\Blin}\|_{L_t^1 L_x^\infty([0, t_0] \times \T^2)} \int_0^{t_0} d\tau \|\ptl_\tau {\Blin}\|_{L_x^\infty} \int dt \int dt_v \ind(t_v < \tau < t) \ind(|t-t_v| \lesssim \rho) (t-t_v)^{-\frac{3}{2}} \\
&\lesssim \rho^{\frac{1}{2}} \|\ptl_t {\Blin}\|_{L_t^1 L_x^\infty([0, t_0] \times \T^2)}^2.
\end{equs}
The desired result now follows.
\end{proof}

\begin{lemma}\label{lemma:perturbation-argument-ibp-error-term}
For all $u \in (0, 1)$, $t_0, \rho \in (0, 1]$, and $N \in \dyadic$, we have that
\begin{equs}
\bigg\| &\int_0^t ds \ind(|t-s| \sim \rho) \int dy_1 dy_2 |\moll_{\leq N}^{*(2)}(y_1 - y_2)| |{\Blin}(w_2) - {\Blin}(w_1)|  |\massp_{{\Blin}_u}(w_1; z) \ptl_u \massp_{{\Blin}_u}(w_2; z)|\bigg\|_{L_t^2 L_x^\infty([0, t_0] \times\T^2)}  \\
&\lesssim \frac{\rho^{\frac{1}{2}} }{N}\|\nabla {\Blin}\|_{L_t^2 L_x^\infty([0, t_0] \times \T^2)} \Big(\|\ptl_t {\Blin}\|_{L_t^1 L_x^\infty([0, t_0] \times \T^2)} + \|\nabla {\Blin}\|_{L_t^2 L_x^\infty([0, t_0] \times \T^2)} + \|\ptl_j \Blin^j \|_{L_t^\infty L_x^\infty([0, t_0] \times \T^2)}\Big).
\end{equs}
\end{lemma}
\begin{proof}
Let $t \in [0, t_0]$ and $z = (t, x)$. Let $J_\rho := \{w \in [0, t] \times \T^2 : |t-s| \sim \rho\}$. We have the pointwise bound in $w_2$
\begin{equs}
f(w_2) := & \int dy_1 |\moll_{\leq N}^{*(2)}(y_1 - y_2)| |{\Blin}_u(w_2) - {\Blin}_u(w_1)| |\massp_{{\Blin}_u}(w_1; z)| \\
&\lesssim \|\nabla {\Blin}(s)\|_{L_x^\infty} \int dy_1 |\moll_{\leq N}^{*(2)} (y_1 - y_2)| |y_1 - y_2| p(w_1; z) \lesssim \frac{1}{N}(t-s)^{-1} \|\nabla {\Blin}(s)\|_{L_x^\infty}.
\end{equs}
Thus by Cauchy-Schwarz and Lemma \ref{lemma:ptl-u-p-Au-estimate}, we obtain the pointwise bound in $z$:
\begin{equs}
\bigg| &\int_0^t ds \ind(|t-s| \sim \rho) \int dy_1 dy_2 |\moll_{\leq N}^{*(2)}(y_1 - y_2)| |{\Blin}(w_2) - {\Blin}(w_1)|  |\massp_{{\Blin}_u}(w_1; z) \ptl_u \massp_{{\Blin}_u}(w_2; z)|\bigg| \\
&\lesssim \bigg\|(t-s)^{\frac{3}{4}} f(w_2)\bigg\|_{L_{w_2}^2(J_\rho)} \bigg\|(t-s)^{-\frac{3}{4}} \ptl_u \massp_{{\Blin}_u}(w_2; z)\bigg\|_{L_{w_2}^2(J_\rho)} \\
&\lesssim \Big(\rho^{\frac{1}{4}} N^{-1} \|\nabla {\Blin}\|_{L_t^2 L_x^\infty([0, t] \times \T^2)} \Big) \Big( I_1(t) + \rho^{\frac{1}{4}} \big(\|\nabla {\Blin}(t)\|_{L_x^\infty} + \|\ptl_j B^j\|_{L_t^\infty L_x^\infty([0, t] \times \T^2)}\big)\Big),
\end{equs}
where $I_1(t)$ is such that $\|I_1\|_{L_t^2} \lesssim \rho^{\frac{1}{4}} \|\ptl_t {\Blin}\|_{L_t^1 L_x^\infty}$. Taking $L_t^2 L_x^\infty$, the desired result now follows.
\end{proof}

\begin{lemma}\label{lemma:pA-FA-covd-pA-localized-estimate}
For all $t_0, \rho \in (0, 1]$, $N \in \dyadic$, we have that
\begin{equs}
\bigg\| \int_0^t ds \ind(|t - s| \sim \rho) \int dy_1 dy_2 \moll_{\leq N}^{*(2)}(y_1 - y_2) \ovl{\massp_{\Blin}(w_1; z)} &\int dv \massp_{\Blin}(v; z) F_{\Blin}(v) \covd_{{\Blin}(v)} \massp_{\Blin}(w_2; v)\bigg\|_{L_t^2 L_x^\infty([0, t_0] \times \T^2)}  \\
&\lesssim \rho^{\frac{1}{2}} \|F_{\Blin}\|_{L_t^2 L_x^\infty([0, t_0] \times \T^2)}.
\end{equs}
\end{lemma}
\begin{proof}
Fix $t \in [0, t_0]$ and $z = (t, x)$. Define $J_\rho = \{w \in [0, t] \times \T^2 : |t-s| \lesssim \rho\}$. For $\alpha = \frac{1}{2}$ (say), we bound by Cauchy-Schwarz and the dual endpoint time-weighted estimate (Corollary \ref{cor:dual-endpoint-time-weighted-estimate})
\begin{equs}
\bigg| \int_0^t &ds\ind(|t - s| \sim \rho) \int dy_1 dy_2 \moll_{\leq N}^{*(2)}(y_1 - y_2) \ovl{\massp_{\Blin}(w_1; z)} \int dv \massp_{\Blin}(v; z) F_{\Blin}(v) \covd_{{\Blin}(v)} \massp_{\Blin}(w_2; v)\bigg| \\
&\lesssim \bigg\| p_{\Blin}(w; z) (t-s)^{\frac{\alpha}{2}} \bigg\|_{L_w^2(J_\rho)} \bigg\| (t-s)^{-\frac{\alpha}{2}} \int dv p_{\Blin}(v; z) F_{\Blin}(v) \covd_{{\Blin}(v)} p_{\Blin}(w; v) \bigg\|_{L_w^2(J_\rho)} \\
&\lesssim \rho^{\frac{\alpha}{2}} \bigg\|(t-t_v)^{-\frac{\alpha}{2}} p(v; z)^{\frac{1}{2}} F_{\Blin}(v)\bigg\|_{L_v^2(J_\rho)} \\
&\lesssim \rho^{\frac{\alpha}{2}} \bigg(\int_0^t dt_v \ind(|t-t_v| \lesssim \rho) (t-t_v)^{-\alpha} \|F_{\Blin}(t_v)\|_{L_x^\infty}^2 \bigg)^{\frac{1}{2}}.
\end{equs}
Thus taking $L_t^2 L_x^\infty$, we obtain the bound
\begin{equs}
\rho^{\frac{\alpha}{2}} \bigg( \int_0^{t_0} dt \int_0^t dt_v\ind(|t-t_v| \lesssim \rho) (t-t_v)^{-\alpha} \|F_{\Blin}(t_v)\|_{L_x^\infty}^2 \bigg)^{\frac{1}{2}} \lesssim \rho^{\frac{\alpha}{2}} \rho^{\frac{1-\alpha}{2}} \|F_{\Blin}\|_{L_t^2 L_x^\infty([0, t_0] \times \T^2)},
\end{equs}
as desired.
\end{proof}

\begin{lemma}\label{lemma:p-H-p-contribution}
For all $t_0, \rho \in (0, 1]$, $N \in \dyadic$, we have that
\begin{equs}
\bigg\| \int_0^t ds \ind(|t-s| \sim \rho) \int dy_1 dy_2 \moll_{\leq N}^{*(2)} (y_1 - y_2) \Im\bigg(\icomplex \ovl{\massp_{\Blin}(w_1; z)} &\int dv \massp_{\Blin}(v; z) H(v) \massp_{\Blin}(w_2; v)\bigg)\bigg\|_{L_t^2 L_x^\infty([0, t_0] \times \T^2)} \\
&\lesssim \rho^{\frac{1}{2}}\|H\|_{L_t^1 L_x^\infty([0, t_0] \times \T^2)} .
\end{equs}
\end{lemma}
\begin{proof}
Let $t \in [0, t_0]$ and let $J_\rho := \{w \in [0, t] \times \T^2 : |t-s| \sim \rho\}$. For fixed $z = (t, x)$ we may obtain the pointwise bound (applying the semigroup property of the heat kernel in the second inequality)
\begin{equs}
\bigg| &\int_0^t ds \ind(|t-s| \sim \rho) \int dy_1 dy_2 \moll_{\leq N}^{*(2)} (y_1 - y_2) \Im\bigg(\icomplex \ovl{\massp_{\Blin}(w_1; z)} \int dv \massp_{\Blin}(v; z) H(v) \massp_{\Blin}(w_2; v)\bigg)\bigg| \\
&\leq  \int dv \ind(|t-t_v| \lesssim \rho) p(v; z) |H(v)| \int_0^t ds \ind(|t-s| \sim \rho) \int dy_1 dy_2 |\moll_{\leq N}^{*(2)}(y_1 - y_2)| p(w_1; z) p(w_2; v) \\
&\lesssim \int dv \ind(|t-t_v| \lesssim \rho) p(v; z) |H(v)| \int_0^t ds \ind(|t-s| \sim \rho) p_{t + t_v - 2s}(x, x_v) \\
&\lesssim  \int dv \ind(|t-t_v| \lesssim \rho) p(v; z) |H(v)| 
\leq \int dt_v \ind(|t-t_v| \lesssim \rho) \|H(t_v)\|_{L_x^\infty}.
\end{equs}
We may thus bound the $L_t^1 L_x^\infty$ norm by
\begin{equs}
\int_0^{t_0} dt \int dt_v \ind(|t-t_v| \lesssim \rho) \|H(t_v)\|_{L_x^\infty} \lesssim \rho \|H\|_{L_t^1 L_x^\infty([0, t_0] \times \T^2)}.
\end{equs}
We may also bound the $L_t^\infty L_x^\infty$ norm by $\|H\|_{L_t^1 L_x^\infty([0, t_0] \times \T^2)}$. We finish by interpolation.
\end{proof}

\begin{lemma}\label{lemma:ptl-u-p-Au--D-Au-estimate}
For all $u \in (0, 1)$, $t_0, \rho \in (0, 1]$, $N \in \dyadic$, we have that 
\begin{equs}
\bigg\| \int_0^t ds \int dy_1 dy_2 \moll_{\leq N}^{*(2)} &(y_1 - y_2)\ind(|t-s| \sim \rho) \ovl{\ptl_u \massp_{{\Blin}_u}(w_1; z)} \covd_{-{\Blin}_u(w_2)} \massp_{{\Blin}_u}(w_2;z)\bigg\|_{L_t^2 L_x^\infty([0, t_0] \times \T^2)} \\
&\lesssim \rho^{\frac{1}{2}} \big(\|\ptl_t {\Blin}\|_{L_t^1 L_x^\infty([0, t_0] \times \T^2)} + \|\nabla {\Blin}\|_{L_t^2 L_x^\infty([0, t_0] \times \T^2)} + \|\ptl_j \Blin^j\|_{L_t^\infty L_x^\infty([0, t_0] \times \T^2)} \big).
\end{equs}
\end{lemma}
\begin{proof}
Let $t \in [0, t_0]$ and $z = (t, x)$. Let $J := [0, t] \times \T^2$ and $J_\rho := \{w \in J : |t-s| \lesssim \rho\}$. Take $\alpha = \frac{1}{2}$. By Cauchy-Schwarz, the $L^2$ covariant derivative estimate Lemma \ref{lemma:covd-pa-l2-spacetime-norm}, and Lemma \ref{lemma:ptl-u-p-Au-estimate},
we have the pointwise bound in $z$:
\begin{equs}
\bigg| \int_0^t ds \int dy_1 dy_2 &\moll_{\leq N}^{*(2)}(y_1 - y_2)\ind(|t-s| \sim \rho) \ovl{\ptl_u \massp_{{\Blin}_u}(w_1; z)} \covd_{-{\Blin}_u(w_2)} \massp_{{\Blin}_u}(w_2;z)\bigg|   \\
&\leq \bigg\|(t-s)^{-\frac{\alpha+1}{2}} \ptl_u p_{{\Blin}_u}(w; z)\bigg\|_{L_w^2(J_\rho)} \bigg\| (t-s)^{\frac{\alpha+1}{2}} \covd_{-{\Blin}_u(w)} p_{{\Blin}_u}(w; v)\bigg\|_{L_w^2(J_\rho)} \\
&\lesssim \rho^{\frac{1}{4}} \bigg\|(t-s)^{-\frac{\alpha+1}{2}} \ptl_u p_{{\Blin}_u}(w; z)\bigg\|_{L_w^2(J_\rho)} \lesssim \rho^{\frac{1}{4}}(I_1(t) + \rho^{\frac{1}{4}} \|\nabla {\Blin}(t)\|_{L_x^\infty} + \rho^{\frac{1}{4}} \|\ptl_j \Blin^j\|_{L_t^\infty L_x^\infty}),
\end{equs}
where $\|I_1\|_{L_t^2} \lesssim \rho^{\frac{1}{4}} \|\ptl_t {\Blin}\|_{L_t^1 L_x^\infty}$. Taking $L_t^2 L_x^\infty$, the desired result now follows.
\end{proof}

\begin{lemma}\label{lemma:p-squared-difference-A-estimate}
For all $t_0, \rho \in (0, 1]$, $N \in \dyadic$, we have that
\begin{equs}
\bigg\| \int_0^t ds \ind(|t-s| \sim \rho) \int dy_1 dy_2 |\moll_{\leq N}^{*(2)}(y_1 - y_2)| &p(w_1; z) p(w_2; z) |{\Blin}(w_2) - {\Blin}(z)| \bigg\|_{L_t^2 L_x^\infty([0, t_0] \times \T^2)}  \\
&\lesssim\rho^{\frac{1}{2}} \big(\|\ptl_t {\Blin}\|_{L_t^1 L_x^\infty([0, t_0] \times \T^2)} + \|\nabla {\Blin}\|_{L_t^2 L_x^\infty([0, t_0] \times \T^2)}\big) .
\end{equs}
\end{lemma}
\begin{proof}
Let $t \in [0, t_0]$ and fix $z = (t, x)$. We may bound (using that $\|p(w_1; z)\|_{L_{y_1}^\infty} \lesssim (t-s)^{-1}$ and then integrating out the $y_1$ variable)
\begin{equs}
\int_0^t ds &\ind(|t-s| \sim \rho) \int dy_1 dy_2 |\moll_{\leq N}^{*(2)}(y_1 - y_2)| p(w_1; z) p(w_2; z) |{\Blin}(w_2) - {\Blin}(z)| \\
&\lesssim \int_0^t ds \ind(|t-s| \sim \rho) (t-s)^{-1} \int dy_2 p(w_2; z) |B(w_2) - B(z)|.
\end{equs}
For fixed $s$, we may bound
\begin{equs}
\int dy_2 p(w_2; z) |B(w_2) - B(z)| &\leq \int dy_2 p(w_2; z) |B(s, y_2) - B(t, y_2)| + \int dy_2 p(w_2; z) |B(t, y_2) - B(z)| \\
&\lesssim \int_s^t d\tau \|\ptl_\tau B\|_{L_x^\infty} + (t - s)^{\frac{1}{2}} \|\nabla B(t)\|_{L_x^\infty}.
\end{equs}
Inserting this bound, we obtain
\begin{equs}
\int_0^t ds &\ind(|t-s| \sim \rho) \int dy_1 dy_2 |\moll_{\leq N}^{*(2)}(y_1 - y_2)| p(w_1; z) p(w_2; z) |{\Blin}(w_2) - {\Blin}(z)| \\
&\lesssim \rho^{\frac{1}{2}} \|\nabla B(t)\|_{L_x^\infty} + \rho^{-1} \int_0^t ds \ind(|t-s| \sim \rho) \int_s^t d\tau \|\ptl_\tau B\|_{L_x^\infty}.
\end{equs}
To finish, observe that the $L_t^2 L_x^\infty$ norm of the first term is clearly acceptable. The bound for the second term follows by interpolating the following two bounds:
\begin{align*}
\bigg\|\rho^{-1} \int_0^t ds \ind(|t-s| \sim \rho) \int_s^t d\tau \|\ptl_\tau \Blin\|_{L_x^\infty}\bigg\|_{L_t^1 L_x^\infty([0, t_0 ] \times \T^2)} &\lesssim \rho \|\ptl_t B\|_{L_t^1 L_x^\infty([0, t_0] \times \T^2)} \\
\bigg\|\rho^{-1} \int_0^t ds \ind(|t-s| \sim \rho) \int_s^t d\tau \|\ptl_\tau \Blin\|_{L_x^\infty}\bigg\|_{L_t^\infty L_x^\infty([0, t_0 ] \times \T^2)} &\lesssim  \|\ptl_t B\|_{L_t^1 L_x^\infty([0, t_0] \times \T^2)}. \qedhere
\end{align*}
\end{proof}

\section{Covariant estimates for the Abelian-Higgs equations}\label{section:Abelian-Higgs}

In this section, we obtain covariant estimates for the stochastic Abelian-Higgs equations from \eqref{intro:eq-SAH}. In Subsection \ref{section:AH-Ansatz} and \ref{section:AH-hypothesis}, we state our Ansatz for the connection one-form $A$ and scalar field $\phi$. Furthermore, we state probabilistic and continuity hypotheses (Hypothesis \ref{AH:hypothesis-probabilistic} and \ref{AH:hypothesis-continuity}), which will be used throughout this section. In Subsection \ref{section:AH-psi}, we obtain covariant estimates of the scalar field $\phi$. The main ingredient will be the covariant monotonicity formula for the covariant heat equation (Proposition \ref{prop:monotonicity}). In Subsection \ref{section:AH-Z}, we then control the connection one-form $A$. The main ingredients will be our earlier covariant, stochastic estimate of the derivative-nonlinearity (Theorem \ref{intro:thm-cshe}, or more precisely, Propositions \ref{prop:derivative-nonlinearity-non-resonant} and \ref{cshe:prop-resonant}) and the covariant estimates from Subsection \ref{section:AH-psi}.

\subsection{Ansatz}\label{section:AH-Ansatz} 

We consider the stochastic Abelian-Higgs equations given by 
\begin{align}
\big(\partial_t - \Delta\big) A &= - \leray \Im \Big( \overline{\phi} \,  \covd_A \phi \Big) + \gaugerenorm A + \leray \xi, 
\label{AH:eq-evolution-A} \\ 
\big(\partial_t - \biglcol \,  \covd_{A}^j \covd_{A,j} \, \bigrcol  \big) \phi &=
-\biglcol \, |\phi|^{q-1} \phi \,  \bigrcol + \zeta,
\label{AH:eq-evolution-phi}
\end{align}
which have been shown to be locally well-posed in Proposition \ref{prop:lwp-para-sah}. That is, we directly work with the limiting dynamics of the smooth approximations from \eqref{intro:eq-SAH-smooth}. In \eqref{AH:eq-evolution-phi}, $\biglcol \, |\phi|^{q-1} \phi \,  \bigrcol$ is the Wick-ordered nonlinearity and 
\begin{equation*}
\biglcol \,  \covd_{A}^j \covd_{A,j} \, \bigrcol = \partial_j \partial^j + 2 \icomplex A^j \partial_j - \hspace{-0.2ex} \biglcol\,  | A |^2 \, \bigrcol 
\end{equation*}
is the Wick-ordered, covariant Laplacian. 

\begin{remark}\label{remark:limiting-A-nonlinearity}
Recall from Remark \ref{remark:local-theory-remarks}\ref{item:A-equation-no-diverging-counterterm} that the $A$-equation in fact has no diverging counterterms, and thus we somewhat informally write $\ovl{\phi} \covd_A\phi$, which should be understood as the $N \toinf$ limit of the frequency-truncated analog $\ovl{\phi_{\leq N}} \covd_{A_{\leq N}} \phi_{\leq N}$. Since the bulk of our analysis in Sections \ref{section:Abelian-Higgs} and \ref{section:decay} will be on a low-frequency component of $\phi$, this consideration will only really enter in Appendix \ref{section:high}.
\end{remark}

We let $t_0\geq 0$ be an initial time, let $\nregA_0 \in \Cs_x^{-\kappa}(\T^2 \rightarrow \R^2)$ and $\pregA_0 \in \Cs_x^\eta(\T^2 \rightarrow \R^2)$ be initial connection one-forms satisfying the Coulomb gauge condition and let 
$\nregphi_0\in \GCs^{-\kappa}(\T^2\rightarrow \C)$ and $\pregphi_0 \in L_x^r(\T^2 \rightarrow \C)$ be initial scalar fields. In the initial data, the superscripts $\nreg$ and $\preg$ represent terms with negative and positive (or at least non-negative) regularities, respectively. Then, we impose the initial conditions 
\begin{equation}\label{AH:eq-initial}
A(t_0) = \nregA_0 + \pregA_0 \qquad \text{and} \qquad \phi(t_0) = \nregphi_0 + \pregphi_0. 
\end{equation}

\begin{remark}[\protect{On $(\nregA_0,\nregphi_0)$ versus $(\pregA_0,\pregphi_0)$}]\label{AH:rem-initial}
The initial data $(\nregA_0,\nregphi_0)$ in \eqref{AH:eq-initial} will be chosen as (gauge transformed) linear stochastic objects, see e.g. the proofs of Lemma \ref{decay:lem-exponential-growth} and Lemma \ref{decay:lem-decay-unit}. While $\nregA_0$ and $\nregphi_0$ therefore will have negative regularity, both $\nregA_0$ and $\nregphi_0$ will always be, at least morally, of size $\sim 1$. Due to this, the impact of $\nregA_0$ and $\nregphi_0$ on the dynamics will be rather insignificant. In contrast to $\nregA_0$ and $\nregphi_0$, $\pregA_0$ and $\pregphi_0$ will have positive (or at least non-negative) regularities. However, the size of $\pregA_0$ and $\pregphi_0$ in our argument can be large, and the impact of $\pregA_0$ and $\pregphi_0$ on the dynamics can therefore be substantial. Overall, $\pregA_0$ and $\pregphi_0$ will play a much more important role in this article than $\nregA_0$ and $\nregphi_0$. 
\end{remark}

We let $L \in \dyadic$ be a parameter, which remains to be chosen. Throughout this section, we assume that $L$ satisfies the condition
\begin{equation}\label{AH:eq-L}
L \geq C_0 \max \bigg(
1,\| \pregA_0 \|_{\Cs_x^\eta}^{\frac{2}{\eta_3}}, \| \pregphi_0 \|_{L_x^r}^{\frac{2}{\eta_3}} \bigg),
\end{equation}
where $C_0$ is as in Subsection \ref{section:parameters}.
The factors of $2$ in the exponents in \eqref{AH:eq-L} are not essential, but the reader should keep in mind that the exponents are of size $\sim \eta_3^{-1}$.
Furthermore, we assume that 
\begin{equation}\label{AH:eq-L-nreg}
\| \nregA_0 \|_{\Cs_x^{-\kappa}} \leq L^{\kappa} \qquad \text{and} \qquad 
\| \nregphi_0 \|_{\GCs^{-\kappa}} \leq L^{\kappa}.
\end{equation}

\begin{remark}[Size of $L$] 
As discussed in Remark \ref{AH:rem-initial}, the size of $\|\pregA_0\|_{\Cs_x^\eta}$ and $\|\pregphi_0\|_{L_x^r}$ can be arbitrarily large, but the size of $\| \nregA_0 \|_{\Cs_x^{-\kappa}}$ and 
$\| \nregphi_0 \|_{\GCs^{-\kappa}}$ is essentially $\sim 1$. As a result, we will be able to choose $L$ such that $L^{\kappa}$ and $L^{\kappa_j}$, where $1\leq j \leq 4$, will essentially be of size $\sim 1$, and yet \eqref{AH:eq-L} is still satisfied. 
\end{remark}

In the following, we use $L$ as a threshold between low and high-frequencies. Due to \eqref{AH:eq-L}, the contributions of all high-frequency terms to \eqref{AH:eq-evolution-A}-\eqref{AH:eq-evolution-phi} will be negligible, and we encourage the reader to focus entirely on the low-frequency terms. We now derive our Ansatz for the connection one-form $A$ and the scalar field $\phi$.

\subsubsection{\protect{Ansatz for $A$}}
First, we define the linear stochastic object $\initiallinear[t_0]\colon [t_0,\infty) \times \T^2 \rightarrow \R^2$ as the solution of 
\begin{equation}\label{AH:eq-A-linear}
\big( \partial_t - \Delta +1 \big) \initiallinear[t_0] = \leray \xi, \qquad \initiallinear[t_0] (t_0)=0. 
\end{equation}
The subscript $t_0$ in $\initiallinear[t_0]$ indicates that the zero initial condition is imposed at the initial time $t=t_0$. In the current section, the initial time is always given by $t_0$, and we therefore simply write 
\begin{equation}\label{AH:eq-linear-simplified}
\linear[] := \initiallinear[t_0]. 
\end{equation}
In Section \ref{section:decay-unit}, however, we iterate in time, and then the dependence of $\initiallinear[t_0]$ on the initial time $t_0$ will be important. Second, we define $\Blin\colon[t_0,\infty) \times \T^2\rightarrow \R^2$ and $\Slin\colon [t_0,\infty) \times \T^2 \rightarrow \R^2$ as the linear heat flow of $\pregA_0$ and $\nregA_0$, respectively. To be more precise, we define $\Blin$ and $\Slin$ as the solutions of
\begin{alignat}{3}
\big( \partial_t - \Delta \big) \Blin &=0, \qquad & \qquad \Blin(t_0)=\pregA_0, \label{AH:eq-Blin} \\ 
\big( \partial_t - \Delta \big) \Slin &=0, \qquad & \qquad \Slin(t_0)=\nregA_0. \label{AH:eq-Slin}
\end{alignat}
We note that, due to the size constraint in \eqref{AH:eq-L-nreg}, the impact of $\Slin$ on the dynamics is rather insignificant compared to the impact of $\Blin$. On first reading, we encourage the reader to ignore all terms involving $\Slin$. Equipped with $\linear[]$, $\Blin$, and $\Slin$, our Ansatz for $A$ then takes the form 
\begin{equs}\label{AH:eq-decomposition-A}
A = \linear[] + \Blin + \Slin + Z,
\end{equs}
where $Z$ is a nonlinear remainder. Due to \eqref{AH:eq-evolution-A}, \eqref{AH:eq-A-linear}, \eqref{AH:eq-Blin}, and \eqref{AH:eq-Slin}, the nonlinear remainder $Z$ then satisfies the nonlinear heat equation 
\begin{equs}\label{AH:eq-Z}
(\partial_t - \Delta) Z = - \leray \Im \Big( \overline{\phi}\, \covd_A \phi \Big) + \gaugerenorm A + \linear[], \qquad Z(t_0)=0. 
\end{equs}
For later use, we also define
\begin{alignat}{3}
\linear[\lo] &:= P_{\leq L} \linear[], \qquad & \qquad \linear[\hi] &:= \linear[] - \linear[\lo], \label{AH:eq-A-convenient-1} \\ 
\Slin_{\lo} &:= P_{\leq L} \Slin, \qquad    \qquad & \qquad         \Slin_{\hi} &:= \Slin - \Slin_{\lo}, 
\label{AH:eq-A-convenient-2} \\ 
A_{\lo} &:= \linear[\lo] + \Blin  + \Slin_{\lo} + Z. \label{AH:eq-A-convenient-3} 
\end{alignat}

\subsubsection{\protect{Ansatz for $\phi$}} 
We first define a covariant linear stochastic object, which is defined using a covariant stochastic heat equation. As mentioned in Remark \ref{intro:rem-cshe}.\ref{intro:item-choice-B}, the covariant stochastic object is not defined using the connection one-form $A$, but rather\footnote{It may also be possible to work with the covariant stochastic heat equation based on $\scalebox{0.8}{$\linear[\lo]$}+\Blin$ or 
$\scalebox{0.8}{$\linear[\lo]$}+\Slin_{\lo}+\Blin$. However, since the interactions between $\scalebox{0.8}{$\linear[\lo]$}$ and $\scalebox{0.8}{$\philinear[\Blin,\lo]$}$  
as well as between $\Slin_{\lo}$ and $\scalebox{0.8}{$\philinear[\Blin,\lo]$}$ can easily be handled perturbatively, we chose to work with the covariant stochastic heat equation induced by $\Blin$.}
using the connection one-form $\Blin$. The reason is that $A$ and $\zeta\big|_{[t_0,\infty)}$ are probabilistically dependent, whereas $\Blin$ and $\zeta\big|_{[t_0,\infty)}$ are probabilistically independent.
To be more precise, we define the covariant, low-frequency linear stochastic object $\philinear[\Blin,t_0,\lo]$ as the solution of 
\begin{equs}\label{AH:eq-philinear-lo}
\big(\partial_t - \covd_{\Blin}^j \covd_{\Blin,j} +1  \big) \, \initialphilinear[\Blin,t_0,\lo] = P_{\leq L} \zeta,  \qquad 
\initialphilinear[\Blin,t_0,\lo](t_0) =0.
\end{equs}
We also introduce the non-covariant linear stochastic object $\philinear[t_0]$ and non-covariant, high-frequency linear stochastic $\initialphilinear[t_0,\hi]$ as the solutions of 
\begin{alignat}{3}
\big( \partial_t - \partial^j \partial_j +1  \big) \, \initialphilinear[t_0] &=  \zeta,  \qquad & \qquad 
\philinear[t_0](t_0) &=0,  \label{AH:eq-philinear} \\ 
\big( \partial_t - \partial^j \partial_j +1  \big) \, \initialphilinear[t_0,\hi] &=  P_{>L}\zeta,  \qquad & \qquad 
\philinear[t_0,\hi](t_0) &=0. \label{AH:eq-philinear-hi}
\end{alignat}
Similar as in \eqref{AH:eq-linear-simplified}, we also write
\begin{equation}\label{AH:eq-philinear-simplified}
\philinear[\Blin,\lo] := \initialphilinear[\Blin,t_0,\lo],  \qquad \philinear[\hi] := \initialphilinear[t_0,\hi], \qquad 
\text{and} \qquad \philinear[] := \initialphilinear[t_0].
\end{equation}
Due to our choice of $L$ in \eqref{AH:eq-L}, all interactions involving $\philinear[\hi]$ can be handled perturbatively. On first reading, we therefore encourage the reader to ignore all contributions of $\philinear[\hi]$. Equipped with $\philinear[\Blin,\lo]$ and $\philinear[\hi]$, we then write the solution $\phi$ as 
\begin{equation}\label{AH:eq-decomposition-phi-first}
\phi = \philinear[\Blin,\lo] + \philinear[\hi] + \psi .
\end{equation}
From the evolution equations \eqref{AH:eq-evolution-phi}, \eqref{AH:eq-philinear-lo}, and \eqref{AH:eq-philinear-hi}, it then follows that the nonlinear remainder $\psi$ solves
\begin{align}
\big( \partial_t - \hspace{-0.2ex} \biglcol \,  \covd_{A}^j \covd_{A,j} \, \bigrcol \hspace{-0.2ex} + 1\big) \psi 
&= \big(\hspace{-0.2ex} \biglcol \,  \covd_{A}^j \covd_{A,j} \, \bigrcol \hspace{-0.2ex} - \covd_{\Blin}^j \covd_{\Blin,j}  \big) \,  \philinear[\Blin,\lo] 
\label{AH:eq-psi-e1}\\
&+  \big( \hspace{-0.2ex} \biglcol \,  \covd_{A}^j \covd_{A,j} \, \bigrcol \hspace{-0.2ex} - \partial^j \partial_j  \big) \, \philinear[\hi]  \label{AH:eq-psi-e2}\\
&+ \philinear[\Blin,\lo] + \philinear[\hi] + \psi \label{AH:eq-psi-e2p} \\
&- \biglcol \, \big| \philinear[\Blin,\lo] + \philinear[\hi] + \psi \big|^{q-1} \big( \philinear[\Blin,\lo] + \philinear[\hi] + \psi\big) \,  \bigrcol \label{AH:eq-psi-e3} 
\end{align}
and has the initial data
\begin{equation}\label{AH:eq-psi-initial}
\psi(t_0) = \nregphi_0 +\pregphi_0. 
\end{equation}
In order to control $\psi$, we would like to use the covariant monotonicity formula from Proposition \ref{prop:monotonicity}, which can only yield good estimates of $\psi$ if its covariant derivatives have finite $L_t^2 L_x^2$-norms. However, due to the negative regularity of the initial data $\nregphi_0$ in \eqref{AH:eq-psi-initial} and the negative regularity of $\philinear[\hi]$ in \eqref{AH:eq-psi-e2}, the $L_t^2 L_x^2$-norm of covariant derivatives of $\psi$ has to be infinite. To circumvent these problems, we decompose $\psi$ as
\begin{equs}\label{AH:eq-decomposition-psi}
\psi = \psi_{\lo} + \psi_{\hi} + \varphi_{\hi},
\end{equs}
which are defined as follows: First, we define $\varphi_{\hi}$ as the solution of 
\begin{equs}
\big(\partial_t - \partial^j \partial_j + 1  \big)  \, \varphi_{\hi} =0, \qquad \varphi_{\hi}(t_0) = P_{>L} \nregphi_0.  \label{eq:AH-eq-varphi-hi}
\end{equs} 
That is, we define $\varphi_\hi$ as the linear heat flow of the high-frequency component of $\nregphi_0$.  Next, we define $\psi_{\hi}$ as a solution of a non-covariant heat equation which absorbs contributions containing at least one factor of $\linear[\hi]$, $\Slin_{\hi}$, $\philinear[\hi]$, or $\varphi_{\hi}$. To be more precise, we define $\psi_{\hi}$ as the solution of evolution equation 
\begin{align}
& \big(\partial_t - \partial^j \partial_j + 1  \big)  \, \psi_{\hi} \notag \\
=&\, \big( \hspace{-0.2ex} \biglcol \,  \covd_{A}^j \covd_{A,j} \, \bigrcol \hspace{-0.2ex} - \covd_{A_{\lo}}^j \covd_{A_{\lo},j}  - 2 \sigma_{\leq L}^2\big) \big( \philinear[\Blin,\lo] + \psi_{\lo} + \psi_{\hi} \big) 
\label{AH:eq-psi-hi-e1}\\ 
+&\,\big( \hspace{-0.2ex} \biglcol \,  \covd_{A}^j \covd_{A,j} \, \bigrcol \hspace{-0.2ex} - \partial^j \partial_j  +1  \big) \, \big( \philinear[\hi]  + \varphi_{\hi} \big) \label{AH:eq-psi-hi-e2} \\
-&\, \Big(  \biglcol \, \big| \philinear[\Blin,\lo] + \philinear[\hi] + \psi \big|^{q-1} \big( \philinear[\Blin,\lo] + \philinear[\hi] + \psi \big) \,  \bigrcol 
- \biglcol \, \big| \philinear[\Blin,\lo] + \psi_{\lo} + \psi_{\hi} \big|^{q-1} \big( \philinear[\Blin,\lo]  + \psi_{\lo} + \psi_{\hi} \big) \,  \bigrcol \Big).  \label{AH:eq-psi-hi-e3}
\end{align}
with initial data 
\begin{equation*}
    \psi_{\hi}(t_0) = 0.
\end{equation*}
\begin{remark}\label{remark:wick-ordered-abuse-of-notation}
With a slight abuse of notation, we use expressions of the form $\biglcol \, |\varphi|^{q-1} \varphi\, \bigrcol $ for two different Wick-ordered nonlinearities. The term $\biglcol \, | \philinear[\Blin,\lo] + \philinear[\hi] + \psi |^{q-1} ( \philinear[\Blin,\lo] + \philinear[\hi] + \psi ) \,  \bigrcol$ refers to the limiting Wick-ordered nonlinearity, whereas 
$\biglcol \, | \philinear[\Blin,\lo]  +\psi_{\lo} + \psi_{\hi}|^{q-1} ( \philinear[\Blin,\lo]  + \psi_{\lo} + \psi_{\hi} ) \,  \bigrcol$ refers to the Wick-ordered nonlinearity corresponding to the frequency-truncation parameter $L$ (and hence variance $\sigma_{\leq L}^2$). In the following, all Wick-ordered nonlinearities involving $\philinear[\Blin,\lo] + \philinear[\hi]$ or $\philinear[\lo] + \philinear[\hi]$ will be limiting Wick-ordered nonlinearities, whereas all Wick-ordered nonlinearities involving only $\philinear[\Blin,\lo]$ or $\philinear[\lo]$ will be defined using the variance $\sigma_{\leq L}^2$.
\end{remark}
Similar as for $\philinear[\hi]$, our choice of $L$ from \eqref{AH:eq-L} allows us to treat all interactions involving $\varphi_{\hi}$ or $\psi_{\hi}$ perturbatively, and such interactions may therefore be ignored on first reading. The evolution equations for $\psi$ and $\psi_{\hi}$ from \eqref{AH:eq-psi-e1}-\eqref{AH:eq-psi-e3} and \eqref{AH:eq-psi-hi-e1}-\eqref{AH:eq-psi-hi-e2} then determine the evolution equation for $\psi_{\lo}$, which is given by 
\begin{align}
\big(\partial_t - \covd_{A_{\lo}}^j \covd_{A_{\lo},j} - 2 \sigma_{\leq L}^2 +1 \big)  \, \psi_{\lo} 
&= \big( \covd_{A_{\lo}}^j \covd_{A_{\lo},j} - \covd_{\Blin}^j \covd_{\Blin,j}  +2  \sigma_{\leq L}^2 \big) \, \philinear[\Blin,\lo] \label{AH:eq-psi-lo-original-e1} \\ 
&+ \big( \covd_{A_{\lo}}^j \covd_{A_{\lo},j} - \partial^j \partial_j  + 2 \sigma_{\leq L}^2 \big) \psi_{\hi} \label{AH:eq-psi-lo-original-e2} \\ 
&+ \philinear[\Blin,\lo]  + \psi_{\lo} +\psi_{\hi} \label{AH:eq-psi-lo-original-e2p} \\
&-  \biglcol \, \big| \philinear[\Blin,\lo]  + \psi_{\lo} + \psi_{\hi} \big|^{q-1} \big( \philinear[\Blin,\lo]  + \psi_{\lo} + \psi_{\hi} \big) \,  \bigrcol . \label{AH:eq-psi-lo-original-e3}
\end{align}
Furthermore, the initial condition for $\psi_{\lo}$ is given by 
\begin{equation}\label{AH:eq-psi-lo-initial}
\psi_{\lo}(t_0)= P_{\leq L} \nregphi_0 + \pregphi_0. 
\end{equation}

\begin{remark}
In the initial condition \eqref{AH:eq-psi-lo-initial}, we combined $P_{\leq L} \nregphi_0$ and $\pregphi_0$.
In contrast, in the initial conditions in \eqref{AH:eq-Blin} and \eqref{AH:eq-Slin}, we completely separated the terms $\nregA_0$ and $\pregA_0$. The reason for this is that nonlinear remainders in $\phi$ will be controlled in $L_x^r$, whereas nonlinear remainders in $A$ will be controlled in $\Beta$.
Due to the additional $\eta$-derivatives, including $P_{\leq L} \nregA_0$ in the initial condition for  $\Blin$ would lead to unacceptable losses.
\end{remark}

Due to our choice of $L$ from \eqref{AH:eq-L} and due to $\sigma_{\leq L}^2 \sim \log(L)$, the $\sigma_{\leq L}^2$-terms in 
\eqref{AH:eq-psi-lo-original-e1}-\eqref{AH:eq-psi-lo-original-e3} can be treated perturbatively. For this reason, we no longer view them as a renormalization of the $\covd_{A_{\lo}}^j \covd_{A_{\lo},j}$-operator, but rather as a perturbative term.\footnote{Using similar arguments, it is possible to treat the massive $+\psi_{\lo}$-term on the left-hand side of \eqref{AH:eq-psi-lo-original-e1} perturbatively. However, in order to keep our notation consistent with the definitions of the stochastic objects, we keep the massive $+\psi_{\lo}$-term on the left-hand side of \eqref{AH:eq-psi-lo-original-e1}.} We also group together all terms in \eqref{AH:eq-psi-lo-original-e2p} with the corresponding $\sigma_{\leq L}^2$-terms. Furthermore, we write 
\begin{equation}\label{AH:eq-power-decomposition}
\begin{aligned}
& \biglcol \, \big| \philinear[\Blin,\lo]  + \psi_{\lo} + \psi_{\hi} \big|^{q-1} \big( \philinear[\Blin,\lo]  + \psi_{\lo} + \psi_{\hi} \big) \,  \bigrcol \\ 
=&\,  |\psi_{\lo}|^{q-1} \psi_{\lo} + \Big( \biglcol \, \big| \philinear[\Blin,\lo]  + \psi_{\lo} + \psi_{\hi} \big|^{q-1} \big( \philinear[\Blin,\lo] + \psi_{\lo} + \psi_{\hi} \big) \,  \bigrcol  - |\psi_{\lo}|^{q-1} \psi_{\lo}\Big),
\end{aligned}
\end{equation}
since the first summand in \eqref{AH:eq-power-decomposition} will be treated non-perturbatively (see the proof of Proposition \ref{AH:prop-psi-lo-estimate}). Finally, we use the identity $\Blin=A_{\lo} - \big( \linear[\lo]+ \Slin_{\lo} + Z \big)$ and Lemma \ref{prelim:lem-derivatives}.\ref{prelim:item-difference} to write all derivatives in terms of $\covd_{A_{\lo}}$, which will be crucial for our covariant energy estimates. All in all, we reformulate \eqref{AH:eq-psi-lo-original-e1}-\eqref{AH:eq-psi-lo-original-e3} as 
\begin{align}
& \big(\partial_t - \covd_{A_{\lo}}^j \covd_{A_{\lo},j} +1  \big)  \, \psi_{\lo} + |\psi_{\lo}|^{q-1} \psi_{\lo}  \notag \\ 
=&\,  2 \icomplex \covd_{A_{\lo}}^j \Big( \big( \linear[\lo] + \Slin_{\lo}  + Z \big)_j \, \philinear[\Blin,\lo] + \big( \linear[\lo] + B + \Slin_{\lo}  + Z \big)_j \psi_{\hi} \Big) \label{AH:eq-psi-lo-e1} \\ 
+&\,  \big| \linear[\lo] + \Slin_{\lo}  + Z \big|^2 \, \philinear[\Blin,\lo] 
+ \big| \linear[\lo] + B + \Slin_{\lo}  +  Z \big|^2 \psi_{\hi} 
+ \big( 2 \sigma_{\leq L}^2 +1 \big) \big( \philinear[\Blin,\lo] + \psi_{\hi} \big) \label{AH:eq-psi-lo-e2} \\ 
+&\,   \big( 2 \sigma_{\leq L}^2 +1 \big) \psi_{\lo}  
- \Big( \biglcol \, \big| \philinear[\Blin,\lo] + \psi_{\lo} + \psi_{\hi} \big|^{q-1} \big( \philinear[\Blin,\lo]  + \psi_{\lo} + \psi_{\hi} \big) \,  \bigrcol 
- |\psi_{\lo}|^{q-1} \psi_{\lo}  \Big) . \label{AH:eq-psi-lo-e3}
\end{align}
We remark that the terms in \eqref{AH:eq-psi-lo-e1}-\eqref{AH:eq-psi-lo-e2} are grouped together differently than in \eqref{AH:eq-psi-lo-original-e1}-\eqref{AH:eq-psi-lo-original-e2}, since we grouped all terms with $\covd_{A_{\lo}}$-operators together.\\

In total, we have introduced the nine terms
\begin{equs}\label{AH:eq-term-collection}
\Big( \linear[], \Blin, \Slin, Z, \philinear[\Blin,\lo], \philinear[\hi],  \psi_{\lo},\psi_{\hi}, \varphi_{\hi} \Big). 
\end{equs}
Since the six terms $\linear[]$, $\Blin$, $\Slin$, $\philinear[\Blin,\lo]$, $\philinear[\hi]$, and $\varphi_{\hi}$ have already been estimated, it remains to control the three terms $Z$, $\psi_{\lo}$, and $\psi_{\hi}$, which is the main subject of this section (and Appendix \ref{section:high}).

\subsection{Continuity and probabilistic hypothesis}\label{section:AH-hypothesis}

In order to obtain our nonlinear estimates for $Z$, $\psi_{\lo}$, and $\psi_{\hi}$, we require several probabilistic estimates of our stochastic objects. For expository purposes, it is convenient to collect all required probabilistic estimates in a probabilistic hypothesis (Hypothesis \ref{AH:hypothesis-probabilistic}), which can easily be referenced.

\begin{hypothesis}[Probabilistic hypothesis]\label{AH:hypothesis-probabilistic}
Let $t_0\in [0,\infty)$, let $0<\tau\leq 1$, and let $L\in \dyadic$. Then, we make the following assumptions: 
\begin{enumerate}[label=(\roman*)]
\item\label{AH:item-probabilistic-A-linear} (Linear stochastic objects) It holds that 
\begin{equation*}
\max\Big(\big\|\,  \linear[\lo]\big\|_{L_t^\infty L_x^\infty([t_0,t_0+\tau])}, 
\big\|\,  \linear[]\big\|_{L_t^\infty \Cs_x^{-\kappa}([t_0,t_0+\tau])},
\big\|\,  \philinear[]\big\|_{L_t^\infty \GCs^{-\kappa}([t_0,t_0+\tau])} \Big) \leq L^{\kappa}. 
\end{equation*}
\item\label{AH:item-probabilistic-covariant-linear} (Covariant linear stochastic object $\philinear[\Blin,\lo]$) It holds that  $\big\|\,  \philinear[\Blin, \lo]\big\|_{L_t^\infty L_x^\infty([t_0,t_0+\tau])}\leq L^\kappa$.
\item\label{AH:item-probabilistic-derivative} (Derivative nonlinearity) It holds that 
\begin{equs}\label{AH:eq-probabilistic-derivative}
\Big\| \leray \Im \Duh \Big[ \, \overline{\philinear[\Blin,\lo]} \covd_{\Blin} \philinear[\Blin,\lo] - \tfrac{1}{8\pi} \Blin  \Big] \Big\|_{C_t^0 \Cs_x^{2\eta} \cap C_t^{\eta} \Cs_x^0([t_0, t_0 + \tau])}
\leq \tau^{\frac{3}{4}-\eta} \big\| \Blin(t_0) \big\|_{\Cs_x^\eta}^{\frac{1}{2}+\eta} + L^\kappa. 
\end{equs}
\item\label{AH:item-high} (High-frequency estimates) The high-frequency probabilistic hypothesis (Hypothesis \ref{high:hypothesis-probabilistic}) is satisfied.
\end{enumerate}
\end{hypothesis}

\begin{remark} 
While estimates of the stochastic objects $\linear$ and $\philinear$ in Hypothesis \ref{AH:hypothesis-probabilistic}.\ref{AH:item-probabilistic-A-linear} will not be needed until the proof of \mbox{Proposition \ref{decay:prop-short}}, but it is convenient to state them here. The reasons for postponing the statement of the probabilistic estimates in Hypothesis \ref{AH:hypothesis-probabilistic}.\ref{AH:item-high} until Appendix \ref{section:high} are two-fold: First, the estimates will only be directly used in Appendix \ref{section:high}, and not appear in this section. Second, the estimates are rather complicated, since they contain all estimates needed in the local theory of \eqref{AH:eq-evolution-A}-\eqref{AH:eq-evolution-phi}.  
\end{remark}

In the next definition, we introduce the event determined by the estimates in Hypothesis \ref{AH:hypothesis-probabilistic}. 

\begin{definition}[Event for probabilistic hypothesis]\label{decay:def-probabilistic-event}
Let $t_0 \geq 0$, let $\tau \in (0,1)$, let $L\in \dyadic$, let $\pregA_0 \in \Cs_x^\eta$,
and let $\Blin$ be as in \eqref{AH:eq-Blin}. Then, we define
\begin{equation*}
E\big( t_0, \tau, L, \pregA_0 \big)
\end{equation*}
as the event on which the estimates in Hypothesis \ref{AH:hypothesis-probabilistic} are satisfied.
\end{definition}

In the next proposition, we prove that the probabilistic hypothesis (Hypothesis \ref{AH:hypothesis-probabilistic}) is satisfied with high probability. 

\begin{proposition}\label{AH:prop-probabilistic}
Let $T_0 \geq 0$ be a stopping time, let $\tau \colon (\Omega,\Fc)\rightarrow (0, \infty)$, let $L\colon (\Omega,\Fc)\rightarrow \dyadic$, and let $\pregA_0\colon (\Omega,\Fc)\rightarrow \Cs_x^\eta$. Furthermore, assume that $\tau$, $L$, and $\pregA_0$ are $\Fc_{T_0}$-measurable and, similar as in \eqref{AH:eq-L}, assume that 
\begin{equation}\label{AH:eq-probabilistic-L-condition}
L \geq C_0 \max\Big( 1 , \big\| \pregA_0 \big\|_{\Cs_x^\eta}^{\frac{2}{\eta_3}}\Big) \quad \text{ and } \quad \tau \leq c_0. 
\end{equation}
Then, it holds that
\begin{equation}\label{AH:eq-probabilistic}
\bP \Big( \big( \Omega \backslash E\big( T_0, \tau, L, \pregA_0 \big) \big) \Big| \Fc_{T_0} \Big) 
\leq c^{-1} \exp \Big( - c L^{\frac{\kappa}{2}} \Big).
\end{equation}
\end{proposition}

\begin{proof}[Proof of Proposition \ref{AH:prop-probabilistic}:]
By the Strong Markov property of Brownian motion, it suffices to assume that $T_0 = t_0$ is constant. We write $\| \Blin \|_{\resnorm}:= \| \Blin (t_0+\cdot)\|_{1,\resnorm}$, where $\| \cdot \|_{1,\resnorm}$ is as in Definition \ref{intro:def-norm}. Using the linear heat equation \eqref{AH:eq-Blin}, the Coulomb condition $\partial_j B^j=0$, and \eqref{AH:eq-probabilistic-L-condition}, we obtain that 
\begin{equation}\label{AH:eq-probabilistic-p1}
\| \Blin \|_{\resnorm} \lesssim \| \Blin(t_0) \|_{\Cs_x^\eta} =  \| \pregA_0\|_{\Cs_x^\eta} \lesssim L^{\frac{\eta_3}{2}}. 
\end{equation}
In the following, all stated estimates only hold after possibly restricting to events satisfying \eqref{AH:eq-probabilistic}, but this will not be repeated below. The estimate in 
\ref{AH:item-probabilistic-A-linear} follow directly from Lemma \ref{lemma:linear-gc}. The estimate in 
\ref{AH:item-probabilistic-covariant-linear} follows directly from Lemma \ref{lemma:covariant-linear-object-mollified-noise-L-infty-bound}. To obtain the estimate in \ref{AH:item-probabilistic-derivative}, 
we combine Proposition \ref{prop:derivative-nonlinearity-non-resonant} and Proposition \ref{cshe:prop-resonant} (with $\alpha=2\eta$ and $\lambda=\frac{1}{2}+\eta$), which yield
\begin{align}
&\, \Big\| \leray \Im \Duh \Big[ \, \overline{\philinear[\Blin,\lo]} \covd_{\Blin} \philinear[\Blin,\lo] - \tfrac{1}{8\pi} \Blin  \Big] \Big\|_{C_t^0 \Cs_x^{2\eta} \cap C_t^{\eta} \Cs_x^0([t_0, t_0 + \tau])} \notag \\
\lesssim&\, \tau^{\frac{3}{4}-\frac{\eta}{2}} 
\Big( 1 + \| \Blin \|_{\resnorm} + L^{-1} \| \Blin\|_{\resnorm}^2 \Big)^{\frac{1}{2}+\eta}
\Big( 1 + L^{-1} \| \Blin \|_{\resnorm} \Big)^{\frac{1}{2}-\eta}
+ L^{-1}\tau^{\frac{1}{2}-\eta} \| \Blin \|_{\resnorm}^3 + \| \Blin \|_{\resnorm}^{\kappa}. \label{AH:eq-probabilistic-p2}
\end{align}
Due to \eqref{AH:eq-probabilistic-p1}, all terms in \eqref{AH:eq-probabilistic-p2} involving $L^{-1}$-factors are negligible. Together with $\tau \leq c_0$ (so that $\tau^{\frac{\eta}{2}} \ll 1$), this easily implies that
\begin{equation*}
\eqref{AH:eq-probabilistic-p1}\leq\tau^{\frac{3}{4}-\eta} \| \Blin(t_0) \|_{\Cs_x^\eta}^{\frac{1}{2}+\eta} + L^{\kappa}. 
\end{equation*}
Finally, the estimates in \ref{AH:item-high} follow from Proposition \ref{high:prop-high-frequency-probabilistic} below. 
\end{proof}

As mentioned previously, all interactions involving $\linear[\hi]$, $\Slin_{\hi}$, $\philinear[\hi]$, $\varphi_{\hi}$, or $\psi_{\hi}$ are negligible and can be treated perturbatively (see e.g. the proof of Lemma \ref{AH:lem-psi-lo-auxiliary} or Appendix \ref{section:high}). In order to state simple estimates for all such interactions, it is convenient to make the following continuity hypothesis.

\begin{hypothesis}[Continuity hypothesis]\label{AH:hypothesis-continuity}
Let $t_0 \in [0,\infty)$ and let $\tau\in (0,1]$ be a time-scale. Then, the continuity hypothesis consists of the estimates 
\begin{align}
\| Z \|_{C_t^0 \Cs_x^{2\eta} \cap C_t^{\eta} \Cs_x^0([t_0,t_0+\tau])} &\leq L^{\eta_{2}}, \label{AH:eq-continuity-Z} \\  
\| \psi_{\lo} \|_{C_t^0 L_x^r([t_0,t_0+\tau])} &\leq L^{\eta_{2}}, \label{AH:eq-continuity-psi-lo} \\
\| \psi_{\hi} \|_{C_t^0 \Cs_x^\eta([t_0,t_0+\tau])} &\leq L^{\eta_1-\eta}. \label{AH:eq-continuity-psi-hi} 
\end{align}
\end{hypothesis}

\begin{remark}
Our final estimates of $Z$, $\psi_{\lo}$, and $\psi_{\hi}$ will be much stronger than  \eqref{AH:eq-continuity-Z}, \eqref{AH:eq-continuity-psi-lo}, and \eqref{AH:eq-continuity-psi-hi}, see e.g. Lemma \ref{decay:lem-short-bounds} and Lemma \ref{high:lem-para-controlled}. Nevertheless,  \eqref{AH:eq-continuity-Z}, \eqref{AH:eq-continuity-psi-lo}, and \eqref{AH:eq-continuity-psi-hi} will be sufficient to control several error terms involving $\psi_{\hi}$. 
\end{remark}

\subsection{\protect{Estimates of $\psi_{\lo}$}}\label{section:AH-psi}

In this subsection, we prove two estimates for $\psi_{\lo}$. The first estimate is a covariant monotonicity formula for the $L_x^p$-norm of $\psi_{\lo}$, which will be used to obtain decay of $\psi_{\lo}$. The second estimate is a localized version of a covariant derivative estimate, which will be used to control $Z$. In the following, recall the parameter $\nu$ from \eqref{prelim:eq-parameter-new-eta-nu}, which is small but much bigger than $\eta$.

\begin{proposition}[Covariant monotonicity formula for $\psi_{\lo}$]\label{AH:prop-psi-lo-estimate}
Let the probabilistic hypothesis and continuity hypothesis, i.e., Hypothesis \ref{AH:hypothesis-probabilistic} and \ref{AH:hypothesis-continuity}, be satisfied. 
Furthermore, let $2 \leq p \leq r$, let $\delta \in (0,1)$, let $0<\tau^\ast\leq \tau$, and let $t\in [t_0,t_0+\tau^\ast]$. Then, it holds that
\begin{equation}\label{AH:eq-psi-lo-estimate}
\begin{aligned}
&\, \frac{1}{p} \frac{\mathrm{d}}{\mathrm{d}t} \, \big\|  \psi_{\lo}(t)\big\|_{L_x^p}^{p}
+ (1-\delta) \big\| |\psi_{\lo}(t)|^{\frac{p-2}{2}} |\covd_{A_{\lo}} \psi_{\lo}(t)| \big\|_{L_x^2}^2 
+ (1-\delta) \big\| \psi_{\lo}(t) \big\|_{L_x^{p+q-1}}^{p+q-1} \\
\leq& \,  C_{\delta,p,q} \| Z \|_{L_{t,x}^\infty([t_0,t_0+\tau^\ast])}^{\frac{2}{q} (1+\nu) (p+q-1)} + C_{\delta,p,q} L^{\kappaone (p+q-1)}. 
\end{aligned}
\end{equation}
\end{proposition}

We encourage the reader to ignore the $L^{\kappaone}$-term in \eqref{AH:eq-psi-lo-estimate}, which will not be important. The essential aspect of \eqref{AH:eq-psi-lo-estimate} is the exponent $\frac{2}{q}$ of the $\| Z\|_{L_{t,x}^\infty}$-norm, since it will play a crucial role in Section \ref{section:decay}.

\begin{proposition}[Localized estimate of $\psi_{\lo}$]\label{AH:prop-psi-lo-localized}
Let the probabilistic hypothesis and continuity hypothesis, i.e., Hypothesis \ref{AH:hypothesis-probabilistic} and \ref{AH:hypothesis-continuity}, be satisfied. 
Furthermore, let $t_1\in (t_0,t_0+\tau)$, let $\rho:=t_1-t_0$, and let $x_1 \in \T^2$. Then, it holds that  
\begin{equation}\label{AH:psi-lo-localized}
\begin{aligned}
\big\| p^{\frac{1}{2}}(t,x;t_1,x_1) \covd_{A_{\lo}}\psi_{\lo} \big\|_{L_{t,x}^2([t_0,t_1])} 
\lesssim  \rho^{-\frac{1}{r}} \big\| \psi_{\lo} \big\|_{L_t^\infty L_x^r([t_0,t_1])} 
+ \rho^{\frac{1}{2}} 
 \| Z \|_{L_{t,x}^\infty([t_0,t_1])}^{\frac{q+1}{q} (1+\nu)} 
+  \rho^{\frac{1}{2}}  L^\kappaone.
\end{aligned}
\end{equation}
\end{proposition}

\begin{remark}
Since Proposition \ref{AH:prop-psi-lo-localized} is proven using the covariant monotonicity formula (Proposition \ref{prop:monotonicity}), the $p^{\frac{1}{2}}$-factor is natural. Together with Lemma \ref{prelim:lem-Duhamel-weighted}, Proposition \ref{AH:prop-psi-lo-localized} can be used to control Duhamel integrals involving $\covd_{A_{\lo}} \psi_{\lo}$, see e.g. Lemma \ref{AH:lem-derivative-psi}.
\end{remark}

Before we turn to the proofs of Proposition \ref{AH:prop-psi-lo-estimate} and Proposition \ref{AH:prop-psi-lo-localized}, we prove the following auxiliary lemma, which concerns the forcing terms in \eqref{AH:eq-psi-lo-e1} and \eqref{AH:eq-psi-lo-e2}. 

\begin{lemma}[Forcing terms in $\psi_{\lo}$-equation]
\label{AH:lem-psi-lo-auxiliary}
Let the probabilistic hypothesis and continuity hypothesis, i.e., Hypothesis \ref{AH:hypothesis-probabilistic} and \ref{AH:hypothesis-continuity}, be satisfied. 
Furthermore, let 
\begin{align*}
G_{\lo,j} &:= 2\icomplex \big( \linear[\lo] + \Slin_{\lo} +  Z \big)_j \, \philinear[\Blin,\lo] + 2\icomplex \big( \linear[\lo] + B + \Slin_{\lo} + Z \big)_j \psi_{\hi}, \\ 
H_{\lo} &:=  \big| \linear[\lo] + \Slin_{\lo} + Z \big|^2 \, \philinear[\Blin,\lo] + \big| \linear[\lo] + B + \Slin_{\lo} +   Z \big|^2 \psi_{\hi} 
+ \big( 2 \sigma_{\leq L}^2 +1 \big) \big( \philinear[\Blin,\lo] + \psi_{\hi} \big). 
\end{align*}
Then, it holds that 
\begin{align}
    \| G_{\lo} \|_{L_{t,x}^\infty}  
   &\lesssim  \| Z \|_{L_{t,x}^\infty}^{1+\nu} + L^{\kappa \frac{1+\nu}{\nu}}   \\ 
   \| H_{\lo} \|_{L_{t,x}^\infty}  
   &\lesssim  \| Z \|_{L_{t,x}^\infty}^{2(1+\nu)} + L^{2\kappa \frac{1+\nu}{\nu}}. 
\end{align}
\end{lemma}

\begin{proof} We first recall estimates of $\linear[\lo]$, $\philinear[\Blin,\lo]$, $\Blin$, $\Slin$, and $\sigma_{\leq L}^2$. Due to Hypothesis \ref{AH:hypothesis-probabilistic}, it holds that 
\begin{equation*}
\big\| \linear[\lo] \big\|_{L_{t,x}^\infty} \lesssim L^\kappa, \qquad 
\big\| \philinear[\Blin,\lo] \big\|_{L_{t,x}^\infty} \lesssim L^\kappa, \qquad \text{and} \qquad 
\sigma_{\leq L}^2 \lesssim \log(L). 
\end{equation*}
Furthermore, due to Lemma \ref{prelim:lem-heat-flow-bound-decay}, \eqref{AH:eq-L}, and \eqref{AH:eq-L-nreg}, it also holds that
\begin{equation*}
\big\| \Blin \big\|_{L_{t,x}^\infty} \lesssim \big\| \pregA_0 \big\|_{L_x^\infty} \lesssim L^{\eta_3} 
\qquad \text{and} \qquad 
\big\| \Slin_{\lo} \big\|_{L_{t,x}^\infty} \lesssim \big\| P_{\leq L} \nregA_0 \big\|_{L_x^\infty}\lesssim L^{3\kappa}. 
\end{equation*}
Using the algebra property of $L_{t,x}^\infty$, it then follows that 
\begin{align*}
\| G_{\lo} \|_{L_{t,x}^\infty} 
&\lesssim 
     \big\|  \linear[\lo] + \Slin_{\lo} + Z \big\|_{L_{t,x}^\infty} \, \big\| \philinear[\Blin,\lo] \big\|_{L_{t,x}^\infty} 
     + \big\| \linear[\lo] + B + \Slin_{\lo} + Z \big\|_{L_{t,x}^\infty} \big\| \psi_{\hi} \big\|_{L_{t,x}^\infty}  \\
&\lesssim \big( L^{3\kappa}  + \| Z \|_{L_{t,x}^\infty} \big) L^\kappa   + 
\big( L^{\eta_3} + \| Z\|_{L_{t,x}^\infty} \big) \| \psi_{\hi} \|_{L_{t,x}^\infty}. 
\end{align*}
Due to the continuity hypothesis (Hypothesis \ref{AH:hypothesis-continuity}), it holds that 
\begin{align*}
\big( L^{\eta_3} + \| Z\|_{L_{t,x}^\infty} \big) \| \psi_{\hi} \|_{L_{t,x}^\infty}
\lesssim L^{\eta_1+\eta_2 -\eta} \lesssim 1 \lesssim L^{4\kappa}. 
\end{align*}
Using \eqref{prelim:eq-parameter-new-eta-nu} and Young's inequality, we then obtain 
\begin{equation*}
\| G_{\lo} \|_{L_{t,x}^\infty}  \lesssim L^\kappa \| Z \|_{L_{t,x}^\infty} 
+ L^{4\kappa} \lesssim \| Z \|_{L_{t,x}^\infty}^{1+\nu} + L^{\kappa \frac{1+\nu}{\nu}} + L^{4\kappa}
\lesssim \| Z \|_{L_{t,x}^\infty}^{1+\nu} + L^{\kappa \frac{1+\nu}{\nu}},
\end{equation*}
which yields the desired estimate of $G_{\lo}$. The estimate of $H_{\lo}$ is similar, and we omit the details.
\end{proof}

Equipped with Lemma \ref{AH:lem-psi-lo-auxiliary}, we can now turn to the proof of Proposition \ref{AH:prop-psi-lo-estimate}. 

\begin{proof}[Proof of Proposition \ref{AH:prop-psi-lo-estimate}:] 
Using $G_{\lo}$ and $H_{\lo}$ from Lemma \ref{AH:lem-psi-lo-auxiliary}, we can write the evolution equation for $\psi_{\lo}$ from \eqref{AH:eq-psi-lo-e1}-\eqref{AH:eq-psi-lo-e3} as 
\begin{align}
& \big(\partial_t - \covd_{A_{\lo}}^j \covd_{A_{\lo},j} +1   \big)  \, \psi_{\lo} + |\psi_{\lo}|^{q-1} \psi_{\lo}  \notag \\ 
=&\,   \covd_{A_{\lo}}^j G_{\lo,j} + H_{\lo} +  \big( 2 \sigma_{\leq L}^2 +1 \big) \psi_{\lo} \label{AH:eq-psi-lo-estimate-p1} \\ 
-&\,  \Big( \biglcol \, \big| \philinear[\Blin,\lo] + \psi_{\lo} + \psi_{\hi} \big|^{q-1} \big( \philinear[\Blin,\lo]  + \psi_{\lo} + \psi_{\hi} \big) \,  \bigrcol 
- |\psi_{\lo}|^{q-1} \psi_{\lo}  \Big) . \label{AH:eq-psi-lo-estimate-p2}
\end{align}
Using the covariant monotonicity formula from Proposition \ref{prop:monotonicity}, it then follows that 
\begin{align}
&\frac{1}{p} \frac{\mathrm{d}}{\mathrm{d}t} \, \big\|  \psi_{\lo}\big\|_{L_x^p}^{p}
+  \big\| |\psi_{\lo}|^{\frac{p-2}{2}} |\covd_{A_{\lo}} \psi_{\lo}| \big\|_{L_x^2}^2 
+ \tfrac{p-2}{4} \big\| |\psi_{\lo}|^{\frac{p-4}{2}} \nabla |\psi_{\lo}|^2 \big\|_{L_x^2}^2
+ \big\|  \psi_{\lo}\big\|_{L_x^p}^{p} +  \big\| \psi_{\lo} \big\|_{L_x^{p+q-1}}^{p+q-1} \notag \\
=& \Re \int_{\T^2} \dx \,   |\psi_{\lo}|^{p-2} \overline{\psi_{\lo}} \Big( \covd_{A_{\lo}}^j G_{\lo,j} + H_{\lo} + \big( 2 \sigma_{\leq L}^2 +1 \big) \psi_{\lo} \Big) \label{AH:eq-psi-lo-estimate-p3} \\ 
-& \Re \int_{\T^2} \dx \, |\psi_{\lo}|^{p-2} \overline{\psi_{\lo}}  \Big( \biglcol \, \big| \philinear[\Blin,\lo]  + \psi_{\lo} + \psi_{\hi} \big|^{q-1} \big( \philinear[\Blin,\lo] + \psi_{\lo} + \psi_{\hi} \big) \,  \bigrcol 
- |\psi_{\lo}|^{q-1} \psi_{\lo}  \Big). \label{AH:eq-psi-lo-estimate-p4}
\end{align}
We now estimate the contributions of the terms in \eqref{AH:eq-psi-lo-estimate-p3} and \eqref{AH:eq-psi-lo-estimate-p4} separately.\\

\emph{Contribution of the $G_{\lo}$-term:} We first use integration by parts (Lemma \ref{prelim:lem-derivatives}), which yields that 
\begin{equs}
\int_{\T^2} \dx \,   \overline{|\psi_{\lo}|^{p-2} \psi_{\lo}}  \, \covd_{A_{\lo}}^j G_{\lo,j} 
&= - \int_{\T^2} \dx \,  \overline{\covd_{A_{\lo}}^j \big( |\psi_{\lo}|^{p-2} \psi_{\lo}\big)} \,  G_{\lo,j}.
\end{equs}
Using the product formula and diamagnetic inequality (Lemma \ref{prelim:lem-diamagnetic}), it holds that 
\begin{equs}
\big| \covd_{A_{\lo}}^j \big( |\psi_{\lo}|^{p-2} \psi_{\lo}\big) \big| 
\leq \big| \partial^j \big( |\psi_{\lo}|^{p-2}\big) \, \psi_{\lo} \big| 
+ |\psi_{\lo}|^{p-2} \big| \covd_{A_{\lo}}^j \psi_{\lo} \big|
\lesssim |\psi_{\lo}|^{p-2} \big| \covd_{A_{\lo}}^j \psi_{\lo} \big|.
\end{equs}
Together with H\"{o}lder's inequality, we obtain that
\begin{equs}
\Big| \int_{\T^2} \dx \,  \overline{\covd_{A_{\lo}}^j \big( |\psi_{\lo}|^{p-2} \psi_{\lo}\big)} \,  G_{\lo,j} \Big| 
\lesssim \big\| |\psi_{\lo}|^{p-2} \covd_{A_{\lo}} \psi_{\lo} \big\|_{L_x^1} \big\| G_{\lo} \big\|_{L_x^\infty} 
\lesssim \big\| |\psi_{\lo}|^{\frac{p-2}{2}} \covd_{A_{\lo}} \psi_{\lo} \big\|_{L_x^2} 
\big\| \psi_{\lo} \big\|_{L_x^{p-2}}^{\frac{p-2}{2}} \big\| G_{\lo} \big\|_{L_x^\infty} .
\end{equs}
Using Young's inequality, we further obtain that 
\begin{equs}\label{AH:eq-psi-lo-estimate-p5}
&\big\| |\psi_{\lo}|^{\frac{p-2}{2}} \covd_{A_{\lo}} \psi_{\lo} \big\|_{L_x^2} 
\big\| \psi_{\lo} \big\|_{L_x^{p-2}}^{\frac{p-2}{2}} \big\| G_{\lo} \big\|_{L_x^\infty}\\
\leq&\, \delta \big\| |\psi_{\lo}|^{\frac{p-2}{2}} \covd_{A_{\lo}} \psi_{\lo} \big\|_{L_x^2}^2 
+ \delta \big\| \psi_{\lo} \big\|_{L_x^{p-2}}^{p+q-1}
+ C_{\delta,p,q} \big\| G_{\lo} \big\|_{L_x^\infty}^{\frac{2}{q+1} (p+q-1)}.
\end{equs}
The first and second summand in \eqref{AH:eq-psi-lo-estimate-p5} are clearly acceptable. Using Lemma \ref{AH:lem-psi-lo-auxiliary}, the third summand in \eqref{AH:eq-psi-lo-estimate-p5} can be estimated by 
\begin{equation*}
    \big\| G_{\lo} \big\|_{L_x^\infty}^{\frac{2}{q+1} (p+q-1)} 
    \lesssim  \| Z \|_{C_t^0 \Cs_x^\eta}^{\frac{2}{q+1} (1+\nu) (p+q-1)} 
    + L^{\kappa \frac{2}{q+1} \frac{1+\nu}{\nu} (p+q-1)}. 
\end{equation*}
Due to the trivial estimate $\frac{2}{q+1}\leq \frac{2}{q}$ and the definition of $\kappaone$, this also yields an acceptable contribution.\\

\emph{Contribution of the $H_{\lo}$-term:} Using H\"{o}lder's inequality and Young's inequality, we obtain that
\begin{equs}
\Big| Re \int_{\T^2} \dx \,   |\psi_{\lo}|^{p-2} \overline{\psi_{\lo}}  H_{\lo} \Big| 
\leq \| \psi_{\lo} \|_{L_x^{p-1}}^{p-1} \| H_{\lo}\|_{L_x^\infty} 
\leq \delta  \| \psi_{\lo} \|_{L_x^{p-1}}^{p+q-1} 
+ C_{\delta,p,q}  \| H_{\lo}\|_{L_x^\infty}^{\frac{p+q-1}{q}}
\end{equs}
The first summand is clearly acceptable. Using Lemma \ref{AH:lem-psi-lo-auxiliary},  the second summand can be estimated by 
\begin{equs}
 \| H_{\lo}\|_{L_x^\infty}^{\frac{p+q-1}{q}} 
 \lesssim \| Z \|_{C_t^0 \Cs_x^\eta}^{\frac{2}{q} (1+\nu) (p+q-1)} + L^{\frac{2}{q} \kappa \frac{1+\nu}{\nu} (p+q-1)},
\end{equs}
which is acceptable.\\

\emph{Contribution of the $\sigma_{\leq L}^2 \psi_{\lo}$-term:}
Using Young's inequality, we obtain that 
\begin{equs}
  \big( 2 \sigma_{\leq L}^2 +1 \big)    \Big| Re \int_{\T^2} \dx \,   |\psi_{\lo}|^{p-2} \overline{\psi_{\lo}} \psi_{\lo} \Big|  
   =  \big( 2 \sigma_{\leq L}^2 +1 \big)  \| \psi_{\lo} \|_{L_x^p}^p 
   \leq \delta \| \psi_{\lo} \|_{L_x^p}^{p+q-1} + C_{\delta,p,q}  \big( 2 \sigma_{\leq L}^2 +1 \big)^{\frac{p+q-1}{q-1}}.
\end{equs}
Since $\sigma_{\leq L}^2 \sim \log(L)$, this clearly yields an acceptable contribution.\\

\emph{Contribution of power-type nonlinearity:} Using Hypothesis \ref{AH:hypothesis-probabilistic} and \ref{AH:hypothesis-continuity} and using $\sigma_{\leq L}^2\sim \log(L)$, we obtain the pointwise estimate
\begin{equation}\label{AH:eq-psi-lo-estimate-p6} 
\begin{aligned}
&\Big|  \biglcol \, \big| \philinear[\Blin,\lo]  + \psi_{\lo} + \psi_{\hi} \big|^{q-1} \big( \philinear[\Blin,\lo]  + \psi_{\lo} + \psi_{\hi} \big) \,  \bigrcol 
- |\psi_{\lo}|^{q-1} \psi_{\lo}  \Big| \\
\lesssim&\, \Big( \big|  \philinear[\Blin,\lo] \big| + \big| \psi_{\lo} \big| + \big|  \psi_{\hi}\big| + \sigma_{\leq L} \Big)^{q-1}
\times \Big( \big|  \philinear[\Blin,\lo] \big|  + \big|  \psi_{\hi}\big| + \sigma_{\leq L} \Big) 
\lesssim L^\kappa \big( \big| \psi_{\lo} \big| + L^\kappa \big)^{q-1}.
\end{aligned}
\end{equation}
As a result, it follows that 
\begin{align*}
&\,\Big| \eqref{AH:eq-psi-lo-estimate-p4} \Big| 
\lesssim L^\kappa \int_{\T^2} \dx \big( \big| \psi_{\lo} \big| + L^\kappa \big)^{p+q-2}  \\
\lesssim&\,  L^{\kappa} \big( \| \psi_{\lo} \|_{L_x^{p+q-2}} + L^\kappa \big)^{p+q-2} 
\leq \delta \big\| \psi_{\lo} \big\|_{L_x^{p+q-1}}^{p+q-1} + C_{\delta,p,q} L^{\kappa (p+q-1) }. 
\end{align*}
Since $\kappa \leq \kappaone$, this yields an acceptable contribution.
\end{proof}

\begin{proof}[Proof of Proposition \ref{AH:prop-psi-lo-localized}:] 
Let $\delta>0$ be a sufficiently small absolute constant
and, to simplify the notation, let $I:= [t_0,t_1]$. 
Using the forcing terms $G_{\lo}$ and $H_{\lo}$ from Lemma \ref{AH:lem-psi-lo-auxiliary}, we can write the evolution equation for $\psi_{\lo}$ from \eqref{AH:eq-psi-lo-e1}-\eqref{AH:eq-psi-lo-e3} as 
\begin{align}
& \big(\partial_t - \covd_{A_{\lo}}^j \covd_{A_{\lo},j}  +1 \big)  \, \psi_{\lo} + |\psi_{\lo}|^{q-1} \psi_{\lo}  \notag \\ 
=&\,   \covd_{A_{\lo}}^j G_{\lo,j} + H_{\lo} +  \big( 2 \sigma_{\leq L}^2 +1 \big) \psi_{\lo} \label{AH:eq-psi-lo-localized-p1} \\ 
-&\,  \Big( \biglcol \, \big| \philinear[\Blin,\lo]  + \psi_{\lo} + \psi_{\hi} \big|^{q-1} \big( \philinear[\Blin,\lo]  + \psi_{\lo} + \psi_{\hi} \big) \,  \bigrcol 
- |\psi_{\lo}|^{q-1} \psi_{\lo}  \Big) . \label{AH:eq-psi-lo-localized-p2}
\end{align}
Let $\varepsilon>0$ be arbitrary and choose 
\begin{equation*}
K(t,x):= p(t,x;t_1+\varepsilon,x_1), 
\end{equation*}
which is a solution of the backwards heat equation. From the definition of $K$, it directly follows that
\begin{equs}\label{AH:eq-psi-lo-localized-K-estimate}
\big\| K \big\|_{L_t^1 L_x^1(I\times \T^2)} \leq \rho \big\| K \big\|_{L_t^\infty L_x^1(I\times \T^2)} = \rho. 
\end{equs}
Using the monotonicity formula from Proposition \ref{prop:monotonicity}, we obtain that 
\begin{align}
&\, \tfrac{1}{2} \big\| K^{\frac{1}{2}}(t_1) \psi_{\lo}(t_1) \big\|_{L_x^2}^2
+ \big\| K^{\frac{1}{2}} \covd_{A_{\lo}} \psi_{\lo} \big\|_{L_{t,x}^2(I)}^2
+ \big\| K^{\frac{1}{q+1}} \psi_{\lo} \big\|_{L_{t,x}^{q+1}(I)}^{q+1} \label{AH:eq-psi-lo-localized-p3} \\
\leq&\, \tfrac{1}{2} \big\| K^{\frac{1}{2}}(t_0) \psi_{\lo}(t_0) \big\|_{L_x^2}^2 
+ \int_{I} \dt \int \dx\, K \overline{\psi_{\lo}} \big( \covd_{A_{\lo}}^j G_{\lo,j} + H_{\lo} +  ( 2 \sigma_{\leq L}^2 +1 ) \psi_{\lo} \big) \label{AH:eq-psi-lo-localized-p4} \\
-&\, \int_{I} \dt \int \dx \, K \overline{\psi_{\lo}}
 \big( \biglcol \, \big| \philinear[\Blin,\lo]  + \psi_{\lo} + \psi_{\hi} \big|^{q-1} \big( \philinear[\Blin,\lo]  + \psi_{\lo} + \psi_{\hi} \big) \,  \bigrcol 
- |\psi_{\lo}|^{q-1} \psi_{\lo}  \big).  \label{AH:eq-psi-lo-localized-p5}
\end{align}
We now claim that 
\begin{equation}\label{AH:eq-psi-lo-localized-K}
\big\| K^{\frac{1}{2}} \covd_{A_{\lo}} \psi_{\lo} \big\|_{L_{t,x}^2(I)}^2
\lesssim \rho^{-\frac{2}{r}} \big\| \psi_{\lo} \big\|_{L_t^\infty L_x^r(I)}^2
+ \rho
 \| Z \|_{L_{t,x}^\infty(I)}^{2 \frac{q+1}{q} (1+\nu)} 
+  \rho  L^{2\kappaone}.
\end{equation}
Once \eqref{AH:eq-psi-lo-localized-K} has been shown, the desired estimate \eqref{AH:psi-lo-localized} can then be obtained using a limiting argument, i.e., by letting $\varepsilon\downarrow 0$. In order to prove \eqref{AH:eq-psi-lo-localized-K}, we separately estimate the terms in \eqref{AH:eq-psi-lo-localized-p4} and \eqref{AH:eq-psi-lo-localized-p5}. \\ 

\emph{Contribution of the $\psi_{\lo}(t_0)$-term:}
Using H\"{o}lder's inequality, it holds that
\begin{equs}
\big\| K^{\frac{1}{2}}(t_0) \psi_{\lo}(t_0) \big\|_{L_x^2}^2 
\lesssim \big\|  K^{\frac{1}{2}}(t_0) \big\|_{L_x^{\frac{2r}{r-2}}}^2 
\big\| \psi_{\lo}(t_0) \big\|_{L_x^r}^2 \lesssim \rho^{-\frac{2}{r}} \big\| \psi_{\lo} \big\|_{L_t^\infty L_x^r(I)}^2.
\end{equs}
This yields an acceptable contribution to the right-hand side of \eqref{AH:eq-psi-lo-localized-K}. \\ 

\emph{Contribution of the $G_{\lo}$-term:} Using integration by parts (Lemma \ref{prelim:lem-derivatives}), we obtain that 
\begin{equs}\label{AH:eq-psi-lo-localized-p6}
\Big|  \int_{I} \dt \int \dx\, K \overline{\psi_{\lo}} \, \covd_{A_{\lo}}^j G_{\lo,j}  \Big|   
\leq  \Big| \int_{I} \dt \int \dx\, K \overline{\covd_{A_{\lo}}^j \psi_{\lo}} G_{\lo,j} \Big| 
+ \Big| \int_{I} \dt \int \dx\, (\partial^j K) \overline{\psi_{\lo}} G_{\lo,j} \Big|. 
\end{equs}
We first estimate the first summand in \eqref{AH:eq-psi-lo-localized-p6}. Using Young's inequality and Lemma \ref{AH:lem-psi-lo-auxiliary}, we obtain that
\begin{equs}
\Big| \int_{I} \dt \int \dx\, K \overline{\covd_{A_{\lo}}^j \psi_{\lo}} G_{\lo,j} \Big| 
\leq \delta \int_{I} \dt \int \dx\, K |\covd_{A_{\lo}} \psi_{\lo}|^2 + C_{\delta}
\int_{I} \dt \int \dx\, K |G_{\lo}|^2.  
\end{equs}
The first summand can be absorbed in \eqref{AH:eq-psi-lo-localized-p3}. Using Lemma \ref{AH:lem-psi-lo-auxiliary} and \eqref{AH:eq-psi-lo-localized-K-estimate}, the second summand can be estimated by 
\begin{equs}
\int_{I} \dt \int \dx\, K |G_{\lo}|^2 \lesssim
\| K \|_{L_{t,x}^1(I)} \| G_{\lo} \|_{L_{t,x}^\infty(I)}^2
\lesssim \rho \Big( \| Z \|_{L_{t,x}^\infty (I)}^{2(1+\nu)} + L^{2\kappa \frac{1+\nu}{\nu}} \Big).
\end{equs}
Since $1\leq \frac{q+1}{q}$ and $\kappa \frac{1+\nu}{\nu}\leq \kappaone$, this yields an acceptable contribution to the right-hand side of \eqref{AH:eq-psi-lo-localized-K}. We now estimate the second summand in \eqref{AH:eq-psi-lo-localized-p6}. Using H\"{o}lder's and Young's inequality, it holds that
\begin{align*}
&\, \Big| \int_{I} \dt \int \dx\, (\partial^j K) \overline{\psi_{\lo}} G_{\lo,j} \Big|
\lesssim \big\| \partial^j K \big\|_{L_t^1 L_x^{r^\prime}(I)} 
\big\| \psi_{\lo} \big\|_{L_t^\infty L_x^r(I)} \big\| G_{\lo} \big\|_{L_{t,x}^\infty(I)} \\ 
\lesssim&\,  \rho^{\frac{1}{2}-\frac{1}{r}} \big\| \psi_{\lo} \big\|_{L_t^\infty L_x^r(I)} \big\| G_{\lo} \big\|_{L_{t,x}^\infty(I)} 
\lesssim \rho^{-\frac{2}{r}} \big\| \psi_{\lo} \big\|_{L_t^\infty L_x^r(I)}^2 
+ \rho \big\| G_{\lo} \big\|_{L_{t,x}^\infty(I)}^2.
\end{align*}
After estimating $\big\| G_{\lo} \big\|_{L_{t,x}^\infty}$ as before, this also yields an acceptable contribution. \\

\emph{Contribution of the $H_{\lo}$-term:} Using Young's inequality, it holds that 
\begin{equs}
\Big|\int_I \dt \int \dx \, K \overline{\psi_{\lo}} H_{\lo}\Big| 
\leq \delta \int_I \dt \int \dx K |\psi_{\lo}|^{q+1} 
+ C_{\delta,q} \int_I \dt \int \dx K |H_{\lo}|^{\frac{q+1}{q}}. 
\end{equs}
The first summand can be absorbed in \eqref{AH:eq-psi-lo-localized-p3}. Using Lemma \ref{AH:lem-psi-lo-auxiliary}, the second term can be estimated by 
\begin{align*}
&\,\int_I \dt \int \dx K |H_{\lo}|^{\frac{q+1}{q}}
\lesssim \| K \|_{L_{t,x}^1(I)} \| H_{\lo} \|_{L_{t,x}^\infty(I)}^{\frac{q+1}{q}} \\ 
\lesssim&\, \rho \big( \|Z\|_{L_{t,x}^\infty(I)}^{2 (1+\nu)} + L^{2 \kappa  \frac{1+\nu}{\nu}} \big)^{\frac{q+1}{q}}
\lesssim \rho \|Z\|_{L_{t,x}^\infty(I)}^{2 \frac{q+1}{q} (1+\nu)}
+ \rho L^{2\kappa \frac{1+\nu}{\nu} \frac{q+1}{q}}.
\end{align*}
Since $\kappa \frac{1+\nu}{\nu} \frac{q+1}{q}\leq \kappaone$, this yields an acceptable contribution to the right-hand side of \eqref{AH:eq-psi-lo-localized-K}. \\ 

\emph{Contribution of the $\sigma_{\leq L}^2 \psi_{\lo}$-term:} Using Young's inequality, we estimate
\begin{equs}
\big( 2 \sigma_{\leq L}^2 + 1 \big) \Big| \int_{I} \dt \int \dx\, K \overline{\psi_{\lo}} \psi_{\lo}  \Big| 
\leq \delta  \int_{I} \dt \int \dx \, K |\psi_{\lo}|^{q+1} + C_{\delta,q} \big( 2 \sigma_{\leq L}^2 +1 \big)^{\frac{q+1}{q-1}} \int_{I} \dt \int \dx \, K. 
\end{equs}
The first summand can be absorbed in \eqref{AH:eq-psi-lo-localized-p3}. 
Using that $\sigma_{\leq L}^2\sim \log(L)$, the second summand easily yields an acceptable contribution to the right-hand side of \eqref{AH:eq-psi-lo-localized-K}.  \\

\emph{Contribution of the power-type nonlinearity:} Similarly as in \eqref{AH:eq-psi-lo-estimate-p6}, we have that 
\begin{align*}
&\, \Big| \overline{\psi_{\lo}}
 \big( \biglcol \, \big| \philinear[\Blin,\lo]  + \psi_{\lo} + \psi_{\hi} \big|^{q-1} \big( \philinear[\Blin,\lo]  + \psi_{\lo} + \psi_{\hi} \big) \,  \bigrcol 
- |\psi_{\lo}|^{q-1} \psi_{\lo}  \big) \Big| \\
\lesssim&\, L^\kappa \big( |\psi_{\lo}| + L^\kappa \big)^q 
\leq \delta |\psi_{\lo}|^{q+1} + C_{\delta,q} L^{\kappa (q+1)}. 
\end{align*}
As a result, it follows that 
\begin{equation*}
\Big| \eqref{AH:eq-psi-lo-localized-p5} \Big| 
\leq \delta 
 \int_{I} \dt \int \dx \, K  |\psi_{\lo}|^{q+1} 
 + C_{\delta,q}  \int_{I} \dt \int \dx \, K  L^{\kappa (q+1)} 
\leq  \delta 
 \int_{I} \dt \int \dx \, K  |\psi_{\lo}|^{q+1}  +  C_{\delta,q}  \rho   L^{\kappa (q+1)} . 
\end{equation*}
The first summand can be absorbed in \eqref{AH:eq-psi-lo-localized-p3} and the second summand yields an acceptable contribution. 
\end{proof}

\subsection{\protect{Estimates of $Z$}}\label{section:AH-Z} 

In this subsection, we control the remainder $Z$ from the decomposition in \eqref{AH:eq-decomposition-A}, and our main estimate is contained in the following proposition.

\begin{proposition}[Estimate of $Z$]\label{AH:prop-Z}
Let the probabilistic hypothesis and continuity hypothesis, i.e., Hypothesis \ref{AH:hypothesis-probabilistic} and \ref{AH:hypothesis-continuity}, be satisfied, let $0<\tau^\ast \leq \tau$, and let $I:=[t_0,t_0+\tau^\ast]$. Then, it holds that
\begin{equation}\label{AH:eq-Z-estimate}
\begin{aligned}
\big\| Z \big\|_{C_t^0 \Cs_x^{2\eta} \cap C_t^{\eta} \Cs_x^0(I)}
&\lesssim \tau^{\frac{3}{4}-\eta} \| \pregA_0 \|_{\Cs_x^\eta}^{\frac{1}{2}+\eta} 
+  \tau^{\frac{1}{2}-\frac{1}{r}-2\eta} \big\| \psi_{\lo} \big\|_{C_t^0 L_x^r(I)}^2  \\ 
&+ \tau^{1-2\eta} \big( \big\| \psi_{\lo} \big\|_{C_t^0 L_x^r(I)} + L^\kappa \big) \big\| Z \big\|_{C_t^0 \Cs_x^\eta(I)}^{\frac{q+1}{q} (1+\nu)} 
+ L^{2\kappaone}.
\end{aligned}
\end{equation}
\end{proposition}

\begin{remark}
The time regularity of $Z$ is only needed in Appendix \ref{section:high}, where we prove Lemma \ref{AH:lem-derivative-high}.
\end{remark}

The proof of Proposition \ref{AH:prop-Z} is postponed until the end of the subsection. As a first step towards a proof of Proposition \ref{AH:prop-Z}, we prove the following lemma, which yields the symmetry of the bilinear form $\leray \Im ( \,  \overline{\phi} \,  \covd_A \varphi)$. 

\begin{lemma}[Symmetry]\label{AH:lem-symmetry}
Let $\phi,\varphi\colon \T^2 \rightarrow \C$ and let $A\colon \T^2 \rightarrow \R^2$. Then, it holds that
\begin{equs}
\leray \Im \big( \, \overline{\phi} \, \covd_A \varphi \big) 
=\leray \Im \big( \, \overline{\varphi}\, \covd_A \phi \big). 
\end{equs}
\end{lemma}

\begin{proof} Using the product formula from Lemma \ref{prelim:lem-derivatives}.\ref{prelim:item-product} and using that $\Im \overline{z}=-\Im z$ for all $z\in \C$, we obtain that
\begin{equs}
\leray \Im \big( \, \overline{\phi} \, \covd_A \varphi \big) 
= \leray \Im \big( \nabla  \big( \,  \overline{\phi} \varphi \big) \big) 
- \leray \Im \big( \, \overline{\covd_A \phi} \, \varphi \big)
= \leray \Im \big( \nabla  \big(\,  \overline{\phi} \varphi \big) \big)
+ \leray \Im \big( \, \overline{\varphi}\, \covd_A \phi \big). 
\end{equs}
Since $\leray \nabla =0$, this yields the desired identity.
\end{proof}

We now turn towards estimates of the derivative nonlinearity $\leray \Im ( \,  \overline{\phi} \,  \covd_A \phi)$, which are the subject of the following three lemmas. In Lemma \ref{AH:lem-derivative-stochastic}, we control the 
$\scalebox{0.9}{$\overline{\philinear[\Blin,\lo]}$} \covd_{A_{\lo}} \scalebox{0.9}{$\philinear[\Blin,\lo]$}$-interaction using our estimates for the covariant stochastic heat equation (Theorem \ref{intro:thm-cshe}, or more precisely, Proposition \ref{prop:derivative-nonlinearity-non-resonant} and Proposition \ref{cshe:prop-resonant}). In Lemma \ref{AH:lem-derivative-psi}, we control the 
$\scalebox{0.9}{$\overline{\philinear[\Blin,\lo]}$} \covd_{A_{\lo}} \psi_{\lo}$ and  
$\overline{\psi_{\lo}} \covd_{A_{\lo}} \psi_{\lo}$-interactions 
using our localized covariant energy estimate (Proposition \ref{AH:prop-psi-lo-localized}). Finally, in \mbox{Lemma \ref{AH:lem-derivative-high}}, we control all interactions involving either $\linear[\hi]$, $\Slin_{\hi}$, $\philinear[\hi]$, $\varphi_{\hi}$, or $\psi_{\hi}$ using estimates from the local theory of \eqref{AH:eq-evolution-A}-\eqref{AH:eq-evolution-phi}. 

\begin{lemma}[\protect{The 
$\scalebox{0.9}{$\overline{\philinear[\Blin,\lo]}$}\covd_{A_{\lo}} \scalebox{0.9}{$\philinear[\Blin,\lo]$}$-interaction}]\label{AH:lem-derivative-stochastic} 
Let the probabilistic hypothesis and continuity hypothesis, i.e., Hypothesis \ref{AH:hypothesis-probabilistic} and \ref{AH:hypothesis-continuity}, be satisfied, let $0<\tau^\ast \leq \tau$, and let $I:=[t_0,t_0+\tau^\ast]$. Then, it holds that 
\begin{equation}\label{AH:eq-derivative-stochastic}
\begin{aligned}
\, \Big\| \leray \Im \Duh \Big[ \, \overline{\philinear[\Blin,\lo]} \covd_{A_{\lo}} \philinear[\Blin,\lo] - \tfrac{1}{8\pi} A_{\lo} \Big] \Big\|_{C_t^0 \Cs_x^{2\eta} \cap C_t^{\eta} \Cs_x^0(I)}\lesssim\,  \tau^{\frac{3}{4}-\eta} \big\| \pregA_0 \big\|_{\Cs_x^\eta}^{\frac{1}{2}+\eta}
+ \tau^{1-2\eta} L^{2\kappa} \big\|  Z \big\|_{L_{t,x}^\infty(I)} 
+ L^{6\kappa}.
\end{aligned}
\end{equation}
\end{lemma}

As mentioned above, the proof of Lemma \ref{AH:lem-derivative-stochastic} heavily relies\footnote{To be more precise, the proof of Lemma \ref{AH:lem-derivative-stochastic} relies on the estimate \eqref{AH:eq-probabilistic-derivative} from Hypothesis \ref{AH:hypothesis-probabilistic}. Proposition \ref{prop:derivative-nonlinearity-non-resonant} and Proposition \ref{cshe:prop-resonant} are used to prove that the estimate \eqref{AH:eq-probabilistic-derivative} is satisfied with high probability.} on our estimates of the derivative nonlinearity, whose proof occupied the majority of Section \ref{section:cshe}.

\begin{proof} 
Using the decomposition of $A_{\lo}$ from \eqref{AH:eq-A-convenient-3}, we decompose 
\begin{equation*}
\overline{\philinear[\Blin,\lo]} \covd_{A_{\lo}} \philinear[\Blin,\lo] 
- \tfrac{1}{8\pi} A_{\lo}
= \Big( \, \overline{\philinear[\Blin,\lo]} \covd_{\Blin} \philinear[\Blin,\lo] - 
\tfrac{1}{8\pi} \Blin \Big)
+  \Big( \big( \big| \, \philinear[\Blin,\lo] \big|^2 - \tfrac{1}{8\pi} \big) \big( \linear[\lo] + \Slin_{\lo} + Z \big) \Big).
\end{equation*}
Using \ref{AH:item-probabilistic-derivative} from Hypothesis \ref{AH:hypothesis-probabilistic}, we obtain that
\begin{equation*}
\Big\| \leray \Im \Duh \Big[ \, \overline{\philinear[\Blin,\lo]} \covd_{\Blin} \philinear[\Blin,\lo] - \tfrac{1}{8\pi} \Blin \Big] \Big\|_{C_t^0 \Cs_x^{2\eta} \cap C_t^{\eta}\Cs_x^0(I)}
\lesssim \tau^{\frac{3}{4}-\eta} \big\| \Blin(t_0) \big\|_{\Cs_x^\eta}^{\frac{1}{2}+\eta} + L^{\kappa},
\end{equation*}
which is acceptable. Using the Duhamel integral estimate (Lemma \ref{prelim:lem-Duhamel-weighted}), Lemma \ref{prelim:lem-leray}, \eqref{AH:eq-L-nreg}, and the probabilistic hypothesis (Hypothesis \ref{AH:hypothesis-probabilistic}), we also obtain that 
\begin{align*}
&\, \Big\| \leray \Im \Duh \Big[ \Big( \big| \, \philinear[\Blin,\lo] \big|^2 - 
\tfrac{1}{8\pi} \Big) \big( \linear[\lo] + \Slin_{\lo} + Z \big)  \Big] \Big\|_{C_t^0 \Cs_x^{2\eta} \cap C_t^{\eta}\Cs_x^0(I)} \\ 
\lesssim&\,  \Big\| \Big( \big| \, \philinear[\Blin,\lo] \big|^2 - 
\tfrac{1}{8\pi} \Big) \big( \linear[\lo] + \Slin_{\lo} + Z \big) \Big\|_{L_t^{\frac{1}{1-2\eta}}L_x^\infty(I)} \\ 
\lesssim&\, \tau^{1-2\eta} 
\Big( \big\| \, \philinear[\Blin,\lo]\big\|_{L_{t,x}^\infty(I)}^2  + 1 \Big) \Big( \big\| \, \linear[\lo]  \big\|_{L_{t,x}^\infty(I)} + \big\| \Slin_{\lo} \big\|_{L_{t,x}^\infty(I)} + \big\|  Z \big\|_{L_{t,x}^\infty(I)} \Big) 
\lesssim \tau^{1-2\eta} L^{2\kappa} \big( L^{4\kappa} + \big\|  Z \big\|_{L_{t,x}^\infty(I)} \big), 
\end{align*}
which is acceptable.
\end{proof}

\begin{lemma}[The $\overline{\psi_{\lo}} \covd_{A_{\lo}} \psi_{\lo}$ and $\scalebox{0.9}{$\overline{\philinear[\Blin,\lo]}$} \covd_{A_{\lo}} \psi_{\lo}$-interactions]\label{AH:lem-derivative-psi}
Let the probabilistic hypothesis and continuity hypothesis, i.e., Hypothesis \ref{AH:hypothesis-probabilistic} and \ref{AH:hypothesis-continuity}, be satisfied, let $0<\tau^\ast \leq \tau$, and let $I:=[t_0,t_0+\tau^\ast]$.  Then, it holds that 
\begin{equation}\label{AH:eq-derivative-psi}
\begin{aligned}
&\, \Big\| \leray \Im \Duh \Big[ \overline{\psi_{\lo}} \covd_{A_{\lo}} \psi_{\lo} \, \Big] \Big\|_{C_t^0 \Cs_x^{2\eta} \cap C_t^\eta \Cs_x^0 (I)} 
+ \Big\| \leray \Im \Duh \Big[ \, \overline{\philinear[\Blin,\lo]} \covd_{A_{\lo}} \psi_{\lo} \Big] \Big\|_{C_t^0 \Cs_x^{2\eta} \cap C_t^\eta \Cs_x^0 (I)}  \\ 
\lesssim&\,
\tau^{\frac{1}{2}-\frac{1}{r}-2\eta}  \big\| \psi_{\lo} \big\|_{L_t^\infty L_x^r(I)}^2 
+ \tau^{1-2\eta}  \big( \big\| \psi_{\lo} \big\|_{L_t^\infty L_x^r(I)} + L^\kappa \big)
 \| Z \|_{L_t^\infty L_x^\infty(I)}^{\frac{q+1}{q} (1+\nu)}  + \tau^{\frac{1}{2}-\frac{1}{r}-2\eta} L^{2\kappaone}.  
\end{aligned}
\end{equation}
\end{lemma}

As mentioned above, the proof of Lemma \ref{AH:lem-derivative-psi} heavily relies on the localized monotonicity formula from Proposition \ref{AH:prop-psi-lo-localized}.

\begin{proof} 
Using Lemma \ref{prelim:lem-Duhamel-weighted}, Lemma \ref{prelim:lem-leray}, and $\frac{1}{r}\ll \eta$, we obtain that
\begin{equation}\label{AH:eq-derivative-psi-p1}
\begin{aligned}
&\, \Big\| \leray \Im \Duh \Big[ \overline{\psi_{\lo}} \covd_{A_{\lo}} \psi_{\lo} \, \Big] \Big\|_{C_t^0 \Cs_x^{2\eta}\cap  C_t^\eta \Cs_x^0  (I)} \\ 
\lesssim&\, \sup_{t_1\in (t_0,t_0+\tau^\ast]} \sup_{x_1\in \T^2} \big\| p^{\frac{1}{2}}(t,x;t_1,x_1) \big( \, \overline{\psi_{\lo}} \covd_{A_{\lo}} \psi_{\lo}\big)(t,x) \big\|_{L_t^{\frac{1}{1-2\eta}} L_x^{\frac{2r}{r+2}}([t_0,t_1])}.
\end{aligned}
\end{equation}
We now fix arbitrary $t_1\in (t_0,t_0+\tau^\ast]$ and $x_1\in \T^2$. Furthermore, we let $\rho:=t_1-t_0$, which clearly satisfies $\rho\leq \tau^\ast\leq \tau$. Using H\"{o}lder's inequality, it then holds that 
\begin{equation}\label{AH:eq-derivative-psi-p11}
\begin{aligned}
&\,  \big\| p^{\frac{1}{2}}(t,x;t_1,x_1) \big( \, \overline{\psi_{\lo}} \covd_{A_{\lo}} \psi_{\lo}\big)(t,x) \big\|_{L_t^{\frac{1}{1-2\eta}} L_x^{\frac{2r}{r+2}}([t_0,t_1])} \\ 
\lesssim&\, \big\| \psi_{\lo} \big\|_{L_t^{\frac{2}{1-4\eta}}L_x^r([t_0,t_1])} 
\big\| p^{\frac{1}{2}}(t,x;t_1,x_1) \big(  \covd_{A_{\lo}} \psi_{\lo}\big)(t,x) \big\|_{L_t^2 L_x^{2}([t_0,t_1])} \\
\lesssim&\, \rho^{\frac{1}{2}-2\eta} \big\| \psi_{\lo} \big\|_{L_t^{\infty} L_x^r([t_0,t_1])}
\big\| p^{\frac{1}{2}}(t,x;t_1,x_1) \big(  \covd_{A_{\lo}} \psi_{\lo}\big)(t,x) \big\|_{ L_t^2 L_x^{2}([t_0,t_1])}.    
\end{aligned}
\end{equation}
Furthermore, due to Proposition \ref{AH:prop-psi-lo-localized}, it holds that 
\begin{equation}\label{AH:eq-derivative-psi-p2}
\begin{aligned}
&\, \big\| p^{\frac{1}{2}}(t,x;t_1,x_1) \big(  \covd_{A_{\lo}} \psi_{\lo}\big)(t,x) \big\|_{L_t^2 L_x^{2}([t_0,t_1])}
\lesssim 
\rho^{-\frac{1}{r}}  \big\| \psi_{\lo} \big\|_{L_t^\infty L_x^r([t_0,t_1])} 
+ \rho^{\frac{1}{2}} \| Z \|_{L_t^\infty L_x^\infty([t_0,t_1])}^{\frac{q+1}{q} (1+\nu)} + \rho^{\frac{1}{2}} L^{\kappaone}. 
\end{aligned}
\end{equation}
By combining \eqref{AH:eq-derivative-psi-p1}, \eqref{AH:eq-derivative-psi-p11} and using Young's inequality, we then obtain that 
\begin{align*}
&\, \Big\| \leray \Im \Duh \Big[ \overline{\psi_{\lo}} \covd_{A_{\lo}} \psi_{\lo} \, \Big] \Big\|_{C_t^0 \Cs_x^{2\eta} \cap C_t^\eta \Cs_x^0 (I)}\\
\lesssim&\, \sup_{0<\rho\leq \tau} \bigg( 
\rho^{\frac{1}{2}-2\eta}  \big\| \psi_{\lo} \big\|_{L_t^\infty L_x^r(I)}  
\times \big( \rho^{-\frac{1}{r}}  \big\| \psi_{\lo} \big\|_{L_t^\infty L_x^r(I)} 
+ \rho^{\frac{1}{2}} \| Z \|_{L_t^\infty L_x^\infty(I)}^{\frac{q+1}{q} (1+\nu)} + \rho^{\frac{1}{2}} L^{\kappaone}  \big) \bigg) \\ 
\lesssim&\, \tau^{\frac{1}{2}-\frac{1}{r}-2\eta}  \big\| \psi_{\lo} \big\|_{L_t^\infty L_x^r(I)}^2 
+ \tau^{1-2\eta}  \big\| \psi_{\lo} \big\|_{L_t^\infty L_x^r(I)}
 \| Z \|_{L_t^\infty L_x^\infty(I)}^{\frac{q+1}{q} (1+\nu)}  + \tau^{1-2\eta} L^{2\kappaone}.
\end{align*}
As a result, we obtain the desired estimate of the first summand in \eqref{AH:eq-derivative-psi}. 
To bound the second summand in \eqref{AH:eq-derivative-psi}, we first recall from  Hypothesis \ref{AH:hypothesis-probabilistic}.\ref{AH:item-probabilistic-covariant-linear} that $\| \philinear[\Blin,\lo]\|_{L_t^\infty L_x^r} \lesssim \| \philinear[\Blin,\lo]\|_{L_t^\infty L_x^\infty} \lesssim L^\kappa$. Using a similar argument as above, it then follows that 
\begin{align*}
   &\,  \Big\| \leray \Im \Duh \Big[ \, \overline{\philinear[\Blin,\lo]} \covd_{A_{\lo}} \psi_{\lo} \Big] \Big\|_{C_t^0 \Cs_x^{2\eta} \cap C_t^\eta \Cs_x^0 (I)} \\ 
   \lesssim&\, \tau^{\frac{1}{2}-2\eta} \big\| \philinear[\Blin,\lo]\big\|_{L_t^\infty L_x^r(I)} \times 
\Big( \tau^{-\frac{1}{r}}  \big\| \psi_{\lo} \big\|_{L_t^\infty L_x^r(I)} 
+ \tau^{\frac{1}{2}} \| Z \|_{L_t^\infty L_x^\infty(I)}^{\frac{q+1}{q} (1+\nu)} + \tau^{\frac{1}{2}} L^{\kappaone}   \Big) \\
\lesssim&\, \tau^{\frac{1}{2}-\frac{1}{r}-2\eta} \| \psi_{\lo} \|_{L_t^\infty L_x^r(I)}^2 
+ \tau^{1-2\eta} L^{\kappa} \| Z \|_{L_t^\infty L_x^\infty(I)}^{\frac{q+1}{q} (1+\nu)} + \tau^{\frac{1}{2}-\frac{1}{r}-2\eta} L^{2\kappaone},
\end{align*}
which implies the desired estimate for the second summand in \eqref{AH:eq-derivative-psi}. 
\end{proof}

In order to simplify the notation in the next lemma, we write 
\begin{equs}\label{AH:eq-phi-lo}
\phi_{\lo} := \philinear[\Blin,\lo] + \psi_{\lo}.
\end{equs}
\begin{lemma}[Interactions involving high-frequency terms]\label{AH:lem-derivative-high}
Let the probabilistic hypothesis and continuity hypothesis, i.e., Hypothesis \ref{AH:hypothesis-probabilistic} and \ref{AH:hypothesis-continuity}, be satisfied, let $0<\tau^\ast \leq \tau$, and let $I:=[t_0,t_0+\tau^\ast]$.  Then, it holds that 
\begin{equs}\label{AH:eq-derivative-high}
\Big\| \leray \Im \Duh \big[ \, \overline{\phi} \covd_A \phi - \overline{\phi}_{\lo} \covd_{A_{\lo}} \phi_{\lo} \big] \Big\|_{C_t^0 \Cs_x^{2\eta} \cap C_t^\eta \Cs_x^0 (I)} 
\lesssim L^{10\eta_1 - \eta}.
\end{equs}
\end{lemma}
Due to the definition of $A_{\lo}$ and $\phi_{\lo}$, all interactions in \eqref{AH:eq-derivative-high} contain at least one factor of $\linear[\hi]$, $\Slin_{\hi}$, $\philinear[\hi]$, $\varphi_{\hi}$, or $\psi_{\hi}$, and can therefore be treated perturbatively. As a result, Lemma \ref{AH:lem-derivative-high} can be proven using perturbative methods from the earlier works \cite{BC23,S21}.  In order to not interrupt the flow of the main argument, we postpone the proof of Lemma \ref{AH:lem-derivative-high} until Appendix \ref{section:high}.\\

Equipped with Lemmas \ref{AH:lem-derivative-stochastic}, \ref{AH:lem-derivative-psi}, and \ref{AH:lem-derivative-high}, we can now prove the main estimate of this subsection.

\begin{proof}[Proof of Proposition \ref{AH:prop-Z}] 
We first recall from \eqref{AH:eq-decomposition-A}, \eqref{AH:eq-Z}, and \eqref{AH:eq-A-convenient-3} that the nonlinear remainder $Z$ is given by
\begin{equs}
Z = - \leray \Im \Duh \big[ \, \overline{\phi} \covd_A \phi \big] + \gaugerenorm\Duh\big[A_{\lo} + \linear[\hi] + \Slin_{\hi} \big] + \Duh \big[ \, \linear[]  \big].
\end{equs}
Using Lemma \ref{prelim:lem-Duhamel-weighted} and Hypothesis \ref{AH:hypothesis-probabilistic}, it follows that
\begin{equs}
\big\| \Duh \big[ \, \linear[]  \big] \big\|_{C_t^0 \Cs_x^{2\eta} \cap C_t^\eta \Cs_x^0 }
\lesssim \big\| \Duh \big[ \, \langle \nabla \rangle^{-2\kappa} \,  \linear[]  \big] \big\|_{C_t^0 \Cs_x^{2\eta+2\kappa} \cap C_t^\eta \Cs_x^{2\kappa}}
\lesssim \big\| \, \langle \nabla \rangle^{-2\kappa} \, \linear[]  \big\|_{L_t^{\frac{1}{1-2\eta}}L_x^\infty}
\lesssim \tau^{1-2\eta} L^{\kappa},
\end{equs}
which is acceptable. A similar argument also works for the $\Duh\big[\, \linear[\hi]\big]$-term and, due to \eqref{AH:eq-L-nreg}, also for the $\Duh\big[\Slin_{\hi} \big]$-term. 
Using our Ansatz for $\phi$ from  \eqref{AH:eq-decomposition-phi-first}, \eqref{AH:eq-decomposition-psi}, and \eqref{AH:eq-phi-lo} and using Lemma \ref{AH:lem-symmetry}, we then decompose
\begin{align*}
 \leray \Im \Big( \, \overline{\phi} \,  \covd_A \phi  \Big) - \gaugerenorm A_{\lo} &=\,   \leray \Im \Big( \, \overline{\phi_{\lo}} \,  \covd_{A_{\lo}} \phi_{\lo}  \Big) - \gaugerenorm A_{\lo}
+ \leray \Im \Big( \, \overline{\phi} \,  \covd_A \phi -  \overline{\phi_{\lo}} \,  \covd_{A_{\lo}} \phi_{\lo}  \Big) \\   &=\, \leray \Im \Big( \, \, \overline{\philinear[\Blin,\lo]} \covd_{A_{\lo}} \, \philinear[\Blin,\lo] \Big) - \gaugerenorm A_{\lo}  
+ 2 \leray \Im \Big( \, \, \overline{\philinear[\Blin,\lo]} \covd_{A_{\lo}} \, \psi_{\lo} \Big) 
\\
&\quad+ \leray \Im \Big( \, \overline{\psi_{\lo}} \,  \covd_{A_{\lo}} \psi_{\lo}  \Big)  
+ \leray \Im \Big( \, \overline{\phi} \,  \covd_A \phi -  \overline{\phi_{\lo}} \,  \covd_{A_{\lo}} \phi_{\lo}  \Big). 
\end{align*}
By combining the estimates from Lemma \ref{AH:lem-derivative-stochastic}, Lemma \ref{AH:lem-derivative-psi}, and Lemma \ref{AH:lem-derivative-high}, we then obtain the desired estimate \eqref{AH:eq-Z-estimate} of $Z$. 
\end{proof}

\section{Decay estimates for Abelian-Higgs equations}\label{section:decay}

In this section, we obtain decay estimates for the stochastic Abelian-Higgs equations. In Subsection \ref{section:decay-short}, we obtain decay estimates on an admissible time-scale $\tau$, which is carefully chosen depending on the size of the initial data. In the proof of the decay estimate on admissible time-scales, we heavily rely on the covariant estimates from Section \ref{section:Abelian-Higgs}. In Subsection \ref{section:decay-unit}, we obtain a decay estimate on unit time-scales by iterating our earlier decay estimate for admissible time-scales. Finally, in Subsection \ref{section:decay-proof}, we prove our main theorem (Theorem \ref{intro:thm-abelian-higgs}).

\subsection{Decay estimate on admissible time-scales}\label{section:decay-short}

In this subsection, we use the notation from Section \ref{section:Abelian-Higgs}, i.e., we use the terms from \eqref{AH:eq-A-linear}, \eqref{AH:eq-Blin}, \eqref{AH:eq-Slin}, \eqref{AH:eq-Z}, 
\eqref{AH:eq-philinear-lo}, \eqref{AH:eq-philinear-hi}, and \eqref{AH:eq-decomposition-psi}.
In order to obtain the gauge invariant, uniform-in-time bounds on the connection one-form $A$ from Theorem \ref{intro:thm-abelian-higgs}.\ref{intro:item-AH-2}, we need to use the decay estimate for the heat flow acting on mean-zero functions from \eqref{prelim:eq-heat-flow-decay}. To this end, we need separate the mean and mean-zero components of $A$ and, as in \eqref{prelim:eq-mean-mean-zero}, we therefore write 
\begin{equation}\label{decay:eq-A-mean-mean-zero}
\sfint A(t) := \frac{1}{(2\pi)^2} \int \dx \, A(t,x) 
\qquad \text{and} \qquad 
\bigdot{A}(t) := A(t) - \sfint A(t). 
\end{equation}
 In addition, due to Remark \ref{prelim:rem-decay-fail}, we need to abandon the $\Cs_x^\alpha$-norms which were used in Section \ref{section:Abelian-Higgs}, and instead work with $\Balpha$-norms. Since there is leeway within the regularity parameters in Proposition \ref{AH:prop-Z}, the switch from $\Cs_x^\alpha$-norms to $\Balpha$-norms will not cause any problems. \\ 

In the following definition, we introduce the parameters $\alpha,\beta$, and $\gamma$ and admissible time-scales $\tau$, which will be important for the rest of this section.

\begin{definition}[The parameters $\alpha$, $\beta$, and $\gamma$ and admissible time-scales $\tau$]\label{decay:def-tau}
We first define parameters $\alpha,\beta$, and $\gamma$ as 
\begin{equation}
\alpha := \frac{3}{4}, \qquad 
\beta := \frac{5}{2}, \qquad \text{and} \qquad 
\gamma = \frac{2}{7}. 
\end{equation}
Then, a time-scale $\tau>0$ is called admissible if the lower and upper bounds
\begin{equation}\label{decay:eq-tau}
\tfrac{1}{4} c_3 \max \Big( \big\| \pregA_0 \big\|_{\Beta}^{\alpha}, \big\| \pregphi_0 \big\|_{L_x^r}^\beta, L^{\kappathree} 
\Big)^{-1}\leq  \tau \leq 4 c_3 \max \Big( \big\| \pregA_0 \big\|_{\Beta}^{\alpha}, \big\| \pregphi_0 \big\|_{L_x^r}^\beta, L^{\kappathree} 
\Big)^{-1}
\end{equation}
are satisfied. Furthermore, a time-scale $\tau>0$ is called upper-admissible if (only) the upper bound is satisfied. 
\end{definition}

While the chosen values of $\alpha$, $\beta$, and $\gamma$ from Definition \ref{decay:def-tau} are rather simple, the choice itself is quite difficult. In the proof of Proposition \ref{decay:prop-short} and in Remark \ref{decay:rem-parameters}, we will see that our argument relies on several conditions involving $\alpha$, $\beta$, and $\gamma$. For expository purposes, we thus rarely insert the explicit values from Definition \ref{decay:def-tau}, and instead keep writing $\alpha$, $\beta$, and $\gamma$. \\

Equipped with \eqref{decay:eq-A-mean-mean-zero} and Definition \ref{decay:def-tau}, we can now state the main estimate of this subsection.

\begin{proposition}[Decay estimate on admissible time-scales]\label{decay:prop-short}
Let $q\geq 3$, let $t_0\in [0,\infty)$, let the initial data be as in \eqref{AH:eq-initial}, 
let $L\in \dyadic$, let $\tau$ be admissible (as defined in Definition \ref{decay:def-tau}), and let 
\begin{equation}\label{decay:eq-short-new}
\pregA := A(t_0+\tau) - \linear[](t_0+\tau) 
\qquad \text{and} \qquad 
\pregphi := \phi(t_0+\tau) - \philinear(t_0+\tau). 
\end{equation}
In addition, let
\eqref{AH:eq-L}, \eqref{AH:eq-L-nreg}, and Hypothesis \ref{AH:hypothesis-probabilistic} be satisfied. Then, we have the estimates 
\begin{align}
\| \pregphi \|_{L_x^r}&\leq 
e^{-c_0 \tau} \max \Big( \| \pregA_0 \|_{\Beta}^\gamma,  \| \pregphi_0 \|_{L_x^r}, C_4 L^{\kappafour} \Big), 
\label{decay:eq-short-phi} \\
\| \pregA \|_{\Beta}^\gamma &\leq 
e^{c_0 \tau} \max \Big(  \| \pregA_0 \|_{\Beta}^\gamma, \| \pregphi_0 \|_{L_x^r}, C_4 L^{\kappafour} \Big). \label{decay:eq-short-A} 
\end{align}
Furthermore, we have the refined estimates 
\begin{align}
\| \pregA \|_{\Beta}^\gamma &\leq e^{-c_0 \tau} \max \Big( e^{2 c_0 \tau} |\sfint \pregA_0 |^\gamma, \| \pregA_0 \|_{\Beta}^\gamma,  \| \pregphi_0 \|_{L_x^r}, C_4 L^{\kappafour} \Big), \label{decay:eq-short-A-refined} \\
| \sfint \pregA |^{\gamma} &\leq 
e^{c_0 \tau} \max\Big( |\sfint \pregA_0 |^\gamma, \| \pregA_0 \|_{\Beta}^{(1-\nu)\gamma},  \| \pregphi_0 \|_{L_x^r}^{1-\nu}, \frac{3}{4} C_4 L^{\kappafour}  \Big). \label{decay:eq-short-A-mean} 
\end{align}
Finally, we have the $L_t^\infty$-estimate 
\begin{equation}\label{decay:eq-short-Linfty}
\max \Big(  \| A \|_{L_t^\infty \Cs_x^{-\kappa}([t_0,t_0+\tau])}^\gamma, \| \phi \|_{L_t^\infty \GCs^{-\kappa}([t_0,t_0+\tau])} \Big)
\leq C_0  \max \Big(  \| \pregA_0 \|_{\Beta}^\gamma, \| \pregphi_0 \|_{L_x^r}, C_4 L^{\kappafour} \Big). 
\end{equation}
\end{proposition}

The proof of Proposition \ref{decay:prop-short} will occupy the remainder of this subsection. Before we turn to the proof, we make several remarks.

\begin{remark}\label{decay:rem-result} 
We make the following remarks regarding Proposition \ref{decay:prop-short}. 
\begin{enumerate}[label=(\roman*)]
\item\label{decay:item-exponential-factor} (Exponential factors) We recall that $\pregA$ contains the linear heat flow $\Blin(t_0+\tau)=e^{\tau\Delta} \pregA_0$. If the mean of $\pregA_0$ is small compared to the $\Beta$-norm of $\pregA_0$, it follows from Lemma \ref{prelim:lem-heat-flow-bound-decay} that $\| \Blin(t_0+\tau)\|_{\Beta}$ can be bounded by $e^{-c\tau} \| \pregA_0\|_{\Beta}$, where $c$ is a constant depending only on $r$. Since the decay of $\pregA$ only stems from the linear heat flow, this explains the $e^{-c_0 \tau} \| \pregA_0\|_{\Beta}^\gamma$-term in \eqref{decay:eq-short-A-refined}. 
The remaining $e^{-c_0 \tau}$ and 
$e^{c_0 \tau}$-factors in \eqref{decay:eq-short-phi}, 
\eqref{decay:eq-short-A},
\eqref{decay:eq-short-A-refined}, 
and \eqref{decay:eq-short-A-mean} can likely be replaced by other expressions, but the $e^{-c_0 \tau}$ and 
$e^{c_0 \tau}$-factors are very convenient in the proofs of Lemma \ref{decay:lem-exponential-growth} and Lemma \ref{decay:lem-decay-unit}, since they allow us to easily iterate Proposition \ref{decay:prop-short}. 
\item\label{decay:item-mean} (Mean of $\pregA$) Loosely speaking, \eqref{decay:eq-short-A-mean} states that, on its own, the mean grows at most exponentially in time. Furthermore, it states that the impact of $\bigdotpregA_0$ and $\pregphi_0$ on $\sfint \pregA$ can be controlled using a slightly smaller exponent than in \eqref{decay:eq-short-phi}, \eqref{decay:eq-short-A}, and \eqref{decay:eq-short-A-refined}. 
\item\label{decay:item-phi-psi} (Comparing $\pregphi$ and $\psi_{\lo}$) From \eqref{AH:eq-decomposition-phi-first}, 
\eqref{AH:eq-decomposition-psi}, and \eqref{decay:eq-short-new}, it follows that 
\begin{equation*}
\pregphi = \big( \philinear[\Blin,\lo] - \philinear[\lo]\big)(t_0+\tau) + \psi_{\lo}(t_0+\tau) + \psi_{\hi}(t_0+\tau) + \varphi_{\hi}(t_0+\tau). 
\end{equation*}
As we will see in the proof of Proposition \ref{decay:prop-short}, it will be most important to control $\psi_{\lo}(t_0+\tau)$, since the contributions of $\philinear[\Blin,\lo](t_0+\tau)$, $\philinear[\lo](t_0+\tau)$,  $\psi_{\hi}(t_0+\tau)$, and $\varphi_{\hi}(t_0+\tau)$ to the $L_x^r$-norm will be of size $\lesssim L^{2\kappa}$. Thus, the reader may think of $\pregphi$ as being essentially equal to $\psi_{\lo}(t_0+\tau)$.
\item\label{decay:item-coming-down} (Coming down from infinity) For $\phi^4_d$-models, it can be shown that the scalar field comes down from infinity, see e.g. \cite{MW17,MW2020}. In comparison, the stated estimate for the scalar field in \eqref{decay:eq-short-phi} is rather weak, and this has two reasons: First, depending on the relative size of $\pregA_0$ and $\pregphi_0$, the time-scale $\tau$ can be small compared to inverse powers of $\|\psi_{\lo}(t_0)\|_{L_x^r}$. As explained in Remark \ref{decay:rem-ODE} below, there may therefore not be sufficient time for $\psi_{\lo}$ to come down from infinity. Second, as we discussed in \ref{decay:item-exponential-factor}, the estimate \eqref{decay:eq-short-phi} can easily be iterated in time, and this makes \eqref{decay:eq-short-phi} more convenient than more intricate estimates.
\item\label{decay:item-q} (Impact of $q$) It is far from obvious that our previous estimates \eqref{AH:eq-psi-lo-estimate} and \eqref{AH:eq-Z-estimate} can be used to prove the decay estimate \eqref{decay:eq-short-phi}-\eqref{decay:eq-short-A-mean} for small values of $q$, such as $q=3$ or $q=5$. For small values of $q$, both the choice of the parameters in Definition \ref{decay:def-tau} and the proofs of \eqref{decay:eq-short-phi}-\eqref{decay:eq-short-A-mean} are rather delicate. However, it is relatively clear that \eqref{AH:eq-psi-lo-estimate} and \eqref{AH:eq-Z-estimate} can be used to prove \eqref{decay:eq-short-phi}-\eqref{decay:eq-short-A-mean} for sufficiently large values of $q$, such as $q\geq 11$. For large values of $q$, the $\| \psi_{\lo}(t)\|_{L_x^{r+q-1}}^{r+q-1}$-term in \eqref{AH:eq-psi-lo-estimate} leads to a sharp decrease of $\|\psi_{\lo}(t)\|_{L_x^r}^r$. Furthermore, the $(2/q)$-factor in the exponent of the $\| Z\|_{C_t^0 \Cs_x^\eta}$-term in \eqref{AH:eq-psi-lo-estimate} implies that bootstrap arguments involving both \eqref{AH:eq-psi-lo-estimate} and \eqref{AH:eq-Z-estimate} and easily be closed. 
\end{enumerate}
\end{remark}

We also make a remark concerning the ordinary differential equation $y^\prime=-y^q$, which will serve as a guiding principle during the proof of Proposition \ref{decay:prop-short}.

\begin{remark}[On $y^\prime=-y^q$]\label{decay:rem-ODE} 
Since the decay of $\psi_{\lo}$ is driven by the power-type nonlinearity in \eqref{AH:eq-evolution-phi}, it is instructive to consider the ordinary differential equation $y^\prime(t)=-y(t)^q$, where $t\geq 0$ and $y(0)=y_0 \geq 1$. The solution $y(t)$ is explicitly given by 
\begin{equation*}
y(t) =  y_0 \Big( (q-1) y_0^{q-1} t +1 \Big)^{-\frac{1}{q-1}}.  
\end{equation*}
Due to Definition \ref{decay:def-tau}, we now set $t=y_0^{-\beta}$, where $0\leq \beta <\infty$. At a heuristic level, one may then use the following description of $y(t)$:
\begin{equation}\label{decay:eq-ODE}
y(t) \simeq 
\begin{cases}
\begin{tabular}{ll}
$y_0 - \mathcal{O}(1)\hspace{8ex}$                     &\text{if }    $q < \beta <\infty$,  \\
$y_0-y_0^{q-\beta}$       &\text{if }             $q-1<\beta\leq q$,          \\
$e^{-\frac{\ln(q)}{q-1}} y_0$                     &\text{if }  $\beta=q-1$,          \\
$e^{-\frac{\ln(q-1)}{q-1}}  y_0^{\frac{\beta}{q-1}}$                     &\text{if }  $0<\beta<q-1$, \\
$e^{-\frac{\ln(q-1)}{q-1}}$                     &\text{if } $\beta=0$.
\end{tabular}
\end{cases}
\end{equation}
The $q$-dependent pre-factors in \eqref{decay:eq-ODE} are not important for our analysis and can safely be ignored. According to \eqref{decay:eq-ODE}, if $\beta>q$, then $y(t)$ has not decreased by more than a term of order one. If $q-1<\beta\leq q$, then $y(t)$ has decreased by much more than one, but the decrease is much smaller than $y_0$. If $\beta=q-1$, then $y(t)$ has decreased by a constant factor. If $0<\beta<q-1$, then $y(t)$ has changed its magnitude. And finally, if $\beta=0$, then $y(t)$ has come down from infinity. 
\end{remark}

Before we start our estimates, we make two elementary observations regarding admissible time-scales $\tau$. Due to the condition \eqref{AH:eq-L-nreg}, it holds that 
\begin{equation*}
\big| \| \psi_{\lo}(t_0) \|_{L_x^r} - \| \pregphi_0 \|_{L_x^r} \big| 
\leq \| P_{\leq L} \nregphi_0 \|_{L_x^r} \lesssim L^{4\kappa}.
\end{equation*}
Since $\tau$ is admissible, it then directly follows that 
\begin{equation}\label{decay:eq-tau-equivalence}
\tau \sim c_3 \max \Big( \big\| \pregA_0 \big\|_{\Beta}^{\alpha}, \big\| \psi_{\lo}(t_0) \big\|_{L_x^r}^\beta, L^{\kappathree} 
\Big)^{-1}.
\end{equation}
In other words, the $\|\pregphi_0\|_{L_x^r}$-term in Definition \ref{decay:def-tau} can safely be replaced by $\| \psi_{\lo}(t_0)\|_{L_x^r}$. Furthermore, due to \eqref{prelim:eq-parameter-new-kappa-kappa-j}, \eqref{AH:eq-L}, and $\alpha,\beta\leq 20$, it holds that 
\begin{equation}\label{decay:eq-tau-eta}
\tau^{-\eta} \lesssim \max\big( \big\| \pregA_0 \big\|_{\Beta}^{20}, \big\| \pregphi_0 \big\|_{L_x^r}^{20}, L^{\kappathree} \big)^{\eta} 
\lesssim \max \big( L^{10 \eta \eta_3} , L^{\eta \kappathree} \big) = \max \big( L^{\frac{\kappaone}{10}}, L^{\eta \nu^{-20} \kappaone} \big) 
= L^{\frac{\kappaone}{10}}.
\end{equation}
Due to \eqref{decay:eq-tau-eta}, we can later afford to lose $\tau^{-\eta}$-factors, see e.g. the proof of Lemma \ref{decay:lem-exponential-growth}. Before we prove Proposition \ref{decay:prop-short}, we now first prove the following weaker estimate. 

\begin{lemma}[Bounds on upper-admissible time-scales]\label{decay:lem-short-bounds}
Let the same assumptions as in Proposition \ref{decay:prop-short} be satisfied, except that the time-scale $\tau>0$ only needs to be upper-admissible. Then, it holds that 
\begin{align}
\big\| \psi_{\lo} \big\|_{C_t^0 L_x^r([t_0,t_0+\tau])}
&\leq \big\| \psi_{\lo}(t_0) \big\|_{L_x^r} + C_2 L^{\kappatwo}, \label{decay:eq-bounds-psi-lo}\\
\big\| \psi_{\hi} \big\|_{C_t^0 \Cs_x^\eta([t_0,t_0+\tau])} &\leq  L^{\eta_1-\eta} \label{decay:eq-bounds-psi-hi}, \\ 
\big\| Z \big\|_{C_t^0 \Cs_x^{2\eta} \cap C_t^{\eta} \Cs_x^0([t_0,t_0+\tau])}
&\leq \big\| \psi_{\lo}(t_0) \big\|_{L_x^r}^{2-\frac{\beta}{2}+\nu} 
+ C_2 L^{\kappatwo}. \label{decay:eq-bounds-Z}
\end{align}
\end{lemma}

We emphasize that the $\nu$-loss in the exponent in \eqref{decay:eq-bounds-Z} will not cause any problems in subsequent estimates, and can therefore be safely ignored.

\begin{proof}
Let $\tau^\ast \in [0,\tau]$ be defined as 
\begin{align*}
\tau^\ast 
:= \inf \Big\{&\,  
\tau^\prime \in [0,\tau] \colon 
\big\| \psi_{\lo} \big\|_{C_t^0 L_x^r([t_0,t_0+\tau^\prime])}
\geq \big\| \psi_{\lo}(t_0) \big\|_{L_x^r} + C_2 L^{\kappatwo}, \,  \big\|\psi_{\hi} \big\|_{C_t^0 \Cs_x^\eta([t_0,t_0+\tau^\prime])} \geq L^{\eta_1-\eta}, \\ 
&\, \big\| Z \big\|_{C_t^0 \Cs_x^{2\eta} \cap C_t^{\eta} \Cs_x^0([t_0,t_0+\tau^\prime])} 
\geq \big\| \psi_{\lo}(t_0) \big\|_{L_x^r}^{2-\frac{\beta}{2}+\nu } + C_2 L^{\kappatwo},  \text{ or } \tau^\prime = \tau \Big\}.
\end{align*}
Our goal is to show that $\tau^\ast=\tau$, which then implies the desired estimates \eqref{decay:eq-bounds-psi-lo}, \eqref{decay:eq-bounds-psi-hi}, and \eqref{decay:eq-bounds-Z}. Using a standard continuity argument (see e.g. \cite[Section 1.3]{Tao06}), it then suffices to prove that  
\begin{align}
\big\| \psi_{\lo} \big\|_{C_t^0 L_x^r([t_0,t_0+\tau^\ast])}
&\leq \big\| \psi_{\lo}(t_0) \big\|_{L_x^r} + \frac{1}{2} C_2 L^{\kappatwo},\label{decay:eq-bounds-p1} \\ 
\big\| \psi_{\hi} \big\|_{C_t^0 \Cs_x^\eta([t_0,t_0+\tau^\ast])}
&\leq \frac{1}{2} L^{\eta_1-\eta}, \label{decay:eq-bounds-q1} \\ 
\big\| Z \big\|_{C_t^0 \Cs_x^{2\eta} \cap C_t^{\eta} \Cs_x^0([t_0,t_0+\tau^\ast])} 
&\leq \big\| \psi_{\lo}(t_0) \big\|_{L_x^r}^{2-\frac{\beta}{2}+\nu } + \frac{1}{2} C_2 L^{\kappatwo}.\label{decay:eq-bounds-p2} 
\end{align} 
We emphasize that the bounds in \eqref{decay:eq-bounds-p1}, \eqref{decay:eq-bounds-q1}, and \eqref{decay:eq-bounds-p2} are stronger than the corresponding bounds in Hypothesis \ref{AH:hypothesis-continuity}, i.e., our previous continuity hypothesis. The proofs of \eqref{decay:eq-bounds-p1}, \eqref{decay:eq-bounds-q1}, and \eqref{decay:eq-bounds-p2} are distributed over the following three steps. \\

\emph{Step 1: Estimate of $\psi_{\lo}$.} 
Using Proposition \ref{AH:prop-psi-lo-estimate} and the definition of $\tau^\ast$, we obtain for all $t\in [t_0,t_0+\tau^\ast]$ that
\begin{equation}\label{decay:eq-bounds-p3}
\begin{aligned}
\frac{1}{r} \frac{\mathrm{d}}{\mathrm{d}t} \| \psi_{\lo}(t)\|_{L_x^r}^r 
&\leq - \frac{1}{2} \| \psi_{\lo}(t) \|_{L_x^{r+q-1}}^{r+q-1}
+ C_1 \| Z \|_{C_t^0 \Cs_x^{2\eta}([t_0,t_0+\tau^\ast])}^{\frac{2}{q} (1+\nu) (r+q-1)} + C_1 L^{\kappaone (r+q-1)} \\
&\leq  - \frac{1}{2} \| \psi_{\lo}(t) \|_{L_x^{r+q-1}}^{r+q-1}
+ C_1 \Big( \| \psi_\lo(t_0) \|_{L_x^r}^{2-\frac{\beta}{2}+\nu } + C_2 L^{\kappatwo} \Big)^{\frac{2}{q} (1+\nu) (r+q-1)} + C_1 L^{\kappaone (r+q-1)}.
\end{aligned}
\end{equation}
Now, let $t_{\max}\in [t_0,t_0+\tau^\ast]$ be chosen such that
\begin{equation}\label{decay:eq-bounds-p4}
\| \psi_{\lo}(t_{\max})\|_{L_x^r}
= \| \psi_{\lo} \|_{C_t^0 L_x^r([t_0,t_0+\tau^\ast])}.
\end{equation}
In the case $t_{\max}=t_0$, it holds that 
$\| \psi_{\lo} \|_{C_t^0 L_x^r([t_0,t_0+\tau^\ast])}=\| \psi_{\lo}(t_0)\|_{L_x^r}$, which directly implies \eqref{decay:eq-bounds-p1}. In the case $t_{\max}>t_0$, the right-hand side of \eqref{decay:eq-bounds-p3} has to be non-negative at $t=t_{\max}$. Using \eqref{decay:eq-bounds-p3}, it follows that 
\begin{equation}\label{decay:eq-bounds-p5}
\begin{aligned}
\| \psi_{\lo}(t_{\max})\|_{L_x^r}
\lesssim \| \psi_{\lo}(t_{\max})\|_{L_x^{r+q-1}}
&\lesssim_{C_1} \| \psi_{\lo}(t_0)\|_{L_x^r}^{ \frac{2}{q} (2-\frac{\beta}{2}+\nu)  (1+\nu)} 
+ \big( C_2 L^{\kappatwo} \big)^{\frac{2}{q} (1+\nu)}
+ L^{\kappaone}.
\end{aligned}
\end{equation}
Since 
\begin{equation}\label{decay:eq-bounds-parameter-1}
\frac{2}{q} \Big( 2 - \frac{\beta}{2} \Big) + \mathcal{O}(\nu) < 1, \qquad \frac{2}{q} + \mathcal{O}(\nu) < 1, 
\end{equation}
and $C_2$ has been chosen as large depending on $C_1$ and the parameters in \eqref{prelim:eq-parameter-new-eta-nu}-\eqref{prelim:eq-parameter-new-r}, the desired estimate \eqref{decay:eq-bounds-p1} can then be obtained from Young's inequality.\\

\emph{Step 2: Estimate of $\psi_{\hi}$.} 
The estimate \eqref{decay:eq-bounds-q1} follows directly from Corollary \ref{high:cor-psi-hi}, since the bounds on $\psi_{\lo}$ and $Z$ in \eqref{decay:eq-bounds-p1} and \eqref{decay:eq-bounds-p2} are stronger than the bounds in Hypothesis \ref{AH:hypothesis-continuity}. \\

\emph{Step 3: Estimate of $Z$.} In the following argument, all factors of $L^\kappa$, $L^\kappaone$, or $L^\kappatwo$ can easily be absorbed using a factor of $\tau^{\frac{\nu}{10}}$, i.e., a small power of $\tau$. On first reading, we therefore encourage the reader to ignore all factors of  $L^\kappa$, $L^\kappaone$, or $L^\kappatwo$ below. 
Using Proposition \ref{AH:prop-Z}, we obtain that
\begin{equation}\label{decay:eq-bounds-p6}
\begin{aligned}
 \big\| Z \big\|_{C_t^0 \Cs_x^{2\eta} \cap C_t^{\eta} \Cs_x^0([t_0,t_0+\tau^\ast])}
&\leq C_1 \tau^{\frac{3}{4}-\eta} \big\| \pregA_0 \big\|_{\Beta}^{\frac{1}{2}+\eta} +  C_1 \tau^{\frac{1}{2}-\frac{1}{r}-2\eta} \big\| \psi_{\lo} \big\|_{C_t^0 L_x^r([t_0,t_0+\tau^\ast])}^2  \\ 
&+ C_1 \tau^{1-2\eta} \big( \big\| \psi_{\lo} \big\|_{C_t^0 L_x^r([t_0,t_0+\tau^\ast])} + L^\kappa \big) \big\| Z \big\|_{C_t^0 \Cs_x^{2\eta}([t_0,t_0+\tau^\ast])}^{\frac{q+1}{q} (1+\nu)} 
+ C_1 L^{2\kappaone}. 
\end{aligned}
\end{equation}
Using that $\tau$ is upper-admissible (see Definition \ref{decay:def-tau}) and using the definition of $\tau^\ast$, we can now estimate the terms in \eqref{decay:eq-bounds-p6} as follows: For the first term, we have 
\begin{align*}
 \tau^{\frac{3}{4}-\eta} \big\| \pregA_0 \big\|_{\Beta}^{\frac{1}{2}+\eta}
\lesssim c_3^{\frac{3}{4}-\eta} \max\big( \big\| \pregA_0 \big\|_{\Beta}^{\alpha},L^{\kappathree}\big)^{-\frac{3}{4}+\eta} 
\big\|  \pregA_0 \big\|_{\Beta}^{\frac{1}{2}+\eta}\lesssim c_3^{\frac{3}{4}-\eta}. 
\end{align*}
In the last inequality, we used the parameter condition 
\begin{equation}\label{decay:eq-bounds-parameter-2}
- \tfrac{3}{4} \alpha + \tfrac{1}{2} + \mathcal{O}(\nu)\leq 0. 
\end{equation}
For the second term, we have that
\begin{align*}
 \tau^{\frac{1}{2}-\frac{1}{r}-2\eta} \big\| \psi_{\lo} \big\|_{C_t^0 L_x^r([t_0,t_0+\tau^\ast])}^2 
&\lesssim c_3^{\frac{1}{2}-\frac{1}{r}-2\eta} 
\max \big( \| \psi_{\lo}(t_0) \|_{L_x^r}^{\beta}, L^{\kappathree} \big)^{-\frac{1}{2}+\frac{1}{r}+2\eta} \max\big(\| \psi_{\lo}(t_0) \|_{L_x^r}, C_2 L^{\kappatwo} \big)^2 \\ 
& \lesssim C_2^2 c_3^{\frac{1}{2}-\frac{1}{r}-2\eta} 
\max \Big( \| \psi_{\lo}(t_0) \|_{L_x^r}^{2-\frac{\beta}{2}+\nu} , L^{-\frac{\nu \kappathree}{100}} \Big).
\end{align*}
In the last inequality, we used that $\nu \kappathree \gg \kappa_2$ and that $\frac{1}{r},\eta \ll \nu$. For the third term, we have that 
\begin{align*}
&\, \tau^{1-2\eta} \max \big( \big\| \psi_{\lo} \big\|_{C_t^0 L_x^r([t_0,t_0+\tau^\ast])},  L^\kappa \big) \big\| Z \big\|_{C_t^0 \Cs_x^{2\eta}([t_0,t_0+\tau^\ast])}^{\frac{q+1}{q} (1+\nu)} \\
\lesssim&\,  c_3^{1-2\eta} 
 \max \big( \| \psi_{\lo}(t_0) \|_{L_x^r}^{\beta}, L^{\kappathree} \big)^{-1+2\eta} 
\max \big( \| \psi_{\lo}(t_0) \|_{L_x^r},  C_2 L^{\kappatwo} \big) 
\max \big( \| \psi_{\lo}(t_0) \|_{L_x^r}^{2-\frac{\beta}{2}+\nu}, C_2 L^{\kappatwo} \big)^{\frac{q+1}{q} (1+\nu)}. 
\end{align*}
Using the parameter condition 
\begin{equation}\label{decay:eq-bounds-parameter-3}
- \beta + 1 + \tfrac{q+1}{q} \big( 2- \tfrac{\beta}{2} \big) + \mathcal{O}(\nu) \leq 2 - \tfrac{\beta}{2}
\end{equation}
and $\nu \kappathree \gg \kappatwo$, this also yields an acceptable contribution. Finally, for the fourth term, it clearly holds that $ L^{2\kappaone}\leq L^{\kappatwo}$, which yields an acceptable contribution. 
This completes our estimates of the terms in \eqref{decay:eq-bounds-p6}, and hence completes our proof of \eqref{decay:eq-bounds-p2}.
\end{proof}

\begin{remark}[Conditions on $\alpha$ and $\beta$ from the proof of Lemma \ref{decay:lem-short-bounds}]\label{decay:rem-parameters-first}
In the following discussion, we omit all $\mathcal{O}(\nu)$-terms. In \eqref{decay:eq-bounds-parameter-2}, we imposed the condition
\begin{equation}\label{decay:eq-alpha-twothirds}
\alpha > \frac{2}{3},
\end{equation}
which was used to bound $\tau^{\frac{3}{4}} \| \pregA_0\|_{\Beta}^{\frac{1}{2}}\lesssim 1$. As a result of \eqref{decay:eq-alpha-twothirds}, the bounds of $\psi_{\lo}$ and $Z$ in Lemma \ref{decay:lem-short-bounds} do not directly\footnote{The bounds of $\psi_{\lo}$ and $Z$ only depend on $\|\pregA_0\|_{\Beta}$ through the $L^{\kappatwo}$-term. Due to our choice of $L$ in  \eqref{AH:eq-L}, the bounds therefore only involve a tiny power of  $\| \pregA_0\|_{\Beta}$.} depend on $\| \pregA_0\|_{\Beta}$. While this is not absolutely necessary for our argument, it simplifies the numerology in the proofs of Proposition \ref{decay:prop-short} and Lemma \ref{decay:lem-short-bounds} significantly. 
In \eqref{decay:eq-bounds-parameter-1} and \eqref{decay:eq-bounds-parameter-3}, we imposed the condition
\begin{equation}\label{decay:eq-beta-mild}
\beta > \max \Big( \frac{q+2}{q+\frac{1}{2}}, 4-q \Big), 
\end{equation}
which is necessary for our continuity argument. Since $q\geq 3$, it is sufficient to impose $\beta>\frac{10}{7}$, which does not lead to any major limitations. 
\end{remark}

\begin{proof}[Proof of Proposition \ref{decay:prop-short}]
In the following proof, our estimates rely on several conditions involving the parameters $\alpha$, $\beta$, and $\gamma$, which will be stated explicitly in \eqref{decay:proof-parameter-1}, \eqref{decay:proof-parameter-2}, \eqref{decay:proof-parameter-3}, \eqref{decay:proof-parameter-4}, and \eqref{decay:proof-parameter-5} below. Using the explicit values of $\alpha$, $\beta$, and $\gamma$ from Definition \ref{decay:def-tau}, each of the conditions can be verified using basic arithmetic. For expository purposes, we split the proof into several steps. \\

\emph{First step: Preparations.}
We first show that 
\begin{align}
 \big\| \pregA - \big( \Blin(t_0+\tau) + Z(t_0+\tau) \big) \big\|_{\Beta}  &\lesssim L^{\kappaone}, \label{decay:eq-prep-1} \\ 
 \big\| \pregphi - \psi_{\lo}(t_0+\tau) \big\|_{L_x^r} &\lesssim L^{\kappaone}, \label{decay:eq-prep-2} \\ 
 \big\| \pregphi_0 - \psi_{\lo}(t_0) \big\|_{L_x^r} &\lesssim L^{\kappaone}. \label{decay:eq-prep-3}
\end{align}
In order to prove \eqref{decay:eq-prep-1}, we use Lemma \ref{lemma:heat-flow-smoothing} and \eqref{AH:eq-L-nreg}, which yield that
\begin{equation*}
\big\| \pregA - \big( \Blin(t_0+\tau) + Z(t_0+\tau) \big) \big\|_{\Beta} 
= \big\| \Slin(t_0+\tau) \big\|_{\Beta} \lesssim \tau^{-(\eta+\kappa)} \big\| \nregA_0 \big\|_{\Cs_x^{-\kappa}} \lesssim \tau^{-(\eta+\kappa)} L^{\kappa}. 
\end{equation*}
Together with \eqref{decay:eq-tau-eta}, this directly yields \eqref{decay:eq-prep-1}. In order to prove \eqref{decay:eq-prep-2}, we first recall that 
\begin{equation*}
\pregphi - \psi_{\lo}(t_0+\tau) = \big( \philinear[\Blin,\lo] - \philinear[\lo] \big)(t_0+\tau) + \varphi_{\hi}(t_0+\tau) + \psi_{\hi}(t_0+\tau). 
\end{equation*}
Using Lemma \ref{lemma:heat-flow-smoothing}, \eqref{AH:eq-L-nreg}, Hypothesis \ref{AH:hypothesis-probabilistic}, and Lemma \ref{decay:lem-short-bounds}, it follows that 
\begin{alignat*}{3}
\big\| \philinear[\Blin,\lo] (t_0+\tau) \big\|_{L_x^\infty} &\lesssim L^{\kappa}, \qquad & \qquad 
\big\| \philinear[\lo] (t_0+\tau) \big\|_{L_x^\infty} &\lesssim L^{\kappa}, \\ 
\big\| \varphi_{\hi}(t_0+\tau) \big\|_{L_x^\infty}&\lesssim \tau^{-\kappa} \big\| \nregphi_0 \|_{\Cs_x^{-\kappa}} 
\lesssim \tau^{-\kappa} L^{\kappa}, \qquad &\text{and} \qquad \big\| \psi_{\hi}(t_0+\tau) \big\|_{L_x^\infty} &\lesssim L^{\eta_1-\eta}. 
\end{alignat*}
From \eqref{decay:eq-tau-eta}, it directly follows that $\tau^{-1}\lesssim L$, which then implies \eqref{decay:eq-prep-2}. Finally, in order to prove \eqref{decay:eq-prep-3}, we simply use \eqref{AH:eq-L-nreg}. Thus, this completes the proof of \eqref{decay:eq-prep-1}, \eqref{decay:eq-prep-2}, and \eqref{decay:eq-prep-3}. 
In order to prove \eqref{decay:eq-short-phi}, \eqref{decay:eq-short-A}, \eqref{decay:eq-short-A-refined}, and \eqref{decay:eq-short-A-mean}, it then suffices to prove that\footnote{In our reduction to \eqref{decay:eq-short-main-estimate-e1}-\eqref{decay:eq-short-main-estimate-e4}, we do not use \eqref{decay:eq-prep-1}, since it would increase the length of the inequalities. Instead, \eqref{decay:eq-prep-1} will only be used in the proofs of \eqref{decay:eq-short-main-estimate-e1}-\eqref{decay:eq-short-main-estimate-e4}.} 
\begin{align}
 \| \psi_{\lo}(t_0+\tau) \|_{L_x^r}+ C_1 L^{\kappaone}
&\leq e^{-c_0 \tau } \max \Big( \| \pregA_0 \|_{\Beta}^\gamma, \| \psi_{\lo}(t_0) \|_{L_x^r}, C_4 L^{\kappafour} \Big), \label{decay:eq-short-main-estimate-e1} \\ 
 \| \pregA \|_{\Beta}^\gamma  + C_1 L^{\kappaone} 
 &\leq e^{c_0 \tau} \max \Big( \| \pregA_0 \|_{\Beta}^\gamma, \| \psi_{\lo}(t_0) \|_{L_x^r}, C_4 L^{\kappafour} \Big), 
\label{decay:eq-short-main-estimate-e2} \\ 
\| \pregA \|_{\Beta}^\gamma + C_1 L^{\kappaone} &\leq e^{-c_0 \tau}  \max \Big( e^{2 c_0 \tau} |\sfint \pregA_0 |^\gamma, \| \pregA_0 \|_{\Beta}^\gamma, \| \psi_{\lo}(t_0) \|_{L_x^r}, C_4 L^{\kappafour} \Big), 
\label{decay:eq-short-main-estimate-e3} \\
| \sfint \pregA|^\gamma + C_1 L^{\kappaone} &\leq e^{c_0 \tau}  \max \Big( |\sfint \pregA_0 |^\gamma, \| \pregA_0 \|_{\Beta}^{(1-\nu)\gamma}, \| \psi_{\lo}(t_0) \|_{L_x^r}^{1-\nu}, \frac{3}{4} C_4 L^{\kappafour} \Big).
\label{decay:eq-short-main-estimate-e4}
\end{align}
At first sight, the $L^{\kappaone}$-terms on the left-hand sides of \eqref{decay:eq-short-main-estimate-e1}-\eqref{decay:eq-short-main-estimate-e4} may appear insignificant, but they should be taken seriously. For example, when the maximum on the right hand side of \eqref{decay:eq-short-main-estimate-e1} is achieved by the $\psi_{\lo}(t_0+\tau)$-term, it forces us to show that $\| \psi_{\lo}(t)\|_{L_x^r}$ decreases by more than one, which imposes restrictions on the choice of the parameter $\beta$ (see Remark \ref{decay:rem-ODE}). 

We further note that the estimates \eqref{decay:eq-short-main-estimate-e1}-\eqref{decay:eq-short-main-estimate-e4} holds if both $\|\pregA_0\|_{\Beta}^\gamma$ and $\|\psi_{\lo}(t_0)\|_{L_x^r}$ are almost bounded. Indeed, if $\|\pregA_0\|_{\Beta}^\gamma \leq \frac{1}{2} C_4 L^{\kappafour}$ and  $\|\psi_{\lo}(t_0)\|_{L_x^r}\leq \frac{1}{2} C_4 L^{\kappafour}$, then  \eqref{decay:eq-short-main-estimate-e1}-\eqref{decay:eq-short-main-estimate-e4} follow from Lemma \ref{decay:lem-short-bounds}. For this reason, we can assume that
\begin{equation}\label{decay:proof-first-1}
\max \big( \|\pregA_0\|_{\Beta}^\gamma, \|\psi_{\lo}(t_0)\|_{L_x^r} \big) 
\geq \tfrac{1}{2} C_4 L^{\kappafour}. 
\end{equation}
As a result of \eqref{decay:proof-first-1}, it follows that the $L^{\kappathree}$-argument in \eqref{decay:eq-tau-equivalence} is irrelevant, and it therefore follows that  
\begin{equation}\label{decay:proof-first-2}
\tau \sim  c_3 \max \Big( \big\| \pregA_0 \big\|_{\Beta}^{\alpha}, \big\| \psi_{\lo}(t_0) \big\|_{L_x^r}^\beta 
\Big)^{-1}.
\end{equation}
We now recall several estimates of $B$, $Z$, and $\psi_{\lo}$, in which we restrict to $t\in [t_0,t_0+\tau]$. From Lemma \ref{prelim:lem-heat-flow-bound-decay}, we obtain that 
\begin{equation}\label{decay:proof-Blin}
\big\| \Blin(t) \big\|_{\Beta} 
\leq \big\| \Blin(t_0) \big\|_{\Beta}, 
\quad 
\big\| \bigdot{\Blin}(t) \big\|_{\Beta}
\leq e^{-c(t-t_0)} \big\| \bigdot{\Blin}(t_0) \big\|_{\Beta}, \quad \text{and} \quad
\big| \sfint \Blin(t) \big| = \big| \sfint \Blin(t_0) \big|, 
\end{equation}
where $c=c(r)$ is a constant depending only on $r$.
From Lemma \ref{decay:lem-short-bounds}, we obtain that 
\begin{equation}\label{decay:proof-Z}
\big\| Z(t) \big\|_{\Cs_x^{2\eta}} 
\leq \big\| \psi_{\lo}(t_0) \big\|_{L_x^r}^{2-\frac{\beta}{2}+\nu} + C_2 L^{\kappatwo}. 
\end{equation}
Finally, from Proposition \ref{AH:prop-psi-lo-estimate}, we obtain that 
\begin{equation}\label{decay:proof-psilo}
\begin{aligned}
&\, \frac{1}{r} \frac{\mathrm{d}}{\mathrm{d}t} \| \psi_{\lo}(t)\|_{L_x^r}^r
+ \frac{1}{2} \| \psi_{\lo}(t) \|_{L_x^{r+q-1}}^{r+q-1} \\ 
\leq&\, 
C_1 \Big( \| \psi_\lo(t_0) \|_{L_x^r}^{2-\frac{\beta}{2}+\nu } + C_2 L^{\kappatwo} \Big)^{\frac{2}{q} (1+\nu) (r+q-1)} + C_1 L^{\kappaone (r+q-1)}.    
\end{aligned}
\end{equation}

\emph{Second step: A decay estimate for $\psi_{\lo}$.}
In this step, we prove that 
\begin{equation}\label{decay:proof-second-1}
\begin{aligned}
\big\| \psi_{\lo}(t_0+\tau) \big\|_{L_x^r}
\leq \max \Big( e^{-\frac{1}{2^{q+1}} \| \psi_{\lo}(t_0) \|_{L_x^r}^{q-1} \tau} \big\| \psi_{\lo}(t_0) \big\|_{L_x^r},
\frac{1}{2} \big\| \psi_{\lo}(t_0) \big\|_{L_x^r}, \frac{1}{2} C_4 L^{\kappafour} \Big). 
\end{aligned}
\end{equation}
In order to prove \eqref{decay:proof-second-1}, we first recall that the parameter condition 
\begin{equation}\label{decay:proof-parameter-1}
\tfrac{2}{q} \big( 2 - \tfrac{\beta}{2} \big) + \mathcal{O}(\nu) \leq 1,
\end{equation}
which was previously used in \eqref{decay:eq-bounds-parameter-1}, is satisfied.  Using \eqref{decay:proof-psilo} and using the H\"{o}lder estimate $\| \psi_{\lo}\|_{L_x^{r}}^{r+q-1}\leq C \| \psi_{\lo}\|_{L_x^{r+q-1}}^{r+q-1}$, where $C$ only depends on $q$ and $r$,  we then obtain for all $t\in [t_0,t_0+\tau]$ that either
\begin{align*}
\big\| \psi_{\lo}(t) \big\|_{L_x^r} 
&\leq \tfrac{1}{2} \max\big(  \big\| \psi_{\lo}(t_0) \big\|_{L_x^r},  C_4 L^{\kappafour}  \big) \\ 
\text{or} \qquad 
\tfrac{1}{r} \tfrac{\mathrm{d}}{\mathrm{d}t}  \big\| \psi_{\lo}(t) \big\|_{L_x^r}^r 
&\leq - \tfrac{1}{4C} \big\| \psi_{\lo}(t) \big\|_{L_x^r}^{r+q-1}
\leq - \tfrac{1}{2^{q+1}C} \big\| \psi_{\lo}(t_0) \big\|_{L_x^r}^{q-1} \big\| \psi_{\lo}(t) \big\|_{L_x^r}^{r}.
\end{align*}
Using a continuity argument and Gronwall's inequality, it then follows for all $t\in [t_0,t_0+\tau]$ that
\begin{equation*}
    \big\| \psi_{\lo}(t) \big\|_{L_x^r}
\leq \max \Big( e^{-\frac{1}{2^{q+1} C} \| \psi_{\lo}(t_0) \|_{L_x^r}^{q-1} (t - t_0)} \big\| \psi_{\lo}(t_0) \big\|_{L_x^r},
\frac{1}{2} \big\| \psi_{\lo}(t_0) \big\|_{L_x^r}, \frac{1}{2} C_4 L^{\kappafour} \Big). 
\end{equation*}
After choosing $t=t_0+\tau$, this completes the proof of \eqref{decay:proof-second-1}. \\

\emph{Third step: Another decay estimate for $\psi_{\lo}$.} 
In this step, we prove that if 
\begin{equation}\label{decay:proof-third-1}
\big\| \psi_{\lo}(t_0) \big\|_{L_x^r}^{q} \tau \geq 2^{q+4} C C_1 L^{\kappaone},  
\end{equation}
then it holds that 
\begin{equation}\label{decay:proof-third-2} 
\big\| \psi_{\lo}(t_0+\tau) \big\|_{L_x^r} + C_1 L^{\kappaone}
\leq e^{-c_0 \tau} \max \big( \big\| \psi_{\lo}(t_0) \|_{L_x^r}, C_4 L^{\kappafour} \big). 
\end{equation}
In order to prove \eqref{decay:proof-third-2}, we first use \eqref{decay:proof-second-1} from the second step of this proof, which implies that 
\begin{equation}\label{decay:proof-third-3} 
\big\| \psi_{\lo}(t_0+\tau) \big\|_{L_x^r} + C_1 L^{\kappaone}
\leq \max \Big( e^{-\frac{1}{2^{q+1}C} \| \psi_{\lo}(t_0) \|_{L_x^r}^{q-1} \tau } \big\| \psi_{\lo}(t_0) \big\|_{L_x^r},
\tfrac{1}{2} \big\| \psi_{\lo}(t_0) \big\|_{L_x^r}, \tfrac{1}{2} C_4 L^{\kappafour} \Big) + C_1 L^{\kappaone}. 
\end{equation}
If the maximum in \eqref{decay:proof-third-3} is achieved by the second or third argument, then it is easy to see that the $L^{\kappaone}$-term can be absorbed. Indeed, since $\tau\leq1$, it holds that $e^{-c_0\tau} \geq \frac{3}{4}$, and therefore 
\begin{align*}
\max \Big( \tfrac{1}{2} \big\| \psi_{\lo}(t_0) \big\|_{L_x^r}, \tfrac{1}{2} C_4 L^{\kappafour} \Big) + C_1 L^{\kappaone}  
&\leq \big( \tfrac{1}{2} + \tfrac{C_1}{C_4} \big)
\max \Big(  \big\| \psi_{\lo}(t_0) \big\|_{L_x^r},  C_4 L^{\kappafour} \Big) \\
&\leq e^{-c_0\tau} \max \Big(  \big\| \psi_{\lo}(t_0) \big\|_{L_x^r},  C_4 L^{\kappafour} \Big). 
\end{align*}
Thus, we may assume that the maximum in \eqref{decay:proof-third-3} is achieved by the first argument. In particular, we may assume that 
\begin{equation}\label{decay:proof-third-4}
\big\| \psi_{\lo}(t_0) \big\|_{L_x^r} \geq \tfrac{1}{2} C_4 L^{\kappafour}.   
\end{equation}
In order to absorb the $L^{\kappaone}$-term, we then need to show that 
\begin{equation*}
e^{-\frac{1}{2^{q+1}C} \| \psi_{\lo}(t_0) \|_{L_x^r}^{q-1} \tau } \big\| \psi_{\lo}(t_0) \big\|_{L_x^r}  + C_1 L^{\kappaone} \leq e^{-c_0 \tau} \big\| \psi_{\lo}(t_0)\big\|_{L_x^r}.
\end{equation*}
Equivalently, we need to show that 
\begin{equation}\label{decay:proof-third-5}
\Big( e^{-c_0 \tau} - e^{-\frac{1}{2^{q+1}C} \| \psi_{\lo}(t_0) \|_{L_x^r}^{q-1} \tau } \Big) \big\| \psi_{\lo}(t_0) \big\|_{L_x^r} \geq C_1 L^{\kappaone}. 
\end{equation}
In order to prove \eqref{decay:proof-third-5}, we distinguish two regimes. In the regime $\| \psi_{\lo}(t_0)\|_{L_x^r}^{-(q-1)}\leq \tau \leq 1$, we estimate
\begin{equation*}
\Big( e^{-c_0 \tau} - e^{-\frac{1}{2^{q+1}C} \| \psi_{\lo}(t_0) \|_{L_x^r}^{q-1} \tau } \Big) \big\| \psi_{\lo}(t_0) \big\|_{L_x^r} 
\geq \Big( e^{-c_0} - e^{-\frac{1}{2^{q+1}C}} \Big) \big\| \psi_{\lo}(t_0) \big\|_{L_x^r} . 
\end{equation*}
Using the lower bound on $\psi_{\lo}(t_0)$ from \eqref{decay:proof-third-4} and the definition of $c_0$, this implies \eqref{decay:proof-third-5}, and it therefore remains to treat the regime
$\tau \leq \| \psi_{\lo}(t_0)\|_{L_x^r}^{-(q-1)}$. To this end, we recall the elementary inequality $e^{-y}-e^{-z}\geq e^{-1} (z-y)$, which holds for all $0\leq y \leq z \leq 1$. From this inequality, we obtain that 
\begin{align*}
\Big( e^{-c_0 \tau} - e^{-\frac{1}{2^{q+1}C} \| \psi_{\lo}(t_0) \|_{L_x^r}^{q-1} \tau } \Big) \big\| \psi_{\lo}(t_0) \big\|_{L_x^r} 
\geq  e^{-1} \Big( \tfrac{1}{2^{q+1}C}   \| \psi_{\lo}(t_0) \|_{L_x^r}^{q-1} - c_0 \Big) \tau \| \psi_{\lo}(t_0) \|_{L_x^r} 
\geq \tfrac{1}{2^{q+2} C e}  \| \psi_{\lo}(t_0) \|_{L_x^r}^q \tau. 
\end{align*}
Due to our assumption on $\tau$ from \eqref{decay:proof-third-1}, this implies the desired estimate \eqref{decay:proof-third-5}.\\

\emph{Fourth step: Final estimate for $\psi_{\lo}$.} 
In this step, we prove \eqref{decay:eq-short-main-estimate-e1}, i.e., the desired estimate for $\psi_{\lo}$. If the condition in \eqref{decay:proof-third-1} is satisfied, then \eqref{decay:eq-short-main-estimate-e1} follows directly from \eqref{decay:proof-third-2}. It therefore only remains to treat the case
\begin{equation}\label{decay:proof-fourth-q1}
\big\| \psi_{\lo}(t_0) \big\|_{L_x^r}^{q} \tau \leq 2^{q+4} C C_1 L^{\kappaone}. 
\end{equation}
Due to the parameter condition 
\begin{equation}\label{decay:proof-parameter-2}
q- \beta + \mathcal{O}(\nu) \geq 0,
\end{equation}
the inequality in \eqref{decay:proof-fourth-q1} cannot be satisfied in the case $\tau \sim c_3 \| \psi_{\lo}(t_0) \|_{L_x^r}^{-\beta}$.  Due to \eqref{decay:proof-first-2} and \eqref{decay:proof-fourth-q1}, it must therefore hold that 
\begin{equation}\label{decay:proof-fourth-q2}
\| \psi_{\lo}(t_0) \|_{L_x^r} \leq \big( 2^{q+4} C C_1 L^{\kappaone} \tau^{-1} \big)^{\frac{1}{q}} \sim \big( 2^{q+4}  C C_1 C_3 L^{\kappaone} \| \pregA_0 \|_{\Beta}^{\alpha} \big)^{\frac{1}{q}}. 
\end{equation}
Using the parameter condition
\begin{equation}\label{decay:proof-parameter-3}
\frac{\alpha}{q} + \mathcal{O}(\nu) \leq \gamma, 
\end{equation}
using Lemma \ref{decay:lem-short-bounds}, and using Young's inequality, it then follows that
\begin{equation}\label{decay:proof-fourth-q3}
\begin{aligned}
  \| \psi_{\lo}(t_0+\tau) \|_{L_x^r}  + C_1 L^{\kappaone} 
  &\lesssim \big( 2^{q+4} C C_1 C_3 L^\kappa \| \pregA_0 \|_{\Beta}^{\alpha} \big)^{\frac{1}{q}} + 2 C_2 L^{\kappatwo}  \\
  &\leq \tfrac{1}{2} \max \big( \| \pregA_0\|_{\Beta}^\gamma, C_4 L^{\kappafour} \big).
\end{aligned}
\end{equation}
From \eqref{decay:proof-fourth-q3}, we then directly obtain \eqref{decay:eq-short-main-estimate-e1}. \\ 

\emph{Fifth step: Estimate of $\pregA$.} 
 In this step, we prove the desired estimate \eqref{decay:eq-short-main-estimate-e2}, whose proof requires a case distinction. In the case $\| \pregA_0\|_{\Beta}^\alpha \leq \| \psi_{\lo}(t_0)\|_{L_x^r}^\beta$, we use \eqref{decay:eq-prep-1}, the first estimate in \eqref{decay:proof-Blin}, and \eqref{decay:proof-Z}, which yield that
\begin{equation}\label{decay:proof-fifth-q2}
\begin{aligned}
\big\| \pregA \big\|_{\Beta} + C_2 L^{\kappatwo}
&\leq \big\| \pregA_0 \big\|_{\Beta}
+ \big\| \psi_{\lo}(t_0) \big\|_{L_x^r}^{(2-\frac{\beta}{2}+\nu)} + 2 C_2 L^{\kappatwo} \\
&\leq  \big\| \psi_{\lo}(t_0) \big\|_{L_x^r}^{\frac{\beta}{\alpha}}
+  \big\| \psi_{\lo}(t_0) \big\|_{L_x^r}^{(2-\frac{\beta}{2}+\nu)} +2  C_2 L^{\kappatwo}.  
\end{aligned}
\end{equation}
Using the parameter condition
\begin{equation}\label{decay:proof-parameter-4} 
\max\big( \tfrac{\beta}{\alpha}, 2 - \tfrac{\beta}{2} \big) \gamma + \mathcal{O}(\nu) \leq 1
\end{equation}
and Young's inequality, we then obtain that
\begin{equation}\label{decay:proof-fifth-q3}
\eqref{decay:proof-fifth-q2} \leq 2^{-\frac{1}{\gamma}} \max\Big( \| \psi_{\lo}(t_0)\|_{L_x^r}^{1-\nu}, \big( C_4 L^{\kappafour} \big)^{1-\nu} \Big)^{\frac{1}{\gamma}}.
\end{equation}
From \eqref{decay:proof-fifth-q3}, we then directly obtain the desired estimate  \eqref{decay:eq-short-main-estimate-e2}.  In the case $\| \pregA_0\|_{\Beta}^\alpha \geq \| \psi_{\lo}(t_0)\|_{L_x^r}^\beta$, we only use the first inequality in \eqref{decay:proof-fifth-q2}. Then, the proof of \eqref{decay:eq-short-main-estimate-e2} can be reduced to the two inequalities 
\begin{equation}
    \big\| \psi_{\lo}(t_0) \big\|_{L_x^r}^{2-\frac{\beta}{2}+\nu}  
    \leq \tfrac{1}{2} c_3 \tau \big\| \pregA_0 \big\|_{\Beta}^{1-\nu}
    \qquad \text{and} \qquad 
    2 C_2 L^{\kappatwo} \leq \tfrac{1}{2} c_3 \tau \big\| \pregA_0 \big\|_{\Beta}^{1-\nu}.
    \label{decay:proof-fifth-q4}
\end{equation}
Both of the inequalities in \eqref{decay:proof-fifth-q4} can easily be obtained from \eqref{decay:proof-first-1}, $\tau \sim c_3 \| \pregA_0 \|_{\Beta}^{-\alpha}$, and the parameter conditions 
\begin{equation}\label{decay:proof-parameter-5}
\tfrac{\alpha}{\beta} \big( 2 - \tfrac{\beta}{2} \big) + \mathcal{O}(\nu) \leq 1-\alpha 
\qquad \text{and} \qquad 
1- \alpha \geq \mathcal{O}(\nu).
\end{equation}

\emph{Sixth step: Refined estimates of $\pregA$.} We now prove \eqref{decay:eq-short-main-estimate-e3} and \eqref{decay:eq-short-main-estimate-e4}, which are obtained from a modification of the proof of \eqref{decay:eq-short-main-estimate-e2}. In the case $\| \pregA_0\|_{\Beta}^\alpha \leq \| \psi_{\lo}(t_0)\|_{L_x^r}^\beta$, we argue exactly as in \eqref{decay:proof-fifth-q2} and \eqref{decay:proof-fifth-q3}, which then yield \eqref{decay:eq-short-main-estimate-e3} and \eqref{decay:eq-short-main-estimate-e4}. In the case $\| \pregA_0\|_{\Beta}^\alpha \geq \| \psi_{\lo}(t_0)\|_{L_x^r}^\beta$, we use the second and third inequality in \eqref{decay:proof-Blin}, \eqref{decay:proof-Z}, and \eqref{decay:proof-fifth-q4}, which imply that 
\begin{align}
\big\| \bigdotpregA \big\|_{\Beta} + C_2 L^{\kappatwo}
&\leq e^{-c \tau} \big\| \bigdotpregA_0 \big\|_{\Beta} + c_3 \tau \big\| \pregA_0 \big\|_{\Beta}^{1-\nu}, \label{decay:proof-sixth-1}\\
\big| \sfint \pregA \big| + C_2 L^{\kappatwo}
&\leq \big| \sfint \pregA_0 \big|  + c_3 \tau \big\| \pregA_0 \big\|_{\Beta}^{1-\nu}. \label{decay:proof-sixth-2}
\end{align}
Compared to the previous step, the important difference is the $e^{-c\tau}$-factor in \eqref{decay:proof-sixth-1}. 
From \eqref{decay:proof-sixth-2}, we directly obtain \eqref{decay:eq-short-main-estimate-e4}, and it therefore only remains to prove \eqref{decay:eq-short-main-estimate-e3}. From Definition \ref{prelim:def-besov},  we obtain that 
\begin{align}
&\big\| \pregA \big\|_{\Beta}  + C_2 L^{\kappatwo} \notag \\
=&\, \max \Big(\big\| \bigdotpregA \big\|_{\Beta},  \big| \sfint \pregA \big|\Big) + C_2 L^{\kappatwo}  \notag \\
\leq&\, \max \Big( e^{-c \tau} \big\| \bigdotpregA_0 \big\|_{\Beta}, \big| \sfint \pregA_0 \big|  \Big) + c_3 \tau \big\| \pregA_0 \big\|_{\Beta} \notag \\ 
\leq&\,  \max \Big( e^{-c \tau} \big\| \bigdotpregA_0 \big\|_{\Beta}, \big| \sfint \pregA_0 \big|  \Big) + c_3 \tau 
 \max \Big( \big\| \bigdotpregA_0 \big\|_{\Beta}, \big| \sfint \pregA_0 \big|  \Big). \label{decay:proof-sixth-3}
\end{align}
If $e^{-c \tau} \big\| \bigdotpregA_0 \big\|_{\Beta} \geq  \big| \sfint \pregA_0 \big|$, one can estimate
\begin{equation*}
\eqref{decay:proof-sixth-3}
\leq \big( e^{-c \tau} +  c_3 \tau  \big) \big\| \bigdotpregA_0 \big\|_{\Beta} \leq e^{-(c-2c_3) \tau}\big\| \bigdotpregA_0 \big\|_{\Beta}  \leq  e^{-(c-2c_3) \tau}\big\| \pregA_0 \big\|_{\Beta}.
\end{equation*}
Since $c_0,c_3\ll c$, this directly implies \eqref{decay:eq-short-main-estimate-e3}. Alternatively, if $e^{-c \tau} \big\| \bigdotpregA_0 \big\|_{\Beta} \leq  \big| \sfint \pregA_0 \big|$, one can estimate
\begin{equation*}
\eqref{decay:proof-sixth-3}
\leq  \big( 1 + c_3 \tau e^{c\tau} \big) \big| \sfint \pregA_0 \big| 
\leq  \big( 1 + 2 c_3 \tau  \big) \big| \sfint \pregA_0 \big| 
\leq e^{4c_3 \tau} \big| \sfint \pregA_0 \big|,
\end{equation*}
where we used $e^{c\tau}\leq 2$. Due to the $\sfint \pregA_0$-term on the right-hand side of \eqref{decay:eq-short-main-estimate-e3} and the parameter condition $c_3 \ll c_0$, this directly implies \eqref{decay:eq-short-main-estimate-e3}. \\

\emph{Seventh step: $L_t^\infty$-estimate.} 
It only remains to prove the $L_t^\infty$-estimate \eqref{decay:eq-short-Linfty}, which is rather simple. Using Lemma \ref{prelim:lem-heat-flow-bound-decay}, Hypothesis \ref{AH:hypothesis-probabilistic}, \eqref{decay:proof-Z}, \eqref{decay:proof-parameter-4}, and the embeddings $\Cs_x^{2\eta}\hookrightarrow \Beta \hookrightarrow \Cs_x^{-\kappa}$, we obtain on the time-interval $[t_0,t_0+\tau]$ that
\begin{equation}\label{decay:eq-short-Linfty-p1}
\begin{aligned}
 \big\| A \big\|_{L_t^\infty \Cs_x^{-\kappa}}^\gamma 
\lesssim&\,  \big\| \, \linear \big\|_{L_t^\infty \Cs_x^{-\kappa}}^\gamma
+ \big\| \Blin \big\|_{L_t^\infty \Beta}^\gamma 
+ \big\| \Slin \big\|_{L_t^\infty \Cs_x^{-\kappa}}^\gamma 
+ \big\| Z \big\|_{L_t^\infty \Cs_x^{2\eta}}^\gamma \\
\leq &\, C_0 \max\Big( \|A_0 \|_{\Beta}^\gamma, \| \pregphi_0 \|_{L_x^r}, C_4 L^{\kappafour} \Big).
\end{aligned}
\end{equation}
Similarly, using Lemma \ref{prelim:lem-heat-flow-bound-decay}, Hypothesis \ref{AH:hypothesis-probabilistic}
and the embeddings $\Cs_x^{\eta} \hookrightarrow L_x^\infty \hookrightarrow L_x^r \hookrightarrow \GCs^{-\kappa}$, we obtain on the time-interval $[t_0,t_0+\tau]$ that
\begin{equation}\label{decay:eq-short-Linfty-p2}
\begin{aligned}
\big\| \phi \big\|_{L_t^\infty \GCs^{-\kappa}}
&\lesssim \big\| \, \philinear[\Blin,\lo] \big\|_{L_t^\infty L_x^\infty}
+  \big\| \, \philinear[\lo] \big\|_{L_t^\infty L_x^\infty}
+  \big\| \, \philinear[] \big\|_{L_t^\infty \GCs^{-\kappa}}
+ \big\| \psi_{\lo} \big\|_{L_t^\infty L_x^r}
+ \big\| \psi_{\hi} \big\|_{L_t^\infty \Cs_x^\eta}
+ \big\| \varphi_{\hi} \big\|_{L_t^\infty \GCs^{-\kappa}} \\
&\leq C_0  \max \Big(  \| \pregA_0 \|_{\Beta}^\gamma, \| \pregphi_0 \|_{L_x^r}, C_4 L^{\kappafour} \Big).
\end{aligned}
\end{equation}
By combining \eqref{decay:eq-short-Linfty-p1} and \eqref{decay:eq-short-Linfty-p2}, we obtain \eqref{decay:eq-short-Linfty}, which completes the proof.
\end{proof}

\begin{remark}[Conditions on $\alpha$, $\beta$, and $\gamma$]\label{decay:rem-parameters} In the following discussion, we omit all $\mathcal{O}(\nu)$-terms. In the proof of Proposition \ref{decay:prop-short}, we imposed the parameter conditions \eqref{decay:proof-parameter-1}, \eqref{decay:proof-parameter-2}, \eqref{decay:proof-parameter-3}, \eqref{decay:proof-parameter-4}, and \eqref{decay:proof-parameter-5}.
In the proof of Lemma \ref{decay:lem-short-bounds}, we previously imposed the parameter conditions \eqref{decay:eq-bounds-parameter-1}, \eqref{decay:eq-bounds-parameter-2}, and \eqref{decay:eq-bounds-parameter-3}. All together, we imposed the parameter conditions
\begin{align}
\frac{2}{3} &<\alpha <1, \label{decay:eq-parameter-collected-1} \\ 
\max \Big( \frac{q+2}{q+\frac{1}{2}}, 4-q \Big) 
&<\beta < q, \label{decay:eq-parameter-collected-2} \\
\max \Big( \frac{\beta}{\alpha}, 2 - \frac{\beta}{2} \Big) &<\frac{1}{\gamma} <\frac{q}{\alpha},  \label{decay:eq-parameter-collected-3} \\
\frac{2\alpha}{1-\frac{\alpha}{2}} &< \beta. \label{decay:eq-parameter-collected-4}
\end{align}
The most significant conditions are the lower bound in \eqref{decay:eq-parameter-collected-1}, the upper bound in \eqref{decay:eq-parameter-collected-2}, and \eqref{decay:eq-parameter-collected-4}, which imply that
\begin{equation}\label{decay:eq-parameter-collected-5}
2 < \beta < q. 
\end{equation}
In the important case $q=3$, this prevents\footnote{As discussed in Remark \ref{decay:rem-parameters-first}, it may be possible to relax the condition $\alpha>\frac{2}{3}$, and it therefore may be possible to relax the lower bound in \eqref{decay:eq-parameter-collected-5}.} us from choosing $\beta=q-1$. Even in the case $\tau \sim c_1 \|\psi_{\lo}(t_0)\|_{L_x^r}^{-\beta}$, Remark \ref{decay:rem-ODE} suggests that the size of $\| \psi_{\lo}(t)\|_{L_x^r}$ may not decrease by a constant multiple of $\| \psi_{\lo}(t_0)\|_{L_x^r}$, which illustrates the delicate nature of Proposition \ref{decay:prop-short}.  
\end{remark}

\begin{remark}[Polynomial growth]\label{decay:rem-polynomial}
We note that, in the proof of \eqref{decay:eq-short-A}, there is room in the exponents, see e.g. \eqref{decay:proof-fifth-q2}-\eqref{decay:proof-fifth-q3} and \eqref{decay:proof-fifth-q4}-\eqref{decay:proof-parameter-5}. For this reason, it should be possible to prove a variant of \eqref{decay:eq-short-A} which leads to polynomial rather than exponential growth in Theorem \ref{intro:thm-abelian-higgs}.\ref{intro:item-AH-1}. However, since our main interest lies in the gauge invariant, uniform-in-time bounds in Theorem \ref{intro:thm-abelian-higgs}.\ref{intro:item-AH-2}, we do not pursue this further.
\end{remark}

At the end of this subsection, we state and prove a crude growth estimate. The crude growth estimate is only needed to circumvent a technical problem in the proofs of Lemma \ref{decay:lem-exponential-growth} and Lemma \ref{decay:lem-decay-unit}, and can be safely skipped on first reading. 

\begin{lemma}[Crude growth estimate]\label{decay:lem-crude}
Let the same assumptions as in Proposition \ref{decay:prop-short} be satisfied, except that the time-scale $\tau>0$ only needs to be upper-admissible. Then, it holds that 
\begin{equation}\label{decay:eq-crude}
\max \Big( \| \pregA \|_{\Beta}^\gamma, \| \pregphi \|_{L_x^r} \Big) 
\leq e^{c_1}  \max \Big( \| \pregA_0 \|_{\Beta}^\gamma,  \| \pregphi_0 \|_{L_x^r}, C_4 L^{\kappafour} , \tau^{-\eta}\Big).
\end{equation}
Furthermore, it holds that 
\begin{equation}\label{decay:eq-crude-mean}
| \sfint \pregA |^{\gamma} \leq 
e^{c_1} \max\Big( |\sfint \pregA_0 |^\gamma, \| \pregA_0 \|_{\Beta}^{(1-\nu)\gamma},  \| \pregphi_0 \|_{L_x^r}^{1-\nu}, \frac{3}{4} C_4 L^{\kappafour} \Big) 
\end{equation} 
and 
\begin{equation}\label{decay:eq-crude-Linfty}
\max \Big(  \| A \|_{L_t^\infty \Cs_x^{-\kappa}([t_0,t_0+\tau])}^\gamma, \| \phi \|_{L_t^\infty \GCs^{-\kappa}([t_0,t_0+\tau])} \Big)
\leq C_0  \max \Big(  \| \pregA_0 \|_{\Beta}^\gamma, \| \pregphi_0 \|_{L_x^r}, C_4 L^{\kappafour} \Big). 
\end{equation}
\end{lemma}

We remark that, since Lemma \ref{decay:lem-crude} will only be used once in the proofs of Lemma \ref{decay:lem-exponential-growth} and Lemma \ref{decay:lem-decay-unit}, the $e^{c_1}$-factor in \eqref{decay:eq-crude} does not cause any difficulties. 

\begin{proof}
We only prove \eqref{decay:eq-crude}, since \eqref{decay:eq-crude-mean} and \eqref{decay:eq-crude-Linfty} can then be obtained using similar arguments as in the sixth and seventh step of the proof of Proposition \ref{decay:prop-short}. We note that the $\tau^{-\eta}$ term is needed in neither \eqref{decay:eq-crude-mean} nor \eqref{decay:eq-crude-Linfty}, since $|\sfint S(t_0 + \tau)|$, $\|\Slin\|_{L_t^\infty \Cs_x^{-\kappa}([t_0, t_0+\tau])}$, and $\|\varphi_{\hi}\|_{L_t^\infty \GCs^{-\kappa}([t_0, t_0 + \tau])}$ are controlled by $\|\Slin(t_0)\|_{\Cs_x^{-\kappa}}$ and  $\|\nregphi_0\|_{\GCs^{-\kappa}}$.  Due to \eqref{AH:eq-decomposition-A}, \eqref{AH:eq-decomposition-phi-first}, \eqref{AH:eq-decomposition-psi}, and \eqref{decay:eq-short-new}, it suffices to prove the five estimates 
\begin{align}
\| \Blin \|_{C_t^0 \Beta([t_0,t_0+\tau])}^\gamma &\leq \| \pregA_0 \|_{\Beta}^{\gamma}, \label{decay:eq-crude-p1} \\ 
\| Z \|_{C_t^0 \Beta([t_0,t_0+\tau])}^{\gamma} &\leq c_2 \max \big( \| \pregphi_0 \|_{L_x^r}, C_4 L^{\kappafour} \big), \label{decay:eq-crude-p2} \\
\| \philinear[\Blin,\lo] - \philinear[\lo] \|_{C_t^0 L_x^r([t_0,t_0+\tau])} 
&\leq c_2 C_4 L^{\kappafour}, \label{decay:eq-crude-p3} \\ 
\| \psi_{\lo} \|_{C_t^0 L_x^r([t_0,t_0+\tau])} &\leq e^{c_2} \max \big( \| \pregphi_0 \|_{L_x^r}, C_4 L^{\kappafour} \big), \label{decay:eq-crude-p4} \\ 
\| \psi_{\hi} \|_{C_t^0 L_x^r([t_0,t_0+\tau])} &\leq c_2 C_4 L^{\kappafour}, 
\label{decay:eq-crude-p5}
\end{align}
which are uniform over the time-interval $[t_0,t_0+\tau]$, and the two estimates 
\begin{align}
\| \Slin(t_0+\tau) \|_{\Beta}^\gamma &\leq c_2 \max\big( C_4 L^{\kappafour}, \tau^{-\eta} \big), \label{decay:eq-crude-p6} \\ 
\| \varphi_{\hi}(t_0 + \tau) \|_{L_x^r} &\leq c_2 \max\big( C_4 L^{\kappafour}, \tau^{-\eta} \big), \label{decay:eq-crude-p7}
\end{align}
which only concern the time $t=t_0+\tau$. In the following, we often use that $C_4$ is sufficiently large depending on $c_2$, which will not be repeated below. The first estimate \eqref{decay:eq-crude-p1} follows directly from Lemma \ref{prelim:lem-heat-flow-bound-decay}. The second estimate \eqref{decay:eq-crude-p2} follows directly from Lemma \ref{decay:lem-short-bounds} and the parameter condition \eqref{decay:proof-parameter-4}. The third estimate \eqref{decay:eq-crude-p3} follows directly from Hypothesis \ref{AH:hypothesis-probabilistic}. The fourth and fifth estimate \eqref{decay:eq-crude-p4} and \eqref{decay:eq-crude-p5} follow directly from Lemma \ref{decay:lem-short-bounds}. In order to obtain the sixth estimate \eqref{decay:eq-crude-p6}, we use Lemma \ref{lemma:heat-flow-smoothing} and \eqref{AH:eq-L-nreg}, which yield that 
\begin{equation*}
\| S(t_0+\tau) \|_{\Beta} \lesssim \tau^{-(\eta+\kappa)} \| \nregA_0 \|_{\Cs_x^{-\kappa}} \lesssim \tau^{-(\eta+\kappa)}  L^{\kappa}.
\end{equation*}
Since $\kappa$ is much smaller than $\kappafour$ and $\eta$ and $\gamma=\frac{2}{7}$ is smaller than one, this then implies \eqref{decay:eq-crude-p6}. Similarly, using Lemma  \ref{lemma:heat-flow-smoothing} and \eqref{AH:eq-L-nreg}, we also obtain
\begin{equation*}
\| \varphi_{\hi}(t_0+\tau) \|_{L_x^r} \lesssim \tau^{-\kappa} \| \nregphi_0 \|_{\Cs_x^{-\kappa}} \lesssim \tau^{-\kappa}  L^{\kappa}.
\end{equation*}
Since $\kappa$ is much smaller than $\kappafour$ and $\eta$, this implies \eqref{decay:eq-crude-p7}. 
\end{proof}

\subsection{Decay estimates on unit time-scales}\label{section:decay-unit}

We now estimate our solution on unit time-scales, which is primarily done by iterating our estimates on admissible time-scales from Proposition \ref{decay:prop-short}. Since we iterate in time, it is important to keep track of the initial time in the definition of our stochastic objects from \eqref{AH:eq-A-linear}, \eqref{AH:eq-philinear-lo}, and \eqref{AH:eq-philinear-hi}. For instance, in Lemma \ref{decay:lem-exponential-growth} and Lemma \ref{decay:lem-decay-unit}, the statement involves $\initiallinear[\bfzero]$ rather than $\initiallinear[t_0]$.
For expository purposes, we first prove a growth estimate (Lemma \ref{decay:lem-exponential-growth}), which is not invariant under the discrete gauge symmetry from \eqref{intro:eq-group-action-Zd}. Then, we prove a gauge invariant decay estimate (Lemma \ref{decay:lem-decay-unit}). 

\begin{lemma}[Exponential growth estimate]\label{decay:lem-exponential-growth} Let $t_0 \geq 0$, let $\pregA_0\colon (\Omega,\Fc) \rightarrow \Beta$, and let $\pregphi_0\colon (\Omega,\Fc) \rightarrow L_x^r$. Furthermore, assume that $\pregA_0$ and $\pregphi_0$ are $\Fc_{t_0}$-measurable. Then, the solution of the stochastic Abelian-Higgs equations \eqref{AH:eq-evolution-A}-\eqref{AH:eq-evolution-phi} with initial data 
\begin{equs}
A(t_0) = \initiallinear[\bfzero] (t_0) + \pregA_0 \qquad \text{and} \qquad 
\phi(t_0) = \initialphilinear[\bfzero] (t_0) + \pregphi_0
\end{equs}
almost surely exists on the time-interval $[t_0,t_0+1]$. Furthermore, for all $p\geq 1$, we have the growth estimate
\begin{equation}\label{decay:eq-growth}
\begin{aligned}
&\, \E \Big[ \max \Big( \big\| A(t_0+1) - \initiallinear[\bfzero] (t_0+1) \big\|_{\Beta}^\gamma, 
\big\| \phi(t_0+1) - \initialphilinear[\bfzero] (t_0+1) \big\|_{L_x^r} \Big)^p \Big]^{\frac{1}{p}} \\ 
\leq&\, e^{2c_0}  \E \Big[ \max \Big( \big\| A(t_0) - \initiallinear[\bfzero] (t_0) \big\|_{\Beta}^\gamma, 
\big\| \phi(t_0) - \initialphilinear[\bfzero] (t_0) \big\|_{L_x^r} \Big)^p \Big]^{\frac{1}{p}} 
+ C_7 \, p^{\frac{1}{\kappa}}.
\end{aligned}
\end{equation}
In addition, we also have the $L_t^\infty$-estimate 
\begin{equation}\label{decay:eq-growth-linfty}
\begin{aligned}
&\, \E \Big[ \max \Big( \big\| A \big\|_{L_t^\infty \Cs_x^{-\kappa}([t_0,t_0+1])}^\gamma, 
\big\| \phi \big\|_{L_t^\infty \GCs^{-\kappa}([t_0,t_0+1])} \Big)^p \Big]^{\frac{1}{p}} \\ 
\leq&\, C_0 e^{2c_0}  \E \Big[ \max \Big( \big\| A(t_0) - \initiallinear[\bfzero] (t_0) \big\|_{\Beta}^\gamma, 
\big\| \phi(t_0) - \initialphilinear[\bfzero] (t_0) \big\|_{L_x^r} \Big)^p \Big]^{\frac{1}{p}} 
+ C_0 C_7 \, p^{\frac{1}{\kappa}}.
\end{aligned}
\end{equation}
\end{lemma}

\begin{proof} 
Since \eqref{decay:eq-growth-linfty} can easily be obtained by combining the argument below with \eqref{decay:eq-short-Linfty} and \eqref{decay:eq-crude-Linfty}, we only prove \eqref{decay:eq-growth}. While the proof of \eqref{decay:eq-growth} is technical, the idea behind it is rather simple: We just iterate Proposition~\ref{decay:prop-short} and, once we are close to the time $t_0+1$, use Lemma~\ref{decay:lem-crude} for the last step in the iteration. The only reason for using Lemma~\ref{decay:lem-crude} in the last step, rather than to keep using Proposition~\ref{decay:prop-short}, is that otherwise the iteration may take us to times larger than $t_0+1$.
For expository purposes, we separate the proof of \eqref{decay:eq-growth} into four steps.\\

\emph{First step: Setup.}
We first recall that $(\Fc_t)_{t\geq 0}$ is the filtration generated by the space-time white noises $\xi$ and $\zeta$. Then, 
we let $\lambda\in \dyadic$ and let $L$ be defined as the smallest dyadic number satisfying  
\begin{equation}\label{decay:eq-growth-p1}
L \geq C_5 \max \Big( \lambda^{\frac{4}{\kappa}}, \big\| \pregA_0 \big\|_{\Beta}^{\frac{20}{\eta_3}}, 
\big\| \pregphi_0 \big\|_{L_x^r}^{\frac{20}{\eta_3}}, \big\| \initiallinear[\bfzero] (t_0) \big\|_{\Cs_x^{-\kappa}}^{\frac{2}{\kappa}}, \big\| \initialphilinear[\bfzero](t_0) \big\|_{\GCs^{-\kappa}}^{\frac{2}{\kappa}} \Big).
\end{equation}
We note that, by definition, $L$ is $\Fc_{t_0}$-measurable. 
The number of total steps $J\in \mathbb{N}\medcup \{ \infty \}$, time-scales $(\tau_j)_{j=0}^{J-1} \subseteq (0,1)$, times $(t_j)_{j=1}^{J}\subseteq [t_0,t_0+1]$, connection one-forms $(\pregA_j)_{j=1}^J$, and scalar fields $(\pregphi_j)_{j=1}^J$ are then iteratively defined as follows: Let $j\geq 1$ and assume that $\tau_k$ has been defined for all $0\leq k \leq j-2$ and $t_k$, $\pregA_k$, and $\pregphi_k$ have been defined for all $1\leq k \leq j-1$. In the case $t_{j-1}=t_0+1$, we stop the iteration, which simply means that we set $J:=j-1$. In the case $t_{j-1}<t_0+1$, we first define $\tau_{j-1}^\ast$ as
\begin{equation}\label{decay:eq-growth-q1}
\tau_{j-1}^\ast := c_3 \max \Big( \big\| \pregA_{j-1} \big\|_{\Beta}^\alpha, \big\| \pregphi_{j-1} \big\|_{L_x^r}^\beta, C_4 L^{\kappafour} \Big)^{-1}.
\end{equation}
Then, we define $\tau_{j-1}$ as\footnote{The reason behind \eqref{decay:eq-growth-q1} and \eqref{decay:eq-growth-q2} is to obtain the inequality in \eqref{decay:eq-growth-q4}. This makes sure that the last time-scale $\tau_{J-1}$ is not too small, and this allows us to control the growth from the last time-step using Lemma \ref{decay:lem-crude}.}
\begin{equation}\label{decay:eq-growth-q2}
\tau_{j-1} 
:= 
\begin{cases}
\begin{tabular}{ll}
$\tau_{j-1}^\ast$ &if $t_{j-1}+2\tau_{j-1}^\ast <t_0+1$,\\
$(t_0+1)-t_{j-1}$  &if  $t_{j-1} + 2 \tau_{j-1}^\ast \geq t_0+1$.
\end{tabular}
\end{cases}
\end{equation}
In the first or second case in \eqref{decay:eq-growth-q2}, the time-scale $\tau_{j-1}$ is admissible or upper-admissible, respectively.  We note that, by definition, $\tau_{j-1}$ is $\Fc_{t_{j-1}}$-measurable.  Furthermore, we define  
\begin{equation}\label{decay:eq-growth-q3}
t_j := t_{j-1} + \tau_{j-1}, \qquad \pregA_j := A(t_j) - \initiallinear[t_{j-1}] (t_j) \qquad \text{and} \qquad 
\pregphi_j := \phi(t_j) - \initialphilinear[t_{j-1}] (t_j),
\end{equation}
where $\initiallinear[t_{j-1}] (t_j) $ and $\initialphilinear[t_{j-1}] (t_j)$ are the linear stochastic objects with zero initial data at time $t=t_{j-1}$.
We note that, if the second case in \eqref{decay:eq-growth-q2} occurs, then the iteration terminates after that step. In that case, it then holds that
\begin{equation*}
t_0 +1 = t_{J} = t_{J-2} + \tau_{J-2} + \tau_{J-1} \qquad \text{and} \qquad 
t_0 + 1 > t_{J-2} + 2 \tau_{J-2}, 
\end{equation*}
which implies 
\begin{equation}\label{decay:eq-growth-q4} 
\tau_{J-1} \geq \tau_{J-2}.
\end{equation}
In addition, for all $0\leq j \leq J-1$, we define the event 
\begin{equation}\label{decay:eq-growth-p3}
E^{\lambda}_j := E \big( t_j, \tau_j, L, \pregA_j),
\end{equation}
where $E(\cdot)$ is as in Definition \ref{decay:def-probabilistic-event}. For  all $0\leq j \leq J-1$, we also let 
\begin{equation}\label{decay:eq-growth-p4}
E^\lambda_{\leq j} := \bigcap_{k=0}^j E^\lambda_k \qquad \text{and} \qquad E^\lambda := E^{\lambda}_{\leq J-1}.
\end{equation}
Finally, we define 
\begin{equation*}
\pregA := A(t_0+1) - \initiallinear[\bfzero](t_0+1) 
\qquad \text{and} \qquad 
\pregphi := \phi(t_0+1) - \initialphilinear[\bfzero](t_0+1),
\end{equation*}

\emph{Second step: Bounds on $A$ and $\phi$.} We now claim that, for all $0\leq j\leq J-1$, it holds that 
\begin{equation}\label{decay:eq-growth-p5}
\begin{aligned}
&\ind_{E^{\lambda}_{\leq j}} \max \Big( \big\| \pregA_{j+1} \big\|_{\Beta}^\gamma, \big\| \pregphi_{j+1} \big\|_{L_x^r} \Big) \\ 
\leq&\, e^{c_0 (t_{j+1}-t_0)} \big( \ind_{j<J-1} + e^{c_1} \ind_{j=J-1} \big) \max \Big(  \big\| \pregA_0 \big\|_{\Beta}^\gamma, \big\| \pregphi_0 \big\|_{L_x^r}, C_4 L^{\kappafour} \Big).    
\end{aligned}
\end{equation}
We remark that \eqref{decay:eq-growth-p5} directly implies that, on the event $E^\lambda$, the time-scales $(\tau_j)_{j=0}^{J-1}$ are uniformly bounded below and, therefore, the number of total steps $J$ is finite. 
In order to obtain \eqref{decay:eq-growth-p5}, we use Proposition \ref{decay:prop-short} in the first case from \eqref{decay:eq-growth-q2} and use Lemma \ref{decay:lem-crude} in the second case from \eqref{decay:eq-growth-q2}. Therefore, we need to verify that the conditions \eqref{AH:eq-L} and \eqref{AH:eq-L-nreg} are satisfied at every step of the iteration and that the $\tau$-term in \eqref{decay:eq-crude} is irrelevant. For this, it suffices to verify that, for all $0\leq j \leq J-1$, 
\begin{align}
\ind_{E^{\lambda}_{\leq j-1}} \max \Big( 1, \big\| \pregA_j \big\|_{\Beta}^{\frac{2}{\eta_3}}, \big\| \pregphi_j \big\|_{L_x^r}^{\frac{2}{\eta_3}} \Big) &\leq L, \label{decay:eq-growth-p5p} \\ 
 \ind_{E^{\lambda}_{\leq j-1}}  \max \Big( \big\| \initialphilinear[t_{j-1}] (t_j) \big\|_{\GCs^{-\kappa}}, 
\big\| \initiallinear[t_{j-1}] (t_j) \big\|_{\Cs_x^{-\kappa}} \Big)&\leq L^{\kappa}, 
\label{decay:eq-growth-p5pp} \\ 
  \ind_{E^{\lambda}_{\leq J-1}} \tau_{J-1}^{-\eta} &\leq C_3 L^{\kappathree}. \label{decay:eq-growth-p5ppp} 
\end{align}
The first bound \eqref{decay:eq-growth-p5p} follows from the previous step in the iteration procedure, i.e., \eqref{decay:eq-growth-p5} with $j$ replaced by $j-1$. The $L^{\kappafour}$-term in \eqref{decay:eq-growth-p5} causes no problems in \eqref{decay:eq-growth-p5p} due to the inequality 
\begin{equation*}
\big( 2 C_4 L^{\kappafour}\big)^{\frac{20}{\eta_3}} \leq L, 
\end{equation*}
which holds since $L\geq C_5$ is large and $\kappafour$ is much smaller than $\eta_3$. The second bound \eqref{decay:eq-growth-p5pp} holds due to the inclusion $E^\lambda_{\leq j-1}\subseteq E^\lambda_{j-1}$ and the definition of $E^\lambda_{j-1}$ (see Definition \ref{decay:def-probabilistic-event}).
The third bound \eqref{decay:eq-growth-p5ppp} can be obtained by first using \eqref{decay:eq-growth-q4}, then using \eqref{decay:eq-growth-p5p} for $j=J-2$, and finally using that $\eta$ is much smaller than the quotient~$\frac{\kappathree}{\eta_3}=100\nu^{-20}\eta$. \\

As a direct consequence of \eqref{decay:eq-growth-p5}, we obtain that
\begin{equation}\label{decay:eq-growth-end-time}
\ind_{E^{\lambda}} \max \Big( \big\| \pregA_J \big\|_{\Beta}^\gamma, \big\| \pregphi_J \big\|_{L_x^r} \Big)
\leq e^{c_0+c_1} \max \Big(  \big\| \pregA_0 \big\|_{\Beta}^\gamma, \big\| \pregphi_0 \big\|_{L_x^r}, C_4 L^{\kappafour} \Big).
\end{equation}
From the definitions of $\pregA_J$, $\pregA$, $\pregphi_J$, and $\pregphi$, we have that
\begin{align*}
\pregA_J - \pregA 
&= \initiallinear[\bfzero](t_0+1) - \initiallinear[t_{J-1}](t_0+1) = e^{\tau_{J-1} (\Delta-1)} \initiallinear[\bfzero](t_{J-1})  \\ 
\text{and} \qquad \qquad 
\pregphi_J - \pregphi &= \initialphilinear[\bfzero](t_0+1) - \initialphilinear[t_{J-1}](t_0+1) = e^{\tau_{J-1} (\Delta-1)} \initialphilinear[\bfzero](t_{J-1}). 
\end{align*}
Using Lemma \ref{lemma:heat-flow-smoothing} and \eqref{decay:eq-growth-p5ppp}, it follows that
\begin{equation}\label{decay:eq-growth-p8}
\big\| \pregA_J - \pregA \big\|_{\Beta} \lesssim \tau_{J-1}^{-\eta-\kappa} \big\| \initiallinear[\bfzero](t_{J-1}) \big\|_{\Cs_x^{-\kappa}}
\lesssim (C_3 L^{\kappathree})^2 \big\| \initiallinear[\bfzero] \big\|_{L_t^\infty \Cs_x^{-\kappa}([t_0,t_0+1])}.
\end{equation}
Similarly, it also holds that
\begin{equation}\label{decay:eq-growth-p9}
\big\| \pregphi_J - \pregphi \big\|_{L_x^r} \lesssim \tau_{J-1}^{-\kappa} \big\| \initialphilinear[\bfzero] \big\|_{L_t^\infty \Cs_x^{-\kappa}([t_0,t_0+1])}
 \lesssim   (C_3 L^{\kappathree})^2  \big\| \initialphilinear[\bfzero] \big\|_{L_t^\infty \GCs^{-\kappa}([t_0,t_0+1])}.
\end{equation}
By combining \eqref{decay:eq-growth-end-time}, \eqref{decay:eq-growth-p8}, and \eqref{decay:eq-growth-p9}
and using Young's inequality, we obtain that 
\begin{equation}\label{decay:eq-growth-p10}
\begin{aligned}
   &\, \ind_{E^{\lambda}} \max \Big( \big\| \pregA \big\|_{\Beta}^\gamma, \big\| \pregphi \big\|_{L_x^r} \Big) \\
\leq&\, e^{c_0+c_1} \max \Big(  \big\| \pregA_0 \big\|_{\Beta}^\gamma, \big\| \pregphi_0 \big\|_{L_x^r} \Big)
+ 2 C_4 L^{\kappafour} + C_4 \max\Big( \big\| \initiallinear[\bfzero] \big\|_{L_t^\infty \Cs_x^{-\kappa}([t_0,t_0+1])},\big\| \initialphilinear[\bfzero] \big\|_{L_t^\infty \GCs^{-\kappa}([t_0,t_0+1])} \Big)^2.
\end{aligned}
\end{equation}
The estimate \eqref{decay:eq-growth-p10} is almost in the desired form, and it only remains to express the right-hand side in terms of $\lambda$ rather than $L$. Using the definition of $L$, using Young's inequality, and using that $\kappafour$ is much smaller than $\eta_3$, we estimate
\begin{align*}
2C_4 L^{\kappafour}
&\lesssim C_4 C_5 \max \Big( \lambda^{\frac{4\kappafour}{\kappa}}, \big\| \pregA_0 \big\|_{\Beta}^{\frac{20\kappafour}{\eta_3}}, 
\big\| \pregphi_0 \big\|_{L_x^r}^{\frac{20\kappafour}{\eta_3}}, \big\| \initiallinear[\bfzero] (t_0) \big\|_{\Cs_x^{-\kappa}}^{\frac{2\kappafour}{\kappa}}, \big\| \initialphilinear[\bfzero](t_0) \big\|_{\GCs^{-\kappa}}^{\frac{2\kappafour}{\kappa}} \Big) \\
&\leq c_1 \max \Big(  \big\| \pregA_0 \big\|_{\Beta}^\gamma, \big\| \pregphi_0 \big\|_{L_x^r} \Big) 
 + \tfrac{1}{2} C_6 \max\Big( \big\| \initiallinear[\bfzero] \big\|_{L_t^\infty \Cs_x^{-\kappa}([t_0,t_0+1])},\big\| \initialphilinear[\bfzero] \big\|_{L_t^\infty \GCs^{-\kappa}([t_0,t_0+1])} \Big)^{\frac{2\kappafour}{\kappa}}
 + C_6 \lambda^{\frac{4\kappafour}{\kappa}}. 
\end{align*}
By inserting this back into \eqref{decay:eq-growth-p10} and using that $c_1$ is much smaller than $c_0$, we then obtain that 
\begin{equation}\label{decay:eq-growth-p11}
\begin{aligned}
   &\, \ind_{E^{\lambda}} \max \Big( \big\| \pregA \big\|_{\Beta}^\gamma, \big\| \pregphi \big\|_{L_x^r} \Big) \\
\leq&\, e^{2c_0} \max \Big(  \big\| \pregA_0 \big\|_{\Beta}^\gamma, \big\| \pregphi_0 \big\|_{L_x^r} \Big)
+ C_6 \max\Big( \big\| \initiallinear[\bfzero] \big\|_{L_t^\infty \Cs_x^{-\kappa}([t_0,t_0+1])},\big\| \initialphilinear[\bfzero] \big\|_{L_t^\infty \GCs^{-\kappa}([t_0,t_0+1])} \Big)^{\frac{2\kappafour}{\kappa}} + C_6 \lambda^{\frac{4\kappafour}{\kappa}}. 
\end{aligned}
\end{equation}

\emph{Third step: Probability of $E^\lambda$.} 
For notational convenience, we first extend our definitions of the times $t_j$ and events $E^\lambda_j$ and $E^\lambda_{\leq j}$ to all $j\geq 0$, rather than just $0\leq j\leq J$ or $0\leq j\leq J-1$, respectively.
We define
\begin{equation}\label{decay:eq-growth-r1} 
\widetilde{t}_j := 
\begin{cases} 
\begin{tabular}{ll}
$t_j$ &if $j\leq J$, \\ 
$t_0+1$ &if $j>J$, 
\end{tabular}
\end{cases} \quad 
\widetilde{E}^\lambda_j := 
\begin{cases} 
\begin{tabular}{ll}
$E^\lambda_j$ &if $j\leq J-1$, \\ 
$\Omega$ &if $j>J-1$, 
\end{tabular}
\end{cases}
\quad \text{and} \qquad \quad 
\widetilde{E}^\lambda_{\leq j} := \bigcap_{k=0}^j \widetilde{E}^\lambda_{k}.
\end{equation}
We note that, by definition, it holds that $J\geq j+1$ if and only if $\widetilde{t}_j<t_0+1$. Furthermore, we define $\widetilde{J} := c_3^{-1} C_4 L^{10\eta_3}$. We now show for all $j\geq \widetilde{J}$ that 
\begin{equation}\label{decay:eq-growth-r2}
\big\{ J - 1 \geq j   \big\} \medcap \widetilde{E}^\lambda_{\leq j-1} = \emptyset.
\end{equation}
Indeed, on the event $\{ J - 1 \geq j \} \medcap \widetilde{E}^\lambda_{\leq j-1}$, it follows from \eqref{decay:eq-growth-q1}, \eqref{decay:eq-growth-q2}, \eqref{decay:eq-growth-q4}, and \eqref{decay:eq-growth-p5p} that 
\begin{equation}\label{decay:eq-growth-r3}
\min_{0\leq k \leq j-1} \tau_k \geq c_3 \max_{0\leq k \leq j-1} \max \Big( \big\| \pregA_k \big\|_{\Beta}^\alpha, \big\| \pregphi_k \big\|_{L_x^r}^\beta, C_4 L^{\kappafour} \Big)^{-1} 
\geq c_3 C_4^{-1} L^{-10\eta_3} = \frac{1}{\widetilde{J}},
\end{equation}
where we also used that $\alpha,\beta\leq 20$. On the event $\{ J - 1 \geq j \} \medcap \widetilde{E}^\lambda_{\leq j-1}$, it also holds that 
\begin{equation}\label{decay:eq-growth-r4}
\min_{0\leq k \leq j-1} \tau_k \leq \frac{t_j - t_0}{j} < \frac{(t_0+1)-t_0}{j}=\frac{1}{j}.
\end{equation}
In the case  $j\geq \widetilde{J}$, the estimates in \eqref{decay:eq-growth-r3} and \eqref{decay:eq-growth-r4} are in contradiction, which then implies \eqref{decay:eq-growth-r2}. From \eqref{decay:eq-growth-r1} and \eqref{decay:eq-growth-r2}, it follows that 
\begin{equation}\label{decay:eq-growth-r5} 
\Omega \backslash E^\lambda = \big( \Omega \backslash E_0^\lambda \big) \medcup \bigcup_{j=1}^{J-1} \big( E^\lambda_{\leq j-1} \backslash E^\lambda_j \big) 
= \big( \Omega \backslash \widetilde{E}_0^\lambda \big) \medcup \bigcup_{j=1}^{\widetilde{J}-1}
\big( \widetilde{E}^\lambda_{\leq j-1} \backslash \widetilde{E}^\lambda_j \big).
\end{equation}
From Proposition \ref{AH:prop-probabilistic} and \eqref{decay:eq-growth-p5p}, we also obtain for all $j\geq 0$ that 
\begin{equation}\label{decay:eq-growth-r6}
\ind_{\widetilde{E}^\lambda_{\leq j-1}} \mathbb{P} \Big( \big( \Omega \backslash \widetilde{E}^\lambda_j \big) \Big| \Fc_{\widetilde{t}_j} \Big) \leq c^{-1} \exp \big( - c L^{\frac{\kappa}{2}} \big),
\end{equation}
where $c$ is an absolute constant depending only on the parameters from  \eqref{prelim:eq-parameter-new-eta-nu}-\eqref{prelim:eq-parameter-new-r}. By combining \eqref{decay:eq-growth-r5} and \eqref{decay:eq-growth-r6}, we then obtain that 
\begin{align*}
\bP \Big( \big( \Omega \backslash E^\lambda \big) \Big| \Fc_{t_0} \Big) 
&\leq \bP \Big( \big( \Omega \backslash \widetilde{E}^\lambda_0 \big) \Big| \Fc_{\widetilde{t}_0} \Big) 
+ \sum_{j=1}^{\widetilde{J}-1} \bP \Big( \big( \widetilde{E}_{\leq j-1}^\lambda \backslash \widetilde{E}_j^\lambda \big) \Big| \Fc_{\widetilde{t}_0} \Big) \\
&\leq  \bP \Big( \big( \Omega \backslash \widetilde{E}^\lambda_0 \big) \Big| \Fc_{\widetilde{t}_0} \Big) 
+ \sum_{j=1}^{\widetilde{J}-1} \E \Big[ \ind_{\widetilde{E}^\lambda_{\leq j-1}} \bP \Big( \big( \Omega \backslash \widetilde{E}_j^\lambda \big) \Big| \Fc_{\widetilde{t}_j} \Big) \Big| \Fc_{\widetilde{t}_0} \Big] \\
&\leq c^{-1} \widetilde{J} \exp \big( -c L^{\frac{\kappa}{2}} \big).
\end{align*}
Using the definitions of $\widetilde{J}$ and $L$ and using that $C_5$ is large depending on $c_3$, $c_4$, and the parameters from \eqref{prelim:eq-parameter-new-eta-nu}-\eqref{prelim:eq-parameter-new-r}, we then obtain that 
\begin{equation*}
c^{-1} \widetilde{J} \exp \big( -c L^{\frac{\kappa}{2}} \big) \leq c^{-1} c_3^{-1} C_4 L^{10\eta_3} \exp \big( -c L^{\frac{\kappa}{2}} \big) \leq \exp \big( - \tfrac{1}{2} c  L^{\frac{\kappa}{2}} \big)  
\leq \exp\big( - \lambda^2 \big).
\end{equation*}
As a result, it follows that 
\begin{equation}\label{decay:eq-growth-final-probability}
\bP \Big( \big( \Omega \backslash E^\lambda \big) \Big| \Fc_{t_0} \Big)  \leq \exp\big( - \lambda^2 \big).
\end{equation}

\emph{Fourth step: Conclusion.} 
We now define the non-negative random variables  
\begin{align*}
X&:= \max\Big( \big\| \pregA \big\|_{\Beta}^\gamma, \big\| \pregphi\big\|_{L_x^r} \Big) \\  
\qquad \text{and} \qquad 
Y&:= e^{2c_0} \max \Big(  \big\| \pregA_0 \big\|_{\Beta}^\gamma, \big\| \pregphi_0 \big\|_{L_x^r} \Big)
+ C_6 \max\Big( \big\| \initiallinear[\bfzero] \big\|_{L_t^\infty \Cs_x^{-\kappa}([t_0,t_0+1])},\big\| \initialphilinear[\bfzero] \big\|_{L_t^\infty \GCs^{-\kappa}([t_0,t_0+1])} \Big)^{\frac{2\kappafour}{\kappa}}.
\end{align*}
Due to our estimates from \eqref{decay:eq-growth-p11} and \eqref{decay:eq-growth-final-probability}, the conditions in \eqref{prelim:eq-comparison} from Lemma \ref{prelim:lem-comparison} are satisfied. As a result, we obtain from Lemma \ref{prelim:lem-comparison} that
\begin{align*}
 \E \big[ X^p \big]^{\frac{1}{p}} 
\leq  \E \big[ Y^p \big]^{\frac{1}{p}}
+ \frac{1}{2} C_7  p^{\frac{4\kappafour}{\kappa}} 
\leq e^{2c_0} \E \Big[ \max \Big(  \big\| \pregA_0 \big\|_{\Beta}^\gamma, \big\| \pregphi_0 \big\|_{L_x^r} \Big)^p \Big]^{\frac{1}{p}} +  C_7  p^{\frac{4\kappafour}{\kappa}},
\end{align*}
where the last inequality was obtained using the triangle inequality in $L^p(\Omega)$ and Lemma \ref{lemma:linear-gc}. 
Since $\kappafour$ is much smaller than one, this implies the desired estimate \eqref{decay:eq-growth}.
\end{proof}

As a direct consequence of Lemma \ref{decay:lem-exponential-growth} (and the local theory discussed in Section \ref{section:gauge-covariance}), we already obtain the following corollary. 

\begin{corollary}[Global well-posedness]\label{decay:cor-gwp}
Let $A_0 \in \Beta(\T^2 \rightarrow \R^2)$ and let $\phi_0 \in L_x^r(\T^2\rightarrow \C)$.
Then, the stochastic Abelian-Higgs equations \eqref{AH:eq-evolution-A}-\eqref{AH:eq-evolution-phi} have a global solution 
\begin{equation*}
(A,\phi) \in C_{t,\textup{loc}} \Cs_x^{-\kappa}([0,\infty)\times \T^2 \rightarrow \R^2 \times \C), 
\end{equation*}
which coincides with the limit of the smooth approximations $(A_{\leq N},\phi_{\leq N})$ obtained from \eqref{intro:eq-SAH-smooth}. 
\end{corollary}

The growth in Lemma \ref{decay:lem-exponential-growth} is due to the choice of norms in \eqref{decay:eq-growth}, which are not invariant under the discrete gauge symmetry from \eqref{intro:eq-group-action-Zd}. By working only with gauge invariant norms, it is possible to obtain a decay estimate, which is the subject of our next lemma.

\begin{lemma}[Gauge invariant exponential decay estimate]\label{decay:lem-decay-unit} 
Let $(A,\phi)$ be as in Corollary \ref{decay:cor-gwp}, i.e., a global solution of \eqref{intro:eq-SAH}. Furthermore, let $t_0\geq 0$ and let $\nregphi_0\colon \T^2 \rightarrow \C$ be an $\Fc_{t_0}$-measurable random distribution such that, for all $p\geq 2$, 
\begin{equation}\label{decay:eq-gauge-varphi-0}
\E \big[ \| \nregphi_0 \|_{\GCs^{-\kappa}(\T^2 \rightarrow \C)}^p \big]^{\frac{1}{p}} \leq C_0 p^{\frac{1}{2}}.
\end{equation}
Then, there exists an $\Fc_{t_0+1}$-measurable random distribution $\nregphi \colon \T^2\rightarrow \C$ such that, for all $p\geq 2$, 
\begin{equation}\label{decay:eq-gauge-varphi}
\E \big[ \| \nregphi \|_{\GCs^{-\kappa}(\T^2 \rightarrow \C)}^p \big]^{\frac{1}{p}} \leq C_0 p^{\frac{1}{2}}
\end{equation}
and 
\begin{equation}\label{decay:eq-gauge-decay}
\begin{aligned}
&\, \E \Big[ \max \Big( \big\| A(t_0+1) - \initiallinear[\bfzero] (t_0+1) \big\|_{\GBeta}^\gamma, 
\big\| \phi(t_0+1) - \nregphi \big\|_{L_x^r} \Big)^p \Big]^{\frac{1}{p}} \\ 
\leq&\, e^{-\frac{1}{2}c_0}  \E \Big[ \max \Big( \big\| A(t_0) - \initiallinear[\bfzero](t_0) \big\|_{\GBeta}^\gamma, 
\big\| \phi(t_0) - \nregphi_0 \big\|_{L_x^r} \Big)^p \Big]^{\frac{1}{p}} 
+ C_7 \, p^{\frac{1}{\kappa}}.
\end{aligned}
\end{equation}
Furthermore, we have the $L_t^\infty$-estimate
\begin{equation}\label{decay:eq-gauge-decay-linfty}
\begin{aligned}
&\, \E \Big[ \max \Big( \big\| A \big\|_{L_t^\infty \GCs^{-\kappa}([t_0,t_0+1])}^\gamma, 
\big\| \phi \big\|_{L_t^\infty \GCs^{-\kappa}([t_0,t_0+1])} \Big)^p \Big]^{\frac{1}{p}} \\ 
\leq&\, C_0  \E \Big[ \max \Big( \big\| A(t_0) - \initiallinear[\bfzero] (t_0) \big\|_{\GBeta}^\gamma, 
\big\| \phi(t_0) - \nregphi_0 \big\|_{L_x^r} \Big)^p \Big]^{\frac{1}{p}} 
+ C_0 C_7 \, p^{\frac{1}{\kappa}}.
\end{aligned}
\end{equation}
\end{lemma}

\begin{proof}
Since \eqref{decay:eq-gauge-decay-linfty} can easily be obtained by combining the argument below with \eqref{decay:eq-short-Linfty}, we only prove \eqref{decay:eq-gauge-decay}.
For expository purposes, we split the argument into two steps. \\ 

\emph{First step: Using gauge covariance.}
Let $n_0 \colon (\Omega,\Fc)\rightarrow \Z^2$ be any $\Fc_{t_0}$-measurable random integer vector. Then, both the assumption in \eqref{decay:eq-gauge-varphi-0} and estimates in \eqref{decay:eq-gauge-varphi} and \eqref{decay:eq-gauge-decay} are invariant under the gauge transformation
\begin{equation}
\big( \nregphi_0, \nregphi , A , \phi \big) \mapsto \big( e^{-\icomplex n_0 x} \nregphi_0, e^{-\icomplex n_0 x}  \nregphi, A + n_0, e^{-\icomplex n_0 x}  \phi \big).  
\end{equation}
From the gauge covariance of our solution theory, i.e., Theorem \ref{thm:gauge-covariance}, we also have that $(A+n_0, e^{-\icomplex n_0 x} \phi)$ is a solution of the stochastic Abelian-Higgs equations with stochastic forcing $\xi$ and $e^{-\icomplex n_0 x} \zeta$ on the time-interval $[t_0,t_0+1]$. Since $n_0$ is $\Fc_{t_0}$-measurable, it further holds that $\Law(e^{-\icomplex n_0 x} \zeta|_{[t_0,t_0+1]})=\Law(\zeta|_{[t_0,t_0+1]})$, i.e., the stochastic forcing $\zeta|_{[t_0,t_0+1]}$ is gauge invariant in law. 
After possibly replacing $\nregphi_0,\nregphi,A,\phi$, and $\zeta$ with their gauge transformed counterparts, we may therefore assume that, almost surely, 
\begin{equation}\label{decay:eq-gauge-decay-p1}
\Big| \sfint \big( A - \initiallinear[\bfzero] \big) (t_0) \Big| \leq 2. 
\end{equation}
Due to Lemma \ref{lemma:linear-gc}, $\initialphilinear[\bfzero](t_0+1)$ satisfies \eqref{decay:eq-gauge-varphi}, and it therefore now suffices to prove that 
\begin{equation}\label{decay:eq-gauge-decay-p2}
\begin{aligned}
&\, \E \Big[ \max \Big( \big\| A(t_0+1) - \initiallinear[\bfzero] (t_0+1) \big\|_{\Beta}^\gamma, 
\big\| \phi(t_0+1) - \initialphilinear[\bfzero] (t_0+1) \big\|_{L_x^r} \Big)^p \Big]^{\frac{1}{p}} \\ 
\leq&\, e^{-\frac{1}{2}c_0}  \E \Big[ \max \Big( \big\| A(t_0) - \initiallinear[\bfzero]  (t_0) \big\|_{\Beta}^\gamma, 
\big\| \phi(t_0) - \nregphi_0 \big\|_{L_x^r} \Big)^p \Big]^{\frac{1}{p}} 
+ \tfrac{1}{2} C_7 \, p^{\frac{1}{\kappa}}.
\end{aligned}
\end{equation}
We note that, since the $\GBeta$-norm and $C_7$ have been replaced by the $\Beta$-norm and $\tfrac{1}{2} C_7$, respectively, \eqref{decay:eq-gauge-decay-p2} is slightly stronger than \eqref{decay:eq-gauge-decay}. In particular, the contribution of the remaining mean from \eqref{decay:eq-gauge-decay-p1} on the right-hand side of \eqref{decay:eq-gauge-decay-p2} is irrelevant. \\

\emph{Second step: Modification of the proof of Lemma \ref{decay:lem-exponential-growth}.}
The rest of the proof of Lemma \ref{decay:lem-decay-unit} is similar as in the proof of Lemma \ref{decay:lem-exponential-growth}, and we only sketch the remaining argument. To simplify the notation, we set
\begin{alignat*}{3}
\pregA_0 &:= A(t_0)-\initiallinear[\bfzero](t_0),  \qquad & \qquad 
\pregA &:=A(t_0+1)-\initiallinear[\bfzero](t_0+1), \\ 
\pregphi_0 &:= \phi(t_0) - \nregphi_0,   \qquad & \qquad 
\quad \pregphi &:= \phi(t_0+1) - \initialphilinear[\bfzero](t_0+1).
\end{alignat*}
Similar as in \eqref{decay:eq-growth-p1}, we let $L$ be the smallest dyadic integer satisfying 
\begin{equation*}
L \geq C_5 \max \Big( \lambda^{\frac{4}{\kappa}}, \big\| \pregA_0 \big\|_{\Beta}^{\frac{2}{\eta_3}}, 
\big\| \pregphi_0 \big\|_{L_x^r}^{\frac{2}{\eta_3}}, \big\| \initiallinear[\bfzero] (t_0) \big\|_{\Cs_x^{-\kappa}}^{\frac{2}{\kappa}}, \big\| \nregphi_0 \big\|_{\GCs^{-\kappa}}^{\frac{2}{\kappa}} \Big).
\end{equation*}
We let the number of total steps $J$, time-scales $(\tau_j)_{j=0}^{J-1}$, times  $(t_j)_{j=1}^J$, connection one-forms $(\pregA_j)_{j=1}^J$, and scalar fields $(\pregphi_j)_{j=1}^J$ be defined as in the proof of Lemma \ref{decay:lem-exponential-growth}, i.e., as in \eqref{decay:eq-growth-q1}-\eqref{decay:eq-growth-q3}. We also let the events $(E^\lambda_j)_{j=0}^{J-1}$, $(E^\lambda_{\leq j})_{j=0}^{J-1}$,  and $E^\lambda$ be defined as in the proof of Lemma \ref{decay:lem-exponential-growth}, i.e., as in \eqref{decay:eq-growth-p3}-\eqref{decay:eq-growth-p4}. \\

As in the proof of \eqref{decay:eq-growth-p5}, we can use the estimates \eqref{decay:eq-short-phi} and \eqref{decay:eq-short-A} from Proposition \ref{decay:prop-short} and \eqref{decay:eq-crude} from Lemma \ref{decay:lem-crude} to obtain that
\begin{equation}\label{decay:eq-gauge-decay-p3}
 \ind_{E^\lambda} \max_{j=1,\hdots,J} \max \Big( \big\| \pregA_j \big\|_{\Beta}^\gamma, \big\| \pregphi_j \big\|_{L_x^r} \Big) 
\leq e^{c_0+c_1} \max \Big( \big\| \pregA_0 \big\|_{\Beta}^\gamma, \big\| \pregphi_0 \big\|_{L_x^r}, C_4 L^{\kappafour} \Big). 
\end{equation}
By using \eqref{decay:eq-short-A-mean} from Proposition \ref{decay:prop-short}, using \eqref{decay:eq-crude-mean} from Lemma \ref{decay:lem-crude}, using \eqref{decay:eq-gauge-decay-p1}, and using \eqref{decay:eq-gauge-decay-p3}, we obtain that
\begin{equation}\label{decay:eq-gauge-decay-p4}
\begin{aligned}
 &\, \ind_{E^\lambda} \max_{j=1,\hdots,J} \big| \sfint \pregA_j \big|^\gamma \\
\leq&\,  e^{c_0+c_1} \max \Big( 2^{\gamma}, \big( e^{c_0+c_1} \big\| \pregA_0 \big\|_{\Beta}\big)^{\gamma (1-\nu)}, \big( e^{c_0+c_1} \big\| \pregphi_0 \big\|_{L_x^r}\big)^{1-\nu},  e^{c_0+c_1} \frac{3}{4} C_4 L^{\kappafour} \Big). 
\end{aligned}
\end{equation}
Since $C_4$ has been chosen as sufficiently large depending on $\nu$, and $c_0, c_1$ are sufficiently small, \eqref{decay:eq-gauge-decay-p4} directly implies that 
\begin{equation}\label{decay:eq-gauge-decay-p5}
 \ind_{E^\lambda} \max_{j=1,\hdots,J} \big| \sfint \pregA_j \big|^\gamma 
 \leq \frac{7}{8} \max \Big( \big\| \pregA_0 \big\|_{\Beta}^\gamma, \big\| \pregphi_0 \big\|_{L_x^r}, C_4 L^{\kappafour} \Big). 
\end{equation}
Equipped with \eqref{decay:eq-gauge-decay-p5}, we now iterate our estimates \eqref{decay:eq-short-phi} and \eqref{decay:eq-short-A-refined} from Proposition \ref{decay:prop-short} and our estimate \eqref{decay:eq-crude} from Lemma \ref{decay:lem-crude}, which yields that 
\begin{equation}\label{decay:eq-gauge-decay-p6}
\begin{aligned}
& \ind_{E^\lambda} \max \Big( \big\| \pregA_j \big\|_{\Beta}^\gamma, \big\| \pregphi_j \big\|_{L_x^r} \Big) \\
\leq&\, e^{-c_0 (t_j-t_0)} \Big( \ind_{j\leq J-1} + e^{c_1} e^{c_0 (t_J-t_{J-1})} \ind_{j=J}\Big) \max \Big( \big\| \pregA_0 \big\|_{\Beta}^\gamma, \big\| \pregphi_0 \big\|_{L_x^r}, e^{c_0 (t_j-t_0)} C_4 L^{\kappafour} \Big)
\end{aligned}
\end{equation}
for all $0\leq j \leq J$. We emphasize that, due to \eqref{decay:eq-gauge-decay-p5}, the means  $\sfint \pregA_j$ do not prohibit the decay in \eqref{decay:eq-gauge-decay-p6}. 
To simplify the expression in \eqref{decay:eq-gauge-decay-p6}, 
we recall from \eqref{decay:eq-growth-q1}-\eqref{decay:eq-growth-q3} that 
$t_{J}-t_{J-1}=\tau_{J-1} \leq 2 \tau_{J-1}^\ast\leq 2c_3$, which implies that 
\begin{equation*}
\ind_{E^\lambda} \max \Big( \big\| \pregA_j \big\|_{\Beta}^\gamma, \big\| \pregphi_j \big\|_{L_x^r} \Big) 
\leq e^{-c_0 (t_j-t_0)} e^{2c_1}
\max \Big( \big\| \pregA_0 \big\|_{\Beta}^\gamma, \big\| \pregphi_0 \big\|_{L_x^r}, e^{c_0 (t_j-t_0)} C_4 L^{\kappafour} \Big)
\end{equation*}
for all $0\leq j \leq J$.
By arguing as in the proof of \eqref{decay:eq-growth-p11}, we then obtain that 
\begin{equation}\label{decay:eq-gauge-decay-p7}
\begin{aligned}
&\, \ind_{E^\lambda} \max \Big( \big\| A(t_0+1) - \linear (t_0+1) \big\|_{\Beta}^\gamma, \big\| \phi(t_0+1) - \initialphilinear[\bfzero](t_0+1) \big\|_{L_x^r} \Big) \\ 
\leq&\, e^{-\frac{1}{2}c_0} \max \Big( \big\| \pregA_0 \big\|_{\Beta}^\gamma, \big\| \pregphi_0 \big\|_{L_x^r} \Big) \\
+&\,  C_6 \max\Big( \big\| \initiallinear[\bfzero] \big\|_{L_t^\infty \Cs_x^{-\kappa}([t_0,t_0+1])},\big\| \initialphilinear[\bfzero] \big\|_{L_t^\infty \GCs^{-\kappa}([t_0,t_0+1])}, \big\| \nregphi_0 \big\|_{\GCs^{-\kappa}} \Big)^{\frac{2\kappafour}{\kappa}} + C_6 \lambda^{\frac{4\kappafour}{\kappa}}. 
\end{aligned}
\end{equation}
Due to our assumption \eqref{decay:eq-gauge-varphi-0}, 
the $\nregphi_0$-term in \eqref{decay:eq-gauge-decay-p7} satisfies the same moment estimates as the $\initiallinear[\bfzero]$ and $\initialphilinear[\bfzero]$-terms. The remainder of the proof of 
\eqref{decay:eq-gauge-decay} is therefore exactly as in the proof of 
\eqref{decay:eq-growth}, and we omit the details. 
\end{proof}

\subsection{Proof of Theorem \ref{intro:thm-abelian-higgs}}\label{section:decay-proof}
The main theorem of this article (Theorem \ref{intro:thm-abelian-higgs}) now directly follows from our earlier results. The global well-posedness of the stochastic Abelian-Higgs model \eqref{intro:eq-SAH} follows from \mbox{Corollary \ref{decay:cor-gwp}}. The exponential bound in Theorem \ref{intro:thm-abelian-higgs}.\ref{intro:item-AH-1} follows from \mbox{Lemma \ref{decay:lem-exponential-growth}}. 
To be more precise, the exponential bound in Theorem \ref{intro:thm-abelian-higgs}.\ref{intro:item-AH-1} is obtained by first iterating \eqref{decay:eq-growth} and then using \eqref{decay:eq-growth-linfty} once. The gauge invariant, uniform-in-time bound in Theorem \ref{intro:thm-abelian-higgs}.\ref{intro:item-AH-2} follows similarly  from \mbox{Lemma \ref{decay:lem-decay-unit}}. Finally, the gauge covariance in Theorem \ref{intro:thm-abelian-higgs}.\ref{intro:item-AH-3} follows from Theorem \ref{thm:gauge-covariance}. 

\begin{appendix}

\section{Miscellaneous}\label{appendix:misc}

In this appendix, we collect some miscellaneous results.

\begin{proof}[Proof of the Duhamel integral estimate with heat-kernel weights (Lemma \ref{prelim:lem-Duhamel-weighted}):] To simplify the notation, we write 
\begin{equation*}
\| F \| := \big\| p^\theta(s,y;t,x) F(s,y) \big\|_{L_t^\infty L_x^\infty L_s^q L_y^r(I\times \T^2 \times I \times \T^2)}.
\end{equation*}
Furthermore, we let $(\chi_\delta)_{\delta \in 2^{-\mathbb{N}_0}}$ be a smooth partition of unity such that, for each $\delta \in 2^{-\mathbb{N}_0}$, $\chi_\delta$ is supported on the interval $[\frac{\delta}{2},2\delta]$ and satisfies the estimate $|\chi_\delta^\prime|\lesssim \delta^{-1}$. Using the partition of unity, we then decompose 
\begin{align*}
\Duh \big[ F \big](t,x) 
= \sum_{\delta \in 2^{-\mathbb{N}_0}}\Duh_\delta \big[ F \big](t,x)
:=  \sum_{\delta \in 2^{-\mathbb{N}_0}} \iint \ds \dy \, \chi_\delta(t-s) p(s,y;t,x) F(s,y). 
\end{align*}
For all $\delta \in 2^{-\mathbb{N}_0}$, we now prove the three estimates
\begin{align}
\big\| \Duh_\delta \big[ F \big] \big\|_{L_t^\infty L_x^\infty}
&\lesssim \delta^{1+\theta- \frac{1}{q}-\frac{1}{r}} \| F\|, \label{prelim:eq-Duhamel-p1} \\ 
\big\| \nabla_x \Duh_\delta \big[ F \big] \big\|_{L_t^\infty L_x^\infty}
&\lesssim \delta^{\frac{1}{2}+\theta- \frac{1}{q}-\frac{1}{r}} \| F\|,\label{prelim:eq-Duhamel-p2}\\ 
\big\| \partial_t \Duh_\delta \big[ F \big] \big\|_{L_t^\infty L_x^\infty}
&\lesssim \delta^{\theta- \frac{1}{q}-\frac{1}{r}} \| F\|. \label{prelim:eq-Duhamel-p3}
\end{align}
In order to prove \eqref{prelim:eq-Duhamel-p1}, we first use H\"{o}lder's inequality, which yields 
\begin{equation}\label{prelim:eq-Duhamel-p4}
\big\| \Duh_\delta \big[ F \big] \big\|_{L_t^\infty L_x^\infty} 
\lesssim \big\| \chi_\delta(t-s) p^{1-\theta}(s,y;t,x) \big\|_{L_t^\infty L_x^\infty L_s^{q^\prime} L_y^{r^\prime}}
\| F \|. 
\end{equation}
Using a direct calculation, we obtain that 
\begin{align*}
 \big\| \chi_\delta(t-s) p^{1-\theta}(s,y;t,x) \big\|_{L_t^\infty L_x^\infty L_s^{q^\prime} L_y^{r^\prime}} 
\lesssim&\,  \Big\| \mathbf{1}\big\{ |t-s|\sim \delta\big\} (t-s)^{-1+\theta} 
\exp\Big( - (1-\theta) \tfrac{|x-y|^2}{2(t-s)} \Big)  \Big\|_{L_t^\infty L_x^\infty L_s^{q^\prime} L_y^{r^\prime}} \\ 
\lesssim&\,   \Big\| \mathbf{1}\big\{ |t-s|\sim \delta\big\} (t-s)^{-1+\theta+\frac{1}{r^\prime}} \Big\|_{L_t^\infty L_x^\infty L_s^{q^\prime}}\\
\lesssim&\, \delta^{-1+\theta+\frac{1}{q^\prime}+\frac{1}{r^\prime}} = \delta^{1+\theta-\frac{1}{q}-\frac{1}{r}}. 
\end{align*}
By inserting this into \eqref{prelim:eq-Duhamel-p4}, we obtain \eqref{prelim:eq-Duhamel-p1}. The proofs of \eqref{prelim:eq-Duhamel-p2} and \eqref{prelim:eq-Duhamel-p3} are similar and use that
\begin{align*}
\Big\| \nabla_x \Big( \chi_\delta(t-s)  \nabla_x p(s,y;t,x) \Big) p^{-\theta}(s,y;t,x) \Big\|_{L_t^\infty L_x^\infty L_s^{q^\prime} L_y^{r^\prime}} 
&\lesssim \delta^{\frac{1}{2}+\theta-\frac{1}{q}-\frac{1}{r}},\\ 
\Big\| \partial_t \Big( \chi_\delta(t-s) p(s,y;t,x) \Big) p^{-\theta}(s,y;t,x) \Big\|_{L_t^\infty L_x^\infty L_s^{q^\prime} L_y^{r^\prime}} 
&\lesssim \delta^{\theta-\frac{1}{q}-\frac{1}{r}}. 
\end{align*}
Equipped with \eqref{prelim:eq-Duhamel-p1}-\eqref{prelim:eq-Duhamel-p3}, we can now prove the desired estimate \eqref{prelim:eq-Duhamel-weighted}. For all $\delta \in 2^{-\mathbb{N}_0}$ and $\alpha\in [0,1]$, it holds that 
\begin{equation}\label{prelim:eq-Duhamel-p5}
\begin{aligned}
&\, \big\| \Duh_\delta\big[ F \big] \big\|_{C_t^0 C_x^\alpha}
\leq \big\| \Duh_\delta\big[ F \big]\big\|_{C_t^0 C_x^0}^{1-\alpha}
\big\| \Duh_\delta\big[ F \big]\big\|_{C_t^0 C_x^1}^{\alpha} \\ 
\leq&\,  \big( \delta^{1+\theta- \frac{1}{q}-\frac{1}{r}} \big)^{1-\alpha}
\big( \delta^{\frac{1}{2}+\theta- \frac{1}{q}-\frac{1}{r}} \big)^{\alpha}
\| F \|
= \delta^{1+\theta- \frac{1}{q}-\frac{1}{r}-\frac{\alpha}{2}} \| F \|. 
\end{aligned}
\end{equation}
Due to our assumption \eqref{prelim:eq-Duhamel-p3}, we can sum \eqref{prelim:eq-Duhamel-p5} over all $\delta \in 2^{-\mathbb{N}_0}$, which then yields the desired $C_t^0 C_x^\alpha$-estimate in \eqref{prelim:eq-Duhamel-weighted} (the factor of $|I|^{\gamma}$ arises since we may restrict to summing over $\delta \lesssim |I|$). The proof of the $C_t^{\frac{\alpha}{2}} C_x^0$-estimate is similar.
\end{proof}

\begin{proof}[Proof of complex Hermite polynomial expansion (Lemma \ref{lemma:complex-hermite-polynomial-expansion})]
We proceed by induction. The claim trivially holds if $m$ or  $n$ is zero. Now suppose the claim has been shown for $H_{m, n}$, as well as all $H_{m_1, n_1}$ for $m_1 \leq m$ and $n_1 \leq n$. We show the claim for $H_{m+1, n}$. We omit the proof for $H_{m, n+1}$ since it will be analogous. By the definition of $H_{m+1, n}$, we have that
\begin{equs}
H_{m+1, n}(z+w, \ovl{z+w}) = (z+w) H_{m, n}(z+w, \ovl{z+w}) - n \sigma^2 H_{m, n-1}(z+w, \ovl{z+w}).
\end{equs}
By the inductive assumption, we have that
\begin{equs}
(z+w) H_{m, n}(z+w, \ovl{z+w}) = &~~\sum_{k_1 = 0}^m \sum_{k_2 = 0}^n \binom{m}{k_1} \binom{n}{k_2} w^{m+1-k_1} \bar{w}^{n-k_2} H_{k_1, k_2}(z, \bar{z}) \\
&+ \sum_{k_1 = 0}^m \sum_{k_2 = 0}^n \binom{m}{k_1} \binom{n}{k_2} w^{m-k_1} \bar{w}^{n-k_2} z H_{k_1, k_2}(z, \bar{z}) =: I_1 + I_2,
\end{equs}
as well as (using that $\binom{n-1}{k_2} n = \binom{n}{k_2+1} (k_2 + 1)$)
\begin{equs}
n\sigma^2 H_{m, n-1}(z+w, \ovl{z+w}) &= \sum_{k_1 = 0}^m \sum_{k_2=0}^{n-1} \binom{m}{k_1} \binom{n-1}{k_2}n \sigma^2 w^{m-k_1} \ovl{w}^{n-1-k_2} H_{k_1, k_2}(z, \bar{z}) \\
&= \sum_{k_1 = 0}^m \sum_{k_2=0}^{n-1} \binom{m}{k_1} \binom{n}{k_2+1} (k_2 + 1) \sigma^2 w^{m-k_1} \ovl{w}^{n-1-k_2} H_{k_1, k_2}(z, \bar{z}) \\
&= \sum_{k_1 = 0}^m \sum_{k_2=1}^{n} \binom{m}{k_1} \binom{n}{k_2} k_2 \sigma^2 w^{m-k_1} \ovl{w}^{n-k_2} H_{k_1, k_2-1}(z, \bar{z}) =: I_3.
\end{equs}
Again using the definition of $H_{m, n}$, we have that
\begin{equs}
I_2 - I_3 &= \sum_{k_1 =0}^m \binom{m}{k_1} w^{m-k_1} \bar{w}^n z H_{k_1, 0}(z, \bar{z}) + \sum_{k_1 = 0}^m \sum_{k_2 = 1}^n \binom{m}{k_1} \binom{n}{k_2} w^{m-k_1} \bar{w}^{n-k_2} H_{k_1 + 1, k_2}(z, \bar{z}) \\
&= \sum_{k_1=1}^{m+1} \binom{m}{k_1 -1} \binom{n}{0} w^{m+1-k_1} \bar{w}^n H_{k_1, 0}(z, \bar{z}) + \sum_{k_1 = 1}^{m+1}  \sum_{k_2 = 1}^n \binom{m}{k_1 - 1} \binom{n}{k_2} w^{m+1-k_1} \bar{w}^{n-k_2} H_{k_1, k_2}(z, \bar{z}) \\
&= \sum_{k_1=1}^{m+1} \sum_{k_2 = 0}^n \binom{m}{k_1 - 1} \binom{n}{k_2} w^{m+1-k_1} \bar{w}^{n-k_2} H_{k_1, k_2}(z, \bar{z}).
\end{equs}
Next, observe that
\begin{equs}
I_1 &= \sum_{k_2=0}^n \binom{m+1}{0} \binom{n}{k_2} w^{m+1} \bar{w}^{n-k_2} H_{0, k_2}(z, \bar{z}) + \sum_{k_1=1}^m \sum_{k_2 = 0}^n \binom{m}{k_1} \binom{n}{k_2} w^{m+1-k_1} \bar{w}^{n-k_2} H_{k_1, k_2}(z, \bar{z}) .
\end{equs}
To finish, we add together the expressions for $I_1$ and $I_2 - I_3$, and use the identities $\binom{m+1}{k_1} = \binom{m}{k_1} + \binom{m}{k_1 - 1}$ for $1 \leq k_1 \leq m$, and $\binom{m}{m} = \binom{m+1}{m+1}$.
\end{proof}

\begin{proof}[Proof of the Feynman-Kac-It\^{o} formula (Lemma \ref{lemma:feynman-kac-ito-formula})]
First, suppose $A$ is smooth. Let $T > 0$, $\phi_0 \in C^\infty(\T^2\rightarrow \C)$, and consider a solution to the covariant heat equation:
\begin{equs}\label{eq:feynman-kac-proof-intermediate-1}
(\ptl_t - \covd_A^j \covd_{A, j}) \phi = 0,  ~~ \phi(0) = \phi_0.
\end{equs}
Let $\E_x$ denote expectation with respect to a Brownian motion $W$ in $\T^2$ of rate $2$ (i.e. $dW_t^i dW_t^i = 2dt$ for each $i \in [2]$). 
Define the process
\begin{equs}
P_t := \exp\bigg(\icomplex \int_0^t A(T-u, W_u) \circ dW_u\bigg),
\end{equs}
where the stochastic integral is a Stratonovich integral. By the chain rule for Stratonovich integrals, It\^{o}-Stratonovich conversion, and the It\^{o} formula, we may compute
\begin{equs}
dP_t &= \icomplex P_t A(T - t, W_t) \circ dW_t \\
&= \icomplex P_t A(T - t, W_t) \cdot dW_t + \frac{1}{2} \icomplex d(P_t A_j(T - t, W_t)) dW_t^j \\
&= P_t \Big( \icomplex A(T - t, W_t) \cdot dW_t + \icomplex (\ptl_j A^j)(T - t, W_t) dt - |A(T - t, W_t)|^2 dt\Big).
\end{equs}
Here, $A(T - t, W_t) \cdot dW_t = A_j(T - t, W_t) dW_t^j$ is an It\^{o} differential. Then again by the It\^{o} formula, we may compute that (here $(\covd_A \phi) \cdot dW_t = (\covd_{A, j} \phi) dW_t^j$)
\begin{equs}
d(P_t \phi(T - t, W_t)) &= P_t (\covd_A \phi)(T - t, W_t) \cdot dW_t - P_t \Big((\ptl_t - \covd_{A(T-t)}^j \covd_{A(T-t), j}) \phi\Big)(T - t, W_t) dt \\
&= P_t (\covd_A \phi)(T - t, W_t) \cdot dW_t.
\end{equs}
This shows that $P_t \phi(T - t, W_t)$ is a martingale, and thus
\begin{equs}
\phi(T, x) &= \E_x [P_0 \phi(T - 0, W_0)] = \E_x [P_T \phi(0, W_T)] \\
&=\int dy p(0, x; T, y) \phi_0(y) \E_{(0, x) \ra (T, y)} \bigg[\exp\Big(\icomplex \int_0^T A(T - u, W_u) \circ dW_u\Big) \bigg].
\end{equs}
By time-reversal properties of Brownian bridge and Stratonovich integrals, we further obtain
\begin{equs}
\phi(T, x) = \int dy p(0, y; T, x) \phi_0(y) \E_{(0, y) \ra (T, x)} \bigg[ \exp\Big(-\icomplex \int_0^T A(u, W_u) \circ dW_u \Big) \bigg].
\end{equs}
Finally, since $A$ is smooth, we may apply It\^{o}-Stratonovich conversion to obtain
\begin{equs}
\int_0^T A(u, W_u) \circ dW_u  = \int_0^T A(u, W_u) \cdot dW_u + \int_0^T (\ptl_j A^j)(u, W_u) du.
\end{equs}
The previous two displays imply the desired formula for the covariant heat kernel, in the case that $A$ is smooth.

In the case of general $A : [0, T] \times \T^2 \ra \R^2$ continuous with $\ptl_j A^j : [0, T] \times \T^2 \ra \R$ continuous, we may take a sequence of smooth $\{A_n\}_{n \geq 1}$ such that
\begin{equs}
\lim_{n \toinf} \|A_n - A\|_{C_t^0 C_x^0([0, T] \times \T^2)} + \|\ptl_j (A_n^j - A^j)\|_{C_t^0 C_x^0([0, T] \times \T^2)} = 0.
\end{equs}
Let $\phi_n$ denote the solution to \eqref{eq:feynman-kac-proof-intermediate-1} with $A$ replaced by $A_n$. By standard Schauder estimates, we have that $\|\phi_n - \phi\|_{C_t^0 C_x^0([0, T] \times \T^2)} \ra 0$. By combining this with our previous result applied to $A_n$, as well as It\^{o} isometry and the convergence of the $A_n$, we may obtain
\begin{equs}
\phi(T, x) = \int dy p(0, y; T, x) \phi_0(y) \E_{(0, y) \ra (T, x)} \bigg[\exp\Big(-\icomplex \int_0^T A(u, W_u) \cdot dW_u - \icomplex \int_0^T (\ptl_j A^j)(u, W_u) du\Big) \bigg],
\end{equs}
which implies the desired formula for the heat kernel.
\end{proof}

We now state and prove a variant of our estimates for linear stochastic objects from Proposition \ref{prop:control-of-enhanced-data-set}. 

\begin{lemma}\label{lemma:linear-gc}
There exists a constant $C=C(\kappa)\geq 1$ such that, for all $t_0 \geq 0$ and $\power\geq 2$, it holds that 
\begin{equation}\label{eq:linear-gc}
\E \Big[ \big\| \initiallinear[\bfzero] \big\|_{C_t^0 \Cs_x^{-\kappa}([t_0,t_0+1]\times \T^2 \rightarrow \R^2)}^{\power} \Big]^{\frac{1}{\power}} + \E \Big[ \big\| \initialphilinear[\bfzero] \big\|_{C_t^0 \GCs^{-\kappa}([t_0,t_0+1]\times \T^2 \rightarrow \C)}^{\power} \Big]^{\frac{1}{\power}} \leq C \sqrt{\power},
\end{equation}
where  $\initiallinear[\bfzero]$ and $\initialphilinear[\bfzero]$ are linear stochastic object with zero initial data at time $t=0$ and $\GCs^{-\kappa}$ is as in \eqref{intro:eq-gauge-invariant-norms}. 
\end{lemma}

\begin{proof}[Proof of Lemma \ref{lemma:linear-gc}]
The first estimate in \eqref{eq:linear-gc} follows directly from the proof of \cite[Lemma 5.6]{BC23}. Using standard reductions (see e.g. \cite[Appendix A]{BC23}), it suffices to prove for all $m\in \Z^2$ and $N\geq 1$ that 
\begin{equation}\label{eq:linear-gc-p1}
\sup_{t\in [t_0,t_0+1]} \E \Big[ \big\| e^{-\icomplex m x} P_N \initialphilinear[\bfzero] \big\|_{H_x^{-\kappa}}^2 \Big] \lesssim \max\big( \langle m \rangle, N \big)^{-2\kappa}. 
\end{equation}
Using the same argument as in the proof of \cite[Lemma 5.6]{BC23}, it holds that 
\begin{equation*}
\E \Big[ \big\| e^{-\icomplex m x} P_N \initialphilinear[\bfzero] \big\|_{H_x^{-\kappa}}^2 \Big]
\lesssim \sum_{\substack{n\in \Z^2\colon \\ |n|\sim N}}
\langle m-n \rangle^{-2\kappa} \langle n \rangle^{-2}. 
\end{equation*}
Since the sum can be estimated by the right-hand side of \eqref{eq:linear-gc-p1}, this completes the proof.
\end{proof}

\section{High-frequency interactions in the derivative-nonlinearity}\label{section:high}

In this section, we present the proof of Lemma \ref{AH:lem-derivative-high}, which controls all interactions in the derivative nonlinearity $\leray \Im \big( \overline{\phi} \covd_A \phi \big)$ involving high-frequency terms. The proof of Lemma \ref{AH:lem-derivative-high} relies on the perturbative methods from the earlier works \cite{BC23,S21}, and will be distributed over the following three subsections. For notational simplicity, throughout this section we assume that $t_0 = 0$.

\subsection{High-frequency probabilistic hypothesis}

In this subsection, we examine the probabilistic estimates which are needed for the proof of Lemma \ref{AH:lem-derivative-high}. First, throughout this section, we will primarily work with the massive Duhamel operator, which we define as
\begin{equs}
\mDuh(f)(t) := \int_0^t e^{(t-s)(\Delta - 1)} f(s) ds.
\end{equs}
Next, we introduce the following stochastic objects. 

\begin{definition}[Further stochastic objects]\label{high:def-stochastic}
Recall the low and high-frequency linear objects $\linear[\lo] = P_{\leq L} \linear$, $\linear[\hi] = \linear - \linear[\lo]$, $\philinear[\lo] = P_{\leq L} \philinear$, $\philinear[\hi] = \philinear - \philinear[\lo]$. We define the mixed quadratic objects
\begin{alignat*}{3}
\mixedquadratic[\lo,\lo]
&= \mDuh \Big[ \linear[\lo][r][j] \partial_j \philinear[\lo] \Big], 
\qquad \text{and}& \qquad \mixedquadratic[\lo,\hi]
&= \mDuh \Big[ \linear[\lo][r][j] \partial_j \philinear[\hi] \Big], \\
\mixedquadratic[\hi,\lo]
&= \mDuh \Big[ \linear[\hi][r][j] \partial_j \philinear[\lo] \Big],
\qquad \text{and}& \qquad \mixedquadratic[\hi,\hi]
&= \mDuh \Big[ \linear[\hi][r][j] \partial_j \philinear[\hi] \Big].
\end{alignat*}
To simplify the notation, we then define
\begin{equs}
\mixedquadratic[\lo]
:= \mixedquadratic[\lo,\lo]
\quad \text{and} \quad 
\mixedquadratic[\hi]
:= \mixedquadratic[\hi,\lo] + \mixedquadratic[\lo,\hi] + \mixedquadratic[\hi,\hi].
\end{equs}
Similarly, we define $\Aquadratic[\lo]$ and $\Aquadratic[\hi]$. Next, we define the Wick-ordered quadratic nonlinearities 
\begin{equs}
\Aquadraticnl &:= \biglcol\, \big|\linear\big|^2\, \bigrcol, 
\qquad 
\Aquadraticnl[\lo] := \big| \linear[\leqL] \big|^2 - 2 \sigma_{\leq L}^2, 
\qquad \text{and} \qquad 
\Aquadraticnl[\hi] := \Aquadraticnl  - \Aquadraticnl[\lo], 
\end{equs}
and analogously for $\phiquadraticnl$, $\phiquadraticnl[\lo]$, $\phiquadraticnl[\hi]$.
We define the Wick-ordered polynomial nonlinearities
\begin{equs}
\big(\biglcol\, \philinear[][r][j] \ovl{\philinear[][r][k]} \bigrcol\big)_{\hi} := \biglcol \, \philinear[][r][j] \ovl{\philinear[][r][k]} \bigrcol - \biglcol\, \philinear[\lo][r][j] \ovl{\philinear[\lo][r][k]} \bigrcol, ~~ 0 \leq j \leq \frac{q+1}{2}, 0 \leq k \leq \frac{q-1}{2}.
\end{equs}
In the above, we keep to the convention specified in Remark \ref{remark:wick-ordered-abuse-of-notation}, in that we use the limiting Wick-ordered nonlinearity if $\philinear$ is involved, and otherwise we use the Wick-ordered nonlinearity with variance $\sigma^2_{\leq L}$ if $\philinear[\lo]$ is involved.

Finally, because we are working with the limiting equation in Sections \ref{section:Abelian-Higgs} and \ref{section:decay} (recall Remark \ref{remark:limiting-A-nonlinearity}), we need to be careful to define combinations of stochastic objects whose resonances cancel. Define the combined objects
\begin{equs}
4 \ovl{\philinear} \ptl_i^2 \mDuh\big[\,\philinear\big] + \big|\, \philinear\big|^2 &:= \lim_{N \toinf} 4 \ovl{\philinear[\leqN]} \ptl_i^2 \mDuh\big[\,\philinear[\leqN]\big] + \big|\,\philinear[\leqN]\big|^2, ~~ i \in [2], \\
\ovl{\philinear} \ptl_i \ptl_j \mDuh\big[\,\philinear\big] &:= \lim_{N \toinf} \ovl{\philinear[\leqN]} \ptl_i \ptl_j \mDuh\big[\,\philinear[\leqN]\big], ~~ i \neq j \in [2], \\
4 \ovl{\philinear} \covd \mixedquadratic + \linear \big|\,\philinear\big|^2 &:= \lim_{N \toinf} 4 \ovl{\philinear[\leqN]} \covd \mixedquadratic[\leqN] + \linear[\leqN] \big|\,\philinear[\leqN]\big|^2 .
\end{equs}
Here, the limits are in $C_t^0 \Cs_x^{-\kappa}([0, 1] \times \T^2)$. The fact that these limits almost surely exist follows from Proposition \ref{prop:control-of-enhanced-data-set}\footnote{Technically, Proposition \ref{prop:control-of-enhanced-data-set} is for stationary objects, whereas the objects in Appendix \ref{section:high} are non-stationary. However, the arguments from \cite{BC23} that are cited in the proof sketch of Proposition \ref{prop:control-of-enhanced-data-set} can also be modified to handle the non-stationary setting.} and the fact that in the first and third displays, the objects are precisely combined so that their diverging counterterms cancel each other (recall the definition of the enhanced data set \eqref{eq:enhanced-data-set}).
\end{definition}

Similar as in Section \ref{section:Abelian-Higgs}, we now collect several probabilistic estimates in a single hypothesis.

\begin{remark}
We note here that as in Sections \ref{section:cshe}-\ref{section:decay}, we work with non-stationary objects, for the convenience of our globalization argument (recall Remark \ref{remark:local-theory-remarks}\ref{item:stationary-objects}). On the other hand, in Definition \ref{high:def-stochastic}, we defined the Wick-ordered objects by using the stationary counterterm.  This discrepancy is the reason for certain time-weighted estimates appearing in Hypothesis \ref{high:hypothesis-probabilistic} (recall Remark \ref{intro:rem-cshe}\ref{item:time-weights} for a similar remark).
\end{remark}

\begin{hypothesis}[High-frequency probabilistic hypothesis]\label{high:hypothesis-probabilistic} 
Let $\tau \in (0,1]$ be a time-scale. We then assume that for all $-\frac{1}{2}\leq \alpha \leq -10\kappa$, the following estimates are satisfied. Here, all norms are on $[0, \tau] \times \T^2$.
\begingroup
\allowdisplaybreaks
\begin{alignat}{3}
\big\| \philinear[\hi] \big\|_{C_t^0 \Cs_x^\alpha} &\leq L^{\kappa+\alpha},  \qquad & \qquad 
\big\| \philinear[\lo] \big\|_{C_t^0 L_x^\infty} &\leq L^{\kappa}, 
\label{high:eq-probabilistic-philinear} \\ 
\big\| \linear[\hi] \big\|_{C_t^0 \Cs_x^\alpha} &\leq L^{\kappa+\alpha},  \qquad & \qquad 
\big\| \linear[\lo] \big\|_{C_t^0 L_x^\infty} &\leq L^{\kappa}, 
\label{high:eq-probabilistic-Alinear} \\
\big\|\philinear[B, \lo] - \philinear[\lo]\big\|_{C_t^0 \Cs_x^{\frac{1}{2}}} &\leq L^{\eta_3}, \label{high:eq-probabilistic-difference-linear-objects}\\
\big\| t^\kappa \Aquadraticnl[\hi] \big\|_{C_t^0 \Cs_x^\alpha} &\leq L^{2\kappa+\alpha}, \label{high:eq-probabilistic-quadratic} \\ 
\big\| \linear[\hi] \philinear[\lo] \big\|_{C_t^0 \Cs_x^\alpha} + \big\| \linear[\lo] \philinear[\hi] \big\|_{C_t^0 \Cs_x^\alpha} &\leq L^{2\kappa+\alpha}, \qquad & \qquad 
\big\| \linear \philinear[\hi] \big\|_{C_t^0 \Cs_x^\alpha} &\leq L^{2\kappa+\alpha},
\label{high:eq-probabilistic-linear-linear} \\
\big\| t^\kappa \Aquadraticnl[\hi] \philinear[\lo] \big\|_{C_t^0 \Cs_x^\alpha} &\leq L^{3\kappa+\alpha}, \qquad & \qquad 
\big\| t^\kappa \Aquadraticnl \philinear[\hi] \big\|_{C_t^0 \Cs_x^\alpha} &\leq L^{3\kappa+\alpha}, 
\label{high:eq-probabilistic-quadratic-linear} \\
\big\| \Aquadratic[\hi]\big\|_{C_t^0 \Cs_x^{1+\alpha} \cap C_t^{\frac{1}{4}} \Cs_x^{\frac{1}{2}+\alpha}} &\leq L^{2\kappa + \alpha}, \label{high:eq-probabilistic-Aquadratic}\\
\big\| \mixedquadratic[\hi]\big\|_{C_t^0 \Cs_x^{1+\alpha}} &\leq L^{2\kappa + \alpha}, \label{high:eq-probabilistic-mixed-quadratic}\\
\big\| t^\kappa \linear[\hi] \phiquadraticnl \big\|_{C_t^0 \Cs_x^\alpha} &\leq L^{3\kappa + \alpha}, \qquad & \qquad \big\|t^\kappa \linear[\lo] \phiquadraticnl[\hi]\big\|_{C_t^0 \Cs_x^\alpha} &\leq L^{3\kappa + \alpha}, \label{high:eq-probabilistic-linear-phiquadratic}\\
\Big\|t^\kappa \big(\biglcol\,\philinear[][r][j_1] \ovl{\philinear[][r][j_2]}\bigrcol\big)_{\hi}\Big\|_{C_t^0 \Cs_x^{\alpha}} &\leq L^{q\kappa - \alpha}, \qquad & \qquad \Big\| t^\kappa \biglcol \, \philinear[\lo][r][j_1] \ovl{\philinear[\lo][r][j_2]} \bigrcol \Big\|_{C_t^0 L_x^\infty} &\leq L^{q \kappa} , \label{high:eq-probabilistic-polynomial}  \\
\bigg\|\philinear[\lo] \covd \ptl_j \mDuh\big(\philinear[\hi]\big)\bigg\|_{C_t^0 \Cs_x^\alpha} &\leq L^{2\kappa + \alpha}, \qquad & \qquad  \bigg\|\philinear[\hi] \covd \ptl_j \mDuh\big(\philinear[\lo]\big)\bigg\|_{C_t^0 \Cs_x^\alpha} &\leq L^{2\kappa + \alpha}. \label{high:eq-probabilistic-quadratic-no-resonance} 
\end{alignat} 
\endgroup
In \eqref{high:eq-probabilistic-polynomial}, the parameters $j_1, j_2$ satisfy $0 \leq j_1 \leq \frac{q+1}{2}$, $0 \leq j_2 \leq \frac{q-1}{2}$. Additionally, we assume that
\begin{align}
\bigg\| 4 \ovl{\philinear} \ptl_i^2 \mDuh\big(\,\philinear\big) + \big|\,\philinear\big|^2 - \Big(4 \ovl{\philinear[\lo]} \ptl_i^2 \mDuh\big(\,\philinear[\lo]\big) + \big|\, \philinear[\lo]\big|^2\Big) \bigg\|_{C_t^0 \Cs_x^{\alpha}} &\leq L^{2\kappa + \alpha}, \quad i \in [2], \label{eq:philinear-laplacian-Duh-philinear-phiquadraticnl}\\
\bigg\|\ovl{\philinear} \ptl_i \ptl_j \mDuh\big(\,\philinear\big) - \ovl{\philinear[\lo]} \ptl_i \ptl_j \mDuh\big(\,\philinear[\lo]\big)\bigg\|_{C_t^0 \Cs_x^{\alpha}} &\leq L^{2\kappa + \alpha}, \quad i \neq j \in [2], \label{eq:philinear-laplacian-Duh-philiner}\\
\bigg\|4 \ovl{\philinear} \covd \mixedquadratic + \linear \big|\,\philinear\big|^2 - \Big(4 \ovl{\philinear[\lo]} \covd \mixedquadratic[\lo] + \linear[\lo] \big|\, \philinear[\lo]\big|^2\Big)\bigg\|_{C_t^0 \Cs_x^\alpha} &\leq L^{3\kappa + \alpha} .\label{eq:philinear-mixed-quadratic-linear-phiquadraticnl-combined}
\end{align}
\end{hypothesis}

The following proposition shows that the high-frequency probabilistic hypothesis (Hypothesis \ref{high:hypothesis-probabilistic}) is satisfied on an event with high probability. 

\begin{proposition}[Probabilistic estimates for high-frequency terms]\label{high:prop-high-frequency-probabilistic}
Let $\EventHi \subseteq \Omega$ be the event defined by the conditions in  \eqref{high:eq-probabilistic-philinear}-\eqref{eq:philinear-mixed-quadratic-linear-phiquadraticnl-combined} above, and let $c>0$ be a sufficiently small constant. Then, it holds that 
\begin{equs}
\mathbb{P} \big( \Omega \backslash \EventHi \big) 
\leq c^{-1} \exp\big(- c L^\kappa \big).
\end{equs}
\end{proposition}

\begin{proof}
The estimate \eqref{high:eq-probabilistic-difference-linear-objects} follows from Theorem \ref{intro:thm-cshe}\ref{item:thm:-difference-in-linear-objects}. All other estimates follow from small variations of the estimates in \cite[Section 5]{BC23}. The estimates in the previous paper were proven where everything was frequency-localized, and we always had a power of the maximum frequency in the denominator of our bound. Thus, if we are willing to give up on regularity, we gain a corresponding power of the maximum frequency, which is at least $L$ for all the high-frequency objects in Hypothesis \ref{high:hypothesis-probabilistic}. This explains why we have the power of $L^\alpha$ in all the estimates of high-frequency objects. Of course, technically \cite{BC23} dealt with stationary objects, but the same arguments work for non-stationary objects, at the expense of a small $t^\kappa$ time weight for any object which requires renormalization. The proof sketch of Proposition \ref{prop:control-of-enhanced-data-set} also gives the correspondence between the objects in Hypothesis \ref{high:hypothesis-probabilistic} and the results in the previous paper. 
\end{proof}

\subsection{\protect{Non-covariant estimates of $\psi_{\lo}$ and $\psi_{\hi}$}}

In Section \ref{section:cshe} and Section \ref{section:Abelian-Higgs}, all of our estimates of $\scalebox{0.9}{$\philinear[\Blin,\lo]$}$ and $\psi_{\lo}$ were obtained using covariant methods. In contrast, the following estimates of $\scalebox{0.9}{$\philinear[\Blin,\lo]$}$, $\psi_{\lo}$, and $\psi_{\hi}$ will be obtained using non-covariant methods. This is possible since all interactions involving high-frequency terms, such as the interactions in Lemma \ref{AH:lem-derivative-high}, exhibit gains in $L$, which can be used to pay for factors of $\| B(0)\|_{\Cs_x^\eta}$.\\

In the first lemma of this subsection, we bound non-covariant derivatives of $\psi_{\lo}$.

\begin{lemma}[Non-covariant estimate of $\psi_{\lo}$]\label{high:lem-psi-lo}
Let the probabilistic hypothesis and continuity hypothesis (Hypothesis \ref{AH:hypothesis-probabilistic} and \ref{AH:hypothesis-continuity}) be satisfied and let $1\leq p \leq \frac{r}{\eta r+(1-2\eta)}$. Then, it holds that 
\begin{equs}\label{high:eq-psi-lo}
\big\| \psi_{\lo} \big\|_{L_t^p \Bc^{2\eta,p}_x} 
+ \big\| \psi_{\lo} \big\|_{L_t^2 H_x^1} \lesssim L^{2\eta_2}. 
\end{equs}
\end{lemma}

\begin{proof}
We first note that, due to the probabilistic and continuity hypothesis and the choice of $L$, 
\begin{equs}\label{high:eq-psi-lo-p0}
\| \philinear[\lo] \|_{L_t^\infty L_x^\infty} \lesssim L^\kappa, \quad 
\| \psi_{\lo} \|_{L_t^\infty L_x^r} \lesssim L^{\eta_2}, \quad 
\| B \|_{C_t^0 \Cs_x^\eta} \lesssim \| A(0) \|_{\Cs_x^\eta} \lesssim L^{\eta_3}, \quad \text{and} \quad
\| Z \|_{C_t^0 \Cs_x^\eta} \lesssim L^{\eta_2}. 
\end{equs}
Using Proposition \ref{AH:prop-psi-lo-estimate} and \eqref{high:eq-psi-lo-p0}, we then obtain that
\begin{equs}\label{high:eq-psi-lo-p1} 
\big\| \covd_{A_{\lo}} \psi_{\lo} \big\|_{L_{t,x}^2}
&\lesssim \big\| \psi_{\lo}(0) \big\|_{L_x^2} 
+  \| Z(0) \|_{C_t^0 \Cs_x^\eta}^{\frac{q+1}{q} (1+\nu)} + L^{\kappaone (q+1)} 
\lesssim L^{\eta_2} + L^{\eta_2 \frac{q+1}{q} (1+\nu)} + L^{\kappaone (q+1)}  \lesssim L^{2\eta_2}. \qquad 
\end{equs}
Furthermore, using \eqref{high:eq-psi-lo-p0}, we also have that
\begin{equs}\label{high:eq-psi-lo-p2}
\big\| A_{\lo} \psi_{\lo} \big\|_{L_{t,x}^2} 
\lesssim \Big( \big\| \linear[\lo] \big\|_{L_{t,x}^\infty} 
+ \big\| \Blin \big\|_{L_{t,x}^\infty} 
+ \big\| Z  \big\|_{L_{t,x}^\infty} \Big)
\big\| \psi_{\lo} \big\|_{L_t^\infty L_x^2}
\lesssim L^{2\eta_2}.
\end{equs}
By combining \eqref{high:eq-psi-lo-p0}, \eqref{high:eq-psi-lo-p1} , and \eqref{high:eq-psi-lo-p2}, it then follows that
\begin{equs}
\big\| \psi_{\lo} \big\|_{L_t^2 H_x^1}
\lesssim \big\| \covd_{A_{\lo}}  \psi_{\lo} \big\|_{L_{t,x}^2} 
+ \big\| A_{\lo} \psi_{\lo} \big\|_{L_{t,x}^2} 
+ \big\| \psi_{\lo} \big\|_{L_{t,x}^2}
\lesssim L^{2\eta_2},
\end{equs}
which yields the $L_t^2 H_x^1$-estimate in \eqref{high:eq-psi-lo}. In order to obtain the $L_t^p \Bc^{2\eta,p}_x$-estimate in \eqref{high:eq-psi-lo}, we first note that
\begin{equation*}
p \leq \frac{r}{\eta r+(1-2\eta)} \leq \frac{1}{\eta}. 
\end{equation*}
Using interpolation, it then follows that 
\begin{equs}
\big\| \psi_{\lo} \big\|_{L_t^p \Bc^{2\eta,p}_x} 
\leq \big\| \psi_{\lo} \big\|_{L_t^{\frac{1}{\eta}}\Bc^{2\eta,\frac{r}{\eta r+(1-2\eta)}}_x} 
\lesssim \big\| \psi_{\lo} \big\|_{L_t^\infty L_x^r}^{1-2\eta} 
\big\| \psi_{\lo} \big\|_{L_t^2 H_x^1}^{2\eta} \lesssim L^{2\eta_2},
\end{equs}
which yields the $L_t^p \Bc^{2\eta,p}_x$-estimate in \eqref{high:eq-psi-lo}.
\end{proof}

While our estimates of $\psi_{\lo}$ and $\psi_{\hi}$
(from Hypothesis \ref{AH:hypothesis-continuity} and Lemma \ref{high:lem-psi-lo}) can be used to control  $\overline{\psi}_{\lo} \covd \psi_{\hi}$, these estimates cannot be used to control 
$\overline{\scalebox{0.8}{$\philinear[\Blin,\lo]$}} \covd \psi_{\lo}$ or
$\overline{\scalebox{0.8}{$\philinear[\Blin,\lo]$}}\covd \psi_{\hi}$. For the latter interactions, we need more information on the structure of $\psi_{\lo}$ and $\psi_{\hi}$, which is obtained in the following lemma. 

\begin{lemma}[\protect{Para-controlled structure of $\psi_{\hi}$, $\psi_{\lo}$, and $\philinear[\Blin,\lo]$}]\label{high:lem-para-controlled}
Let the probabilistic and continuity hypothesis be satisfied and let $100q \leq p\leq \frac{r}{\eta r+(1-2\eta)}$. Then, it holds that 
\begin{align}
\Big\| \psi_{\hi} - 2 \icomplex\, \mixedquadratic[\hi] - 2\icomplex \big( S + B+ Z\big)^j \parall \partial_j \mDuh \big[ \philinear[\hi] \big]  - 2\icomplex \Slin_{\hi} \parall \ptl_j \mDuh\big[ \philinear[\lo]\big] \Big\|_{L_t^{p/10} \Bc^{1+\eta_3,p}_x} & \lesssim L^{\eta_1-\eta}, \label{high:eq-para-controlled-e1} \\ 
\Big\| \psi_{\lo} - e^{t(\Delta-1)} \psi_{\lo}(0) - 2 \icomplex \,  \mixedquadratic[\lo] - 2\icomplex   (\Slin_{\lo} + Z)^j \parall \partial_j \mDuh \big[ \philinear[\lo] \big] \Big\|_{L_t^{p/10} \Bc^{1+\eta+\eta_3,p}_x} & \lesssim L^{\eta_1}, \label{high:eq-para-controlled-e2} \\ 
\Big\| \philinear[\Blin,\lo] - \philinear[\lo] - 2\icomplex  B^j \parall \partial_j \mDuh \big[ \philinear[\lo] \big]  \Big\|_{L_t^{p/10} \Bc^{1+\eta+\eta_3,p}_x} & \lesssim L^{\eta_1}. \label{high:eq-para-controlled-e3}
\end{align}
\end{lemma}

\begin{remark} In contrast to the right-hand side of \eqref{high:eq-para-controlled-e1}, the right-hand sides of \eqref{high:eq-para-controlled-e2} and \eqref{high:eq-para-controlled-e3} contain no $L^{-\eta}$-factors, and are in fact unbounded in $L$. Nevertheless, in interactions of $\philinear[\Blin,\lo]$ and $\psi_{\lo}$ with $\philinear[\hi]$, the higher-regularity norm on the left-hand sides of \eqref{high:eq-para-controlled-e2} and \eqref{high:eq-para-controlled-e3} will still allow us to gain a $L^{-\eta}$-factor.
\end{remark}

\begin{proof} We only prove \eqref{high:eq-para-controlled-e1}, since the proofs of \eqref{high:eq-para-controlled-e2} and \eqref{high:eq-para-controlled-e3} are similar. 
To simplify the notation, we first let 
\begin{equs}
\psitilde:= \philinear[\Blin,\lo]-\philinear[\lo]+\psi_{\lo}+\psi_{\hi}.
\end{equs}
By Lemma \ref{high:lem-psi-lo} and Hypotheses \ref{AH:hypothesis-probabilistic}, \eqref{high:hypothesis-probabilistic}, and \eqref{AH:hypothesis-continuity}, we have that
\begin{equs}
\big\| \philinear[\Blin,\lo] - \philinear[\lo] \big\|_{C_t^0 \Cs_x^\eta} \leq L^{\eta_3}, \quad
\big\| \psi_{\lo} \big\|_{L_t^p \Bc_x^{2\eta,p}}
\lesssim L^{2\eta_2}, \quad \text{and} \quad 
\big\| \psi_{\hi} \big\|_{C_t^0 \Cs_x^{\eta}} \leq L^{\eta_1-\eta}, 
\end{equs} 
Since $C_t^0 \Cs_x^{\eta}\hookrightarrow L_t^p \Bc_x^{\eta, p}$ it then follows that
\begin{equs}\label{high:eq-para-controlled-p1}
\big\| \psitilde \big\|_{L_t^p \Bc_x^{\eta, p}}
\lesssim L^{2\eta_2}. 
\end{equs}
By \eqref{AH:eq-L-nreg}, \eqref{AH:eq-Slin}, the heat flow smoothing estimate (Lemma \ref{lemma:heat-flow-smoothing}), and the fact that $p < \frac{1}{\eta}$, we have that
\begingroup
\allowdisplaybreaks
\begin{equs}
\|S\|_{L_t^p \Cs_x^\eta} \lesssim \|t^{-\eta}\|_{L_t^p} \|\Slin(0)\|_{\Cs_x^{-\kappa}} \lesssim L^\kappa, \big\|\Slin_{\hi}\big\|_{L_t^p \Cs_x^{\eta}} &\lesssim \big\|t^{-\eta}\big\|_{L_t^p} \big\|P_{> L} \Slin(0)\big\|_{\Cs_x^{-\eta}} \lesssim L^{\kappa - \eta} ,\label{eq:Slin-estimate}\\
\big\|\varphi_{\hi}\big\|_{L_t^p\Cs_x^\eta} &\lesssim L^{\kappa - \eta} .\label{eq:varphi-hi-estimate}
\end{equs}
\endgroup
Now, by \eqref{high:eq-probabilistic-philinear}, \eqref{eq:Slin-estimate}, and product estimates (Lemma \ref{prelim:lem-para-besov}), we have that
\begin{equs}
\big\|\Slin_{\hi} \parall \ptl_j \mDuh\big[ \philinear[\lo]\big]\big\|_{L_t^p \Cs_x^{1+\eta_3}} \lesssim \|\Slin_{\hi}\|_{L_t^p \Cs_x^\eta} \|\philinear[\lo]\big\|_{L_t^\infty \Cs_x^{\eta_3}} \lesssim L^{\kappa - \eta} \times L^{\kappa + \eta_3} \lesssim L^{2\kappa + \eta_3 - \eta},
\end{equs}
which is acceptable. Next, using the evolution equation for $\psi_{\hi}$ from \eqref{AH:eq-psi-hi-e1}-\eqref{AH:eq-psi-hi-e3} and the decomposition of $A$ from \eqref{AH:eq-decomposition-A}, it then follows that
\begin{align}
\big( \partial_t - \partial^j \partial_j  + 1\big) \psi_{\hi}
&= 2 \icomplex \, \big(\linear[\hi][r][j] + \Slin_{\hi}^j\big) \partial_j \philinear[\lo] 
+ 2 \icomplex \, \big(\linear[\hi][r][j] + \Slin_{\hi}^j\big) \partial_j \psitilde 
\label{high:eq-para-controlled-p2} \\ 
&- \big(\Aquadraticnl[\hi] + |\Slin|^2 - |\Slin_{\lo}|^2\big) \philinear[\lo] - \big(\Aquadraticnl[\hi] + |\Slin|^2 - |\Slin_{\lo}|^2\big) \psitilde 
\label{high:eq-para-controlled-p3} \\ 
&- 2 \big( \linear[\hi][r][j] S_j + \linear[\lo][r][j] \Slin_{\hi, j}\big) \philinear[\lo] - 2\big( \linear[\hi][r][j] S_j + \linear[\lo][r][j] \Slin_{\hi, j}\big) \psitilde \label{high:eq-paracontrolled-linear-Slin}\\
&- 2 \, \big(\linear[\hi][r][j] + \Slin_{\hi}^j\big) \big(B+Z \big)_j \philinear[\lo]
- 2 \, \big(\linear[\hi][r][j] + \Slin_{\hi}^j\big) \big( B+Z \big)_j  \psitilde 
\label{high:eq-para-controlled-p4} \\ 
&+ 2\icomplex \, \linear[][r][j] \partial_j \philinear[\hi] + 2\icomplex \, \linear[][r][j] \partial_j\varphi_{\hi}  
+ 2\icomplex \, \big(S + B + Z \big)^j \partial_j \big(\philinear[\hi] + \varphi_{\hi}\big) 
\label{high:eq-para-controlled-p5} \\
&- \Aquadraticnl(\philinear[\hi] + \varphi_{\hi})
- 2 \big(S + B +Z \big)_j \linear[][r][j] \big(\philinear[\hi] + \varphi_{\hi}\big)
- \big|S+ B +Z \big|^2 \, \big(\philinear[\hi] + \varphi_{\hi}\big)
\label{high:eq-para-controlled-p6} \\
&+ \philinear[\hi] + \varphi_{\hi} \label{high:eq-para-controlled-philinear-varphi}\\
&+ \mrm{poly}, \label{high:eq-para-controlled-poly}
\end{align}
where $\mrm{poly}$ is the difference in poynomial terms given by \eqref{AH:eq-psi-hi-e3}. We now control the contributions of \eqref{high:eq-para-controlled-p2}-\eqref{high:eq-para-controlled-poly} to the Duhamel integral separately. \\

\emph{Contribution of \eqref{high:eq-para-controlled-p2}:} Due to Definition \ref{high:def-stochastic}, the contribution of the $2\icomplex \linear[\hi][r][j] \ptl_j \philinear[\lo]$ part of the first summand in \eqref{high:eq-para-controlled-p2} equals $2\icomplex \, \mixedquadratic[\hi,\lo]$, which is contained in the argument on the left-hand side of \eqref{high:eq-para-controlled-e2}. Next, using the null-form estimate (Lemma \ref{prelim:lem-null}) and \eqref{eq:Slin-estimate}, we have that
\begin{equs}
\big\|\Slin_{\hi}^j \ptl_j \philinear[\lo]\big\|_{L_t^{p} \Cs_x^{-1+\eta_2}} \lesssim \|\Slin_{\hi}\|_{L_t^p \Cs_x^\eta} \|\philinear[\lo]\|_{L_t^\infty \Cs_x^{2\eta_2}} \lesssim L^{\kappa - \eta} \times L^{\kappa + 2\eta_2} = L^{2\kappa + 2\eta_2 - \eta},
\end{equs}
which is acceptable. For the second summand in \eqref{high:eq-para-controlled-p2}, using the null-form estimate (Lemma \ref{prelim:lem-null}), \eqref{high:eq-probabilistic-Alinear}, and \eqref{high:eq-para-controlled-p1}, we have that
\begin{equs}
\big\| \linear[\hi][r][j] \partial_j \tilde{\psi} \big\|_{L_t^p \Bc_x^{-1+\eta_2, p}} &\lesssim \|\linear[\hi]\|_{L_t^\infty \Cs_x^{-(\eta-2\eta_2)}} \|\tilde{\psi} \|_{L_t^p \Bc_x^{\eta, p}} \lesssim L^{\kappa +2\eta_2 - \eta} \times L^{2\eta_2} = L^{\kappa + 4\eta_2 - \eta},
\end{equs}
which is acceptable, and by the null-form estimate (Lemma \ref{prelim:lem-null}), \eqref{eq:varphi-hi-estimate}, and \eqref{high:eq-para-controlled-p1}, we have that
\begin{equs}
\big\|\Slin_{\hi}^j \ptl_j \psitilde\big\|_{L_t^{\frac{p}{2}} \Bc_x^{-1 +\eta, p}} \lesssim \|\Slin_{\hi}\|_{L_t^{p} \Cs_x^\eta} \|\psitilde\|_{L_t^p \Bc_x^{\eta, p}} \lesssim L^{\kappa - \eta} \times L^{2\eta_2} = L^{\kappa + 2\eta_2 -\eta},
\end{equs}
which is acceptable.\\

\emph{Contribution of \eqref{high:eq-para-controlled-p3}:} Using \eqref{high:eq-probabilistic-quadratic-linear}, we have that 
\begin{equs}
\big\| t^\kappa \Aquadraticnl[\hi] \, \philinear[\lo] \big\|_{C_t^0 \Cs_x^{-\eta}} \leq L^{3\kappa-\eta},
\end{equs}
which is acceptable. Using the product estimate (Lemma \ref{prelim:lem-para-besov}), \eqref{high:eq-probabilistic-quadratic}, and \eqref{high:eq-para-controlled-p1}, we have that 
\begin{equs}
\big\| t^\kappa \Aquadraticnl[\hi] \, \tilde{\psi}\big\|_{L_t^p \Bc_x^{-\eta,p}} &\lesssim \|t^\kappa \Aquadraticnl[\hi] \big\|_{L_t^\infty \Cs_x^{-(\eta-\eta_2)}} \|\tilde{\psi}\|_{L_t^p \Bc_x^{\eta, p}} \lesssim L^{\kappa +\eta_2 -\eta} \times L^{2\eta_2}  = L^{\kappa + 3\eta_2 - \eta},
\end{equs}
which is acceptable. The terms involving $|\Slin|^2 - |\Slin_{\lo}|^2$ may be bounded using \eqref{eq:Slin-estimate} and product estimates; we omit the arguments here.\\

\emph{Contribution of \eqref{high:eq-paracontrolled-linear-Slin}:}
All terms may be estimated via \eqref{high:eq-probabilistic-Alinear}, \eqref{high:eq-probabilistic-linear-linear}, \eqref{high:eq-para-controlled-p1}, \eqref{eq:Slin-estimate}, and product estimates.
\\

\emph{Contribution of \eqref{high:eq-para-controlled-p4}:} Any term involving $\Slin_{\hi}$ may be estimated using \eqref{eq:Slin-estimate} and product estimates, thus we omit the argument for those terms. Using product estimates (Lemma \ref{prelim:lem-para-besov}), we have that 
\begin{equation}\label{high:eq-para-controlled-p7}
\begin{aligned}
&\big\| \linear[\hi][r][j] \big( B+Z \big)_j \philinear[\lo]
+ \linear[\hi][r][j] \big( B + Z \big)_j \psitilde \big\|_{L_t^p \Bc_x^{-\eta, p}} \\
\lesssim&\,  \big\| \linear[\hi] \philinear[\lo] \big\|_{C_t^0 \Cs_x^{\eta_3-\eta}}
\big\| B +Z \big\|_{C_t^0 \Cs_x^\eta}
+ \big\|  \linear[\hi] \big\|_{C_t^0 \Cs_x^{\eta_3-\eta}}
\big\| B +Z \big\|_{C_t^0 \Cs_x^\eta}
\big\| \psitilde \big\|_{L_t^p \Bc_x^{\eta, p}}.
\end{aligned}
\end{equation}
The factors in \eqref{high:eq-para-controlled-p7} can be controlled using the continuity hypothesis (Hypothesis \ref{AH:hypothesis-continuity}),  \eqref{high:eq-probabilistic-philinear},  \eqref{high:eq-probabilistic-linear-linear}, and \eqref{high:eq-para-controlled-p1}, which yield that
\begin{equs}
\eqref{high:eq-para-controlled-p7} 
\lesssim L^{2\kappa+\eta_3-\eta} \times L^{\eta_2}
+ L^{\kappa+\eta_3-\eta} \times L^{\eta_2} \times L^{2\eta_2} \lesssim L^{\eta_1-\eta},
\end{equs}
which is acceptable. \\

\emph{Contribution of \eqref{high:eq-para-controlled-p5}:} Due to Definition \ref{high:def-stochastic}, the contribution of the first summand in \eqref{high:eq-para-controlled-p5} equals $2\icomplex \, \big( \mixedquadratic[\lo,\hi]+\mixedquadratic[\hi,\hi]\big)$, which is contained in the argument on the left-hand side of \eqref{high:eq-para-controlled-p2}. To bound the second summand, we use the null-form estimate (Lemma \ref{prelim:lem-null}) and \eqref{eq:varphi-hi-estimate} to obtain
\begin{equs}
\big\|\linear[][r][j] \ptl_j \varphi_{\hi}\big\|_{L_t^p \Cs_x^{-1+\eta_2}} \lesssim \big\|\linear[][r][j]\big\|_{L_t^\infty \Cs_x^{-\kappa}} \big\|\varphi_{\hi}\big\|_{L_t^p\Cs_x^{\eta}} \lesssim L^\kappa \times L^{\kappa - \eta} = L^{2\kappa - \eta},
\end{equs}
which is acceptable. For the $(S + B + Z)^j \ptl_j \philinear[\hi]$ piece of the third summand in \eqref{high:eq-para-controlled-p5}, we combine it with the paraproduct term on the left-hand side of \eqref{high:eq-para-controlled-e1}, so that it suffices for us to bound the difference
\begin{equs}
\bigg\| \mDuh\big((\Slin + B + Z)^j \ptl_j \philinear[\hi]\big) - (\Slin + B + Z)^j \parall \ptl_j \mDuh\big(\philinear[\hi]\big)\bigg\|_{C_t^0 \Cs_x^{1+\eta_3}} \lesssim L^{\eta_1 - \eta}.
\end{equs}
We split this into two terms:
\begin{equs}\label{high:eq-paracontrolled-p5-commutator}
\hspace{-5mm}\mDuh\big((\Slin +B + Z)^j \paragtrsim \ptl_j \philinear[\hi]\big) + \Big(\mDuh\big((\Slin + B + Z)^j \parall \ptl_j \philinear[\hi]\big) - (\Slin + B + Z)^j \parall \ptl_j \mDuh\big(\philinear[\hi]\big)\Big).
\end{equs}
For the first term, we first use the Coulomb gauge condition to write 
\begin{equs}
(\Slin + B + Z)^j \paragtrsim \ptl_j \philinear[\hi] = \ptl_j \big((\Slin + B + Z)^j \paragtrsim \philinear[\hi]\big),
\end{equs}
and then bound by Schauder and then product estimates
\begin{equs}
\Big\| \mDuh\Big( \ptl_j \big((\Slin + B + Z)^j \paragtrsim \philinear[\hi]\big)\Big)\Big\|_{L_t^p \Cs_x^{1+\eta_3}} &\lesssim \Big\| (\Slin + B + Z)^j \paragtrsim \philinear[\hi]\Big\|_{L_t^p\Cs_x^{\eta_2}} \\
&\lesssim \|\Slin + B + Z\|_{L_t^p \Cs_x^\eta} \|\philinear[\hi]\|_{\Cs_x^{-(\eta - \eta_2)}}  \lesssim L^{\eta_1 - \eta},
\end{equs}
which is acceptable. For the second term in \eqref{high:eq-paracontrolled-p5-commutator}, we apply the commutator estimate \cite[Lemma B.4]{BC23} and the continuity and high-frequency probabilistic hypotheses to obtain
\begin{equation}
\begin{aligned}\label{eq:commutator-estimate-application}
\bigg\|t^{\eta}\Big(\mDuh&\big((\Slin + B + Z)^j \parall \ptl_j \philinear[\hi]\big) - (\Slin + B + Z)^j \parall \ptl_j \mDuh\big(\philinear[\hi]\big)\Big)\bigg\|_{L_t^\infty \Cs_x^{1+\eta_3}} \\
&\lesssim \big(\|\Slin(0)\|_{\Cs_x^{-\kappa}} + \|B + Z\|_{C_t^0 \Cs_x^\eta \cap C_t^{\eta/2} \Cs_x^0}\big) \big\|\ptl_j \philinear[\hi]\big\|_{C_t^0 \Cs_x^{-(\eta - 2\eta_2) - 1}} \lesssim L^{\eta_2} \times L^{\kappa + 2\eta_2 - \eta} \lesssim L^{\eta_1 - \eta}.
\end{aligned}
\end{equation}
Since $p < \frac{1}{\eta}$, we have for a general function $f$ that
\begin{equs}
\|f\|_{L_t^p \Cs_x^{1+\eta_3}} \leq \| t^{-\eta}\|_{L_t^p} \|t^\eta f\|_{L_t^\infty \Cs_x^{1+\eta_3}} \lesssim  \|t^\eta f\|_{L_t^\infty \Cs_x^{1+\eta_3}},
\end{equs}
and thus the estimate in \eqref{eq:commutator-estimate-application} is acceptable. Finally, for the second piece of the third summand (involving $\ptl_j \varphi_{\hi}$), we bound by \eqref{eq:Slin-estimate}, \eqref{AH:eq-L}, \eqref{AH:eq-continuity-Z}, and \eqref{eq:varphi-hi-estimate},
\begin{equs}
\big\|(S + B + Z)^j \ptl_j \varphi_{\hi}\big\|_{L_t^{\frac{p}{2}} \Cs_x^{-1+\eta}} = \big\|\ptl_j \big((S + B + Z)^j \varphi_{\hi}\big)\big\|_{L_t^{\frac{p}{2}} \Cs_x^{-1+\eta}} \lesssim \big\|S + B + Z\big\|_{L_t^p \Cs_x^\eta} \big\|\varphi_{\hi}\big\|_{L_t^p \Cs_x^{\eta}} \lesssim L^{\eta_1 - \eta},
\end{equs}
which is acceptable. \\

\emph{Contributions of \eqref{high:eq-para-controlled-p6} and \eqref{high:eq-para-controlled-philinear-varphi}:} 
The term \eqref{high:eq-para-controlled-p6} may be estimated by product estimates similar to many estimates before, and thus we omit the details. The term \eqref{high:eq-para-controlled-philinear-varphi} follows directly from \eqref{high:eq-probabilistic-philinear} and \eqref{eq:varphi-hi-estimate}.\\



\emph{Contribution of \eqref{high:eq-para-controlled-poly}:}
For brevity, let $k := \frac{q+1}{2}$ in the following.
Recalling that $\mrm{poly}$ is the difference in polynomial terms given by \eqref{AH:eq-psi-hi-e3}, we first apply the Hermite polynomial expansion (Lemma \ref{lemma:complex-hermite-polynomial-expansion}) to obtain
\begin{equs}
\mrm{poly} &= \biglcol \, \big| \philinear + \tilde{\psi} + \varphi_{\hi} \big|^{q-1} (\philinear + \tilde{\psi} + \varphi_{\hi})\big| \, \bigrcol - \biglcol\, \big|\philinear[\lo] + \tilde{\psi}\big|^{q-1} (\philinear[\lo] + \tilde{\psi})\, \bigrcol \\
&= \sum_{j_1=0}^{k} \sum_{j_2 =0}^{k-1} \binom{k}{j_1} \binom{k}{j_2}  \mrm{poly}(j_1, j_2),
\end{equs}
where
\begin{equs}
\mrm{poly}(j_1, j_2) &:= \biglcol\, \philinear[][r][j_1]\ovl{\philinear[][r][j_2]}\bigrcol (\psitilde + \varphi_{\hi})^{k-j_1}  \ovl{(\psitilde + \varphi_{\hi})}^{k -1 - j_2} - \biglcol\,\philinear[\lo][r][j_1] \ovl{\philinear[\lo][r][j_2]}\bigrcol \psitilde^{k - j_1} \ovl{\psitilde}^{k - 1 - j_2} \\
&= \big(\biglcol\, \philinear[][r][j_1]\ovl{\philinear[][r][j_2]}\bigrcol\big)_{\hi} (\psitilde + \varphi_{\hi})^{k-j_1}  \ovl{(\psitilde + \varphi_{\hi})}^{k -1 - j_2} \\
&\quad + \biglcol\,\philinear[\lo][r][j_1] \ovl{\philinear[\lo][r][j_2]}\bigrcol\Big((\psitilde + \varphi_{\hi})^{k-j_1}  \ovl{(\psitilde + \varphi_{\hi})}^{k -1 - j_2} - \psitilde^{k - j_1} \ovl{\psitilde}^{k - 1 - j_2}\Big).
\end{equs}
By \eqref{high:eq-probabilistic-polynomial}, \eqref{high:eq-para-controlled-p1}, \eqref{eq:varphi-hi-estimate}, and product estimates, we have that
\begin{equs}
\Big\| t^\kappa \big(\biglcol\, \philinear[][r][j_1]\ovl{\philinear[][r][j_2]}\bigrcol\big)_{\hi} (\psitilde + \varphi_{\hi})^{k-j_1}  \ovl{(\psitilde + \varphi_{\hi})}^{k -1 - j_2}\Big\|_{L_t^{\frac{p}{q}} \Cs_x^{-(\eta - \eta_2)}} \lesssim L^{q\kappa + \eta_2 - \eta} \|\psitilde + \varphi_{\hi}\|_{L_t^p \Cs_x^{\eta}}^q \lesssim L^{q\kappa + \eta_2 + 2q\eta_2 - \eta },
\end{equs}
which suffices by classical Schauder estimates and our assumption that $p \geq 100q$.
For the second term, note that when we expand out the difference, at least one factor of $\varphi_{\hi}$ must appear, and then using \eqref{high:eq-probabilistic-polynomial} combined with
the estimate \eqref{eq:varphi-hi-estimate}, we obtain
\begin{equs}
\Big\| t^\kappa\biglcol\,\philinear[\lo][r][j_1] \ovl{\philinear[\lo][r][j_2]}\bigrcol\Big((\psitilde + \varphi_{\hi})^{k-j_1}  \ovl{(\psitilde + \varphi_{\hi})}^{k -1 - j_2} - \psitilde^{k - j_1} \ovl{\psitilde}^{k - 1 - j_2}\Big)\Big\|_{L_t^{\frac{p}{q}} \Bc_x^{0, p/q}}  \lesssim L^{q\kappa + 2q\eta_2 } L^{\kappa - \eta} \lesssim L^{\eta_1 - \eta},
\end{equs}
which again suffices by classical Schauder estimates.
\end{proof}

\begin{corollary}[\protect{Bounds on $\psi_{\hi}$}]\label{high:cor-psi-hi}
Let the probabilistic and continuity hypothesis (Hypothesis \ref{AH:hypothesis-probabilistic} and \ref{AH:hypothesis-continuity}) be satisfied. Then, it holds that
\begin{equation}\label{high:eq-psi-hi-estimate}
\big\| \psi_{\hi} \big\|_{C_t^0 \Cs_x^\eta([t_0,t_0+\tau])}\leq \frac{1}{2} L^{\eta_1-\eta}.
\end{equation}
\end{corollary}
\begin{remark}
We remark that \eqref{high:eq-psi-hi-estimate} is only better than \eqref{AH:eq-continuity-psi-hi} from Hypothesis \ref{AH:hypothesis-continuity} by a factor of $\frac{1}{2}$. However, this is sufficient to close the continuity argument in the proof of Lemma \ref{decay:lem-short-bounds}. 
\end{remark}
\begin{proof}
By arguing as in the proof of Lemma \ref{high:lem-para-controlled}, we can obtain an estimate like in \eqref{high:eq-para-controlled-e1}, but with $L_t^{p/10} \Bc_x^{1+\eta_3, p}$ replaced by $C_t^0 \Cs_x^\eta$, and with a slightly tighter upper bound (because we can use that $L^{\eta_3} \gg 1$):
\begin{equs}
\Big\| \psi_{\hi} - 2 \icomplex\, \mixedquadratic[\hi] - 2\icomplex \big( S + B+ Z\big)^j \parall \partial_j \mDuh \big[ \philinear[\hi] \big]  - 2\icomplex \Slin_{\hi} \parall \ptl_j \mDuh\big[ \philinear[\lo]\big] \Big\|_{C_t^0 \Cs_x^\eta} \leq \frac{1}{10} L^{\eta_1 - \eta}.
\end{equs}
Thus, it remains to estimate the three terms on the left-hand side above. By \eqref{high:eq-probabilistic-mixed-quadratic}, we may estimate
\begin{equs}
\Big\|2\icomplex \mixedquadratic[\hi]\Big\|_{C_t^0 \Cs_x^\eta} \leq L^{-\frac{1}{4}} \leq \frac{1}{10} L^{\eta_1 - \eta}.
\end{equs}
Next, we estimate by \eqref{AH:eq-L}, \eqref{AH:eq-L-nreg}, the continuity hypothesis \eqref{AH:hypothesis-continuity}, and \eqref{high:eq-probabilistic-philinear},
\begin{equs}
\bigg\|2\icomplex (S + B + Z)^j \parall \ptl_j \mDuh\big[\philinear[\hi]\big]\bigg\|_{C_t^0 \Cs_x^{\eta}} \lesssim \|S + B + Z\|_{C_t^0 \Cs_x^{-\kappa}} \big\|\ptl \mDuh \big[\philinear[\hi]\big] \big\|_{C_t^0 \Cs_x^{\frac{1}{2}}} \lesssim L^{\eta_2} L^{\kappa - \frac{1}{2}} \leq \frac{1}{10} L^{\eta_1 - \eta}. 
\end{equs}
Finally, we estimate by \eqref{AH:eq-L-nreg} and \eqref{high:eq-probabilistic-philinear},
\begin{equs}
\bigg\|2\icomplex \Slin_{\hi} \parall \ptl_j \mDuh\big[\philinear[\lo]\big]\bigg\|_{C_t^0 \Cs_x^\eta} \lesssim \|\Slin_{\hi}\|_{C_t^0 \Cs_x^{-\frac{1}{4}}} \big\|\ptl_j \mDuh\big[\philinear[\lo]\big]\big\|_{C_t^0 \Cs_x^{\frac{1}{2}}} \lesssim \|P_{> L} \Slin(0)\|_{\Cs_x^{-\frac{1}{4}}} L^{\kappa} \lesssim L^{2\kappa - \frac{1}{4}} \leq \frac{1}{10} L^{\eta_1 - \eta}.
\end{equs}
The desired result now follows by combining the estimates.
\end{proof}

\subsection{\protect{Proof of Lemma \ref{AH:lem-derivative-high}}}
We want to show that
\begin{equs}
\bigg\| \leray \Im\mDuh\Big[ \ovl{\phi} \covd_A \phi - \ovl{\phi}_{\lo} \covd_{A_\lo} \phi_{\lo}\Big] \bigg\|_{C_t^0 \Cs_x^{2\eta} \cap C_t^\eta \Cs_x^0} \lesssim L^{10\eta_1 - \eta}.
\end{equs}
Take $p \leq \frac{r}{\eta r + (1-2\eta)}$ sufficiently large. From Lemma \ref{high:lem-para-controlled}, we obtain the decomposition
\begin{align}
\phi &= \philinear[] 
+ 2 \icomplex \, \mixedquadratic 
+ 2 \icomplex \,  (\Slin + B+Z)^j \parall \partial_j \mDuh \big[ \, \philinear \big] 
+ e^{t(\Delta-1)} \psi(0) + \eta_{\lo} + \eta_{\hi}, 
\label{high:eq-proof-1} \\
\phi_{\lo} &= \philinear[\lo] 
+ 2 \icomplex \, \mixedquadratic[\lo]  
+ 2 \icomplex \,  (\Slin_{\lo} + B+Z)^j \parall \partial_j \mDuh \big[ \, \philinear[\lo] \big] 
+ e^{t(\Delta-1)} \psi_{\lo}(0) + \eta_{\lo},\label{high:eq-proof-2}
\end{align}
where the nonlinear remainders $\eta_{\lo}$ and $\eta_{\hi}$ satisfy
\begin{equation}\label{high:eq-proof-3}
\big\| \eta_{\lo} \big\|_{L_t^{p/10} \Bc_x^{1+\eta+\eta_3,p}}
\lesssim L^{\eta_1} \qquad \text{and} \qquad
\big\| \eta_{\hi} \big\|_{L_t^{p/10} \Bc_x^{1+\eta_3,p}}
\lesssim L^{\eta_1-\eta}.
\end{equation}
To simplify the notation, we also define
\begin{equation}\label{high:eq-proof-4}
\chi := \phi - \philinear[], \qquad 
\chi_{\lo} := \phi_{\lo} - \philinear[\lo], \qquad \text{and} \qquad 
\chi_{\hi} := \chi - \chi_{\lo}. 
\end{equation}
By \eqref{high:eq-proof-3}, the embedding $\Bc_x^{1+\eta_3, p} \hookrightarrow \Cs_x^{\frac{1}{2}}$, the high-frequency probabilistic hypothesis \eqref{high:hypothesis-probabilistic}, the continuity hypothesis \eqref{AH:hypothesis-continuity}, the assumption $\psi(0) = \nregphi_0 + \pregphi_0$ with $\nregphi_0, \pregphi_0$ satisfying \eqref{AH:eq-L-nreg}, \eqref{AH:eq-L}, and the heat flow smoothing estimate (Lemma \ref{lemma:heat-flow-smoothing}), we have that
\begin{equs}
\|\chi - e^{t (\Delta-1)} \psi(0)\|_{L_t^{p/10} \Cs_x^{\frac{1}{2}}} \lesssim L^{\eta_1}, ~~ \|e^{t (\Delta-1)} \psi(0)\|_{L_t^{p/10} \Cs_x^{\eta + \eta_3}} \lesssim L^{\eta_3}, \label{eq:chi-estimates}
\end{equs}
as well as the higher regularity estimates (note that $e^{t(\Delta-1)} (\psi(0) - \psi_{\lo}(0)) = \varphi_{\hi}$ by definition)
\begin{align}\label{eq:chi-estimates-high-regularity}
\|\chi_{\lo} - e^{t (\Delta-1)} \psi_{\lo}(0)\|_{L_t^{p/10} \Cs_x^{\frac{1}{2}}} &\lesssim L^{\eta_2} , ~~ \|t^{\frac{1}{4} + \frac{1}{r}} e^{t (\Delta-1)} \psi_{\lo}(0)\|_{L_t^\infty \Cs_x^{\frac{1}{2}}} \lesssim L^{\eta_3} , \\
\|\chi_{\hi} - \varphi_{\hi}\|_{L_t^{p/10} \Cs_x^{\frac{3}{4}}} &\lesssim L^{\eta_1 - \eta}, ~~ \|t^{\frac{3+\eta}{8}}\varphi_{\hi}\|_{L_t^\infty \Cs_x^{\frac{3}{4}}} \lesssim L^{\kappa - \eta}.
\end{align}
We also have by \eqref{AH:eq-L}, \eqref{AH:eq-L-nreg}, \eqref{AH:eq-continuity-Z}, and the heat flow smoothing estimate (Lemma \ref{lemma:heat-flow-smoothing}) that
\begin{equs}\label{eq:S-B-Z-estimate}
\|S + B + Z\|_{L_t^p \Cs_x^{\eta}} \lesssim L^{\eta_2}.    
\end{equs}
Using the decomposition of $\phi$ from \eqref{high:eq-proof-1}, we then obtain a decomposition of $\leray \Im \big(\overline{\phi} \covd_A \phi\big)$, which can be written as 
\begin{align}
    \leray &\Im \big(\overline{\phi} \covd_A \phi\big) 
    \\
    &= \leray \Im \big( \overline{\phi} \covd \phi\big) + \leray \big( A | \phi|^2 \big)\notag \\  
    &= \leray \Im \Big(\overline{\, \philinear} \covd \philinear \Big) \label{high:eq-decomp-1}\\
    &+ \leray \bigg(4  \Im \Big( \icomplex\overline{\, \philinear} \covd \mixedquadratic \Big)
    + \linear \big|\,\philinear\big|^2\bigg) \label{high:eq-decomp-2} \\ 
    &+ \leray \bigg( 4 \Im\Big(\icomplex \ovl{\philinear} \covd \ptl_j \mDuh\big[\,\philinear\big]\Big) (S + B + Z)^j + (S + B + Z)\big|\,\philinear\big|^2\bigg) \label{high:eq-decomp-3.1}\\
    &+ 4\leray \Im \bigg(-\icomplex \ovl{\philinear} \Big((S + B+ Z)^j \paragtrsim \covd \ptl_j \mDuh\big[\,\philinear\big]\Big) + \icomplex \ovl{\philinear} \Big( \covd(S + B+ Z)^j \parall \ptl_j \mDuh\big[\, \philinear\big]\Big)\bigg) \label{high:eq-decomp-3.2}\\
    &+ 2 \leray \Im \Big(  \overline{\, \philinear}  \covd  \big( e^{t(\Delta-1)} \psi(0) + \eta_{\lo} + \eta_{\hi} \big) \Big) \label{high:eq-decomp-4}\\
    &+ \leray \Im \Big( \overline{\chi} \covd \chi \Big) \label{high:eq-decomp-5} \\ 
    &+ \leray \Big( \big(\linear + S + B + Z\big) \big(2 \mrm{Re}\big((\phi - \philinear) \ovl{\philinear}\big) + |\phi-\philinear|^2 \big) \Big). 
    \label{high:eq-decomp-6}
\end{align}
We remark that the second term in \eqref{high:eq-decomp-3.1} may appear to be ill-defined, but one should not view this as a separate term. Instead, we will see that one needs to combine the two terms in \eqref{high:eq-decomp-3.1} and utilize the estimates \eqref{eq:philinear-laplacian-Duh-philinear-phiquadraticnl} and \eqref{eq:philinear-laplacian-Duh-philiner}.

By using the decomposition of $\phi_{\lo}$ from \eqref{high:eq-proof-2}, we also obtain a similar decomposition of $\leray \Im ( \overline{\phi}_{\lo} \covd_{A_{\lo}} \phi_{\lo})$. We now separately address the contributions of \eqref{high:eq-decomp-1}-\eqref{high:eq-decomp-6} to the desired estimate \eqref{AH:eq-derivative-high}. \\

\emph{Contribution of \eqref{high:eq-decomp-1}:} It suffices to prove 
\begin{equs}
\Big\|\Aquadratic - \Aquadratic[\lo]\Big\|_{C_t^0 \Cs_x^{\frac{1}{2}} \cap C_t^{\frac{1}{8}} \Cs_x^{\frac{1}{4}}} \leq L^{-\frac{1}{4}},
\end{equs} which coincides with the probabilistic estimate \eqref{high:eq-probabilistic-Aquadratic}. \\ 

\emph{Contribution of \eqref{high:eq-decomp-2}:}
A sufficient estimate follows directly from \eqref{eq:philinear-mixed-quadratic-linear-phiquadraticnl-combined} and standard Schauder estimates.
\\ 

\emph{Contributions of \eqref{high:eq-decomp-3.1} and \eqref{high:eq-decomp-3.2}:} 
First, to interpret \eqref{high:eq-decomp-3.1}, note that for $i \in [2]$, the $i$th coordinate is given by (ignoring the Leray projection) 
\begin{equs}
\Im\Big(  \icomplex \big(4 \ovl{\philinear} \ptl_i^2 \mDuh\big[\, \philinear\big] + \big|\,\philinear\big|^2\big) \Big) (S + B + Z)_i + 4\Im\big( \icomplex \ovl{\philinear} \ptl_i \ptl_j \mDuh\big[\, \philinear\big]\big) (S + B + Z)^j \ind(j \neq i). 
\end{equs}
Upon taking differences with the low-frequency analog (i.e. with $\philinear[\lo]$ replacing $\philinear$ everywhere), we obtain a sufficient estimate by applying \eqref{eq:philinear-laplacian-Duh-philinear-phiquadraticnl}, \eqref{eq:philinear-laplacian-Duh-philiner}, \eqref{eq:S-B-Z-estimate}, and product estimates. The contribution of \eqref{high:eq-decomp-3.2} may be handled similarly; we omit the details.\\

\emph{Contribution of \eqref{high:eq-decomp-4}:} Due to the Duhamel integral estimate (Lemma \ref{prelim:lem-Duhamel-weighted}), it suffices to prove 
\begin{equs}
\Big\| \overline{\, \philinear[\hi]} \covd e^{t(\Delta-1)} \psi_{\lo}(0) \Big\|_{L_t^{2(1-2\eta)} \Bc_x^{-\eta,r}} 
&\lesssim L^{\kappa+\eta_1-\eta}, ~~ \Big\| \overline{\, \philinear} \covd \varphi_{\hi} \Big\|_{L_t^{2(1-2\eta)} \Bc_x^{-\eta, r}} \lesssim L^{\kappa + \eta_1 - \eta},  \\
\Big\| \overline{\, \philinear[\hi]} \covd  \eta_{\lo} + \overline{\, \philinear}  \covd \eta_{\hi} \Big\|_{L_t^{p/10} \Bc_x^{-\eta,p}}
&\lesssim L^{\kappa+\eta_1-\eta}.
\end{equs}
Recalling that $\psi_{\lo}(0) = P_{\leq L}\nregphi_0 + \pregphi_0$ \eqref{AH:eq-psi-lo-initial}, using our product estimate (Lemma \ref{prelim:lem-para-besov}), the heat flow smoothing estimate (Lemma \ref{lemma:heat-flow-smoothing}), and \eqref{high:eq-probabilistic-philinear}, the $e^{t(\Delta-1)} \psi_{\lo}(0)$-term can be estimated by 
\begin{align*}
&\, \Big\| \overline{\, \philinear[\hi]} \covd e^{t(\Delta-1)} \psi_{\lo}(0) \Big\|_{L_t^{2(1-2\eta)} \Bc_x^{-\eta,r}} 
\lesssim \big\| \, \philinear[\hi] \big\|_{C_t^0 \Cs_x^{-\eta}} 
\big\|  \covd e^{t(\Delta-1)} \psi_{\lo}(0) \big\|_{L_t^{2(1-2\eta)} \Bc_x^{\eta+\eta_3,r}} \\ 
\lesssim&\,  L^{\kappa-\eta} \Big( \big\| t^{-\frac{1+\eta+\eta_3}{2}} \big\|_{L_t^{2(1-2\eta)}} \big\| \pregphi_0 \big\|_{\Bc_x^{0,r}} + \big\|t^{-\frac{1+\eta+\eta_3+\kappa}{2}}\big\|_{L_t^{2(1-2\eta)}} \|\nregphi_0\|_{\Cs_x^{-\kappa}}\Big)
\\
\lesssim&\, L^{\kappa-\eta} (L^{\eta_3} + L^{\kappa}),
\end{align*}
where we used that $\| \pregphi_0 \|_{\Bc_x^{0,r}}\lesssim \| \pregphi_0 \|_{L_x^r} \lesssim L^{\eta_3}$ and $\|\nregphi_0\|_{\Cs_x^{-\kappa}} \leq L^\kappa$ \eqref{AH:eq-L-nreg}. This yields an acceptable contribution. The term $\overline{\, \philinear} \covd \varphi_{\hi}$ may be similarly estimated.
Using our product estimate (Lemma \ref{prelim:lem-para-besov}), \eqref{high:eq-probabilistic-philinear}, and \eqref{high:eq-proof-3}, the $\eta_{\lo}$-term can be estimated by
\begin{equation*}
\Big\| \overline{\, \philinear[\hi]} \covd  \eta_{\lo} \Big\|_{L_t^{p/10} \Bc_x^{-\eta,p}} 
\lesssim \big\| \philinear[\hi] \big\|_{C_t^0 \Cs_x^{-\eta}} \big\| \covd \eta_{\lo} \big\|_{L_t^{p/10} \Bc_x^{\eta+\eta_3,p}}
\lesssim L^{\kappa-\eta} \times L^{\eta_1} = L^{\kappa+\eta_1-\eta},
\end{equation*}
which is acceptable. Similarly, the $\eta_{\hi}$-term can be estimated by 
\begin{equation*}
\Big\| \overline{\, \philinear[]} \covd  \eta_{\hi} \Big\|_{L_t^{p/10} \Bc_x^{-\eta,p}} 
\lesssim \big\| \philinear[] \big\|_{C_t^0 \Cs_x^{-\kappa}} \big\| \covd \eta_{\hi} \big\|_{L_t^{p/10} \Bc_x^{\eta_3,p}}
\lesssim L^{\kappa} \times L^{\eta_1-\eta} = L^{\kappa+\eta_1-\eta},
\end{equation*}
which is also acceptable.\\ 

\emph{Contribution of \eqref{high:eq-decomp-5}:} 
The bound 
\begin{equs}
\Big\|\mDuh\Big(\ovl{\chi} \covd \chi - \ovl{\chi_{\lo}} \covd {\chi_{\lo}} \Big)\Big\|_{C_t^0 \Cs_x^\eta \cap C_t^{\eta/2} \Cs_x^0} \lesssim L^{10 \eta_1 - \eta} 
\end{equs}
easily follows from the high regularity estimates \eqref{eq:chi-estimates-high-regularity} and product estimates. The details are omitted.
\\ 

\emph{Contribution of \eqref{high:eq-decomp-6}:} Upon taking differences with the low-frequency analog, we note that at least one factor of $\linear[\hi]$, $\Slin_{\hi}$, $\chi_{\hi}$, or $\philinear[\hi]$ must appear. Combining this with product estimates as we have done before, we obtain a sufficient bound. The details are omitted. \\

This completes the proof of Lemma \ref{AH:lem-derivative-high}.


\end{appendix}

\bibliography{AH_Library}
\bibliographystyle{myalpha}

\end{document}